\documentclass[reqno,11pt]{amsart}
\usepackage{amscd,amssymb,verbatim,array}
\usepackage{hyperref}

\usepackage{mathrsfs}

\numberwithin{equation}{section} \theoremstyle{plain}

\newcommand{\Complex}{\mathbb C}
\newcommand{\Real}{\mathbb R}

\newcommand{\ddbar}{\overline\partial}
\newcommand{\pr}{\partial}
\newcommand{\ol}{\overline}
\newcommand{\Td}{\widetilde}
\newcommand{\norm}[1]{\left\Vert#1\right\Vert}
\newcommand{\set}[1]{\left\{#1\right\}}
\newcommand{\To}{\rightarrow}
\DeclareMathOperator{\Ker}{Ker}

\newcommand{\cali}[1]{\mathscr{#1}}

\newcommand{\sumprime}{{\sum}^\prime}

\newtheorem{theorem}{Theorem}[section]
\newtheorem{lemma}[theorem]{Lemma}
\newtheorem{proposition}[theorem]{Proposition}
\newtheorem{corollary}[theorem]{Corollary}
\newtheorem{definition}[theorem]{Definition}
\newtheorem{ass}[theorem]{Assumption}
\newtheorem{remark}[theorem]{Remark}

\theoremstyle{definition}

\theoremstyle{remark}

\numberwithin{equation}{section}

\newcommand{\abs}[1]{\lvert#1\rvert}

\usepackage{color}

\begin{document}
	
	\title[Heat kernel asymptotics and analytic torsion on CR manifolds]
	{Heat kernel asymptotics and analytic torsion on non-degenerate CR manifolds}

	\author[]{Chin-Yu Hsiao}
	\address{Department of Mathematics, National Taiwan University, Taipei, Taiwan}
	\email{chinyuhsiao@ntu.edu.tw or chinyu.hsiao@gmail.com}
	\thanks{Chin-Yu Hsiao was partially supported by National Science and Technology Council project 113-2115-M-002-011-MY3.}
	
	\author[]{Rung-Tzung Huang}
	\address{Department of Mathematics, National Central University, Zhongli District, Taoyuan City 320317, Taiwan}
	\email{rthuang@math.ncu.edu.tw}
	\thanks{Rung-Tzung Huang was supported by Taiwan Ministry of Science and Technology project 111-2115-M-008-003-MY2 and National Science and Technology Council project 113-2115-M-008-009-MY3.}

	\author[]{Guokuan Shao}
	\address{School of Mathematics (Zhuhai), Sun Yat-sen University, Zhuhai 519082, Guangdong, China}
	\email{shaogk@mail.sysu.edu.cn}
	\thanks{Guokuan Shao was supported by National Key R\& D Program of China No.2024YFA1015200, National Natural Science Foundation of China No. 12471082 and Guangdong Basic and Applied Basic Research Foundation No. 2024A1515011223.}
	

\begin{abstract}
The existence of small-time asymptotics for the heat kernel of the Kohn Laplacian on a general CR manifold has remained an open problem. In this paper, we resolve the problem in the non-degenerate case. More precisely, let $X$ be a compact oriented CR manifold of dimension $2n+1$, $n \ge 1$, with a nondegenerate Levi form of constant signature $(n_-, n_+)$. Suppose that condition $Y(q)$ holds at each point of $X$, we establish the small-time asymptotics of the heat kernel of Kohn Laplacian. Suppose that condition $Y(q)$ fails, we establish the small-time asymptotics of the kernel of the difference of the heat operator and Szeg\H{o} projector. As an application we define the analytic torsion on compact oriented nondegenerate CR manifolds and study its dependence on changes of the metrics. Let $L^k$ be the $k$-th power of a CR complex line bundle $L$ over $X$. We establish the asymptotics, as $k \to \infty$, of the analytic torsion with values in $L^k$, under a variant of spectral gap condition. Furthermore, when $X$ admits a transversal CR $S^1$-action, we establish the small-time asymptotics of the $S^1$-equivariant heat kernel of the Kohn Laplacian with values in $L^k$. As an application we define the $S^1$-equivariant Quillen metric with values in $L^k$ and study its dependence on changes of the metrics. Finally, we establish the asymptotics, as $k \to \infty$, of the $S^1$-equivariant analytic torsion with values in $L^k$.		
\end{abstract}

	\maketitle
    \tableofcontents
	

	\section{Introduction}\label{s-gue190320}
	
	Let $(X, T^{1,0}X)$ be a compact oriented CR manifold of dimension $2n+1$, $n \ge 1$, and let $\ddbar_b: \Omega^{0,q}(X) \to \Omega^{0,q+1}(X)$ be the tangential Cauchy-Riemann operator. Let $\Box_b^{(q)} = \ddbar_b^{\ast}\ddbar_b + \ddbar_b \ddbar_b^{\ast} : \Omega^{0,q}(X) \to \Omega^{0,q}(X)$ be the Kohn Laplacian. The study of small $t$ behavior of the heat operator $e^{-t\Box^{(q)}_b}$ is a classical subject in several complex variables, harmonic and microlocal analysis. Suppose that the Levi form satisfies condition $Y(q)$ at each point of $X$, by developing a Heisenberg pseudodifferential calculus, Beals, Greiner and Stanton \cite[Theorem 7.30]{BGS84} obtained the asymptotic expansion of the heat kernel $e^{-t\Box_b^{(q)}}(x, x)$ as $t \to 0^+$. In particular, when $X$ is strongly pseudoconvex and the metric is a Levi metric, for $0 < q < n$, Beals, Greiner and Stanton showed that the coefficients are integrals of polynomials in the covariant derivatives of the curvature and torsion of the Webster-Stanton connection \cite{St92}, \cite{Web78}. In the case that $X$ is strongly pseudoconvex with a Levi metric, Stanton and Tartakoff \cite{ST85} obtained an exact formula for the kernel of $e^{-t\Box_b^{(q)}}$, $0< q < n$, using successive approximations to solve an integral equation and applied this to give a new proof of the asymptotic expansion in this case. M. Taylor \cite{Ta84}, using a different pseudodifferential calculus, obtained a somewhat less precise expansion in the case of a nondegenerate Levi form. 

Without any assumption on the Levi curvature, the Kohn Laplacian $\Box_{b}^{(q)}$ is, in general, not hypoelliptic, which makes the study of heat kernel asymptotics highly challenging. Consequently, the existence of the small-time asymptotics of $e^{-t\Box_{b}^{(q)}}$ in a general setting has remained an open problem for decades, which has important applications in CR local index problems \cite{CHT} and CR analytic torsion. In particular, for strongly pseudoconvex CR manifolds, this obstruction is manifest on functions $(q=0)$ or at top-degree $(q=n)$, where the condition $Y(q)$ completely fails and the operator $\Box_{b}^{(q)}$ loses its hypoellipticity. Because classical tools heavily depend on the hypoellipticity of the underlying operator, tracking the spectral invariants and heat kernel asymptotics on functions $(q=0)$ has long been recognized as a highly intricate problem in several complex variables and microlocal analysis. Even in highly symmetric settings such as sphere quotients, obtaining explicit spectral asymptotics on functions $(q=0)$ remains a delicate task that requires intricate combinatorial analysis, as recently highlighted by Fan et al. in their study of lens spaces \cite{FK}. In a general setting where such geometric symmetry is absent, the small-time asymptotic behavior of $e^{-t\Box_{b}^{(0)}}$ has remained largely out of reach. In this paper, we resolve the problem in the non-degenerate CR manifold case.  

A principal analytical tool for treating such problems is the celebrated framework developed by Melin and Sj\"ostrand~\cite{MS78} for Fourier integral operators with complex-valued phase functions. This framework, however, is traditionally adapted to hypoelliptic operators where the heat kernel exhibits smoothness for $t>0$. When the condition $Y(q)$ fails, the heat operator $e^{-t\Box_{b}^{(q)}}$ is no longer smoothing, and its analysis appears to go beyond the standard scope of the Melin-Sj\"ostrand category. To study the small-time behavior of the heat kernel under these non-hypoelliptic constraints, the approach we introduce in this paper attempts to extend the traditional boundaries of the theory. Our analysis centers on two main aspects: 

(1) Asymptotic Expansion in the Distribution Sense: We observe that despite the loss of hypoellipticity, a suitable formulation of the small-time asymptotics for the heat kernel can still be established rigorously when interpreted in the sense of distributions, rather than the classical pointwise smooth expansion. This allows us to describe $e^{-t\Box_{b}^{(q)}}$ as a distributional Fourier integral operator. 

(2) Smooth Asymptotic Expansion on the Orthogonal Complement of the Kernel: While the full heat kernel lacks smoothness, we show that upon projecting onto the orthogonal complement of the kernel of $\Box_{b}^{(q)}$-effectively removing the singularities tied to the Szeg\H{o} projection $\Pi^{(q)}$-the modified heat operator $(e^{-t\Box_{b}^{(q)}}(I-\Pi^{(q)}))(x,y)$ becomes smooth for $t>0$. For this modified operator, we prove the existence of a smooth small-time asymptotic expansion on the diagonal under a closed range assumption.

With our results on the small-time asymptotic expansion of the heat kernel $(e^{-t\Box_{b}^{(q)}}(I-\Pi^{(q)}))(x,x)$, for all degree $q$, we are able to define the analytic torsion on a compact oriented CR manifold with a nondegenerate Levi form and a CR vector bundle $E$ over $X$. We study the dependence of the analytic torsion on changes of the Hermitian metrics on $TX$ and $E$.

Let us illustrate our results more precisely. Assume that the Levi form is non-degenerate of constant signature $(n_-,n_+)$ on $X$. If $q \not\in \{ n_-, n_+ \}$, i.e. $Y(q)$ condition holds, then we obtain the asymptotic expansion of the heat kernel $e^{-t\Box_b^{(q)}}(x,y)$ as $t \to 0^+$ and recover the result of Beals, Greiner and Stanton \cite{BGS84}. Assume that $q \in \{ n_-, n_+ \}$, i.e. $Y(q)$ condition fails, then we show that $e^{-t\Box_b^{(q)}}$ is a Fourier integral operator and obtain the asymptotic expansion of the heat kernel of $e^{-t\Box_b^{(q)}}(x, y)$ as $t \to 0^+$ in the sense of distribution. Let $\Pi^{(q)} : L^2_{(0,q)}(X) \to \operatorname{Ker} \Box_b^{(q)}$ be the Szeg\H{o} projection. Assume that $q \in \{ n_-, n_+ \}$, i.e. $Y(q)$ condition fails and suppose that (1) $n_- = n_+$ or (2) $|n_- - n_+| > 1$ or (3) $|n_- - n_+| =1$ and $\Box_b^{(n_-)}$, $\Box_b^{(n_+)}$ have closed range, then we show that the heat kernel $\big( e^{-t\Box_b^{(q)}} ( I - \Pi^{(q)} ) \big)(x, y)$ is smooth and obtain the asymptotic expansion of the heat kernel $\big( e^{-t\Box_b^{(q)}} ( I - \Pi^{(q)} ) \big)(x, x)$ as $t \to 0^+$. The CR heat kernel in non-degenerate case is of importance as it can have potential applications in problems about heat kernel, quantization in contact geometry.
	
In \cite{RS71}, Ray and Singer introduced an analytic counterpart of the classical topological invariant called the Reidemeister torsion. They conjectured that the two invariants coincide for unitarily flat vector bundles. This conjecture was proved by Cheeger \cite{C79} and M\"uller \cite{M78}. Bismut and Zhang \cite{BZ92} extended it to arbitrary flat vector bundles via Witten deformation method. 
    In \cite{RS73}, Ray and Singer extended their work to the holomorphic setting.    
    The work of Ray and Singer has had a profound impact on many areas of mathematics and physics, making it difficult to provide a comprehensive survey of subsequent developments. We therefore refer the reader to recent excellent expositions \cite{BCG22, Lo25}.
    
    A version of analytic torsion on contact manifolds was proposed by Rumin and Seshadri \cite{RS12} by using certain fourth-order Laplacians. 
    It is natural to ask whether the Ray–Singer torsion can be extended to the CR setting; this question was raised to the first and second authors by Bismut during a coffee break at a conference in M\"{u}nster in 2024. With our results on the small-time asymptotic expansion of the heat kernel $\big( e^{-t\Box_b^{(q)}} ( I - \Pi^{(q)} ) \big)(x, x)$, for all degree $q$, we are able to define the analytic torsion on a compact oriented CR manifold with a nondegenerate Levi form and a CR vector bundle $E$ over $X$. We study also the dependence of the analytic torsion on changes of the Hermitian metrics on $TX$ and $E$.

	In \cite{BV89}, Bismut and Vasserot established the asymptotic formula of the holomorphic analytic torsion associated with $p$-th powers of a positive line bundle as $p \to \infty$, by using the heat kernel method of \cite{B} (see also \cite[Subsection 5.5]{MM}). The holomorphic analytic torsion and its asymptotic formula  have profound applications in the theory of determinant bundles of direct images and their Quillen metrics, and Arakelov geometry \cite{BGS1, BGS2, BGS3, S}. In \cite{F18}, Finski showed that the asymptotic expansion contains only the terms of the form $p^{n-1}\log p$ and $p^{n-i}$ for $i \in \mathbb{N}$. The first two leading terms were previously established by Bismut and Vasserot. Finski also established asymptotics for the Ray-Singer torsion on complex orbifolds. In \cite{P23}, Puchol extended the results of Bismut and Vasserot on the asymptotics of holomorphic torsion to the fibration case. More recently, Shen and Yu \cite{SY25} proved geometric Zabrodin-Wiegmann conjecture for an integer Quantum Hall state using asymptotics of analytic torsions \cite{BV89, F18}.
    
    In \cite{HH19}, the first and second authors introduced the Fourier components of the Ray-Singer analytic torsion on a compact connected strongly pseudoconvex CR manifold of dimension $2n+1, n \ge 1$, with a transversal CR $S^1$-action on $X$ with respect to the $S^1$-action and established an asymptotic formula for the Fourier components of the analytic torsion with respect to the $S^1$-action. This generalizes the asymptotic formula of Bismut and Vasserot on the holomorphic Ray-Singer torsion associated with high powers of a positive line bundle to strongly pseudoconvex CR manifolds with a transversal CR $S^1$-action. The results in~\cite{HH19} can be deduced from the results in~\cite{F18}. Inspired by \cite{HH19}, the second and third authors \cite{HS} studied asymptotics of the analytic torsion on CR covering manifolds with $S^1$-action. 
    
    Let $L^k$ be the $k$-th power of a CR complex line bundle $L$ over $X$. When condition $Y(q)$ holds at each point of $X$, in~\cite{HZ23}, the first author and W. Zhu established the asymptotics, as $k \to \infty$, of the heat kernel of the Kohn Laplacian with values in $L^k$. As an application, the first author and W. Zhu gave a heat kernel proof of Morse inequalities on compact CR manifolds. When condition $Y(q)$ fails, assume that the condition $n_-=n_+$ or $\abs{n_--n_+}>1$ hold, we establish the asymptotics, as $k \to \infty$, of the kernel of the difference of the heat operator of the Kohn Laplacian and Szeg\H{o} projector with values in $L^k$. By using the asymptotics of the heat kernel of the Kohn Laplacian of Hsiao-Zhu and the asymptotics of the difference of the heat kernel of the Kohn Laplacian and the Szeg\H{o} kernel of ours, we also establish Bismut-Vasserot type asymptotics (\cite{BV89}) of the analytic torsion with values in $L^k$ under the Assumption \ref{a-gue250426yyd1}. 

Quillen \cite{Q85} introduced the Quillen metric on the determinant line
of the cohomology of a family of Riemann surfaces, which is constructed using the Ray-Singer holomorphic
torsion of the fiber. Its key properties were developed in the works of Bismut-Gillet-Soul\'{e} \cite{BGS1,
BGS2, BGS3} and Bismut-Lebeau \cite{BL91}. Since then, the Quillen metric has found numerous applications in both mathematics and physics, including the quantum Hall effect; see, for example, \cite{SY25} and the references therein.

When $X$ admits a transversal CR $S^1$-action, we also establish the small time-asymptotics of the $S^1$-equivariant heat kernel of Kohn Laplacian with values in $L^k$. As an application, we define the $S^1$-equivariant analytic torsion and the $S^1$-equivariant Quillen metrics on compact CR manifolds with transversal CR $S^1$-action. We study also the dependence of the $S^1$-equivariant Quillen metrics on changes of the Hermitian metrics on $TX$ and $L^k$. By using the $\mathbb{R}$-equivariant heat kernel asymptotics of the Kohn Laplacian of Hsiao-Zhu \cite{HZ23}, we establish the Bismut-Vasserot type asymptotics \cite{BV89} of the $S^1$-equivariant analytic torsion with values in $L^k$.


We now formulate our main results. For the standard notations and our terminology used in Subsection~\ref{s-gue250921yyd} below, we refer the reader to Section~\ref{s:prelim}. 
	
	\subsection{Main results}\label{s-gue250921yyd}

	Let $(X, T^{1,0}X)$ be a compact connected orientable CR manifold of dimension $2n+1$, $n\geq1$, where $T^{1,0}X$ denotes the CR structure of $X$. 
	Fix a global non-vanishing real $1$-form $\omega_0\in\mathcal{C}^\infty(X,T^*X)$ such that $\langle\,\omega_0\,,\,u\,\rangle=0$, for every $u\in T^{1,0}X\oplus T^{0,1}X$. The Levi form of $X$ at $x\in X$ is the Hermitian quadratic form on $T^{1,0}_xX$ given by $\mathcal{L}_x(U,\ol V)=-\frac{1}{2i}\langle\,d\omega_0(x)\,,\,U\wedge\ol V\,\rangle$, $U, V\in T^{1,0}_xX$. In this work, we assume that 
	
	\begin{ass}\label{a-gue170123}
		The Levi form is non-degenerate of constant signature $(n_-,n_+)$ on $X$. That is, the Levi form has exactly $n_-$ negative and $n_+$ positive eigenvalues at each point of $X$, where $n_-+n_+=n$. 
	\end{ass}
    
    Let 
    \[\Box_b: {\rm Dom\,}\Box_b\subset L^2_{(0,\bullet)}(X)\To L^2_{(0,\bullet)}(X)\]
    be the Kohn Laplacian (See Subsection \ref{s-gue250331I} for the definition of the Kohn Laplacian $\Box_b$). Let 
    \[e^{-t\Box_b}: L^2_{(0,\bullet)}(X)\To{\rm Dom\,}\Box_b\]
    be the heat operator of $\Box_b$ 
    and let $e^{-t\Box_b}(x,y)\in\mathcal{D}'(\mathbb R_+\times X\times X,T^{*0,\bullet}X\boxtimes(T^{*0,\bullet}X)^*)$
    be the distribution kernel of $e^{-t\Box_b}$. Let $q\in\set{0,1,\ldots,n}$. Write $\Box^{(q)}_b:=\Box_b|_{{\rm Dom\,}\Box_b\cap L^2_{(0,q)}(X)}$,  $e^{-t\Box^{(q)}_b}:=e^{-t\Box_b}|_{L^2_{(0,q)}(X)}$ and let $e^{-t\Box^{(q)}_b}(x,y)\in\mathcal{D}'(\mathbb R_+\times X\times X,T^{*0,q}X\boxtimes(T^{*0,q}X)^*)$
    be the distribution kernel of $e^{-t\Box^{(q)}_b}$.
    Our first main theorem is the following 

        \begin{theorem}[cf. Theorem \ref{T:011120251117}]\label{T:011120251117-1}
		Suppose that Assumption~\ref{a-gue170123} holds. Let $q \in \{0, 1, \cdots, n\}$ ($q$ can be $n_-$ or $n_+$), we can find $a_j(x, y) \in\mathcal{C}^\infty(X \times X, T^{\ast0, q}X \boxtimes (T^{\ast 0 ,q}X)^\ast)$, $j=1, 2, \cdots$, such that for all $m ,N \in \mathbb{N}$, there is a constant $C_{m, N} > 0$ such that for all $t>0$, $t \ll 1$,
		\[
		\Big\| e^{-t\Box_b^{(q)}}(x, y) - \widehat{A}(t, x, y) - \sum_{j=1}^N a_j(x, y)t^j \Big\|_{\mathcal{C}^m(X \times X)} \le C_{m, N}t^{N+1}, 
		\]
		where $\widehat{A}(t, x, y) \in \mathcal{D}'(\mathbb{R}_+ \times X \times X, T^{\ast0, q}X \boxtimes (T^{\ast 0 ,q}X)^\ast)$ and for every local coordinate patch $D$ with local coordinate $x=(x_1, \cdots, x_{2n+1})$, we have
		\[
        \begin{split}
			& \widehat{A}(t, x, y) = \int_0^\infty e^{i \frac{1}{t}\Phi_-(x, y, s)}b_-(t, x, y, s) ds +  \int_0^\infty e^{i \frac{1}{t} \Phi_+(x, y, s)}b_+(t, x, y, s) ds,\\
            &b_{\pm}(t, x, y, s) \sim \sum_{j=0}^\infty t^{-n-1+j}b^j_{\pm}( x, y, s) \\
            &\qquad \text{in} \quad  S^{n,-n-1}_\varepsilon(\mathbb{R}_+ \times D \times D \times \overline{\mathbb{R}}_+, T^{\ast 0, q}X \boxtimes (T^{\ast 0, q}X)^\ast),\\
            &b^j_{\pm}(x, y, s)\in S^{n-j}_{1,0}(D \times D \times\mathbb R_+, T^{\ast 0, q}X \boxtimes (T^{\ast 0, q}X)^\ast),\ \ j=0,1,\ldots,\\
             &b^0_{\pm}(x,y,s)\neq0,\\
              &\mbox{$b_-(t,x,y,s)\in S^{n,-n-1}_{\varepsilon,{\rm cl\,}}(\mathbb R_+\times D\times D\times\mathbb R_+,T^{*0,q}X\boxtimes(T^{*0,q}X)^*)$ }\\
              &\qquad\mbox{with $\varepsilon>0$ if $q\neq n_-$, $\varepsilon=0$ if $q=n_-$},\\
              &\mbox{$b_+(t,x,y,s)\in S^{n,-n-1}_{\varepsilon,{\rm cl\,}}(\mathbb R_+\times D\times D\times\mathbb R_+,T^{*0,q}X\boxtimes(T^{*0,q}X)^*)$}\\
              &\qquad\mbox{ with $\varepsilon>0$ if $q\neq n_+$, $\varepsilon=0$ if $q=n_+$},
             \end{split}
             \]
             \[
             \begin{split}
             & \operatorname{Im}\Phi_\pm (x, y, s) \ge 0, \\
			& \Phi_\pm(x, x, s) = 0,\ \ \mbox{for all $x\in D$}, \\
            &(d_x\Phi_\pm)(x, x, s) = \pm \omega_0(x)s, \quad \text{for all $(x, s) \in D \times \mathbb{R}_+$},
            \end{split}
            \]
            and for every compact set $K \subset D$, there is a constant $C_K > 0$ such that for all $(x, y) \in K \times K$,
		\[
			\operatorname{Im} \Phi_\pm(x, y, s) \ge C_K \frac{s^2}{1+s} |x'-y'|^2, 
		\]
        $x'=(x_1,\ldots,x_{2n})$, $y'=(y_1,\ldots,y_{2n})$, $S^{n,-n-1}_\varepsilon$ and $S^{n,-n-1}_{\varepsilon,{\rm cl\,}}$ are defined in Definition~\ref{d-gue250403yyd} and $S^{n-j}_{1,0}$ is the H\"{o}rmander symbol space of order $n-j$ and of type $(1, 0)$.
	\end{theorem}

\begin{remark}
If $q \not\in \{n_-, n_+\}$, we deduce that $\widehat{A}(t, x, y)$ is smooth for $t > 0$ and admits a small-time asymptotic expansion as $t \to 0^+$. If $q \in \{n_-, n_+\}$, $\widehat{A}(t, x, y)$ is defined as an oscillatory integral and admits a small-time asymptotic expansion in the sense of distribution.
\end{remark}

From Theorem~\ref{T:011120251117-1}, we get 

\begin{corollary}[cf. Corollary~\ref{C:011120251206}]\label{C:011120251206-1}
		Assume that $q \not\in \{ n_-, n_+ \}$. We can find $a_j(x)\in\mathcal{C}^\infty(X,T^{*0,q}X\boxtimes(T^{*0,q}X)^*)$, $j=0,1,\ldots$, $a_0(x)\neq0$, such that
		\[
		\mbox{$e^{-t\Box_b^{(q)}}(x, x) \sim\sum^{+\infty}_{j=0} t^{-n-1+j}a_j(x)$ in $\tilde S^{-n-1}(\mathbb{R}_+ \times X, T^{\ast 0, q}X \boxtimes (T^{\ast 0, q}X)^\ast)$},
		\]
        where $\tilde S^{-n-1}$ is defined in Definition \ref{d-gue250404yydb}.
	\end{corollary}
	
	Note that, by developing a Heisenberg pseudodifferential calculus, Beals, Greiner and Stanton \cite[Theorem 7.30]{BGS84} obtained the same result of Corollary~\ref{C:011120251206-1}. 
    
For every $q=0,1,\ldots,n$,	let $\Pi^{(q)} : L^2_{(0,q)}(X) \to \operatorname{Ker} \Box_b^{(q)}$ be the orthogonal projection with respect to the $L^2$ inner product $(\,\cdot\,|\,\cdot\,)$ induced by the given Hermitian metric $\langle\,\cdot\,|\,\cdot\,\rangle$ on $\mathbb CTX$ (Szeg\H{o} projection). We have

	\begin{theorem}[cf. Theorem \ref{T:012820251552}]\label{T:0128202515521}
		Assume that $q \in \{ n_-, n_+ \}$. Suppose that $n_-=n_+$ or $\abs{n_--n_+}>1$. Then 
		\[
            \big( e^{-t\Box_b^{(q)}} ( I - \Pi^{(q)} ) \big)(x, y) \in \mathcal{C}^\infty(\mathbb{R}_+ \times X \times X, T^{\ast0, q}X \boxtimes (T^{\ast 0 ,q}X)^\ast)
        \]    
		and we can find $b_j(x)\in\mathcal{C}^\infty(X,T^{*0,q}X\boxtimes(T^{*0,q}X)^*)$, $j=0,1,\ldots$, $b_0(x)\neq0$, such that
        \[
        \begin{split}
			\big( e^{-t\Box_b^{(q)}} ( I - \Pi^{(q)} ) \big)(x, x)& \sim \sum_{j=0}^\infty t^{-n-1+j} b_j(x)) \\
            &\text{in}\quad \tilde S^{-n-1}(\mathbb{R}_+ \times X, T^{\ast 0, q}X \boxtimes (T^{\ast 0, q}X)^\ast).
            \end{split}
		\]
	\end{theorem}
    
\begin{remark}\label{r-gue2507181}
Note that the condition $n_-=n_+$ or $\abs{n_--n_+}>1$ implies that  $\Box_b^{(q)}$ has closed range. When $E$ is a CR vector bundle, which is locally CR trivializable, we establish a vector bundle version (cf. Theorem \ref{t-gue250410ycdbz}).
\end{remark}

 Let $E$ be a CR vector bundle over $X$. Suppose that $E$ is locally CR trivializable. Let
    \[\Box_{b,E}: {\rm Dom\,}\Box_{b,E}\subset L^2_{(0,\bullet)}(X,E)\To L^2_{(0,\bullet)}(X,E)\]
    be the Kohn Laplacian. Let 
    \[e^{-t\Box_{b,E}}: L^2_{(0,\bullet)}(X,E)\To{\rm Dom\,}\Box_{b,E}\]
    be the heat operator of $\Box_{b,E}$ 
    and let $e^{-t\Box_{b,E}}(x,y)\in\mathcal{D}'(\mathbb R_+\times X\times X,(E\otimes T^{*0,\bullet}X)\boxtimes(E\otimes T^{*0,\bullet}X)^*)$
    be the distribution kernel of $e^{-t\Box_{b,E}}$. Let $q\in\set{0,1,\ldots,n}$. We write 
    \begin{equation*}
    \Box^{(q)}_{b,E}:=\Box_{b,E}|_{{\rm Dom\,}\Box_{b,E}\cap L^2_{(0,q)}(X,E)},\quad  e^{-t\Box^{(q)}_{b,E}}:=e^{-t\Box_{b,E}}|_{L^2_{(0,q)}(X,E)} 
    \end{equation*}
    and let $e^{-t\Box^{(q)}_{b,E}}(x,y)\in\mathcal{D}'(\mathbb R_+\times X\times X,(E\otimes T^{*0,q}X)\boxtimes(E\otimes T^{*0,q}X)^*)$
    be the distribution kernel of $e^{-t\Box^{(q)}_{b,E}}$.

Let $N$ be the number operator on $T^{*0,\bullet}X$, i.e. $N$ acts on $T^{*0,q}X$ by multiplication by $q$. Denote by $\Pi^\perp_E : L^2(X, T^{*0,\bullet}X \otimes E) \to ({\rm Ker\,}\Box_{b,E})^\perp$ the orthogonal projection onto the complement of ${\rm Ker\,}\Box_{b,E}$. A key application of Theorem \ref{T:0128202515521} is that $\operatorname{STr}  \lbrack N e^{-t \Box_{b,E}} \Pi^\perp_E \rbrack$ admits small time asymptotic expansion when $\Box_{b,E}$ has $L^2$ closed range. Then we can define for $\operatorname{Re}(z)>n$,
\[
        \theta_{b,E}(z)  = - M \left\lbrack \operatorname{STr}  \lbrack N e^{-t \Box_{b,E}}\Pi^\perp_E  \rbrack \right\rbrack  =   - \operatorname{STr} \left\lbrack N ({\Box}_{b,E})^{-z}\Pi^\perp_E \right\rbrack,  
\]
where $M$ denotes the Mellin transform. Hence we can define analytic torsion on non-degenerate CR manifolds, which gives an affirmative answer to Bismut's question.

\begin{definition}[cf. Definition \ref{d-gue160502w}]\label{d-gue160502w1}
    Assume that $\Box_{b,E}$ has $L^2$ closed range. 
		We define $\exp ( -\frac{1}{2} \theta_{b,E}'(0) )$ to be the $\ddbar_{b}$-torsion for the CR vector bundle $E$ over the CR manifold $X$. 
	\end{definition}

Let $\langle \, \cdot \, |\, \cdot \, \rangle_{s}$ and $\langle \, \cdot \, |\, \cdot \, \rangle_{E, s}, s \in [0, 1]$ be smooth families of rigid Hermitian metrics on $TX$ and $E$, respectively. Let $(\, \cdot \, |\, \cdot \, )_{E,s}$ be the $L^2$ inner products on $\Omega^{p,q}(X,E)$ induced by $\langle \, \cdot \, | \, \cdot \, \rangle_s$ and $\langle \, \cdot \, | \, \cdot \, \rangle_{E,s}$. Denote by $\ast_{b,s}$ the tangential Hodge $\ast$-operators associated to the metrics $\langle \, \cdot \, |\, \cdot \, \rangle_{s}$ (cf. \eqref{e-gue1530a05092026} and \eqref{e-gue1530b05092026}) and $\mu_s$ the induced conjugate linear bundle isomorphisms of $E$ and $E^*$ associated to the metric $\langle \, \cdot \, |\, \cdot \, \rangle_{E, s}$. Let 
$$
\Box_{b,E,s}:= \overline{\partial}_{b}\overline{\partial}^{\ast}_{b,s} + \overline{\partial}^{\ast}_{b,s}\overline{\partial}_b:\Omega^{0,\bullet}(X,E) \to \Omega^{0,\bullet}(X,E),
$$ 
where $\overline{\partial}^{\ast}_{b,s}$ denotes the formal adjoint of $\overline{\partial}_b$ with respect to the $L^2$ scalar product $( \, \cdot \, | \, \cdot \, )_{E,s}$. Let $\theta_{b,E,s}(z)$ be the corresponding function as defined in \eqref{E:5.5.12} and let $\Pi_{E, s}$ be the orthogonal projection operator from $\Omega^{0,\bullet}(X,E)$ on $\operatorname{Ker}\left( D_{b,s}|_{\Omega^{0,\bullet}(X,E)} \right)$ for the Hermitian product $(\,\cdot\, | \,\cdot\,)_{E,s}$ on $\Omega^{0,\bullet}(X,E)$. Set 
\[
Q_{b, s} = - \left(\ast_{b,s} \otimes \mu_s \right)^{-1} \frac{\partial \left(\ast_{b, s} \otimes \mu_{k,s} \right)}{\partial s} = - \left(\ast^{-1}_{b, s} \frac{\partial \ast_{b, s}}{\partial s} + (\mu_{k,s})^{-1} \frac{\partial \mu_{k,s}}{\partial s} \right).
\]
We have  
\begin{theorem}[cf. Theorem~\ref{t303262026}]\label{tvari}
 As $t \to 0^+$, for any $\ell \in \mathbb{N}$, there is an asymptotic expansion
\[
\operatorname{STr} \left\lbrack Q_{b,s} \exp \left( -t\Box_{b,E, s} \right)(I-\Pi_{E,s}) \right\rbrack = \sum_{j=0}^{n+1+\ell} M_{j, s} t^{-n-1+j} +O(t^{\ell+1}), 
\]
where
\[
M_{n+1, s} = \frac{\partial}{\partial s}  \left( \frac{\partial}{\partial z}\theta_{b,E,s}  \right)(0).
\]
\end{theorem}
Note that the corresponding result for the Quillen metric on the determinant of the cohomology of a holomorphic Hermitian vector bundle over a complex manifold were obtained in \cite[Theorem 1.18]{BGS1}, see also \cite[Theorem 5.5.6]{MM}.

We now consider CR line bundle cases. We refer the reader to Section~\ref{s-gue250704} for the details. 
Let $L$ be a CR complex line bundle over $X$. The asymptotic behavior of $e^{-\frac{t}{k}\Box^{(q)}_{b,L^k}}$ is important with applications in Morse inequalities. The case when condition $Y(q)$ holds was obtained by Hsiao and Zhu \cite{HZ23}. The case when condition $Y(q)$ fails is more difficult since $e^{-\frac{t}{k}\Box^{(q)}_{b,L^k}}$ is not smoothing in general. We solve the problem
as shown in our second main theorem. 
We refer the readers to Section \ref{s-gue250411yydI} for more notations.

Let $(L, h^L)$ be a CR complex line bundle over $X$, where $h^L$ denotes the Hermitian metric of $L$. Let ${\rm Spec\,}\Box_{b,L^k}$ be the set of all spectrum of $\Box_{b,L^k}$. Since the Levi form is non-degenerate, it was shown in~\cite{HM17} that  ${\rm Spec\,}\Box_{b,L^k}\cap(0,+\infty)$ consists of point eigenvalues of finite multiplicity. 

For $\lambda\in{\rm Spec\,}\Box_{b,L^k}$, let 
\[E_\lambda(X,L^k):=\set{u\in{\rm Dom\,}\Box_{b,L^k};\, \Box_{b,L^k}u=\lambda u}.\]
For $\mu\geq0$, put 
\[
\begin{split}
&E_{0<\lambda\leq\mu}(X,L^k):=\oplus_{\lambda\in{\rm Spec\,}\Box_{b,L^k},0<\lambda\leq\mu}E_\lambda(X,L^k),\\
&E_{\leq\mu}(X,L^k):=\oplus_{\lambda\in{\rm Spec\,}\Box_{b,L^k},0\leq\lambda\leq\mu}E_\lambda(X,L^k).
\end{split}
\]
For $q\in\set{0,1,\ldots,n}$, put 
\[
\begin{split}
&E^{(q)}_{0<\lambda\leq\mu}(X,L^k):=E_{0<\lambda\leq\mu}(X,L^k)\cap L^2_{(0,q)}(X,L^k),\\
&E^{(q)}_{\leq\mu}(X,L^k):=E_{\leq\mu}(X,L^k)\cap L^2_{(0,q)}(X,L^k). 
\end{split}
\]

Let 
\[
\Pi_{L^k,\mu}: L^2(X,L^k)\To E_{\leq\mu}(X,L^k)
\]
be the orthogonal projection and let $\Pi_{L^k,\mu}(x,y)$ be the distribution kernel of $\Pi_{L^k,\mu}$. When $\mu=0$, $\Pi_{L^k,\mu}=\Pi_{L^k}$, $\Pi_{L^k,\mu}(x,y)=\Pi_{L^k}(x,y)$. For every $q=0,1,\ldots$, put $\Pi^{(q)}_{L^k,\mu}:=\Pi_{L^k,\mu}|_{L^2_{(0,q)}(X,L^k)}$, $\Pi^{(q)}_{L^k}:=\Pi_{L^k}|_{L^2_{(0,q)}(X,L^k)}$. 

Let $\mathcal{R}^L$ be the CR curvature induced by $h^L$ (see Definition~\ref{d-gue250921ycdt}). Let \[
		\begin{split}
			\dot{\mathcal{R}}^\phi_x : T^{1,0}_xX \to T^{1,0}_xX, \\
			\dot{\mathcal{L}}_x : T^{1,0}_xX \to T^{1,0}_xX,
		\end{split}\]
be the linear transformations induced by the given Hermitian metric $\langle\,\cdot\,|\,\cdot\,\rangle$ with respect to $\mathcal{R}^L$ and the Levi form $\mathcal{L}_x$ respectively (see \eqref{e-gue250921ycdq}). For every $\eta \in \mathbb{R}$, let
	\[
	\det (\dot{\mathcal{R}}^\phi_x -2\eta \dot{\mathcal{L}}_x) = \mu_1(x) \cdots \mu_n(x),
	\]
	where $\mu_j(x)$, $j = 1, \cdots, n$, are the eigenvalues of $\dot{\mathcal{R}}^\phi_x -2\eta \dot{\mathcal{L}}_x$ with respect to $\langle \cdot \mid \cdot \rangle$. For $x\in X$, let $\set{U_{1}(x),\ldots,U_{n}(x)}$ be an orthonormal frame of $T^{1,0}_xX$ and let $\{\omega^j \}^n_{j=1} \subset T^{\ast1,0}X$ is the dual frame of $\set{U_{1},\ldots,U_{n}}$. For every $\eta\in\mathbb R$, put 
    \[
		\omega^\eta_x = \sum_{j, l =1}^n (\dot{\mathcal{R}}^\phi_x -2\eta \dot{\mathcal{L}}_x)(U_{l}, \overline{U}_{j}) \overline{\omega}^j \wedge (\overline{\omega}^l \wedge)^\star : T^{\ast0,q}_x X \to T^{\ast0,q}_xX.
	\]

    For every $\eta\in\mathbb R$, we set 
	\[
    \begin{split}
	\mathbb R_x(q):=\{\eta&\in\mathbb R;\, \dot{\mathcal R}^\phi_x-2\eta\dot{\mathcal L}_x\,\text{has exactly $q$ negative eigenvalues} 
	\\&\text{and $n-q$ positive eigenvalues}\},
	\end{split}
	\]
    and let $1_{\mathbb R_x(q)}(\eta)=1$ if $\eta\in\mathbb R_x(q)$, $1_{\mathbb R_x(q)}(\eta)=0$ if $\eta\notin\mathbb R_x(q)$.

Our second main theorem is the following
\begin{theorem}[cf. Theorem \ref{t-gue250529yyd}]\label{t-gue250529yyd1}
We assume that $n_-=n_+$ or $\abs{n_--n_+}>1$. With the notations used above, let $I\subset\mathbb R_+$ be a bounded open interval. Let $q\in\set{0,1,\ldots,n}$. Then, 
\[
\begin{split}
&\lim_{k\To+\infty}k^{-(n+1)}\operatorname{Tr}^{(q)}(e^{-\frac{t}{k}\Box^{(q)}_{b,L^k}}(I-\Pi^{(q)}_{L^k,\frac{k}{\log k}}))(x,x))\\
&=\frac{1}{(2\pi)^{n+1}}\int_{\mathbb R}\Bigr(\dfrac{\det(\dot{\mathcal R}^\phi_x-2\eta\dot{\mathcal L}_x)}{\det\big(1-e^{-t(\dot{\mathcal R}^\phi_x-2\eta\dot{\mathcal L}_x)}\big)}{\rm Tr\,}^{(q)}e^{-t\omega_x^\eta}-\abs{\det(\dot{\mathcal{R}}_x^\phi - 2 \eta \dot{\mathcal{L}}_x)}1_{\mathbb R_x(q)}(\eta)\Bigr)d\eta
\end{split}
\]
in $\mathcal{C}^0(I\times X)$ topology.
\end{theorem} 

We also establish Bismut-Vasserot type asymptotics of the analytic torsion when $E=L^k$ as $k\To+\infty$, where $L$ is a CR line bundle. The spectral gap condition is essential in Bismut and Vasserot's work \cite{BV89}, see also \cite{HH19}. But it is difficult to find examples that satisfy the spectral gap condition (${\rm Spec\,}\Box_{b,L^k}\subset\set{0}\cup[Ck,+\infty)$) in the non-degenerate case. We now take $\langle\,\cdot\,|\,\cdot\,\rangle$ to be Levi metric (see \eqref{e-gue250623yyd}). 
We need to relax the condition as follows.

\begin{ass}[cf. Assumption \ref{a-gue250426yyd}]\label{a-gue250426yyd1}
${\rm dim\,}E_{0<\lambda\leq\frac{k}{\log k}}(X,L^k)=o(k^{n+1})$, for $k\gg1$ and 
\[{\rm Spec\,}\Box_{b,L^k}\subset\set{0}\cup[C,+\infty),\]
for all $k\gg1$, where $C>0$ is a constant independent of $k$. 
\end{ass} 

We introduce some notations. 
For $x \in X$, $t>0$ and $\eta\in\mathbb R$, we set
	\[
		(\dot{\mathcal{R}}^\phi_x-2\eta \dot{\mathcal{L}_x})_t =\frac{1}{2\pi}\det \left( \frac{ \dot{\mathcal{R}}_x^\phi - 2 \eta \dot{\mathcal{L}}_x  }{2\pi} \right) \operatorname{Tr}[(\operatorname{Id} - \exp(t (\dot{\mathcal{R}}_x^\phi - 2 \eta \dot{\mathcal{L}}_x ) ))^{-1}].
	\]

    Fix $C\gg1$ so that if $\eta>C$, $\mathcal{R}_x^\phi - 2 \eta \dot{\mathcal{L}}_x$ is non-degenerate of constant signature $(n_+,n_-)$ and if $\eta<-C$, then $\mathcal{R}_x^\phi - 2 \eta \dot{\mathcal{L}}_x$ is non-degenerate of constant signature $(n_-,n_+)$. Fro ${\rm Re\,}z>1$, let 
\[
\begin{split}
H_x(z):=&\frac{1}{\Gamma(z)}\int^{+\infty}_0\int_{\abs{\eta}\geq C}\Bigr((\dot{\mathcal{R}}^\phi_x-2\eta \dot{\mathcal{L}}_x)_t\\
&\quad-(2\pi)^{-n-1}\sum^n_{q=0}q(-1)^q\abs{\det(\dot{\mathcal{R}}_x^\phi - 2 \eta \dot{\mathcal{L}}_x)}1_{\mathbb R_x(q)}(\eta)\Bigr)t^{z-1}d\eta dt,
\end{split}
\]
where $\Gamma(z)$ is the standard Gamma function. We will show in Lemma~\ref{l-gue250623yyd} below that $H_x(z)$ extends to a meromorphic function on $\mathbb{C}$ with poles contained in $\mathbb Z$
and $H_x(z)$ is holomorphic at $0$. 


Our third main theorem is the following 

	\begin{theorem}[cf. Theorem \ref{T:5.5.8}]\label{T:5.5.81}

    We assume that Assumption~\ref{a-gue250426yyd1} holds and $n_-=n_+$ or $\abs{n_--n_+}>1$. Fix $C \gg 1$. As $k \to+\infty$, we have
\[
        \begin{split}
			&\theta_{b,L^k}'(0)\\   
          &=(\log k)k^{n+1}\Bigr(\int_X\int_{\abs{\eta}\leq C}\frac{1}{(2\pi)^{n+1}}\bigr(\frac{n}{2}{\rm det\,}(\dot{\mathcal{R}}^\phi_x - 2\eta \dot{\mathcal{L}_x})\\
          &\quad-\sum^n_{q=0}(-1)^qq\abs{{\rm det\,}(\dot{\mathcal{R}}^\phi_x - 2\eta \dot{\mathcal{L}_x})}1_{\mathbb R_x(q)}(\eta)\bigr)d\eta dv_X(x)+\int_X H_x(0)dv_X(x)\Bigr)\\
&+k^{n+1}\Bigr(-\int_X H'_x(0)dv_X(x)\\
&+\frac{1}{2} \log (2\pi)(2\pi)^{-n-1} \int_X\int_{\abs{\eta}\leq C} \det \left(\dot{\mathcal{R}}^\phi_x - 2 \eta \dot{\mathcal{L}}_x \right)(2q-n)1_{\mathbb R_x(q)}(\eta)d\eta dv_X(x)\\
                &+\frac{1}{2}(2\pi)^{-n-1} \int_X\int_{\abs{\eta}\leq C} \det \left(\dot{\mathcal{R}}^\phi_x - 2 \eta \dot{\mathcal{L}}_x \right)\bigr(-\log(\abs{{\rm det\,}(\dot{\mathcal{R}}^\phi_x - 2 \eta \dot{\mathcal{L}}_x)_-})\\
                &\quad+\log(\abs{{\rm det\,}(\dot{\mathcal{R}}^\phi_x - 2 \eta \dot{\mathcal{L}}_x)_+})\bigr)1_{\mathbb R_x(q)}(\eta)d\eta dv_X(x)\Bigr)+o(k^{n+1}),
          \end{split}
	\]
		where $H_x(z)$ is given by \eqref{e-gue250603yyd}, Lemma~\ref{l-gue250623yyd} and $H_x(0)$ and $H'_x(0)$ are computed in Lemma \ref{l-gue250623yyd} and Lemma \ref{l-gue250623yydI} respectively. 
        \end{theorem}
\begin{remark}
Note that in Subsection \ref{s-gue250714} we provide a large class of examples, i.e. when $X$ is an irregular Sasakian manifold, one can find a CR line bundle $(L,h^L)$ over $X$ such that all the assumptions of Theorem \ref{T:5.5.81} are satisfied.
\end{remark}

Assume that $(X, T^{1,0}X)$ is a CR manifold of dimension $2n+1, n \ge 1$, with a locally free transversal CR $S^1$-action $e^{i\theta}, \theta \in \mathbb{R}, e^{i\theta} : X \to X, x \mapsto e^{i\theta} \circ x$ (see Definition~\ref{d-215605072026} for the meaning of CR $S^1$-action). We let $T$ be the infinitesimal generator of the $S^1$-action (see \eqref{e-gue220005072026}). 
	Suppose that $X$ admits an $S^1$-invariant complete Hermitian metric $\langle \cdot \, | \, \cdot \rangle$ on $\mathbb{C}TX$ so that we have the orthogonal decomposition
	\[
	\mathbb{C}TX = T^{1, 0}X \oplus T^{0, 1}X \oplus \left\{ \lambda T \mid \lambda \in \mathbb{C} \right\}
	\] 
	and $|T|^2 = \langle\, T \, | \, T \,\rangle =1$.
	Let $(L, h^L)$ be a rigid CR line bundle over $X$ with an $S^1$-invariant Hermitian metric $h^L$ on $L$ (see Definition \ref{D:114511262024} for the meaning of rigid CR line bundles). We will use the same notations as before. 
    Consider the operator
	\[
	-iT : \Omega_c^{0, q}(X, L^k) \to \Omega_c^{0, q}(X, L^k)
	\]
	and we extend $-iT$ to $L^2_{(0,q)}(X, L^k)$ in the standard way (see \eqref{e-gue220805072026}).
	Fix $\delta > 0$, let
	\[
		\begin{split}
			& L^2_{(0, q), \le k\delta}(X, L^k) : = E_{-iT} \left(\left[-k\delta, k\delta \right] \right),  \\
			& \Omega^{0, q}_{\le k \delta} := \Omega^{0, q}(X, L^k) \cap L^2_{(0, q), \le k\delta}(X, L^k),
		\end{split}
	\]
	where $E_{-iT}$ denotes the spectral measure of $-iT$. Let
	\[
	Q_{X, \le k\delta}: L^2_{(0, q)}(X, L^k) \to L^2_{(0, q), \le k \delta}(X, L^k)
	\]
	be the orthogonal projection with respect to $(\,\cdot \, | \,\cdot \,)_{h^{L^k}}$. We have the following partial $\ddbar_{b, k}$-complex:
	\[
	\cdots \longrightarrow \Omega^{0, q-1}_{\le k \delta}(X, L^k) \stackrel{\ddbar_{b, k}}{\longrightarrow}\Omega^{0, q}_{\le k \delta}(X, L^k) \stackrel{\ddbar_{b, k}}{\longrightarrow} \Omega^{0, q+1}_{\le k \delta}(X, L^k) \longrightarrow \cdots
	\]
	and put
	\[
	H^q_{b, \le k\delta}(X, L^k) := \frac{\Ker \ddbar_{b,k}:  \Omega^{0, q}_{\le k \delta}(X, L^k)  \to \Omega^{0, q+1}_{\le k \delta}(X, L^k) }{\mbox{Im} \ddbar_{b, k}: \Omega^{0, q-1}_{\le k \delta}(X, L^k)  \to \Omega^{0, q}_{\le k \delta}(X, L^k) }, \quad 0 \le q \le n.
	\]
	Let
	\[
		\begin{split}
			& \Box^{(q)}_{b, k, \le k\delta} : \mbox{Dom} \, \Box^{(q)}_{b, k, \le k\delta} \subset L^2_{(0, q), \le k\delta}(X, L^k) \to  L^2_{(0, q), \le k\delta}(X, L^k),    \\
			&  \mbox{Dom} \, \Box^{(q)}_{b, k, \le k\delta} =  \mbox{Dom} \, \Box^{(q)}_{b, k}   \cap  L^2_{(0, q), \le k\delta}(X, L^k),  \\
			& \Box^{(q)}_{b, k, \le k\delta} = \ddbar^*_{b, k}\ddbar_{b, k} + \ddbar_{b, k}\ddbar^*_{b, k} \quad \text{on}\ \mbox{Dom} \, \Box^{(q)}_{b, k, \le k\delta}. \ 
		\end{split}
	\] 
	It is known that
	\[
	H^q_{b, \le k\delta}(X, L^k) \cong \Ker \Box^{(q)}_{b, k, \le k\delta},     \quad \dim H^q_{b, \le k\delta}(X, L^k)  < +\infty.
	\]
    Let 
	\[
	e^{-t\Box^{(q)}_{b, k, \le k\delta} } := e^{-t\Box^{(q)}_{b, k} } \circ Q_{X, \le k\delta} : L^2_{(0, q)}(X, L^k) \to  \mbox{Dom} \, \Box^{(q)}_{b, k, \le k\delta}, \quad t>0,
	\]
	and let
	\[
	 e^{-t\Box^{(q)}_{b, k, \le k\delta}}(x, y) \in C^\infty(\mathbb{R}_+ \times X \times X, (T^{*0,q}X \times L^k) \boxtimes (T^{*0,q}X \otimes L^k)^*)
	\]
	be the distribution kernel of $ e^{-t\Box^{(q)}_{b, k, \le k\delta}}$ with respect to $dv_X(x)$.
	As before, we use the canonical identification $\mbox{End}(L^k) = \mathbb{C}$.
    We have the following
\begin{theorem}[cf. Theorem~\ref{t05042026}]\label{t-gue222505072026}
		With the notations and assumptions used before, we can find 
        $a_{j,k} \in \mathbb{R}$, $j=0,1,\ldots$, $a_{0,k}\neq0$, such that
		\[
		\int_X e^{-t\Box^{(q)}_{b, k, \le k\delta} }(x, x)dv_X(x) \sim\sum^{+\infty}_{j=0} a_{j,k} t^{-n-1+j} \quad \text{as $t \to 0^+$}.
		\]
	\end{theorem}
    Let 
	$\Pi_{L^k, \le k \delta}^\perp : L^2(X, T^{*0,\bullet}X \otimes L^k) \to (\Ker\Box_{b, k, \le k\delta})^\perp$, 
	be the orthogonal projection onto the complement of $\Ker\Box_{b, k, \le k\delta}$.
	A key application of Theorem~\ref{t-gue222505072026} is that $
		\operatorname{STr}  \lbrack N e^{-t \Box_{b, k, \le k\delta}} \Pi_{L^k, \le k \delta}^\perp \rbrack$
	admits small-time asymptotics. 
	Then we can define for $\operatorname{Re}(z)>n$, 
	\[
		\theta_{b, L^k, \le k \delta}(z)  = - M \left\lbrack \operatorname{STr}  \lbrack N e^{-t \Box_{b, k, \le k\delta}} (\Pi_{L^k, \le k \delta})^\perp  \rbrack \right\rbrack  =   - \operatorname{STr} \left\lbrack N ({\Box}_{b, k, \le k \delta})^{-z} (\Pi_{L^k, \le k \delta})^\perp \right\rbrack.  
	\]
	Hence we can define the $S^1$-equivariant analytic torsion on non-degenrate CR manifolds.
	\begin{definition}[cf. Definition \ref{d-gue160502wr6}]\label{d-gue223205072026}
		We define $\exp ( -\frac{1}{2} \theta_{b, L^k, \le k \delta}'(0) )$ to be the $\ddbar_{b}$-torsion for the CR line bundle $L^k$ over the CR manifold $X$ with $S^1$-action. 
	\end{definition}	
	Denote by 
$$
H^\bullet_{b, \le k \delta}(X, L^k) = \oplus_{q=0}^n H^q_{b,\le k \delta}(X,L^k).
$$ 
For a finite dimensional vector space $V$, we set 
$$
\det V := \wedge^{\text{max}}V.
$$
We then denote by 
$$
(\det V)^{-1}:= (\det V)^*,
$$ the dual line of $\det V$.
Then
\[
\det H^\bullet_{b,\le k \delta}(X,L^k) = \otimes_{q=0}^n \left( \det H^q_{b,\le k \delta}(X,L^k)  \right)^{(-1)^q}
\]
is the determinant line of the cohomology $H^\bullet_{b,\le k \delta}(X,L^k)$. We define 
\begin{equation}\label{E:5.5.14}
\lambda_{b,\le k \delta}(L^k) = \left( \det H^\bullet_{b,\le k \delta}(X,L^k)  \right)^{-1}. \nonumber
\end{equation}
The rigid Hermitian metrics $\langle \, \cdot \, |\, \cdot \, \rangle$ and $\langle \, \cdot \, |\, \cdot \, \rangle_{L^k}$ on $\mathbb{C}TX$ and $L^k$, respectively, induce a canonical $L^2$-metric $h^{H^\bullet_{b,\le k \delta}(X,L^k)}$ on $H^\bullet_{b,\le k \delta}(X,L^k)$. Let $|\cdot|_{\lambda_{b,\le k \delta}(L^k)}$ be the $L^2$-metric on $\lambda_{b,\le k \delta}(L^k)$ induced by $h^{H^\bullet_{b,\le k \delta}(X,L^k)}$. 
Now we can define the Quillen metric on $\det H^\bullet_{b,\le k \delta}(X,L^k)$.
\begin{definition}[cf. Definition \ref{D:5.5.5}]\label{d-gue223205072026}
Fix $k \in \mathbb{N}$. The Quillen metric $\| \cdot  \|_{\lambda_{b, \le k \delta}(L^k)}$ on $\det H^\bullet_{b,\le k \delta}(X,L^k)$ is defined as
\[
\| \cdot  \|_{\lambda_{b, \le k \delta}(L^k)} \, := \,  |\cdot|_{\lambda_{b,\le k \delta}(L^k)} \cdot \exp ( -\frac{1}{2} \theta_{b,L^k,\le k \delta}'(0)  ).
\]
\end{definition}
Denote by $\mu_k: L^k \to (L^*)^k$ the induced conjugate linear bundle isomorphism from the vector bundle $L$ to its dual vector bundle $(L^*)^k$. We denote by 
$$
\Box_{b,k,\le k\delta, s}:=\Box_{b, s}|_{\Omega^{0,\bullet}_{\le k \delta}(X,L^k)}.
$$ 
Let $\| \cdot  \|_{\lambda_{b, \le k \delta}(L^k), s}$ be the corresponding Quillen metrics  on $\det H^\bullet_{b,\le k \delta}(X,L^k)$. Recall that $Q_{b,s}$ is defined in \eqref{E:qbs}. 

We have (see also Theorem~\ref{tvari})
\begin{theorem}[cf. Thoerem~\ref{t205302042026}]\label{t581}
 As $t \to 0^+$, for any $\ell \in \mathbb{N}$, there is an asymptotic expansion
\[
\operatorname{STr} \left\lbrack Q_{b,s} \exp \left( -t\Box_{b,k,\le k\delta, s} \right) \right\rbrack = \sum_{j=0}^{\ell+n} M_{j,\le k \delta, s} t^{-n-1+j} +O(t^{\ell+1}), 
\]
where
\[
M_{n+1,\le k\delta, s} = \frac{\partial}{\partial s} \log \left(  \| \cdot  \|^{2}_{\lambda_{b, \le k \delta}(L^k), s} \right).
\]
\end{theorem}

Finally, we establish the Bismut-Vasserot type asymptotics (\cite{BV89}) of the $S^1$-equivariant analytic torsion with values in $L^k$.
\begin{theorem}[cf. Theorem~\ref{t-que042520262248}]\label{t582}
		As $k \to+\infty$, we have
		\[
			  \begin{split}
			&\theta_{b,L^k, \le k \delta}'(0)\\   
          &=(\log k)k^{n+1}\Bigr(\int_X\int_{-\delta}^\delta \frac{1}{(2\pi)^{n+1}}\bigr(\frac{n}{2}{\rm det\,}(\dot{\mathcal{R}}^\phi_x - 2\eta \dot{\mathcal{L}_x})\\
          &\quad-\sum^n_{q=0}(-1)^qq\abs{{\rm det\,}(\dot{\mathcal{R}}^\phi_x - 2\eta \dot{\mathcal{L}_x})}1_{\mathbb R_x(q)}(\eta)\bigr)d\eta dv_X(x)\Bigr)\\
&+k^{n+1}\Bigr(\frac{1}{2} \log (2\pi)(2\pi)^{-n-1} \int_X\int_{-\delta}^\delta \det \left(\dot{\mathcal{R}}^\phi_x - 2 \eta \dot{\mathcal{L}}_x \right)(2q-n)1_{\mathbb R_x(q)}(\eta)d\eta dv_X(x)\\
                &+\frac{1}{2}(2\pi)^{-n-1} \int_X\int_{-\delta}^\delta \det \left(\dot{\mathcal{R}}^\phi_x - 2 \eta \dot{\mathcal{L}}_x \right)\bigr(-\log(\abs{{\rm det\,}(\dot{\mathcal{R}}^\phi_x - 2 \eta \dot{\mathcal{L}}_x)_-})\\
                &\quad+\log(\abs{{\rm det\,}(\dot{\mathcal{R}}^\phi_x - 2 \eta \dot{\mathcal{L}}_x)_+})\bigr)1_{\mathbb R_x(q)}(\eta)d\eta dv_X(x)\Bigr)+o(k^{n+1}),
          \end{split}
		\]
	\end{theorem}

It is interesting to consider CR analytic torsion in the degenerate case (e.g. Levi flat case), equivariant case with group action and CR orbifold case which can be related to CR singular points \cite{FH18, HY16, HY17}. Our paper can be regarded as a starting point for such further topics.

The paper is organized as follows. In Section \ref{s:prelim}, we list some terminology and definitions we use throughout this paper. In Section \ref{s-gue250401yyd}, we introduce some new symbol classes and study the asymptotic behavior of $e^{-t\Box_b^{(q)}}$ as $t\to 0+$. In Section \ref{s-gue250704}, we are devoted to exploring the asymptotics of $e^{-\frac{t}{k}\Box^{(q)}_{b,L^k}}$ as $k\to\infty$. In Section \ref{s-gue250411}, we define the CR analytic torsion and study its dependence on changes of metrics. We establish the Bismut-Vasserot type asymptotics of the analytic torsion $\theta_{b,L^k}'(0)$ is established. In Subsection \ref{s-gue250714}, we show non-trivial examples satisfying conditions in Theorem \ref{T:5.5.81}. In Section \ref{s23}, we obtain the small time asymptotics of the $S^1$-equivariant CR heat kernel. In Section \ref{s-220404272026}, we define the $S^1$-equivariant CR analytic torsion and Quillen metric. We study the dependence of the $S^1$-equivariant Quillen metric on changes of metrics. Finally we give the corresponding asymptotic formula of the $S^1$-equivariant CR analytic torsion with values in $L^k$ in Section \ref{sfinal}.


	
	\section{Preliminaries}\label{s:prelim}

    \subsection{Standard notations} \label{s-ssna}

We use the following notations: $\mathbb N=\{1,2,\ldots\}$ is the set of natural numbers excluding $0$ and $\mathbb N_0=\mathbb N\cup\{0\}$, $\mathbb R$ is the set of real numbers, 
$\dot{\mathbb R}=\{x\in\mathbb R;\, x\neq0\}$,
${\mathbb R}_+=\{x\in\mathbb R;\, x>0\}$, $\overline{\mathbb R}_+=\{x\in\mathbb R;\, x\geq0\}$. Furthermore we adopt the standard multi-index notation: we write $\alpha=(\alpha_1,\ldots,\alpha_n)\in\mathbb N^n_0$ 
if $\alpha_j\in\mathbb N_0$, $j=1,\ldots,n$. 

Let $M$ be a smooth paracompact manifold. We let $TM$ and $T^*M$ denote respectively the tangent bundle of $M$ and the cotangent bundle of $M$. The complexified tangent bundle $TM \otimes \mathbb{C}$ of $M$ will be denoted by $\Complex TM$, similarly we write $\Complex T^*M$ for the complexified cotangent bundle of $M$. Consider $\langle\,\cdot\,,\cdot\,\rangle$ to denote the pointwise
duality between $TM$ and $T^*M$; we extend $\langle\,\cdot\,,\cdot\,\rangle$ bi-linearly to $\Complex TM\times\Complex T^*M$. Let $B$ be a smooth vector bundle over $M$. The fiber of $B$ at $x\in M$ will be denoted by $B_x$. Let $E$ be a vector bundle over a smooth paracompact manifold $N$. We write
$B\boxtimes E^*$ to denote the vector bundle over $M\times N$ with fiber over $(x, y)\in M\times N$ consisting of the linear maps from $E_y$ to $B_x$.  

Let $Y\subset M$ be an open set. From now on, the spaces of distribution sections of $B$ over $Y$ and smooth sections of $B$ over $Y$ will be denoted by $\mathcal D'(Y, B)$ and $\mathcal{C}^\infty(Y, B)$, respectively.
Let $\mathcal E'(Y, B)$ be the subspace of $\mathcal D'(Y, B)$ whose elements have compact support in $Y$. Let $\mathcal{C}^\infty_c(Y,B):=\mathcal{C}^\infty(Y,B)\cap\mathcal{E}'(Y,B)$. 
For $m\in\Real$, let $H^m(Y, B)$ denote the Sobolev space
of order $m$ of sections of $B$ over $Y$. Let us denote
\begin{align*}
H^m_{\rm loc\,}(Y, B)=\big\{u\in\mathcal{D}'(Y, B);\, \varphi u\in H^m(Y, B),
    \, \forall\varphi\in \mathcal{C}^\infty_c(Y)\big\}\,,
\end{align*}
and
\begin{align*}
       H^m_{\rm comp\,}(Y, B)=H^m_{\rm loc}(Y, B)\cap\mathcal{E}'(Y, B)\,.
\end{align*}

Let $B$ and $E$ be smooth vector
bundles over paracompact orientable manifolds $M$ and $N$, respectively, equipped with smooth densities of integration. If
$A: \mathcal{C}^\infty_c(N,E)\To\mathcal D'(M,B)$
is continuous, we write $A(x, y)$ to denote the distribution kernel of $A$.
The following two statements are equivalent
\begin{enumerate}
	\item $A$ is continuous: $\mathcal E'(N,E)\To\mathcal{C}^\infty(M,B)$,
	\item $A(x,y)\in\mathcal{C}^\infty(M\times N,B\boxtimes E^*)$.
\end{enumerate}
If $A$ satisfies (1) or (2), we say that $A$ is smoothing on $M \times N$. 
We say that $A$ is properly supported if the restrictions of the two projections 
$(x,y)\mapsto x$, $(x,y)\mapsto y$ to ${\rm supp\,}(A(x,y))$
are proper.

Let $H(x,y)\in\mathcal D'(M\times N,B\boxtimes E^*)$. We write $H$ to denote the unique continuous operator $\mathcal C^\infty_c(N,E)\To\mathcal D'(M,B)$ with distribution kernel $H(x,y)$. In this work, we identify $H$ with $H(x,y)$. 

Let
$B: \mathcal{C}^\infty_c(N,E)\To\mathcal D'(M,B)$
be a continuous operator. We write $A(x, y)\equiv B(x,y)$ or $A\equiv B$ if $A-B$ is a smoothing operator.  
	
	\subsection{CR manifolds} \label{s-gue250331}
	
	Let $(X, T^{1,0}X)$ be a compact, and orientable CR manifold of dimension $2n+1$, $n\geq 1$, where $T^{1,0}X$ is a CR structure of $X$, that is, $T^{1,0}X$ is a subbundle of rank $n$ of the complexified tangent bundle $\mathbb{C}TX$, satisfying $T^{1,0}X\cap T^{0,1}X=\{0\}$, where $T^{0,1}X=\overline{T^{1,0}X}$, and $[\mathcal V,\mathcal V]\subset\mathcal V$, where $\mathcal V=\mathcal C^\infty(X, T^{1,0}X)$. There is a unique subbundle $HX$ of $TX$ such that $\mathbb{C}HX=T^{1,0}X \oplus T^{0,1}X$, i.e. $HX$ is the real part of $T^{1,0}X \oplus T^{0,1}X$. Let $J:HX\To HX$ be the complex structure map given by $J(u+\ol u)=iu-i\ol u$, for every $u\in T^{1,0}X$. 
	By complex linear extension of $J$ to $\mathbb{C}TX$, the $i$-eigenspace of $J$ is $T^{1,0}X \, = \, \left\{ V \in \mathbb{C}HX;\, JV \, =  \,  iV  \right\}.$ We shall also write $(X, HX, J)$ to denote a compact CR manifold $X$.
	
	We fix a real $1$ form $\omega_0\in\mathcal{C}^{\infty}(X,T^*X)$ so that $\omega_0(x)\neq0$, for every $x\in X$,  $\langle\,\omega_0(x)\,,\,u\,\rangle=0$, for every $u\in H_xX$, for every $x\in X$. 
	For each $x \in X$, we define a quadratic form on $HX$ by
	\begin{equation}\label{E:levi}
		\mathcal{L}_x(U,V) =\frac{1}{2}d\omega_0(JU, V), \forall \ U, V \in H_xX.
	\end{equation}
	We extend $\mathcal{L}$ to $\mathbb{C}HX$ by complex linear extension. Then for $U, V \in T^{1,0}_xX$,
	\begin{equation}
		\mathcal{L}_x(U,\overline{V}) = \frac{1}{2}d\omega_0(JU, \overline{V}) = -\frac{1}{2i}d\omega_0(U,\overline{V}).
	\end{equation}
	The Hermitian quadratic form $\mathcal{L}_x$ on $T^{1,0}_xX$ is called Levi form at $x$. We recall that in this paper, we always assume that the Levi form $\mathcal{L}$ on $T^{1,0}X$ is non-degenerate of constant signature $(n_-,n_+)$ on $X$, where $n_-$ denotes the number of negative eigenvalues of the Levi form and $n_+$ denotes the number of positive eigenvalues of the Levi form. Let $T\in\mathcal{C}^\infty(X,TX)$ be the non-vanishing vector field determined by 
	\begin{equation}\label{e-gue170111ry}\begin{split}
			&\omega_0(T)=-1,\\
			&d\omega_0(T,\cdot)\equiv0\ \ \mbox{on $TX$}.
	\end{split}\end{equation}
	Note that $X$ is a contact manifold with contact form $\omega_0$, contact plane $HX$ and $T$ is the Reeb vector field.
	
	Fix a smooth Hermitian metric $\langle \cdot \mid \cdot \rangle$ on $\mathbb{C}TX$ so that $T^{1,0}X$ is orthogonal to $T^{0,1}X$, $\langle u \mid v \rangle$ is real if $u, v$ are real tangent vectors, $\langle\, T\, | \, T \, \rangle=1$ and $T$ is orthogonal to $T^{1,0}X\oplus T^{0,1}X$. For $u \in \mathbb{C}TX$, we write $|u|^2 := \langle u | u \rangle$. Denote by $T^{*1,0}X$ and $T^{*0,1}X$ the dual bundles of $T^{1,0}X$ and $T^{0,1}X$, respectively. They can be identified with subbundles of the complexified cotangent bundle $\mathbb{C}T^*X$. Define the vector bundle of $(0,q)$-forms by $T^{*0,q}X := \wedge^qT^{*0,1}X$. Let
    $\Lambda^\bullet(\mathbb CT^*X):=\oplus^{2n+1}_{j=0}\Lambda^j(\mathbb CT^*X)$.
    The Hermitian metric $\langle\,\cdot\,| \cdot\,\rangle$ on $\mathbb{C}TX$ induces, by duality, a Hermitian metric on $\mathbb{C}T^*X$ and also on $\Lambda^\bullet(\mathbb CT^*X)$. We shall also denote all these induced metrics by $\langle \cdot \mid \cdot \rangle$. For $u\in\Lambda^*(\mathbb CT^*X)$, let $\abs{u}^2:=\langle\,u\,|\,u\,\rangle$. Let $T^{*0,\bullet}X:=\oplus^n_{q=0}T^{*0,q}X$. Note that we have the pointwise orthogonal decompositions:
	\begin{equation}
		\begin{array}{c}
			\mathbb{C}T^*X = T^{*1,0}X \oplus T^{*0,1}X \oplus \left\{ \lambda \omega_0;\, \lambda \in \mathbb{C} \right\}, \\
			\mathbb{C}TX = T^{1,0}X \oplus T^{0,1}X \oplus \left\{ \lambda T;\, \lambda \in \mathbb{C} \right\}.
		\end{array}
	\end{equation}
	
	For $x, y\in X$, let $d(x,y)$ denote the distance between $x$ and $y$ induced by the Hermitian metric $\langle \cdot \mid \cdot \rangle$. Let $A$ be a subset of $X$. For every $x\in X$, let $d(x,A):=\inf\set{d(x,y);\, y\in A}$. 

    \subsection{Kohn Laplacian}\label{s-gue250331I} 
	Let $D\subset X$ be an open subset. Let $\Omega^{0,q}(D)$ denote the space of smooth sections 
	of $T^{*0,q}X$ over $D$ and let $\Omega^{0,q}_c(D)$ be the subspace of
	$\Omega^{0,q}(D)$ whose elements have compact support in $D$. Put $\Omega^{0,\bullet}(D):=\oplus^n_{q=0}\Omega^{0,q}(D)$ and $\Omega^{0,\bullet}_c(D):=\oplus^n_{q=0}\Omega^{0,q}_c(D)$.
	
	Let 
	\begin{equation} \label{e-suIV}
		\ddbar_b:\Omega^{0,\bullet}(X)\To\Omega^{0,\bullet}(X)
	\end{equation}
	be the tangential Cauchy-Riemann operator. Let $dv_X(x)$ be the volume form on $X$ induced by the Hermitian metric $\langle\,\cdot\,|\,\cdot\,\rangle$.
	The natural global $L^2$ inner product $(\,\cdot\,|\,\cdot\,)$ on $\Omega^{0,\bullet}(X)$ 
	induced by $dv_X(x)$ and $\langle\,\cdot\,|\,\cdot\,\rangle$ is given by
	\begin{equation}\label{e:l2}
		(\,u\,|\,v\,):=\int_X\langle\,u(x)\,|\,v(x)\,\rangle\, dv_X(x)\,,\quad u,v\in\Omega^{0,\bullet}(X)\,.
	\end{equation}

    For every $q=0,1,\ldots,n$, we denote by $L^2_{(0,q)}(X)$ 
	the completion of $\Omega^{0,q}(X)$ with respect to $(\,\cdot\,|\,\cdot\,)$. 
	Write $L^2(X):=L^2_{(0,0)}(X)$. We extend $(\,\cdot\,|\,\cdot\,)$ to $L^2_{(0,q)}(X)$ 
	in the standard way. Put $L^2_{(0,\bullet)}(X):=\oplus^n_{q=0}L^2_{(0,q)}(X)$. For $f\in L^2_{(0,\bullet)}(X)$, denote by $\norm{f}^2:=(\,f\,|\,f\,)$.
	We extend
	$\ddbar_{b}$ to $L^2_{(0,\bullet)}(X)$, by
	\begin{equation}\label{e-suVII}
		\ddbar_{b}:{\rm Dom\,}\ddbar_{b}\subset L^2_{(0,\bullet)}(X)\To L^2_{(0,\bullet)}(X)\,,
	\end{equation}
	where ${\rm Dom\,}\ddbar_{b}:=\{u\in L^2_{(0,\bullet)}(X);\, \ddbar_{b}u\in L^2_{(0,\bullet)}(X)\}$, 
	and for any $u\in L^2_{(0,\bullet)}(X)$, $\ddbar_{b} u$ is defined in the sense of distributions.
	We also write
	\begin{equation}\label{e-suVIII}
		\ol{\pr}^{*}_{b}:{\rm Dom\,}\ol{\pr}^{*}_{b}\subset L^2_{(0,\bullet)}(X)\To L^2_{(0,\bullet)}(X)
	\end{equation}
	to denote the Hilbert space adjoint of $\ddbar_{b}$ in the $L^2$ space with respect to $(\,\cdot\,|\,\cdot\, )$.
	Let $\Box_{b}$ denote the (Gaffney extension) of the Kohn Laplacian given by
	\begin{equation}\label{e-suIX}
		\begin{split}
			{\rm Dom\,}\Box_{b}=\Big\{s\in L^2_{(0,\bullet)}(X);&\, 
			s\in{\rm Dom\,}\ddbar_{b}\cap{\rm Dom\,}\ol{\pr}^{*}_{b},\,
			\ddbar_{b}s\in{\rm Dom\,}\ol{\pr}^{*}_{b},\\
			&\quad\ol{\pr}^{*}_{b}s\in{\rm Dom\,}\ddbar_{b}\Big\}\,,\\
			\Box_{b}s&=\ddbar_{b}\ol{\pr}^{*}_{b}s+\ol{\pr}^{*}_{b}\ddbar_{b}s 
			\:\:\text{for  $s\in {\rm Dom\,}\Box_{b}$}\,.
		\end{split}
	\end{equation}
	By a result of Gaffney, $\Box_{b}$ is a non-negative self-adjoint operator 
	(see \cite[Proposition\,3.1.2]{MM}). That is, $\Box_{b}$ is self-adjoint and 
	the spectrum of $\Box_{b}$ is contained in $\ol\Real_+$. Let 
    \[e^{-t\Box_b}: L^2_{(0,\bullet)}(X)\To{\rm Dom\,}\Box_b\]
    be the heat operator of $\Box_b$ 
    and let $e^{-t\Box_b}(x,y)\in\mathcal{D}'(\mathbb R_+\times X\times X,T^{*0,\bullet}X\boxtimes(T^{*0,\bullet}X)^*)$
    be the distribution kernel of $e^{-t\Box_b}$. Let $q\in\set{0,1,\ldots,n}$. Write $\Box^{(q)}_b:=\Box_b|_{{\rm Dom\,}\Box_b\cap L^2_{(0,q)}(X)}$,  $e^{-t\Box^{(q)}_b}:=e^{-t\Box_b}|_{L^2_{(0,q)}(X)}$ and let $e^{-t\Box^{(q)}_b}(x,y)\in\mathcal{D}'(\mathbb R_+\times X\times X,T^{*0,q}X\boxtimes(T^{*0,q}X)^*)$
    be the distribution kernel of $e^{-t\Box^{(q)}_b}$.

    
    The operator $\Box^{(q)}_b$ is not elliptic, for all $q=0,1,\ldots,n$. 
	The characteristic manifold of $\Box^{(q)}_b$ is given by $\Sigma=\Sigma^+\cup\Sigma^-$, where 
	\begin{equation}\label{e-gue250401yyd}\begin{split}
		&\Sigma^+=\set{(x,\lambda\omega_0(x))\in T^*X;\, \lambda>0},\\
		&\Sigma^-=\set{(x,\lambda\omega_0(x))\in T^*X;\, \lambda<0}.\end{split}\end{equation}
	Then $\Sigma$ is a symplectic submanifold in $T^*X$ if and only if the Levi from is non-degenerate (see~\cite[Lemma 2.4]{Hsiao08}). 
	
	\subsection{CR vector bundles}\label{s-gue250404yyds}

We first introduce CR manifolds of high codimension. 

\begin{definition}\label{d-gue250920yyd}
Let $E$ be a smooth manifold of dimension $2n+d$, $n\geq1$, $d\geq1$. Let $T^{1,0}E$ be a subbundle of $\mathbb CTE$. We say that $T^{1,0}E$ is a CR structure of $E$ if 
\begin{itemize}
    \item ${\rm dim\,}T^{1,0}_xE=n$, $\forall x\in E$,
    \item $T^{1,0}E\cap T^{0,1}E=\set{0}$, $T^{0,1}E:=\ol{T^{1,0}E}$, 
    \item $[\mathcal{C}^\infty(E,T^{1,0}E),\mathcal{C}^\infty(E,T^{1,0}E)]\subset\mathcal{C}^\infty(E,T^{1,0}E)$.
\end{itemize}
We call the pair $(E,T^{1,0}E)$ CR manifold of codimension $d$. 
\end{definition}

Let $(E, T^{1,0}E)$ and $(F, T^{1,0}F)$ be two arbitrary CR manifolds. A smooth map $f:E \to F$ is called a CR map if $(d_xf)(T_x^{1,0}E) \subseteq T_{f(x)}^{1,0}F$ for any $x \in E$, where $d_xf$ is the ($\mathbb{C}$-linear extension to $\mathbb{C}TE$ of the) differential of $f$ at $x$. 

   We now introduce the definition of CR vector bundles.
    
\begin{definition}\label{Def:CRVB}
	A complex vector bundle $\pi\,:\,E\rightarrow X$ of rank $d$ is called a CR vector bundle if 
	\begin{itemize}
		\item [(i)] \(E\) is a CR manifold of codimension \(1+d\), 
		\item [(ii)] \(\pi\colon E\to X\) is a CR submersion,
		\item [(iii)] \(E\oplus E\ni(\xi_1,\xi_2)\to \xi_1+\xi_2\in E\) and \(\mathbb C\times E\ni(\lambda,\xi)\to \lambda \xi\in E\) are CR maps.
	\end{itemize}
	A smooth section \(s\in\mathcal{C}^\infty(U,E)\) defined on an open set \(U\subset E\) is called a CR section if the map \(s\colon U\to E\) is CR.
\end{definition}

Note that it is easy to see that $E \oplus E$ and $\mathbb{C} \times E$ are CR manifolds.

Let $E$ be a CR vector bundle over $X$. The tangential Cauchy-Riemann operator can be defined on sections of $E$ in the standard way: 
\[\ddbar_{b}: \Omega^{0,\bullet}(X,E)\To\Omega^{0,\bullet}(X,E),\]
where $\Omega^{0,\bullet}(X,E):=\oplus^n_{q=0}\Omega^{0,q}(X,E)$, $\Omega^{0,q}(X,E):=\mathcal{C}^\infty(X,E\otimes T^{*0,q}X)$, $q=0,1,\ldots,n$. 

\begin{definition}\label{Def:LocTriv}
 A CR vector bundle $E\To X$ is called locally CR trivializable if for any point \(p\in X\)  there exists an open neighborhood \(U\subset X\) such that \(E|_U\) is CR vector bundle isomorphic to the trivial CR vector bundle \(U\times\mathbb C^d\), $d\in\mathbb N$, $d={\rm rank\,}E$.
\end{definition}

Let $E$ be a CR vector bundle over $X$. In this work, we will always assume that $E$ is locally CR trivializable. Let $D\subset X$ be an open subset. Let $\Omega^{0,q}(D,E)$ denote the space of smooth sections 
	of $T^{*0,q}X\otimes E$ over $D$ and let $\Omega^{0,q}_c(D,E)$ be the subspace of
	$\Omega^{0,q}(D,E)$ whose elements have compact support in $D$. Put $\Omega^{0,\bullet}(D,E):=\oplus^n_{q=0}\Omega^{0,q}(D,E)$ and $\Omega^{0,\bullet}_c(D,E):=\oplus^n_{q=0}\Omega^{0,q}_c(D,E)$.

Fix a Hermtian metric $h^E=\langle\,\cdot\,|\,\cdot\,\rangle_E$ on $E$. Let $(\,\cdot\,|\,\cdot\,)_E$ be the natural global $L^2$ inner product on $\mathcal{C}^\infty(X,E\otimes\Lambda^\bullet(\mathbb CT^*X))$ induced by $dv_X(x)$, $\langle\,\cdot\,|\,\cdot\,\rangle$ and $h^E$. 
    For every $q=0,1,\ldots,n$, we denote by $L^2_{(0,q)}(X,E)$ 
	the completion of $\Omega^{0,q}(X,E)$ with respect to $(\,\cdot\,|\,\cdot\,)_E$. 
	Write $L^2(X,E):=L^2_{(0,0)}(X,E)$. We extend $(\,\cdot\,|\,\cdot\,)_E$ to $L^2_{(0,q)}(X,E)$ 
	in the standard way. Put $L^2_{(0,\bullet)}(X,E):=\oplus^n_{q=0}L^2_{(0,q)}(X,E)$. Similarly,  we denote by $L^2(X,E\otimes\Lambda^\bullet(\mathbb CT^*X))$ 
	the completion of $\mathcal{C}^\infty(X,E\otimes\Lambda^\bullet(\mathbb CT^*X))$ with respect to $(\,\cdot\,|\,\cdot\,)_E$. For $f\in L^2(X,E\otimes\Lambda^\bullet(\mathbb CT^*X))$, denote by $\norm{f}^2_E:=(\,f\,|\,f\,)_E$.
    
Let $\Box_{b,E}: {\rm Dom\,}\Box_{b,E}\subset L^2_{(0,\bullet)}(X,E)\To L^2_{(0,\bullet)}(X,E)$ be the Gaffney extension of Kohn Laplacian with values in $E$ as \eqref{e-suIX} and let $e^{-t\Box_{b,E}}$ be the heat operator of $\Box_{b,E}$ 
    and let $e^{-t\Box_{b,E}}(x,y)\in\mathcal{D}'(\mathbb R_+\times X\times X,(E\otimes T^{*0,\bullet}X)\boxtimes(E\otimes T^{*0,\bullet}X)^*)$
    be the distribution kernel of $e^{-t\Box_{b,E}}$. Let $q\in\set{0,1,\ldots,n}$. We write $\Box^{(q)}_{b,E}:=\Box_{b,E}|_{{\rm Dom\,}\Box_{b,E}\cap L^2_{(0,q)}(X,E)}$,  $e^{-t\Box^{(q)}_{b,E}}:=e^{-t\Box_{b,E}}|_{L^2_{(0,q)}(X,E)}$ and let $e^{-t\Box^{(q)}_{b,E}}(x,y)\in\mathcal{D}'(\mathbb R_+\times X\times X,(E\otimes T^{*0,q}X)\boxtimes(E\otimes T^{*0,q}X)^*)$
    be the distribution kernel of $e^{-t\Box^{(q)}_{b,E}}$. We will also consider $e^{-t\Box_{b,E}}(x,y)$, $e^{-t\Box^{(q)}_{b,E}}(x,y)$, $q=0,1,\ldots,n$, to be elements in $\mathcal{D}'(\mathbb R_+\times X\times X,(E\otimes\Lambda^\bullet(\mathbb CT^*X))\boxtimes(E\otimes (\Lambda^\bullet(\mathbb CT^*X))^*)$ and $e^{-t\Box_{b,E}}$, $e^{-t\Box^{(q)}_{b,E}}$, $t>0$, $q=0,1,\ldots,n$, to be continuous operators: 
\[e^{-t\Box_{b,E}}, e^{-t\Box^{(q)}_{b,E}}: L^2(X,E\otimes\Lambda^\bullet(\mathbb CT^*X))\To  L^2(X,E\otimes\Lambda^\bullet(\mathbb CT^*X)),\]
$q=0,1,\ldots,n$.

Let $\Pi_E: L^2_{(0,\bullet)}(X,E)\To{\rm Ker\,}\Box_{b,E}$ be the orthogonal projection with respect to $(\,\cdot\,|\,\cdot\,)_E$ and let $\Pi_E(x,y)\in\mathcal{D}'(X\times X,(E\otimes T^{*0,\bullet}X)\boxtimes(E\otimes T^{*0,\bullet}X)^*)$
    be the distribution kernel of $\Pi_E$.  Let $q\in\set{0,1,\ldots,n}$. We write  $\Pi^{(q)}_E:=\Pi_E|_{L^2_{(0,q)}(X,E)}$ and let $\Pi^{(q)}_E(x,y)\in\mathcal{D}'(X\times X,(E\otimes T^{*0,q}X)\boxtimes(E\otimes T^{*0,q}X)^*)$
    be the distribution kernel of $\Pi^{(q)}_E$. We will also consider $\Pi_E(x,y)$, $\Pi^{(q)}_E(x,y)$, $q=0,1,\ldots,n$, to be elements in $\mathcal{D}'(X\times X,(E\otimes\Lambda^\bullet(\mathbb CT^*X))\boxtimes(E\otimes (\Lambda^\bullet(\mathbb CT^*X))^*)$ and $\Pi_E$, $\Pi^{(q)}_E$, $q=0,1,\ldots,n$, to be continuous operators: 
\[\Pi_E, \Pi^{(q)}_E: L^2(X,E\otimes\Lambda^\bullet(\mathbb CT^*X))\To  L^2(X,E\otimes\Lambda^\bullet(\mathbb CT^*X)),\]
$q=0,1,\ldots,n$.

\subsection{CR manifolds with $S^1$-action}
Assume that $X$ admits an $S^1$-action.
\begin{definition}\label{d-215605072026}
		We say that the $S^1$-action $e^{i\theta}$ is CR if
		\[
		[T, C^\infty(X, T^{1,0}X)] \subset C^\infty(X, T^{1,0}X)
		\]
		and the $S^1$-action is transversal if for each $x \in X$,
		\[
		\mathbb{C}T_xX = \mathbb{C}T(x) \oplus T_x^{1, 0}X \oplus T_x^{0,1}X. 
		\]
	\end{definition}
	
	Assume that $(X, T^{1,0}X)$ is a CR manifold of dimension $2n+1, n \ge 1$, with a transversal CR $S^1$-action $e^{i\theta}$ and we let $T$ be the infinitesimal generator of the $S^1$-action. 
    
	Suppose that $X$ admits an $S^1$-invariant complete Hermitian metric $\langle \cdot \, | \, \cdot \rangle$ on $\mathbb{C}TX$ so that we have the orthogonal decomposition
	\[
	\mathbb{C}TX = T^{1, 0}X \oplus T^{0, 1}X \oplus \left\{ \lambda T \mid \lambda \in \mathbb{C} \right\}
	\] 
	and $|T|^2 = \langle\, T \, | \, T \,\rangle =1$.
	
	For $u \in \Omega^{0, q}(X)$, we define
	\begin{equation}\label{e-gue220005072026}
	Tu := \frac{\partial}{\partial \theta} ((e^{i\theta})^*u) \mid_{\theta = 0} \in \Omega^{0, q}(X), 
	\end{equation}
	where $(e^{i\theta})^* : T^{*0,q}_{e^{i\theta} \circ x}X \to T^{*0,q}_x X$ is the pull-back map of $e^{i\theta}$. Since the $S^1$-action is CR, as in the $S^1$-action case (see Section 2.4 in \cite{HLM21}), we have
	\[
	T\ddbar_b = \ddbar_b T \quad \text{on} \ \Omega^{0, q}(X).
	\]
	
	Let $(L, h^L)$ be a rigid CR line bundle over $X$ with an $S^1$-invariant Hermitian metric $h^L$ on $L$. We refer the reader to Section 2.3 in \cite{HHL22} for the definitions and terminology about rigid CR vector bundles. We recall the following (see Definition 2.9 in \cite{HHL22})
	
	\begin{definition}\label{D:114511262024}
		A CR vector bundle $E$ of rank $r$ over $X$ with a CR vector bundle lift $T^E$ of $T$ is called rigid CR (with respect to $T^E$) if for every point $p \in X$ there exists an open neighborhood $U$ around $p$ and a CR frame $\{f_1, \cdots, f_r \}$ of $E|_U$ with $T^E(f_j) = 0$ for $1 \le j \le r$.
	\end{definition} 
	
	A section $s \in C^\infty(X, E)$ is called a rigid CR section if $T^E s =0$ and $\ddbar_b s = 0$. The frame $\{f_j \}_{j=1}^r$ in Definition~\eqref{D:114511262024} is called a rigid CR frame of $E|_U$. Note that any rigid CR vector bundle is locally CR trivializable.
	
	For every $k \in \mathbb{N}$, we use $T$ to denote the CR lifting $T^{L^k}$. For $u \in \Omega^{0, q}(X, L^k)$, we can define $Tu$ in the standard way (see Section 2.3 in \cite{HHL22}) and we have $Tu \in \Omega^{0, q}(X, L^k)$.
	
    Consider the operator
	\[
	-iT : \Omega^{0, q}(X, L^k) \to \Omega^{0, q}(X, L^k)
	\]
	and we extend $-iT$ to $L^2_{(0,q)}(X, L^k)$ by
	\begin{equation}\label{e-gue220805072026}
		\begin{split}
			& -iT : \mbox{Dom} \, (-iT) \subset   L^2_{(0,q)}(X, L^k) \to L^2_{(0,q)}(X, L^k),  \\
			& \mbox{Dom} \, (-iT) = \left\{ u \in L^2_{(0,q)}(X, L^k) ; -iTu \in L^2_{(0,q)}(X, L^k)  \right\}.
		\end{split}
	\end{equation}
	It is known that \cite[Corollary 4.3]{HMW22} $-iT$ is self-adjoint. 
	
	Fix $\delta > 0$, let
	\begin{equation}
		\begin{split}
			& L^2_{(0, q), \le k\delta}(X, L^k) : = E_{-iT} \left(\left[-k\delta, k\delta \right] \right),  \\
			& \Omega^{0, q}_{\le k \delta} := \Omega^{0, q}(X, L^k) \cap L^2_{(0, q), \le k\delta}(X, L^k),
		\end{split}
	\end{equation}
	where $E_{-iT}$ denotes the spectral measure of $-iT$. Let
	\[
	Q_{X, \le k\delta}: L^2_{(0, q)}(X, L^k) \to L^2_{(0, q), \le k \delta}(X, L^k)
	\]
	be the orthogonal projection with respect to $(\,\cdot \, | \,\cdot \,)_{h^{L^k}}$. We have the following partial $\ddbar_{b, k}$-complex:
	\[
	\cdots \longrightarrow \Omega^{0, q-1}_{\le k \delta}(X, L^k) \stackrel{\ddbar_{b, k}}{\longrightarrow}\Omega^{0, q}_{\le k \delta}(X, L^k) \stackrel{\ddbar_{b, k}}{\longrightarrow} \Omega^{0, q+1}_{\le k \delta}(X, L^k) \longrightarrow \cdots
	\]
	and put
	\[
	H^q_{b, \le k\delta}(X, L^k) := \frac{\Ker \ddbar_{b,k}:  \Omega^{0, q}_{\le k \delta}(X, L^k)  \to \Omega^{0, q+1}_{\le k \delta}(X, L^k) }{\mbox{Im} \ddbar_{b, k}: \Omega^{0, q-1}_{\le k \delta}(X, L^k)  \to \Omega^{0, q}_{\le k \delta}(X, L^k) }, \quad 0 \le q \le n.
	\]
	
	Let
	\begin{equation}
		\begin{split}
			& \Box^{(q)}_{b, k, \le k\delta} : \mbox{Dom} \, \Box^{(q)}_{b, k, \le k\delta} \subset L^2_{(0, q), \le k\delta}(X, L^k) \to  L^2_{(0, q), \le k\delta}(X, L^k),    \\
			&  \mbox{Dom} \, \Box^{(q)}_{b, k, \le k\delta} =  \mbox{Dom} \, \Box^{(q)}_{b, k}   \cap  L^2_{(0, q), \le k\delta}(X, L^k),  \\
			& \Box^{(q)}_{b, k, \le k\delta} = \ddbar^*_{b, k}\ddbar_{b, k} + \ddbar_{b, k}\ddbar^*_{b, k} \quad \text{on}\ \mbox{Dom} \, \Box^{(q)}_{b, k, \le k\delta}. \ 
		\end{split}
	\end{equation} 
	It is known that
	\[
	H^q_{b, \le k\delta}(X, L^k) \cong \Ker \Box^{(q)}_{b, k, \le k\delta},     \quad \dim H^q_{b, \le k\delta}(X, L^k)  < +\infty.
	\]

\section{The small time asymptotics of CR heat kernel}\label{s-gue250401yyd}
	
	In this section, we will establish asymptotic expansions  of $e^{-t\Box^{(q)}_b}(x,y)$ as $t\To0^+$. We first introduce some symbol spaces.  
	
\subsection{Some symbol classes}\label{s-gue250401yydI}	

Let $D$ be a local coordinate patch of $X$ with local coordinates $x=(x_1,\ldots,x_{2n+1})$. 

	\begin{definition}\label{d-gue250401yyd}
		Let $m \in \mathbb{R}$. Let $\widehat{S}^m(\overline{\mathbb{R}}_+ \times D \times \dot{\mathbb{R}}^{2n+1}, T^{\ast 0, q}X \boxtimes (T^{\ast 0, q}X)^\ast)$ be the space of all $a(t, x, \eta) \in\mathcal{C}^\infty(\overline{\mathbb{R}}_+ \times D \times \dot{\mathbb{R}}^{2n+1}, T^{\ast 0, q}X \boxtimes (T^{\ast 0, q}X)^\ast)$ such that for all $\alpha, \beta \in \mathbb{N}_0^{2n+1}$, $\gamma \in \mathbb{N}_0$ and every compact set $K \subset D$, there exists a constant $C_{\alpha, \beta, \gamma, K}>0$ such that
		\[
		|\partial_t^\gamma \partial_x^\alpha \partial_\eta^\beta a(t, x, \eta)| \le C_{\alpha, \beta, \gamma, K}(1+|\eta|)^{m-|\beta|+\gamma}
		\]
		for all $x \in K$, all $|\eta| \ge 1$, and all $t > 0$. 
		
		We say that $a(t, x, \eta) \in \mathcal{C}^\infty(\overline{\mathbb{R}}_+ \times D \times \dot{\mathbb{R}}^{2n+1}, T^{\ast 0, q}X \boxtimes (T^{\ast 0, q}X)^\ast)$ is quasi-homogeneous of degree $j$ if 
		\[
		a(t, x, \lambda \eta) = \lambda^j a(\lambda t, x, \eta), \quad
        \mbox{for all $\lambda > 0$}.
		\]
		Let $a(t, x, \eta) \in \widehat{S}^m(\overline{\mathbb{R}}_+ \times D \times \dot{\mathbb{R}}^{2n+1}, T^{\ast 0, q}X \boxtimes (T^{\ast 0, q}X)^\ast)$. We write 
		\[
		a(t, x, \eta) \sim \sum_{j=0}^\infty a_j(t, x, \eta) \quad \text{in} \quad  \widehat{S}^m(\overline{\mathbb{R}}_+ \times D \times \dot{\mathbb{R}}^{2n+1}, T^{\ast 0, q}X \boxtimes (T^{\ast 0, q}X)^\ast),
		\]
		where $a_j(t, x, \eta) \in \widehat{S}^{m-j}(\overline{\mathbb{R}}_+ \times D \times \dot{\mathbb{R}}^{2n+1}, T^{\ast 0, q}X \boxtimes (T^{\ast 0, q}X)^\ast)$, $j=0, 1, \cdots,$ if for every $N \in \mathbb{N}$,
		\[
		a - \sum_{j=0}^N a_j \in \widehat{S}^{m-N-1}(\overline{\mathbb{R}}_+ \times D \times \dot{\mathbb{R}}^{2n+1}, T^{\ast 0, q}X \boxtimes (T^{\ast 0, q}X)^\ast).
		\]
		
		Let $a_j \in \widehat{S}^{m-j}(\overline{\mathbb{R}}_+ \times D \times \dot{\mathbb{R}}^{2n+1}, T^{\ast 0, q}X \boxtimes (T^{\ast 0, q}X)^\ast)$, $j=0,1,2, \cdots$. By Borel construction, we can find 
		\[
		a \in \widehat{S}^m(\overline{\mathbb{R}}_+ \times D \times \dot{\mathbb{R}}^{2n+1}, T^{\ast 0, q}X \boxtimes (T^{\ast 0, q}X)^\ast)
		\]
		unique modulo 
		\begin{equation*}
        \begin{split}
		\widehat{S}^{-\infty}&(\overline{\mathbb{R}}_+ \times D \times \dot{\mathbb{R}}^{2n+1}, T^{\ast 0, q}X \boxtimes (T^{\ast 0, q}X)^\ast)\\
       & := \cap_{m \in \mathbb{R}} \widehat{S}^m(\overline{\mathbb{R}}_+ \times D \times \dot{\mathbb{R}}^{2n+1}, T^{\ast 0, q}X \boxtimes (T^{\ast 0, q}X)^\ast)
        \end{split}
        \end{equation*}
		 such that
		\[
		a - \sum_{j=0}^N a_j \in \widehat{S}^{m-N-1}(\overline{\mathbb{R}}_+ \times D \times \dot{\mathbb{R}}^{2n+1}, T^{\ast 0, q}X \boxtimes (T^{\ast 0, q}X)^\ast), \quad \text{for all $N \in \mathbb{N}$.}
		\]

        Let $B$ be a vector bundle over $\overline{\mathbb{R}}_+ \times D \times \dot{\mathbb{R}}^{2n+1}$. For $m\in\mathbb R$, we can define 
        $\widehat{S}^m(\overline{\mathbb{R}}_+ \times D \times \dot{\mathbb{R}}^{2n+1}, B)$ and asymptotic sum in $\widehat{S}^m(\overline{\mathbb{R}}_+ \times D \times \dot{\mathbb{R}}^{2n+1}, B)$ in the same ways. 
	\end{definition}
	
	Similarly, we can define
	
	\begin{definition}\label{d-gue250404yyd}
		Let $m \in \mathbb{R}$ and let $\varepsilon \ge 0$. Let $\widehat{S}^m_\varepsilon(\overline{\mathbb{R}}_+ \times D \times D \times\mathbb{R}_+, T^{\ast 0, q}X \boxtimes (T^{\ast 0, q}X)^\ast)$ be the space of all $a(t, x, y, s) \in \mathcal{C}^\infty(\overline{\mathbb{R}}_+ \times D \times D\times\mathbb{R}_+, T^{\ast 0, q}X \boxtimes (T^{\ast 0, q}X)^\ast)$ such that for all $\alpha, \beta \in \mathbb{N}_0^{2n+1}$, $\gamma_1, \gamma_2 \in \mathbb{N}_0$ and every compact set $K \subset D$, there exists a constant $C_{\alpha, \beta, \gamma_1, \gamma_2, K}>0$ such that
		\[
		|\partial_t^{\gamma_1} \partial_s^{\gamma_2} \partial_x^\alpha \partial_y^\beta a(t, x, y, s)| \le C_{\alpha, \beta, \gamma_1, \gamma_2, K}(1+|s|)^{m+\gamma_1-\gamma_2} e^{-\varepsilon t s}
		\]
		for all $(x, y) \in K \times K$ and all $s \ge 1$, and all $t \in \mathbb{R}_+$.

		Let $a(t, x, y, s) \in \widehat{S}^m_\varepsilon(\overline{\mathbb{R}}_+ \times D \times D \times\mathbb{R}_+, T^{\ast 0, q}X \boxtimes (T^{\ast 0, q}X)^\ast)$. We write 
		\[
		a(t, x, y, s) \sim \sum_{j=0}^\infty a_j(t, x, y, s) \quad \text{in} \quad  \widehat{S}^m_\varepsilon(\overline{\mathbb{R}}_+ \times D \times D \times \overline{\mathbb{R}}_+, T^{\ast 0, q}X \boxtimes (T^{\ast 0, q}X)^\ast),
		\]
		where $a_j(t, x, y, s) \in \widehat{S}^{m-j}_\varepsilon(\overline{\mathbb{R}}_+ \times D \times D \times \mathbb{R}_+, T^{\ast 0, q}X \boxtimes (T^{\ast 0, q}X)^\ast)$, $j=0, 1, \cdots,$ if for every $N \in \mathbb{N}$,
		\[
		a - \sum_{j=0}^N a_j \in \widehat{S}^{m-N-1}_\varepsilon(\overline{\mathbb{R}}_+ \times D \times D\times\mathbb{R}_+, T^{\ast 0, q}X \boxtimes (T^{\ast 0, q}X)^\ast).
		\]
		
		Let $a_j \in \widehat{S}^{m-j}_\varepsilon(\overline{\mathbb{R}}_+ \times D \times D \times\mathbb{R}_+, T^{\ast 0, q}X \boxtimes (T^{\ast 0, q}X)^\ast)$, $j=0, 1, 2, \cdots$. By Borel construction, we can find 
		\[
		a \in \widehat{S}^m_\varepsilon(\overline{\mathbb{R}}_+ \times D \times D \times\mathbb{R}_+, T^{\ast 0, q}X \boxtimes (T^{\ast 0, q}X)^\ast)
		\] 
		unique modulo 
		\begin{equation*}
        \begin{split}
		\widehat{S}^{-\infty}_\varepsilon&(\overline{\mathbb{R}}_+ \times D \times D \times \mathbb{R}_+, T^{\ast 0, q}X \boxtimes (T^{\ast 0, q}X)^\ast)\\
        &:= \cap_{m \in \mathbb{R}} \widehat{S}^m_\varepsilon(\overline{\mathbb{R}}_+ \times D \times D \times\mathbb{R}_+, T^{\ast 0, q}X \boxtimes (T^{\ast 0, q}X)^\ast)
        \end{split}
		\end{equation*}
		such that
		\[
		a - \sum_{j=0}^N a_j \in \widehat{S}^{m-N-1}_\varepsilon(\overline{\mathbb{R}}_+ \times D \times D \times\mathbb{R}_+, T^{\ast 0, q}X \boxtimes (T^{\ast 0, q}X)^\ast), \quad \text{for all $N \in \mathbb{N}$.}
		\]

        Let $B$ be a vector bundle over $\overline{\mathbb{R}}_+ \times D\times D \times\mathbb R_+$. For $m\in\mathbb R$, $\varepsilon\geq0$, we can define 
        $\widehat{S}^m_\varepsilon(\overline{\mathbb{R}}_+ \times D \times D\times\mathbb R_+, B)$ and asymptotic sum in $\widehat{S}^m_\varepsilon(\overline{\mathbb{R}}_+ \times D\times D \times\mathbb R_+, B)$ in the same ways. 
	\end{definition} 

    \begin{definition}\label{d-gue250404yydI}
		Let $m \in \mathbb{R}$ and let $\varepsilon \ge 0$. Let $\widehat{S}^m_{\varepsilon,{\rm cl\,}}(\overline{\mathbb{R}}_+ \times D \times D \times\mathbb{R}_+, T^{\ast 0, q}X \boxtimes (T^{\ast 0, q}X)^\ast)$ be the space of all $a(t, x, y, s) \in \widehat{S}^m_\varepsilon(\overline{\mathbb{R}}_+ \times D \times D \times\mathbb{R}_+, T^{\ast 0, q}X \boxtimes (T^{\ast 0, q}X)^\ast)$ such that 
        \[
		a(t, x, y, s) \sim \sum_{j=0}^\infty a_j(t, x, y, s) \quad \text{in} \quad  \widehat{S}^m_\varepsilon(\overline{\mathbb{R}}_+ \times D \times D \times\mathbb{R}_+, T^{\ast 0, q}X \boxtimes (T^{\ast 0, q}X)^\ast),
		\]
		where $a_j(t, x, y, s) \in \widehat{S}^{m-j}_\varepsilon(\overline{\mathbb{R}}_+ \times D \times D \times \mathbb{R}_+, T^{\ast 0, q}X \boxtimes (T^{\ast 0, q}X)^\ast)$, 
        \[\mbox{$a_j(t,x,y,\lambda s)=\lambda^{m-j}a_j(\lambda t,x,y,s)$, \quad for every $s>0$, $\lambda\geq1$, $t\in\mathbb R_+$},\]
        $j=0, 1, \cdots$. 

 Let $B$ be a vector bundle over $\overline{\mathbb{R}}_+ \times D\times D \times\mathbb R_+$. Let $m\in\mathbb R$ and let $\varepsilon\geq0$. We can define 
        $\widehat{S}^m_{\varepsilon,{\rm cl\,}}(\overline{\mathbb{R}}_+ \times D \times D \times\mathbb{R}_+, B)$ in the same way. 
	\end{definition}

Let $\Omega\subset\mathbb R^N$ be an open set of $\mathbb R^N$. Let $\ell\in\mathbb N$. Let $E$ be a smooth vector bundle over $\Omega\times\dot{\mathbb R}^\ell$. For $m\in\mathbb R$, let $S^m_{1,0}(\Omega\times\dot{\mathbb{R}}^\ell,E)$ be the H\"ormander symbol space with values in $E$ of order $m$ and of type $(1,0)$. 
	
\subsection{Small-time asymptotics}\label{s-gue250401yyda}

Recall that we work with the assumption that the Levi form is non-degenerate of constant signature $(n_-, n_+)$, $n_-+n_+=n$. Fix $q \in \{0, 1, \cdots, n\}$. Fix $p \in X$ and let $x = (x_1, x_2, \cdots, x_n)$ be local coordinates of $X$ defined on an open set $D$ of $X$, $p \in D$, $x(p) = 0 \in \mathbb{R}^{2n+1}$, $T = -\frac{\partial}{\partial x_{2n+1}}$ on $D$, $\omega_0(0)=dx_{2n+1}$. 
	
    It was shown in \cite{Hsiao08} that on $D\times D$, we can find 
	\begin{equation}\label{e-gue250402yydI}
	A(t, x, y)=\frac{1}{(2\pi)^{2n+1}}\int e^{i(\Psi(t,x,\eta)-<y,\eta>)}a(t,x,\eta)d\eta,
	\end{equation}
	where there is a constant $C>1$ such that 
    \[\frac{1}{C}d((x,\frac{\eta}{\abs{\eta}}),\Sigma)^2\frac{t\abs{\eta}^2}{1+t\abs{\eta}}\leq{\rm Im\,}\Psi(t,x,\eta)\leq Cd((x,\frac{\eta}{\abs{\eta}}),\Sigma)^2\frac{t\abs{\eta}^2}{1+t\abs{\eta}},\]
    for all $t\in\mathbb R_+$, $\abs{\eta}\geq1$, 
	\[
	\begin{split}
		& \Psi(0, x, \eta) = \langle x, \eta \rangle, \\
        & \Psi(t, x, \eta) = \langle x, \eta \rangle\ \ \mbox{on $\Sigma$},\\
		& \Psi(t, x, \lambda \eta) = \lambda \Psi(\lambda t, x ,\eta), \\
		& d_{x, \eta}(\Psi(t, x, \eta) - \langle x, \eta \rangle) = 0 \quad \text{on} \quad \Sigma, \\
		& a(t,x,\eta)\sim\sum^\infty_{j=0}a_j(t,x,\eta) \quad \text{in} \quad \widehat{S}^0(\overline{\mathbb{R}}_+ \times D \times \dot{\mathbb{R}}^{2n+1}, T^{\ast 0,q}X\boxtimes (T^{\ast 0,q}X)^\ast), \\
		& a_j(t, x, \lambda \eta) = \lambda^{-j}a_j(\lambda t, x, \eta), \quad \mbox{if $|\eta| \ge 1$, $t>0$, $\lambda \ge 1$, $j=0,1,\cdots$},
	\end{split}
	\]
	and for every $\alpha, \beta \in \mathbb{N}^{2n+1}_0$, every $j=0, 1, \cdots$, and every compact set $K \subset D$, there are constants $C_{K,\alpha,\beta}>0$, $C_{K, \alpha, \beta, j}>0$, $j=0,1,\ldots$, such that
	\[\begin{split}
&| \partial_x^\alpha \partial_\eta^\beta (\Psi(t, x, \eta) - \Psi(\infty, x, \eta)) | \le C_{K,\alpha,\beta} e^{-\varepsilon t |\eta|}(1+|\eta|)^{-|\beta|},\\
 &| \partial_x^\alpha \partial_\eta^\beta (a(t, x, \eta) - a(\infty, x, \eta)) | \le C_{K,\alpha,\beta} e^{-\varepsilon t |\eta|}(1+|\eta|)^{-|\beta|},\\
	&| \partial_x^\alpha \partial_\eta^\beta (a_j(t, x, \eta) - a_j(\infty, x, \eta)) | \le C_{K, \alpha,\beta,j} e^{-\varepsilon t |\eta|}(1+|\eta|)^{-j-|\beta|},
	\end{split}\]
	where $\varepsilon >0$ is a constant independent of $j, \alpha, \beta$ and $\Psi(\infty,x,\eta)\in\mathcal{C}^\infty(T^*D)$, $\Psi(\infty,x,\lambda\eta)=\lambda\Psi(\infty,x,\eta)$, for all $\abs{\eta}\geq1$, $\lambda\geq1$, 
 $d_{x, \eta}(\Psi(\infty, x, \eta) - \langle x, \eta \rangle) = 0$ on $\Sigma$, $\Psi(\infty, x, \eta)=\langle\,x\,,\,\eta\,\rangle$ on $\Sigma$, there is a constant $C>1$ such that 
 \[\frac{1}{C}\abs{\eta}d((x,\frac{\eta}{\abs{\eta}}),\Sigma)^2\leq{\rm Im\,}\Psi(\infty,x,\eta)\leq C\abs{\eta}d((x,\frac{\eta}{\abs{\eta}}),\Sigma)^2,\]
    for all $\abs{\eta} \ge 1$, 
\[
	\begin{split}
    &a(\infty,x,\eta)\in S^0_{1,0}(D\times\dot{\mathbb R}^{2n+1},T^{\ast 0,q} X \boxtimes (T^{\ast 0,q}X)^\ast),\\
&\mbox{$a(\infty,x,\eta)\sim\sum^{+\infty}_{j=0}a_j(\infty,x,\eta)$ \quad in \quad $S^0_{1,0}(D\times\dot{\mathbb R}^{2n+1},T^{\ast 0,q}X\boxtimes(T^{\ast 0,q}X)^\ast)$},\\
    & a_j(\infty, x,\eta) \in \mathcal{C}^\infty(T^\ast D, T^{\ast 0,q} X \boxtimes (T^{\ast 0,q}X)^\ast), \\
		& a_j(\infty, x, \lambda \eta) = \lambda^{-j}a_j(\infty, x, \eta), \quad  \text{if $|\eta| \ge 1$,\ $\lambda \ge 1$, \ $j=0, 1, \cdots$}.
	\end{split}
	\]
	When $q \not= n_-$, 
	\[
	a(\infty,x,\eta)=0,\ \  a_j(\infty, x, \eta) = 0 \quad \text{if $\eta_{2n+1} < 0$, \ $j=0,1, \cdots$},
	\]
	and when $q \not= n_+$, 
	\[
	a(\infty,x,\eta)=0,\ \ a_j(\infty, x, \eta) = 0 \quad \text{if $\eta_{2n+1} > 0$, \ $j=0,1, \cdots$},
	\]
	such that
	\begin{equation}\label{E:010820252208}
		\left\{
		\begin{split} &  A'(t) + \Box_b^{(q)}A(t) = R(t) \quad \text{on} \ D, \\
			& A(0) = I,  
		\end{split}
		\right.
	\end{equation}
	where $A(t) : \Omega_c^{0,q}(D) \to \Omega^{0, q}(D)$ is the continuous operator with distribution kernel $A(t, x, y)$, and $R(t) : \Omega^{0,q}_c(D) \to \Omega^{0, q}(D)$ is a continuous operator with distribution kernel $R(t, x, y) \in\mathcal{C}^\infty (\overline{\mathbb{R}}_+ \times D \times D, T^{\ast 0, q}X \boxtimes (T^{\ast 0, q}X)^\ast)$. 
	
	Let $\eta' = (\eta_1, \cdots, \eta_{2n})$, $x' = (x_1, \cdots, x_{2n})$. Then, by \cite[Proposition 3.2]{Hsiao08} and repeat the proof of~\cite[Lemma 7.14]{Hsiao08}, we have
    
	\begin{lemma}\label{L:010820252224}
		For fixed $t > 0$ and $\eta_{2n+1} \not= 0$, 
		\[
		d_{\eta'}(\Psi(t, x, \eta) - \langle y, \eta \rangle)\big|_{\eta'=0, x= y = 0} = 0,
		\]
		and
		\[
		\left( \frac{\partial^2}{\partial \eta_j \partial \eta_l} \left( \Psi(t, x, \eta) - \langle y, \eta \rangle \right)\big|_{\eta'=0, x = y = 0} \right)_{j, l =1}^{2n}
		\]
		is invertible.
	\end{lemma}
	
	From Lemma~\ref{L:010820252224}, we can apply complex stationary phase formula of Melin-Sj\"{o}strand to carry out the integration of $\eta'$ in \eqref{e-gue250402yydI} and get
	
	\begin{theorem}\label{T:0111202511001}
		With the notations used above, on $D\times D$, we have
		\begin{equation}\label{e-gue250402ycdb}
        \begin{split}
		A(t, x, y) =& \int_0^\infty e^{i \varphi_-(t, x, y, s)}a_-(t, x, y, s) ds\\
        &+  \int_0^\infty e^{i \varphi_+(t, x, y, s)}a_+(t, x, y, s) ds + \mathcal{E}(t, x, y),
        \end{split}
		\end{equation}
		where 
		\begin{equation}\label{E:011020251803}
		\begin{split}
                & \mathcal{E}(t, x, y) \in \mathcal{C}^\infty(\overline{\mathbb{R}}_+ \times D \times D, T^{\ast 0,q}X\boxtimes (T^{\ast 0,q}X)^\ast), \\
			& \varphi_\pm(t, x, y, s) \in\mathcal{C}^\infty(\overline{\mathbb{R}}_+ \times D \times D \times\mathbb{R}_+), \\
			& \operatorname{Im} \varphi_\pm (t, x, y, s) \ge 0, \\
			& \varphi_\pm(t, x, x, s) = 0,\ \ \mbox{for all $x\in D$}, \\
            &(d_x\varphi_\pm)(t, x, x, s) = \pm \omega_0(x)s, \quad \text{for all $(t, x, s) \in \mathbb{R}_+ \times D \times \mathbb{R}_+$},\\
			& \varphi_\pm (t, x, y, \lambda s) = \lambda \varphi_\pm(\lambda t, x, y, s), \ \ \mbox{for all $s>0$, $\lambda \ge 1$}, \\
			& a_\pm(t, x, y, s) \in\widehat S^n_{0,{\rm cl\,}}( \overline{\mathbb{R}}_+ \times D \times D \times\mathbb{R}_+, T^{\ast 0, q}X \boxtimes (T^{\ast 0, q}X)^\ast), \\
			& a_\pm(t,x, y, s)\sim\sum^\infty_{j=0}a^j_\pm(t,x, y, s)  \ \text{in}\ \widehat{S}^n_0(\overline{\mathbb{R}}_+ \times D \times D \times \mathbb{R}_+, T^{\ast 0,q}X\boxtimes (T^{\ast 0,q}X)^\ast), \\
			& a^j_\pm(t, x, y, s) \in C^\infty(\overline{\mathbb{R}}_+ \times D \times D \times \mathbb{R}_+, T^{\ast 0,q}X\boxtimes (T^{\ast 0,q}X)^\ast), j=0,1, \cdots, \\
			& a^j_\pm (t, x, y, \lambda s) = \lambda^{n-j}a^j_\pm(\lambda t, x, y, s),\quad t, s \not= 0,\ \lambda \ge 1, \ j=0,1,\cdots.
		\end{split}
		\end{equation}
		Moreover, we can find $a_{\pm}(x,y,s)\in S^n_{1,0}(D\times D\times\mathbb R_+,T^{*0,q}X\boxtimes(T^{*0,q}X)^*)$, 
        \[\mbox{$a_{\pm}(x,y,s)\sim\sum^{+\infty}_{j=0}a^j_{\pm}(x,y)s^{n-j}$ in $S^n_{1,0}(D\times D\times\mathbb R_+,T^{*0,q}X\boxtimes(T^{*0,q}X)^*)$},\] 
        $a^j_{\pm}(x,y)\in\mathcal{C}^\infty(D\times D,T^{*0,q}X\boxtimes(T^{*0,q}X)^*)$, $j=0,1,\ldots$, with  $a_-(x,y,s)=0$, $a^j_-(x,y)=0$, $j=0,1,\ldots$, if $q\neq n_-$, $a_+(x,y,s)=0$, $a^j_+(x,y)=0$, $j=0,1,\ldots$, if $q\neq n_+$, $a^0_-(x,y)\neq0$ if $q=n_-$, $a^0_+(x,y)\neq0$ if $q = n_+$, such that 
        \begin{equation}\label{e-gue250404yydu}
		\begin{split}
		& a_\pm(t, x, y, s) - a_\pm(x, y, s) \in \widehat{S}^n_\varepsilon( \overline{\mathbb{R}}_+ \times D \times D \times \overline{\mathbb{R}}_+, T^{\ast 0,q}X\boxtimes (T^{\ast 0,q}X)^\ast ), \\
		& a^j_\pm(t, x, y, s) - a^j_\pm(x, y)s^{n-j} \in \widehat{S}^{n-j}_\varepsilon( \overline{\mathbb{R}}_+ \times D \times D \times \overline{\mathbb{R}}_+, T^{\ast 0,q}X\boxtimes (T^{\ast 0,q}X)^\ast ), 
		\end{split}
		\end{equation}
		where $j = 0, 1, 2, \cdots$,$\varepsilon > 0$. Furthermore, for every $\alpha, \beta \in \mathbb{N}_0^{2n+1}$ and for every compact set $K \subset D$, there is a constant $C_{\alpha, \beta, K}>0$ such that 
		\begin{equation}\label{e-gue250404yydv}
		|\partial_x^\alpha \partial_y^\beta (\varphi_\pm(t, x, y, s) - s \varphi_\pm(x, y))| \le C_{\alpha, \beta, K}e^{-ts}(1+|s|),
		\end{equation}
		for every $(t, x, y, s) \in \overline{\mathbb{R}}_+ \times K \times K \times \overline{\mathbb{R}}_+$ 
		where $\varphi_\pm(x, y) \in C^\infty(D \times D)$ is the phase appearing in the description of the Szeg\H{o} projection (see~\cite[Theorem 1.2]{Hsiao08},~\cite[Theorem 4.1]{HM17}).
	\end{theorem} 

    The following is known~(see~\cite[Page 147]{MS74}) 

\begin{lemma} \label{l-gue250402yyd}
In some open neighborhood $\Omega$ of $p$ in $D$, for all $t\in\mathbb R_+$, we have
\begin{equation} \label{e:0709212032} 
\begin{split}
{\rm Im\,}&\varphi_{+}(t,x,y,1)  \\
&\geq c\frac{t}{1+t}\inf_{w\in W}\Big({\rm Im\,}\Psi(t, x, (w, 1))
+\abs{d_w(\Psi(t, x, (w, 1))-\langle\,y, (w, 1)\,\rangle)}^2\Big), \\ 
{\rm Im\,}&\varphi_{-}(t,x,y,1) \\
&\geq c\frac{t}{1+t}\inf_{w\in W}\Big({\rm Im\,}\Psi(t, x, (w, -1))
+\abs{d_w(\Psi(t, x, (w, -1))-\langle\,y, (w, -1)\,\rangle)}^2\Big), \\ 
&(x, y)\in \Omega\times \Omega,
\end{split}
\end{equation}
where $c$ is a positive constant and $W$ is some open set of the origin in $\Real^{2n}$.
\end{lemma} 

We can now prove the following

\begin{theorem}\label{T:0111202511002}
		With the notations and assumptions used above, for every compact set $K \subset D$, there is a constant $C_K > 0$ such that 
		\begin{equation}\label{E:011020251802}
			\operatorname{Im} \varphi_\pm(t, x, y, s) \ge C_K \frac{ts^2}{1+ts} |x'-y'|^2, \quad 
            \text{for all $(x, y) \in K \times K$},
		\end{equation} 
        where $x'=(x_1,\ldots,x_{2n})$, $y'=(y_1,\ldots,y_{2n})$.
	\end{theorem}

    \begin{proof}
    Let $\Omega$ be as in Lemma~\ref{l-gue250402yyd}.
For $x\in\Omega$, let $\sigma(x)\in\mathbb R^{2n}$ so that $(x,(\sigma(x),1))\in\Sigma^+$.
From 
$\Psi(t, x, (w, 1))-\langle\,y,\,(w, 1)\,\rangle=\langle\,x-y\,,\, (w, 1)\,\rangle+O(\abs{w-\sigma(x)}^2)$
we can check that
\[d_w\big(\Psi(t, x, (w, 1))-\langle\,y\,,\,(w, 1)\,\rangle\big)=\langle\,x'-y'\,,\,dw\,\rangle+O(\abs{w-\sigma(x)}),\]
where $x'=(x_1,\ldots,x_{2n})$, $y'=(y_1,\ldots,y_{2n})$. Thus, there are
constants $c_1$, $c_2>0$ such that
\[\abs{d_w\big(\Psi(t, x, (w, 1))-\langle\,y\,,\,(w, 1)\,\rangle\big)}^2\geq c_1\abs{x'-y'}^2-c_2\abs{w-\sigma(x)}^2\]
for $(x, w)$ in some compact set of $\Omega\times\dot{\mathbb{R}}^{2n}$. If $\frac{c_1}{2}\abs{(x'-y')}^2\geq c_2\abs{w-\sigma(x)}^2$, then
\begin{equation} \label{e:0709231244}
\abs{d_\omega\big(\Psi(t, x, (\omega,1 ))-\langle\,y\,,(\omega, 1)\,\rangle\big)}^2\geq\frac{c_1}{2}\abs{(x'-y')}^2.
\end{equation}
Now, we assume that $\abs{(x'-y')}^2\leq\frac{2c_2}{c_1}\abs{w-\sigma(x)}^2$. We have
\begin{equation} \label{e:0709231245}
{\rm Im\,}\Psi(t, x, (w, 1))\geq c_3\abs{w-\sigma(x)}^2\geq\frac{c_1c_3}{2c_2}\abs{(x'-y')}^2,
\end{equation}
for $(x, w)$ in some compact set of $\Omega\times\dot{\mathbb{R}}^{2n}$, where $c_3$ is a positive constant.
From Lemma~\ref{l-gue250402yyd}, \eqref{e:0709231244} and \eqref{e:0709231245}, we have
${\rm Im\,}\varphi_+(t,x, y,1)\geq c\frac{t}{1+t}\abs{(x'-y')}^2$
for $x$, $y$ in some neighborhood of $p$, where $c$ is a positive constant. By using quasi-homogeneity of $\varphi_+(t,x,y,s)$, we get \eqref{E:011020251802} for $\varphi_+$. The proof of $\varphi_-$ is similar. 
        \end{proof}
	
	
	\begin{theorem}\label{T:0111202511003}
		With the notations and assumptions used above, let $\chi, \tau \in \mathcal{C}_c^\infty(D)$, $\operatorname{supp} \chi \cap \operatorname{supp} \tau = \emptyset$, then 
		\[
		\chi(x) A(t, x, y)\tau(y) \in \mathcal{C}^\infty ( \overline{\mathbb{R}}_+ \times D \times D, T^{\ast0, q}X \boxtimes (T^{\ast 0 ,q}X)^\ast ).
		\]
	\end{theorem} 

    \begin{proof}
Let $\chi, \tau \in C_c^\infty(D)$, $\operatorname{supp} \chi \cap \operatorname{supp} \tau = \emptyset$. Let $x=(x_1,\ldots,x_{2n+1})=(x',x_{2n+1})\in{\rm Supp\,}\chi$, $y=(y_1,\ldots,y_{2n+1})=(y',y_{2n+1})\in{\rm Supp\,}\tau$. If $x_{2n+1}\neq y_{2n+1}$. Form \eqref{E:011020251803}, we have
 \[(d_x\varphi_\pm)(t, x, x, s) = \pm \omega_0(x)s, \quad \text{for all $(t, x, s) \in \mathbb{R}_+ \times D \times \mathbb{R}_+$}.\]
 From this observation, 
we can integrate by parts in $s$ and get $A(t, x, y)\in\mathcal{C}^\infty ( \overline{\mathbb{R}}_+ \times W_1 \times W_2, T^{\ast0, q}X \boxtimes (T^{\ast 0 ,q}X)^\ast )$, where $W_1$ is an open set of $x$ and $W_2$ is an open set of $y$. 

If $x'\neq y'$. By changing $s\To\frac{s}{t}$ in \eqref{e-gue250402ycdb} in Theorem \ref{T:0111202511001}, we get 
\begin{equation}\label{e-gue250402ycdc}
\begin{split}
		A(t, x, y) =& \int_0^\infty e^{i\frac{1}{t}\varphi_-(1, x, y, s)}t^{-1}a_-(t, x, y, \frac{s}{t}) ds\\
        &+  \int_0^\infty e^{i\frac{1}{t}\varphi_+(1, x, y, s)}t^{-1}a_+(t, x, y, \frac{s}{t}) ds + \mathcal{E}(t, x, y).
        \end{split}
		\end{equation}
    From \eqref{E:011020251802} and \eqref{e-gue250402ycdc}, we see that $A(t, x, y)\in\mathcal{C}^\infty ( \overline{\mathbb{R}}_+ \times\hat W_1 \times\hat W_2, T^{\ast0, q}X \boxtimes (T^{\ast 0 ,q}X)^\ast )$, where $\hat W_1$ is an open set of $x$ and $\hat W_2$ is an open set of $y$. The theorem follows. 
\end{proof}
	
	Assume that $X = \cup_{j=1}^N D_j, N \in \mathbb{N}$, where $D_j$ is an open local coordinate patch as above, $j = 1, 2, \cdots, N$. Let $\chi_j \in C^\infty_c (D_j)$, $j=1, 2, \cdots, N$, with $\sum_{j=1}^N \chi_j \equiv 1$ on $X$. Let $A_j(t) : \Omega^{0,q}_c(D_j) \to \Omega^{0,q}(D_j)$ be the continuous operator as in \eqref{E:010820252208} and let $A_j(t, x, y)$ be the distribution kernel of $A_j(t)$, $j=1, 2, \cdots, N$. Let 
	\[
	\widehat{A}(t) := \sum_{j=1}^N \widetilde{\chi}_j A_j(t) \chi_j : \Omega^{0,q}(X) \to \Omega^{0,q}(X),
	\]
	where $\widetilde{\chi}_j \in C^\infty_c(D_j)$, $\widetilde{\chi}_j \equiv 1$ on $\operatorname{supp} \chi_j$, $j=1,2, \cdots, N$. From Theorem~\ref{T:0111202511003}, we get
	\begin{equation}\label{E:011120251102}
		\left\{
		\begin{split} &  \widehat{A}'(t) + \Box_b^{(q)}\widehat{A}(t) = \delta (t), \\
			& \lim_{t \to 0}\widehat{A}(t) = I,  
		\end{split}
		\right.
	\end{equation}
	where $\widehat{A}(t) : \Omega^{0,q}(X) \to \Omega^{0, q}(X)$, the distribution kernel of $\widehat{A}(t)$ is $\widehat{A}(t, x, y)$, and $\delta(t) : \Omega^{0,q}(X) \to \Omega^{0, q}(X)$ is a continuous operator with distribution kernel $\delta(t, x, y) \in \mathcal{C}^\infty (\overline{\mathbb{R}}_+ \times X \times X, T^{\ast 0, q}X \boxtimes (T^{\ast 0, q}X)^\ast)$.

	\begin{theorem}\label{T:011120251108}
		We can find $a_j(x, y) \in\mathcal{C}^\infty(X \times X, T^{\ast0, q}X \boxtimes (T^{\ast 0 ,q}X)^\ast)$, $j=1, 2, \cdots$, such that for all $m, N \in \mathbb{N}$, there is a constant $C_{m, N}>0$ such that for all $t>0$, $t \ll 1$,
		\[
		\Big\| e^{-t\Box_b^{(q)}}(x, y) - \widehat{A}(t, x, y) - \sum_{j=1}^N a_j(x, y)t^j \Big\|_{\mathcal{C}^m(X \times X)} \le C_{m, N}t^{N+1}.
		\]
	\end{theorem} 
	\begin{proof}
		Let 
		$
		a_1(x, y) = - \delta(0, x, y) \in \mathcal{C}^\infty(X \times X, T^{\ast0, q}X \boxtimes (T^{\ast 0 ,q}X)^\ast),
		$ 
		where $\delta(t, x, y)$ is as in \eqref{E:011120251102}. Let 
		\[
		\widehat{A}_1(t, x, y) := \widehat{A}(t, x ,y) + ta_1(x, y)
		\] 
		and let $\widehat{A}_1(t) : \Omega^{0, q}(X) \to \Omega^{0, q}(X)$ be the continuous operator with distribution kernel $\widehat{A}_1(t, x, y)$. It follows from \eqref{E:011120251102} that
		\begin{equation}\label{E:011120251231}
			\widehat{A}'_1(t) + \Box_b^{(q)} \widehat{A}_1(t) = t\delta_1(t),
		\end{equation}
		where $\delta_1(t, x, y) \in \mathcal{C}^\infty(\ol{\mathbb{R}}_+ \times X \times X, T^{\ast0, q}X \boxtimes (T^{\ast 0 ,q}X)^\ast)$, $\delta_1(t, x, y)$ is the distribution kernel of $\delta_1(t)$. 
		
		Let $a_2(x, y) := - \frac{1}{2} \delta_1(0, x, y)$ and let 
		\[
		\begin{array}{lcl}
			\widehat{A}_2(t, x, y) & := & \widehat{A}_1(t, x, y) + t^2a_2(x, y) \\
			& = & \widehat{A}(t, x, y) + ta_1(x, y) + t^2a_2(x, y).
		\end{array}
		\]
		Let $\widehat{A}_2(t) : \Omega^{0,q}(X) \to \Omega^{0, q}(X)$ be the continuous operator with distribution kernel $\widehat{A}_2(t, x, y)$. From \eqref{E:011120251231}, we have
		\[
		\widehat{A}'_2(t) + \Box_b^{(q)} \widehat{A}_2(t) = t^2\delta_2(t),
		\]
		where $\delta_2(t, x, y) \in \mathcal{C}^\infty(\ol{\mathbb{R}}_+ \times X \times X, T^{\ast0, q}X \boxtimes (T^{\ast 0 ,q}X)^\ast)$, $\delta_2(t, x, y)$ is the distribution kernel of $\delta_2(t)$. Continuing in this way, we get $a_j(x, y) \in\mathcal{C}^\infty(X \times X, T^{\ast0, q}X \boxtimes (T^{\ast 0 ,q}X)^\ast)$, $j=1, 2, \dots$, such that, for every $N \in \mathbb{N}$, let
		\begin{equation}\label{e-gue250529ycdw}
		\widehat{A}_N(t, x, y) := \widehat{A}(t, x ,y) + \sum_{j=1}^N t^ja_j(x, y)
		\end{equation}
		and let $\widehat{A}_N(t) : \Omega^{0, q}(X) \to \Omega^{0, q}(X)$ be the continuous operator with distribution kernel $\widehat{A}_N(t, x, y)$, we have
		\begin{equation}\label{E:011120251243}
			\widehat{A}'_N(t) + \Box_b^{(q)} \widehat{A}_N(t) = t^N\delta_N(t),
		\end{equation}
		where $\delta_N(t, x, y) \in \mathcal{C}^\infty(\ol{\mathbb{R}}_+\times X \times X, T^{\ast0, q}X \boxtimes (T^{\ast 0 ,q}X)^\ast)$, $\delta_N(t, x, y)$ is the distribution kernel of $\delta_N(t)$. Let 
		\begin{equation}\label{E:011120251245}
			\mathcal{E}_N(t):= e^{-t\Box_b^{(q)}} - \widehat{A}_N(t).
		\end{equation}
		We deduce that
		\[
		\mathcal{E}'_N(t) + \Box_b^{(q)} \mathcal{E}_N(t) = -t^N\delta_N(t)
		\]
		and $\mathcal{E}_N(0) = 0$, hence
		\begin{equation}\label{E:011120251249}
			\frac{\partial}{\partial t} \mathcal{E}_N^\ast(t) + \mathcal{E}_N^\ast(t) \Box_b^{(q)} = -t^N\delta_N^\ast(t),
		\end{equation}
		where $\mathcal{E}_N^\ast(t)$ and $\delta_N^\ast(t)$ are adjoints of $\mathcal{E}_N(t)$ and $\delta_N(t)$ with respect to $(\,\cdot\,|\,\cdot\,)$ respectively. From \eqref{E:011120251249}, we have
		\begin{equation}\label{E:011120251254}
			\frac{\partial}{\partial s} \left(\mathcal{E}_N^\ast(s) e^{-(t-s)\Box_b^{(q)}} \right) = -s^N\delta_N^\ast(s) e^{-(t-s)\Box_b^{(q)}}, \quad \text{for all $0 \le s \le t$}.
		\end{equation}
		From \eqref{E:011120251254} and $\mathcal{E}_N^\ast(0) = 0$, we get
		\begin{equation}\label{E:011120251300}
			\mathcal{E}_N^\ast(t) = \int_0^t-s^N
            \delta_N^\ast(s) e^{-(t-s)\Box_b^{(q)}}ds.
		\end{equation}
		It implies from \eqref{E:011120251300} that
		\[
		\mathcal{E}_N^\ast(t) := O(t^{N+1}) : L^2_{(0, q)}(X) \to \Omega^{0, q}(X)
		\]
		is continuous and hence
		\begin{equation}\label{E:011120251304}
			\mathcal{E}_N(t):=O(t^{N+1}) : \mathcal{D}'(X, T^{\ast 0, q}X) \to L^2_{(0, q)}(X) \ \text{is continuous}.
		\end{equation}
		From \eqref{E:011120251304}, we deduce that $e^{-t\Box_b^{(q)}}$ can be continuously extended to 
		\begin{equation}\label{E:011120251330}
			e^{-t\Box_b^{(q)}}: \mathcal{D}'(X, T^{\ast 0, q}X) \to \mathcal{D}'(X, T^{\ast 0, q}X).
		\end{equation}
		From \eqref{E:011120251300} and \eqref{E:011120251330}, we finally get
		\[
		\mathcal{E}_N^\ast(t) = O(t^{N+1}) :\mathcal{D}'(X, T^{\ast 0, q}X) \to \Omega^{0,q}(X)
		\]
		is continuous. The theorem follows.
	\end{proof}

\begin{definition}\label{d-gue250403yyd}
Let $m, \ell\in \mathbb{R}$. Let $\varepsilon \ge 0$. Let $S^{m,\ell}_\varepsilon(\mathbb{R}_+ \times D \times D \times\mathbb{R}_+, T^{\ast 0, q}X \boxtimes (T^{\ast 0, q}X)^\ast)$ be the space of all $a(t, x, y, s) \in \mathcal{C}^\infty(\mathbb{R}_+ \times D \times D\times\mathbb{R}_+, T^{\ast 0, q}X \boxtimes (T^{\ast 0, q}X)^\ast)$ such that for all $\alpha, \beta \in \mathbb{N}_0^{2n+1}$, $\gamma_1, \gamma_2 \in \mathbb{N}_0$ and every compact set $K \subset D$, there exists a constant $C_{\alpha, \beta, \gamma_1, \gamma_2, K}>0$ such that
		\[
		|\partial_t^{\gamma_1} \partial_s^{\gamma_2} \partial_x^\alpha \partial_y^\beta a(t, x, y, s)| \le C_{\alpha, \beta, \gamma_1, \gamma_2, K}t^{\ell-\gamma_1}(1+|s|)^{m-\gamma_2} e^{-\varepsilon s}
		\]
		for all $(x, y) \in K \times K$ and all $s \ge 1$, and all $t \in \mathbb{R}_+$, $t\ll1$.   

Let $a(t, x, y, s) \in S^{m,\ell}_\varepsilon(\mathbb{R}_+ \times D \times D \times\mathbb{R}_+, T^{\ast 0, q}X \boxtimes (T^{\ast 0, q}X)^\ast)$. We write 
		\[
		a(t, x, y, s) \sim \sum_{j=0}^\infty a_j(t, x, y, s) \ \text{in} \  S^{m,\ell}_\varepsilon(\mathbb{R}_+ \times D \times D \times\mathbb R_+, T^{\ast 0, q}X \boxtimes (T^{\ast 0, q}X)^\ast),
		\]
		where $a_j(t, x, y, s) \in S^{m-j,\ell+j}_\varepsilon(\mathbb{R}_+ \times D \times D \times \mathbb{R}_+, T^{\ast 0, q}X \boxtimes (T^{\ast 0, q}X)^\ast)$, $j=0, 1, \cdots,$ if for every $N \in \mathbb{N}$,
		\[
		a - \sum_{j=0}^N a_j \in S^{m-N-1,\ell+N+1}_\varepsilon(\mathbb{R}_+ \times D \times D\times\mathbb{R}_+, T^{\ast 0, q}X \boxtimes (T^{\ast 0, q}X)^\ast).
		\]

        Let $S^{m,\ell}_{\varepsilon,{\rm cl\,}}(\mathbb{R}_+ \times D \times D \times\mathbb{R}_+, T^{\ast 0, q}X \boxtimes (T^{\ast 0, q}X)^\ast)$ be the space of all $a(t, x, y, s) \in 
S^{m,\ell}_\varepsilon(\mathbb{R}_+ \times D \times D \times\mathbb{R}_+, T^{\ast 0, q}X \boxtimes (T^{\ast 0, q}X)^\ast)$  such that 
\[
		a(t, x, y, s) \sim \sum_{j=0}^\infty t^{\ell+j}a_j( x, y, s) \ \text{in} \  S^{m,\ell}_\varepsilon(\mathbb{R}_+ \times D \times D \times\mathbb{R}_+, T^{\ast 0, q}X \boxtimes (T^{\ast 0, q}X)^\ast),
		\]
		where $a_j(x, y, s) \in S^{m-j}_{1,0}(D \times D \times \mathbb{R}_+, T^{\ast 0, q}X \boxtimes (T^{\ast 0, q}X)^\ast)$, $j=0, 1, \cdots$.  

       Let $B$ be a vector bundle over $\mathbb R_+\times D\times D\times\mathbb R_+$. Let $m, \ell\in \mathbb{R}$. Let $\varepsilon \ge 0$. We can define $S^{m,\ell}_{\varepsilon}(\mathbb{R}_+ \times D \times D \times\mathbb{R}_+, B)$, $S^{m,\ell}_{\varepsilon,{\rm cl\,}}(\mathbb{R}_+ \times D \times D \times\mathbb{R}_+, B)$, asymptotic sum in $S^{m,\ell}_{\varepsilon}(\mathbb{R}_+ \times D \times D \times\mathbb{R}_+, B)$ in the same ways. 
		\end{definition}

        Note that for $a(t,x,y,s)\in\hat S^m_{\varepsilon,{\rm cl\,}}(\ol{\mathbb R}_+\times D\times D\times\mathbb R_+,T^{*0,q}X\boxtimes(T^{*0,q}X)^*)$, where $m\in\mathbb R$, $\varepsilon\geq0$, we have 
        \begin{equation}\label{e-gue250404yyd}
t^{\ell}a(t,x,y,\frac{s}{t})\in S^{m,-m+\ell}_{\varepsilon,{\rm cl\,}}(\mathbb R_+\times D\times D\times\mathbb R_+,T^{*0,q}X\boxtimes(T^{*0,q}X)^*),
        \end{equation}
for all $\ell\in\mathbb R$. Changing $s$ by $\frac{s}{t}$ in the integral \eqref{e-gue250402ycdb} and from \eqref{e-gue250404yyd}, we get the following

	\begin{theorem}\label{T:011120251117}
		With the notations used above, we can find $a_j(x, y) \in\mathcal{C}^\infty(X \times X, T^{\ast0, q}X \boxtimes (T^{\ast 0 ,q}X)^\ast)$, $j=1, 2, \cdots$, such that for all $m ,N \in \mathbb{N}$, there is a constant $C_{m, N} > 0$ such that for all $t>0$, $t \ll 1$,
		\[
		\Big\| e^{-t\Box_b^{(q)}}(x, y) - \widehat{A}(t, x, y) - \sum_{j=1}^N a_j(x, y)t^j \Big\|_{\mathcal{C}^m(X \times X)} \le C_{m, N}t^{N+1}, 
		\]
		where $\widehat{A}(t, x, y) \in \mathcal{D}'(\mathbb{R}_+ \times X \times X, T^{\ast0, q}X \boxtimes (T^{\ast 0 ,q}X)^\ast)$ and for every local coordinate patch $D$ with local coordinate $x=(x_1, \cdots, x_{2n+1})$, we have
		\begin{equation}\label{e-gue250404yyda}
		\begin{split}
			& \widehat{A}(t, x, y) = \int_0^\infty e^{i \frac{1}{t}\Phi_-(x, y, s)}b_-(t, x, y, s) ds +  \int_0^\infty e^{i \frac{1}{t} \Phi_+(x, y, s)}b_+(t, x, y, s) ds, \\
			& \Phi_\pm(x, y, s) = \varphi_\pm(1, x, y, s), \quad \text{$\varphi_\pm(t, x, y, s)$ is as in Theorem~\ref{T:0111202511001}}, \\
            &b_{\pm}(t, x, y, s) \sim \sum_{j=0}^\infty t^{-n-1+j}b^j_{\pm}( x, y, s) \\
            & \qquad \text{in} \quad S^{n,-n-1}_\varepsilon(\mathbb{R}_+ \times D \times D \times \overline{\mathbb{R}}_+, T^{\ast 0, q}X \boxtimes (T^{\ast 0, q}X)^\ast),\\
            &b^j_{\pm}(x, y, s)\in S^{n-j}_{1,0}(D \times D \times\mathbb R_+, T^{\ast 0, q}X \boxtimes (T^{\ast 0, q}X)^\ast),\ \ j=0,1,\ldots,\\
             &b^0_{\pm}(x,y,s)\neq0,\\
              &\mbox{$b_-(t,x,y,s)\in S^{n,-n-1}_{\varepsilon,{\rm cl\,}}(\mathbb R_+\times D\times D\times\mathbb R_+,T^{*0,q}X\boxtimes(T^{*0,q}X)^*)$}\\
              &\qquad\mbox{ with $\varepsilon>0$ if $q\neq n_-$, $\varepsilon=0$ if $q=n_-$},\\
              &\mbox{$b_+(t,x,y,s)\in S^{n,-n-1}_{\varepsilon,{\rm cl\,}}(\mathbb R_+\times D\times D\times\mathbb R_+,T^{*0,q}X\boxtimes(T^{*0,q}X)^*)$}\\
              &\qquad\mbox{ with $\varepsilon>0$ if $q\neq n_+$, $\varepsilon=0$ if $q=n_+$}.
             \end{split}
             \end{equation}
	\end{theorem}
	
We also need 

\begin{definition}\label{d-gue250404yydb}
Let $m\in \mathbb{R}$. Let $\tilde{S}^{m}(\mathbb{R}_+ \times D, T^{\ast 0, q}X \boxtimes (T^{\ast 0, q}X)^\ast)$ be the space of all $a(t, x) \in \mathcal{C}^\infty(\mathbb R_+ \times D,T^{\ast 0, q}X \boxtimes (T^{\ast 0, q}X)^\ast)$ such that for all $\alpha\in \mathbb{N}_0^{2n+1}$, $\gamma \in \mathbb{N}_0$ and every compact set $K \subset D$, there exists a constant $C_{\alpha, \gamma, K}>0$ such that
		\[
		|\partial_t^{\gamma}\partial_x^\alpha a(t, x)| \le C_{\alpha, \gamma,K}t^{m-\gamma_1},
		\]
		for all $x \in K $ and all $t \in \mathbb{R}_+$, $t\ll1$.   

Let $a(t, x) \in\tilde S^{m}(\mathbb{R}_+ \times D, T^{\ast 0, q}X \boxtimes (T^{\ast 0, q}X)^\ast)$. We write 
		\begin{equation}\label{e-gue250404yyde}
		a(t, x) \sim \sum_{j=0}^\infty a_j(t, x) \quad \text{in} \quad\tilde{S}^m(\mathbb{R}_+ \times D, T^{\ast 0, q}X \boxtimes (T^{\ast 0, q}X)^\ast),
		\end{equation}
		where $a_j(t, x) \in \tilde S^{m+j}(\mathbb{R}_+ \times D, T^{\ast 0, q}X \boxtimes (T^{\ast 0, q}X)^\ast)$, $j=0, 1, \cdots,$ if for every $N \in \mathbb{N}$,
		\begin{equation}\label{e-gue250920ycd}
		a - \sum_{j=0}^N a_j \in\tilde S^{m+N+1}(\mathbb{R}_+ \times D, T^{\ast 0, q}X \boxtimes (T^{\ast 0, q}X)^\ast).
		\end{equation}

Let $B$ be a vector bundle over $\mathbb R_+\times X$. 
        We can similar define $\tilde S^{m}(\mathbb{R}_+ \times D,B)$, $\tilde S^{m}(\mathbb{R}_+ \times X,B$ and asymptotic sum as \eqref{e-gue250920ycd}.
		\end{definition}

We deduce from Theorem~\ref{T:011120251117} the following

\begin{corollary}\label{C:011120251206}
		Assume that $q \not\in \{ n_-, n_+ \}$. We can find $a_j(x)\in\mathcal{C}^\infty(X,T^{*0,q}X\boxtimes(T^{*0,q}X)^*)$, $j=0,1,\ldots$, $a_0(x)\neq0$, such that
		\[
		\mbox{$e^{-t\Box_b^{(q)}}(x, x) \sim\sum^{+\infty}_{j=0} t^{-n-1+j}a_j(x)$ in $\tilde S^{-n-1}(\mathbb{R}_+ \times X, T^{\ast 0, q}X \boxtimes (T^{\ast 0, q}X)^\ast)$}. 
		\]
	\end{corollary}

	Let $\Pi^{(q)} : L^2_{(0,q)}(X) \to \operatorname{Ker} \Box_b^{(q)}$ be the orthogonal projection with respect to $(\,\cdot\,|\,\cdot\,)$ (Szeg\H{o} projection).  We have the following result about small-time asymptotics of CR heat kernel, which plays a vital role throughout the paper.
	
	\begin{theorem}\label{T:012820251552}
		Assume that $q \in \{ n_-, n_+ \}$. Suppose that  $\Box_b^{(q)}$ has closed range.  Then 
		\begin{equation}\label{E:0111202512141}
			\big( e^{-t\Box_b^{(q)}} ( I - \Pi^{(q)} ) \big)(x, y) \in \mathcal{C}^\infty(\mathbb{R}_+ \times X \times X, T^{\ast0, q}X \boxtimes (T^{\ast 0 ,q}X)^\ast)
		\end{equation}
		and we can find $b_j(x)\in\mathcal{C}^\infty(X,T^{*0,q}X\boxtimes(T^{*0,q}X)^*)$, $j=0,1,\ldots$, $b_0(x)\neq0$, such that
		\begin{equation}\label{E:0111202512142}
        \begin{split}
			\big( e^{-t\Box_b^{(q)}} ( I - \Pi^{(q)} ) \big)(x, x) &\sim \sum_{j=0}^\infty t^{-n-1+j} b_j(x)\\
            & \text{in} \ \tilde S^{-n-1}(\mathbb{R}_+ \times X, T^{\ast 0, q}X \boxtimes (T^{\ast 0, q}X)^\ast). 
            \end{split}
		\end{equation}
	\end{theorem}
    
	\begin{proof}
    Let $x=(x_1,\ldots,x_{2n+1})$ be local coordinates of $X$ defined on an open set $D$ of $X$ with $T=-\frac{\pr}{\pr x_{2n+1}}$ on $D$. 
Since $\Box_b^{(q)}$ has closed range, it was shown in \cite{Hsiao08} that on $D\times D$, 
		\[
		\Pi^{(q)}(x, y)\equiv\int_0^\infty e^{is\varphi_-(x, y)}a_-(x, y, s) ds +  \int_0^\infty e^{is \varphi_+(x, y)}a_+(x, y, s) ds, \\
		\]
		where $\varphi_\pm(x, y)$, $a_{\pm}(x,y,s)$  are as in \eqref{e-gue250404yydv} and \eqref{e-gue250404yydu} respectively. 
Let $A(t,x,y)$ be as in \eqref{e-gue250402ycdb}. We have
		\begin{equation}\label{e-gue250404yydp}
        \begin{split}
			&A(t, x, y) - \Pi^{(q)}(x, y)\\
            &= \int_0^\infty e^{i \varphi_+(t, x, y, s)}(a_+(t, x, y, s) - a_+(x, y, s)) ds\\
			&   + \int_0^\infty (e^{i\varphi_+(t, x, y, s)} - e^{is\varphi_+(x, y)} ) a_+(x, y, s) ds \\ 
			&   + \int_0^\infty e^{i \varphi_-(t, x, y, s)}(a_-(t, x, y, s) - a_-(x, y, s)) ds \\
			&    + \int_0^\infty (e^{i\varphi_-(t, x, y, s)} - e^{is\varphi_-(x, y)} ) a_-(x, y, s) ds+E(t,x,y),
            \end{split}
		\end{equation}
        where $E(t,x,y)\in\mathcal{C}^\infty(\ol{\mathbb R}_+\times D\times D,T^{*0,q}X\boxtimes(T^{*0,q}X)^*)$.
		From \eqref{e-gue250404yydu} and \eqref{e-gue250404yydv}, we see that the integrals in \eqref{e-gue250404yydp} converge and hence 
        \begin{equation}\label{e-gue250404yydm}
            A(t, x, y) - \Pi^{(q)}(x, y)\in\mathcal{C}^\infty(\mathbb R_+\times D\times D,T^{*0,q}X\boxtimes(T^{*0,q}X)^*).\end{equation}
            From Theorem~\ref{T:011120251108}, we see that 
             \begin{equation}\label{e-gue250404yydn}
             \begin{split}
               A(t, x, y) - \Pi^{(q)}(x, y)&-(e^{-t\Box^{(q)}_b}(I-\Pi^{(q)}))(x,y)\\
               &\in\mathcal{C}^\infty(\ol{\mathbb R}_+\times D\times D,T^{*0,q}X\boxtimes(T^{*0,q}X)^*).
               \end{split}
             \end{equation}
             From \eqref{e-gue250404yydm} and \eqref{e-gue250404yydn}, we get \eqref{E:0111202512141}. 

From \eqref{e-gue250404yydp}, we have 
             \begin{equation}\label{e-gue250404yydq}
        \begin{split}
			&A(t, x, x) - \Pi^{(q)}(x, x)\\
            &= \int_0^\infty (a_+(t, x, x, s) - a_+(x, x, s)) ds\\
			&   + \int_0^\infty (a_-(t, x, x, s) - a_-(x, x, s)) ds+E(t,x,x)\\
            &=\int_0^\infty (a_+(t, x, x, \frac{s}{t}) - a_+(x, x, \frac{s}{t}))t^{-1}ds\\
			&   + \int_0^\infty (a_-(t, x, x,\frac{s}{t}) - a_-(x, x, \frac{s}{t}))t^{-1} ds+E(t,x,x).
            \end{split}
		\end{equation}
        From \eqref{e-gue250404yydq} and notice that 
        \begin{equation}\label{e-gue250404yydr}
        A(t, x, x) - \Pi^{(q)}(x, x)=(e^{-t\Box^{(q)}_b}(I-\Pi^{(q)}))+\hat E(t,x,y),\end{equation}
        where $\hat E(t,x,y)\in\mathcal{C}^\infty(\ol{\mathbb R}_+\times D\times D,T^{*0,q}X\boxtimes(T^{*0,q}X)^*)$. From \eqref{e-gue250404yydq} and \eqref{e-gue250404yydr}, we can repeat the proof of Theorem~\ref{T:011120251108} and get \eqref{E:0111202512142}. 
            \end{proof}

Let 
\[\Pi: L^2_{(0,\bullet)}(X)\To{\rm Ker\,}\Box_b\]
be the orthogonal projection. 
            From Corollary~\ref{C:011120251206} and Theorem~\ref{T:012820251552}, we get 

\begin{theorem}\label{t-gue250410ycdb}
Suppose that  $\Box_b$ has closed range. Then 
		\begin{equation}\label{E:0111202512141y}
			\big( e^{-t\Box_b} ( I - \Pi ) \big)(x, y) \in \mathcal{C}^\infty(\mathbb{R}_+ \times X \times X, T^{\ast0, \bullet}X \boxtimes (T^{\ast 0 ,\bullet}X)^\ast)
		\end{equation}
		and we can find $a_j(x)\in\mathcal{C}^\infty(X,T^{*0,\bullet}X\boxtimes(T^{*0,\bullet}X)^*)$, $j=0,1,\ldots$, $a_0(x)\neq0$, such that
		\begin{equation}\label{E:0111202512142y}
        \begin{split}
			\big( e^{-t\Box_b}( I - \Pi ) \big)(x, x) &\sim \sum_{j=0}^\infty t^{-n-1+j} a_j(x)\\
            &\text{in} \ \tilde S^{-n-1}(\mathbb{R}_+ \times X, T^{\ast 0, \bullet}X \boxtimes (T^{\ast 0, \bullet}X)^\ast). 
            \end{split}
		\end{equation}
\end{theorem}

        Let $E$ be a CR vector bundle over $X$ (see Definition~\ref{Def:CRVB}). We recall that we always assume that $E$ is locally CR trivializable.  It should be mentioned that Theorem~\ref{T:011120251108}, Theorem~\ref{T:011120251117} and Theorem~\ref{T:012820251552} hold for $e^{-t\Box^{(q)}_{b,E}}$ with the same proofs. In particular, we have the following results
        
        \begin{corollary}\label{c-gue250410yyd}
		Assume that $q \not\in \{ n_-, n_+ \}$. We can find $a_j(x)\in\mathcal{C}^\infty(X,(E\times T^{*0,q}X)\boxtimes(E\otimes T^{*0,q}X)^*)$, $j=0,1,\ldots$, $a_0(x)\neq0$, such that
		\begin{equation*}
        \begin{split}
		e^{-t\Box_{b,E}^{(q)}}(x, x) &\sim\sum^{+\infty}_{j=0} t^{-n-1+j}a_j(x) \\
        &\text{in}\ \tilde S^{-n-1}(\mathbb{R}_+ \times X, (E\otimes T^{\ast 0, q}X) \boxtimes(E\otimes T^{\ast 0, q}X)^\ast). 
        \end{split}
		\end{equation*}
	\end{corollary}
   
Let $\Pi^{(q)}_E : L^2_{(0,q)}(X,E) \to \operatorname{Ker} \Box_{b,E}^{(q)}$ be the orthogonal projection with respect to $(\,\cdot\,|\,\cdot\,)_E$. 

        \begin{theorem}\label{t-gue250410yydI}
		Assume that $q \in \{ n_-, n_+ \}$. Suppose that  $\Box_{b,E}^{(q)}$ has closed range.  Then 
		\begin{equation}\label{E:0111202512141z}
			\big( e^{-t\Box_{b,E}^{(q)}} ( I - \Pi^{(q)}_E ) \big)(x, y) \in \mathcal{C}^\infty(\mathbb{R}_+ \times X \times X, (E\otimes T^{\ast0, q}X)\boxtimes(E\otimes T^{\ast 0 ,q}X)^\ast)
		\end{equation}
		and we can find $b_j(x)\in\mathcal{C}^\infty(X,(E\otimes T^{*0,q}X)\boxtimes(E\times T^{*0,q}X)^*)$, $j=0,1,\ldots$, $b_0(x)\neq0$, such that
		\begin{equation}\label{E:0111202512142z}
        \begin{split}
			&\mbox{$\big( e^{-t\Box_{b,E}^{(q)}} ( I - \Pi^{(q)}_E ) \big)(x, x) \sim \sum_{j=0}^\infty t^{-n-1+j} b_j(x)$}\\ &\mbox{in $\tilde S^{-n-1}(\mathbb{R}_+ \times X, (E\otimes T^{\ast 0, q}X)\boxtimes(E\otimes T^{\ast 0, q}X)^\ast)$}. 
            \end{split}
		\end{equation}
	\end{theorem}

    Let $\Pi_E : L^2_{(0,\bullet)}(X,E) \to \operatorname{Ker} \Box_{b,E}$ be the orthogonal projection with respect to $(\,\cdot\,|\,\cdot\,)_E$. 

    \begin{theorem}\label{t-gue250410ycdbz}
Suppose that  $\Box_{b,E}$ has closed range.  Then 
		\begin{equation}\label{E:0111202512141yq}
			\big( e^{-t\Box_{b,E}} ( I - \Pi_E) \big)(x, y) \in \mathcal{C}^\infty(\mathbb{R}_+ \times X \times X, (E\otimes T^{\ast0, \bullet}X)\boxtimes (E\otimes T^{\ast 0 ,\bullet}X)^\ast)
		\end{equation}
		and we can find $g_j(x)\in\mathcal{C}^\infty(X,(E\otimes T^{*0,\bullet}X)\boxtimes(E\otimes T^{*0,\bullet}X)^*)$, $j=0,1,\ldots$, $g_0(x)\neq0$, such that
		\begin{equation}\label{E:0111202512142q}
        \begin{split}
			\big( e^{-t\Box_b}( I - \Pi_E) \big)(x, x) &\sim \sum_{j=0}^\infty t^{-n-1+j} g_j(x) \\
            &\text{in}\ \tilde S^{-n-1}(\mathbb{R}_+ \times X, (E\otimes T^{\ast 0, \bullet}X)\boxtimes (E\otimes T^{\ast 0, \bullet}X)^\ast). 
            \end{split}
		\end{equation}
\end{theorem}

    \section{The asymptotics of CR heat kernel for CR line bundles $L^k$}\label{s-gue250704}

Let $L^k$ be the $k$-th tensor power of a CR complex line bundle $L$ over $X$. In this section, we will establish asymptotics of the heat kernel of Kohn Laplacian with values in $L^k$ as $k \to +\infty$. 

\subsection{CR line bundle}\label{s-gue250411yyd}

Let $(L, h^L)$ be a CR complex line bundle over $X$, where $h^L$ is the Hermitian fiber metric on $L$ and its local weight is $\phi$. Namely, let $s$ be a local CR trivializing section of $L$ over $D$, then locally
	\[
	| s(x) |^2_{h^L} = e^{-\phi(x)}, \quad x \in D.
	\]
	
	For $k > 0$, let $L^k$ be the $k$-th tensor power of the line bundle $L$ over $X$, then $h^L$ induces a Hermitian fiber metric $h^{L^k}$ on $L^k$ and $s^k$ is a local CR trivializing section of $L^k$. 
As before, let 
	\[
	\ddbar_{b, k}=\ddbar_b : \Omega^{0, q}(X, L^k) \to  \Omega^{0, q+1}(X, L^k)
	\]
	be the tangential Cauchy-Riemann operator acting on forms with values in $L^k$. Then the $q$-th $\ddbar_{b, k}$-cohomology is given by
	\[
	H^q_b(X, L^k) :=\frac{{\rm Ker\,}(\ddbar_b: \Omega^{0,q}(X,L^k)\To\Omega^{0,q+1}(X,L^k))}{{\rm Im\,}(\ddbar_b: \Omega^{0,q-1}(X,L^k)\To\Omega^{0,q}(X,L^k))},\ \ q=0,1,\ldots,n.
	\]
We recall that we work with the assumption that the Levi form is non-degenerate of constant signature $(n_-,n_+)$. Then ${\rm dim\,}H^q(M,L^k)<+\infty$ if $q\notin\set{n_-,n_+}$. 

\subsection{The asymptotics of $e^{-\frac{t}{k}\Box_{b,L^k}}$ as $k\To+\infty$} \label{s-gue250411yydI} 

In this subsection, we are going to study asymptotic behavior of $e^{-\frac{t}{k}\Box_{b,L^k}}$ as $k\To+\infty$. 
We also write $A_k(t)$ to denote $e^{-t\Box_{b,L^k}}$ and write $A_k(t,x,y)\in\mathcal{D}'(\mathbb R_+\times X\times X,(T^{*0,\bullet}X\otimes L^k)\boxtimes(T^{*0,\bullet}X\otimes L^k)^*)$ to denote the distribution kernel of $A_k(t)$. 
For every $q\in\set{0,1,\ldots,n}$, write $A^{(q)}_k(t)$ to denote $e^{-t\Box^{(q)}_{b,L^k}}$ 
and write $A^{(q)}_k(t,x,y)$ to denote $e^{-t\Box^{(q)}_{b,L^k}}(x,y)$.
If condition $Y(q)$ holds, then 
 \[A^{(q)}_k(t,x,y)\in\mathcal{C}^\infty(\mathbb R_+\times X\times X, (T^{*0,q}X\otimes L^k)\boxtimes(T^{*0,q}X\otimes L^k)^*).\]
 When condition $Y(q)$ holds, the asymptotic behavior of 
$A^{(q)}_k(\frac{t}{k},x,y)$ as $k\To+\infty$ was established in~\cite{HZ23}.

One of the main goal of this subsection is to study the asymptotic behavior of 
 $A^{(q)}_k(\frac{t}{k},x,y)$ as $k\To+\infty$ in the sense of distribution when condition $Y(q)$ fails. 

Let $s$ be a local CR trivializing section of $L$ defined on an open set $D$ of $X$, $\abs{s}^2_{h^L}=e^{-\phi}$. We have the unitary identification: 
\[\begin{split}
U: L^2_{(0,\bullet)}(D,L^k)&\To L^2_{(0,\bullet)}(D),\\
s^k\otimes \hat u&\To e^{-\frac{k\phi}{2}} \hat u.
\end{split}\]
There exists $A_{k,s}(t,x,y)\in\mathcal{D}'(\mathbb R_+\times D\times D, T^{*0,\bullet}X\boxtimes(T^{*0,\bullet}X)^*)$ such that 
\begin{equation}\label{e-gue210303yydI}
(e^{-t\Box_{b,L^k}}u)(x)=s^k(x)\otimes e^{\frac{k\phi(x)}{2}}\int_DA_{k,s}(t,x,y)e^{-\frac{k\phi(y)}{2}}\hat u(y)dv_X(y)\ \ \mbox{on $D$}, 
\end{equation}
for every $u=s^k\otimes\hat u\in\Omega^{0,\bullet}_c(D,L^k)$, $\hat u\in\Omega^{0,\bullet}_c(D)$. Note that 
\begin{equation}\label{e-gue210303yydII}
A_k(t,x,y)=s^k(x)\otimes A_{k,s}(t,x,y)e^{\frac{k\phi(x)-k\phi(y)}{2}}\otimes (s^k(y))^*\ \ \mbox{on $D$}. 
\end{equation}

We first set up the local coordinates. Fix $p\in X$. Let $U_1,\cdots,U_n$ be a smooth orthonormal frame for $T^{1,0}X$ in a neighborhood $D$ of $p$,
 for which $\mathcal L_p(U_j,\ol U_l)=\delta_{j,l}\lambda_j(p)$, $1\le j,l\le n$. Let $\omega^1,\cdots,\omega^n$ be an orthonormal basis for $(1,0)$ forms which is dual to the basis $U_1,\cdots,U_n$. Let $s$ be a local CR trivializing section of $L$ on $D$  with local weight $\phi$, i.e. $|s(x)|^2_{h^L}=e^{-\phi(x)}$ on $D$.
 
We take local coordinates $x=(x_1,\ldots,x_{2n+1})=(z,\theta)=(z_1,\cdots,z_n,\theta)$ on an open set $D$ of $p$, such that 
\begin{equation}\label{local1}
 (z(p),\theta(p))=0,\ \omega_0(p)=d\theta,\ z_j=x_{2j-1}+ix_{2j},\quad \frac{\pr}{\pr z_j}=\frac{1}{2}\left(\frac{\pr}{\pr x_{2j-1}}-i\frac{\pr}{\pr x_{2j}}\right),
\end{equation}
\begin{equation}\label{local2}
\begin{split}
&\langle\,\frac{\partial}{\partial x_j}\,\mid\,\frac{\partial}{\partial x_l}\,\rangle=2\delta_{j,l}+O(|x|), \quad
\langle\,\frac{\partial}{\partial x_j}\,\mid\,\frac{\partial}{\partial\theta}\,\rangle=O(|x|),\\
 &\left|\frac{\partial}{\partial \theta}\right|^2=1+O(|x|),\quad  \left|\frac{\pr}{\pr z_j}\right|^2=1+O(|x|),
 \end{split}
\end{equation}
for $j,l=1,\cdots,2n$.
 Moreover,
\begin{equation}\label{U}
   U_j=\frac{\pr}{\pr z_j}-i\lambda_j\ol z_j\frac{\pr}{\pr \theta}-c_j\theta\frac{\pr}{\pr\theta}+\sum^{2n+1}_{s=1}r_{j,s}(z,\theta)\frac{\partial}{\partial x_s},
\end{equation}
where $r_{j,s}(z,\theta)\in\cali{C}^\infty(D)$, $r_{j,s}(z,\theta)=O(|(z,\theta)|^2)$, $j=1,\ldots,n$, $s=1,\ldots,2n+1$, 
and
\begin{equation}\label{phi}
\begin{split}
\phi(z,\theta)=&\beta\theta+\sum_{j=1}^n(a_jz_j+\ol a_j\ol z_j)+\sum_{j,l=1}^{n}(a_{j,l}z_jz_l+\ol a_{j,l}\ol z_j\ol z_l)+\sum_{j,l=1}^{n}\mu_{j,l}z_j\ol z_l\\
&+O(|\theta|^2+|z||\theta|+|(z,\theta)|^3)
\end{split}
\end{equation}
where $\beta\in\mathbb R$, $c_j$, $a_j$, $a_{j,l}$, $\mu_{j,l}\in\mathbb C$, and $\lambda_j$ are the eigenvalues of $\mathcal L_p$, $j, l=1,\ldots,n$. We then work with this local coordinates and we identify $D$ with some open set in $\mathbb R^{2n+1}$. 

For $u,v\in \Omega_c^{0,\bullet}(D)$, let
$$
(u\mid v)_{k\phi,D}=(u\mid v)_{k\phi}:=\int_{D}\langle\,u\mid v\,\rangle e^{-k\phi}dv_X(x).
$$
$(\,\cdot\mid\cdot\,)_{k\phi}$ is called the inner product on $\Omega_c^{0,\bullet}(D)$ with weight $k\phi$. Denote by $L_{(0,\bullet)}^2(D,k\phi)$ the completion of $\Omega_c^{0,\bullet}(D)$ with respect to $(\,\cdot\mid\cdot\,)_{k\phi}$.
Let $\ddbar_b^{*,k\phi}$ be the formal adjoint of $\ddbar_b$ with respect to $(\,\cdot\mid\cdot\,)_{k\phi}$. From $\ddbar_b=\sum_{j=1}^{n}\bigr(\ol\omega^j\wedge \ol U_j+(\ddbar_b \ol\omega^j)\wedge (\ol\omega^j\wedge)^*\bigr)$, we can check that
\begin{equation}
\ddbar^{*,k\phi}_{b}=\sum_{j=1}^{n}\bigr((\ol\omega^j\wedge)^*(- U_j+kU_j\phi+\beta_j(z,\theta))+\ol\omega^j\wedge (\ddbar_b \ol\omega^j\wedge)^*\bigr),
\end{equation}
where $\beta_j(z,\theta)$ is a smooth function on $D$, independent of $k$, for every $j=1,\ldots,n$, and for $e\in T^{*,r}_xX$, $(e\wedge)^*: T^{*0,\bullet}_xX\To T^{*0,\bullet}_xX$ denotes the adjoint of $e\wedge$ with respect to $\langle\,\cdot\,|\,\cdot\,\rangle$, that is, 
$\langle\,e\wedge u\,|\,v\,\rangle=\langle\,u\,|\,(e\wedge)^*v\,\rangle$, for all $u, v\in T^{*0,\bullet}_xX$. 

Set
\begin{equation}\label{e-gue210531yydu}
\Box_{b,k\phi}=\ddbar^{*,k\phi}_{b}\ddbar_{b}+\ddbar_{b}\ddbar^{*,k\phi}_{b}: \Omega^{0,\bullet}(D)\To\Omega^{0,\bullet}(D).
\end{equation}
For $u\in\Omega^{0,\bullet}(D,L^k)$, there exists a $\check u\in \Omega^{0,\bullet}(D)$ such that 
$u=s^k\check u$ on $D$. Recall that $\ddbar_{b}(s^k\otimes \check u)=s^k\otimes\ddbar_b \check u$ and we have
\begin{equation}\label{e-gue210303yyda}
\Box_{b,L^k}u=s^k\otimes \Box_{b,k\phi}\check u.
\end{equation}
From now on, on $D$, we identify $u$ with $\check u$ and $\Box_{b,L^k}$ with $\Box_{b,k\phi}$. 

Set
\begin{equation}\label{rho}
\rho(z)=\sum_{j=1}^na_jz_j+\sum_{j,l=1}^{n}a_{j,l}z_jz_l
\end{equation}
and 
\begin{equation}\label{phi0}
\begin{aligned}
\phi_0(z,\theta)&=\phi(z,\theta)-\rho(z)-\ol{\rho(z)}\\
&=\beta\theta+\sum_{j,l=1}^{n}\mu_{j,l}z_j\ol z_l+O(|\theta|^2+|z||\theta|+|(z,\theta)|^3).
\end{aligned}
\end{equation}
We consider the following unitary identification:
$$
\begin{cases}
L^2_{(0,\bullet)}(D,k\phi)&\leftrightarrow L^2_{(0,\bullet)}(D,k\phi_0)\\
u&\to\tilde{u}=e^{-k\rho}u\\
u=e^{k\rho}\tilde{u}&\leftarrow\tilde{u}\\
\end{cases}
$$

For $u\in\Omega^{0,\bullet}(D)\cap L^2_{(0,\bullet)}(D,k\phi_0)$,
let $\ddbar_{\rho}=\ddbar_b+k(\ddbar_b\rho)\wedge$, then we have 
\begin{equation}\label{dbarrho}
\ddbar_\rho\tilde u=e^{-k\rho}\ddbar_{b}u=\widetilde{\ddbar_bu}.
\end{equation}
Let $\ddbar_\rho^{*,k\phi_0}$ be the formal adjoint of $\ddbar_{\rho}$ with respect to $(\,\cdot\mid\cdot\,)_{k\phi_0}$. We can check that 
\begin{equation}
\begin{aligned}
\ddbar_\rho=\sum_{j=1}^{n}\bigr(\ol\omega^j\wedge (\ol U_j+k\ol U_j\rho)+(\ddbar_b \ol\omega^j)\wedge (\ol\omega^j\wedge)^*\bigr)
\end{aligned}
\end{equation}
and
\begin{equation}\label{dbarstarrho}
\begin{aligned}
\ddbar_\rho^{*,k\phi_0}=\sum_{j=1}^{n}\bigr((\ol\omega^j\wedge)^*\left(-U_j+kU_j\phi_0+kU_j\ol\rho+\alpha_j(z,\theta)\right)+\ol\omega^j\wedge (\ddbar_b \ol\omega^j\wedge)^*\bigr),
\end{aligned}
\end{equation}\label{dbarstarkphi0}
where $\alpha_j(z,\theta)$ is a smooth function on $D$, independent of $k$, for every $j=1,\ldots,n$. For abbreviation, we write $\ol U_{\rho,j}=\ol U_j+k\ol U_j\rho$ and $\ol U_{\rho,j}^{*,k\phi_0}=-U_j+kU_j\phi_0+kU_j\ol\rho+\alpha_j(z,\theta)$, $j=1,\ldots,n$. 

Put
$$
\Box_{\rho,k\phi_0}=\ddbar^{*,k\phi_0}_{\rho}\ddbar_{\rho}+\ddbar_{\rho}\ddbar^{*,k\phi_0}_{\rho}: \Omega^{0,q}(D)\To\Omega^{0,q}(D).
$$ 
Thorough a straightforward computation, we have
\begin{equation}
\begin{aligned}
\Box_{\rho, k\phi_0}=&\sum_{j=1}^{n}\ol U_{\rho,j}^{*,k\phi_0}\ol U_{\rho,j}
+\sum_{j,l=1}^{n}\ol\omega^j\wedge(\ol\omega^l\wedge)^*\big[\ol U_{\rho,j}\,,\,\ol U_{\rho,l}^{*,k\phi_0}\big] \\
&+\sum_{j=1}^n\Big(\epsilon_j(z,\theta)\ol U_{\rho,j}+s_j(z,\theta)\ol U_{\rho,j}^{*,k\phi_0}\Big)+\gamma(z,\theta),
\end{aligned}
\end{equation}
where $\epsilon_j$, $s_j$ and $\gamma$ are smooth functions on $D$ and independent of $k$, $j=1,\ldots,n$. We can check that
\begin{equation}\label{e-gue210516yyd}
\Box_{b,k\phi}u=e^{k\rho}\Box_{\rho, k\phi_0}(e^{-k\rho}u),\ \ u\in\Omega^{0,\bullet}(D).
\end{equation} 

Let 
\begin{equation}\label{e-gue210305yyd}
\begin{split}
A_{k\phi}(\frac{t}{k},x,y):=&e^{\frac{k\phi(x)}{2}}A_{k,s}(\frac{t}{k},x,y)e^{-\frac{k\phi(y)}{2}},\\
A_{k\phi_0}(\frac{t}{k},x,y):=&e^{-k\rho(x)}A_{k\phi}(\frac{t}{k},x,y)e^{k\rho(y)}\\
=&e^{-k\rho(x)+\frac{k\phi(x)}{2}}A_{k,s}(\frac{t}{k},x,y)e^{k\rho(y)-\frac{k\phi(y)}{2}},
\end{split}
\end{equation}
where $A_{k,s}(\frac{t}{k},x,y)$ is as in \eqref{e-gue210303yydII}. For $t>0$, let 
\[A_{k\phi_0}(t): \Omega^{0,\bullet}_c(D)\To \mathcal{D}'(D,T^{*0,\bullet}X)\]
be the continuous operator with distribution kernel $A_{k\phi_0}(t,x,y)$ with respect to $dv_X$. Note that 
\begin{equation}\label{e-gue210305yydI}
(A_{k\phi_0}(t)u)(x)=\int A_{k\phi_0}(t,x,y)u(y)dv_X(y),\ \ u\in\Omega^{0,\bullet}_c(D).
\end{equation}
It is easy to check that 
\begin{equation}\label{e-gue210305yydII}
\begin{cases}
&A'_{k\phi_0}(t)+\Box_{\rho,k\phi_0}A_{k\phi_0}(t)=0,\\
&\lim_{t\To0^+}A_{k\phi_0}(t)=I.
\end{cases}
\end{equation}

Notice that the reason we take the term $\rho(z)$ out of $\phi(z,\theta)$ is that  the term $kU_j\ol\rho$ in it will approach to zero as $k\To+\infty$ after $\Box_{\rho,k\phi_0}^q$ is scaled. 

We introduce the scaling technique. Let $B_r:=\{(z,\theta)=(x_1,\cdots x_{2n},\theta)\in\mathbb R^{2n+1};\, |x_j|<r, |\theta|<r, j=1,\cdots 2n\}$. Let 
\[\begin{split}
F_k: \mathbb R^{2n+1}&\To\mathbb R^{2n+1},\\
(z,\theta)&\To(\frac{z}{\sqrt{k}},\frac{\theta}{k}),
\end{split}\]
 be the scaling map, then we can choose large $k$ such that $F_k(B_{\log k})\subset D$. For $(z,\theta)\in B_{\log k}$, let 
\begin{equation}\label{e-gue250423ycd}
F_k^*T^{*0,q}_{(z,\theta)}X:=\left\{\sumprime_{|J|=q}a_J\ol\omega^J\left(\frac{z}{\sqrt k},\frac{\theta}{k}\right);\,a_J\in\mathbb C\right\},
\end{equation}
where the sum takes over all $q$-tuple of integers between $1$ and $n$, the prime means that we sum over only strictly increasing multiindices, 
$\ol\omega^J=\ol\omega^{j_1}\wedge\cdots\wedge\ol\omega^{j_q}$, $1\leq j_1<\cdots<j_q\leq n$. Let $F^*_kT^{*0,q}X$ be the vector bundle over $B_{\log k}$ with fiber $F^*_kT^{*0,q}_{x}X$, $x=(z,\theta)\in B_{\log k}$. Put
\[F^*_kT^{*0,\bullet}X:=\oplus^n_{q=0}F^*_kT^{*0,q}X.\]
Define by $F_k^*\Omega^{0,q}(B_r)$ the space of smooth sections of $F_k^*T^{*0,q}X$ over $B_r$. Let 
\[F^*_k\Omega^{0,q}_c(B_r)=\set{u\in F^*_k\Omega^{0,q}(B_r);\, {\rm supp\,}u\Subset B_r}.\] 
Put
\[\begin{split}
&F_k^*\Omega^{0,\bullet}(B_r):=\oplus^n_{q=0}F_k^*\Omega^{0,q}(B_r),\\
&F_k^*\Omega^{0,\bullet}_c(B_r):=\oplus^n_{q=0}F_k^*\Omega^{0,q}_c(B_r).
\end{split}\]

For $u=\sumprime_{|J|=q}u_J\ol\omega^J\in \Omega^{0,q}(F_k(B_{\log k}))$, we define the scaled form $F_k^*u$ by
\begin{equation}\label{e-gue250423yyd}
F_k^*u:=\sumprime_{|J|=q}u_J\left(\frac{z}{\sqrt k},\frac{\theta}{k}\right)\ol\omega^J\left(\frac{z}{\sqrt k},\frac{\theta}{k}\right)\in F_k^*\Omega^{0,q}(B_{\log k}).
\end{equation}
For $u\in \Omega^{0,\bullet}(F_k(B_{\log k}))$, we define 
$F_k^*u\in F_k^*\Omega^{0,\bullet}(B_{\log k})$ as \eqref{e-gue250423yyd} in the same way.

Fix $q\in\set{0,1,\ldots,n}$. With the notations used in the discussions after \eqref{e-gue210303yydII} and \eqref{e-gue250423ycd}, 
\[\set{\overline\omega^J(x);\, J=(j_1,\ldots,j_q), 1\leq j_1<\cdots<j_q\leq n}\] 
is an orthonormal frame for $T^{*0,q}_xX$, for every $x\in D$. Let $\langle\,\cdot\mid\cdot\,\rangle_{F^*_k}$ be the Hermitian metric of $F^*_kT^{*0,q}X$ on $B_{\log k}$ such that 
\[\set{\overline\omega^J(F_kx);\, J=(j_1,\ldots,j_q), 1\leq j_1<\cdots<j_q\leq n}\] is an orthonormal frame at every $x\in B_{\log k}$. On $D$, write 
$dv_X(x)=m(x)dx_1\cdots dx_{2n+1}$, $m(x)\in\mathcal{C}^\infty(D)$. Let $(\,\cdot\mid\cdot\,)_{kF^*_k\phi_0}$ be the $L^2$ inner product on $F^*_k\Omega^{0,\bullet}_c(B_{\log k})$ given by 
\[(\,u\mid v\,)_{kF^*_k\phi_0}=\int\langle\,u\mid v\,\rangle_{F^*_k}e^{-kF^*_k\phi_0}m(F_kx)dx_1\cdots dx_{2n+1},\]
where $u, v\in F^*_k\Omega^{0,\bullet}_c(B_{\log k})$ and let $\norm{\cdot}_{kF^*_k\phi_0}$ be the corresponding norm. 

Let $\ddbar_{\rho,(k)}^{*,kF_k^*\phi_0}$ be the formal adjoint of $\ddbar_{\rho,(k)}$ with respect to $(\,\cdot\mid\cdot\,)_{kF_k^*\phi_0}$. 
Now we define the scaled Kohn-Laplacian:
\begin{equation}\label{e-gue210614yydI}
\Box_{\rho,(k)}=\ddbar_{\rho,(k)}^{*,kF_k^*\phi_0}\ddbar_{\rho,(k)}+\ddbar_{\rho,(k)}\ddbar_{\rho,(k)}^{*,kF_k^*\phi_0}
\end{equation}
on $F_k^*\Omega^{0,\bullet}(B_{\log k})$. Let 
\begin{equation}\label{e-gue210325yyd}
\begin{split}
A_{(k)}(t,x,y)&:=k^{-(n+1)}A_{k\phi_0}(\frac{t}{k},F_kx,F_ky)\\
&\in\mathcal{D}'(\mathbb R_+\times B_{\log k}\times B_{\log k}, F^*_kT^{*0,\bullet}X\boxtimes(F^*_kT^{*0,\bullet}X)^*),
\end{split}
\end{equation}
where $F_kx=(\frac{z}{\sqrt{k}},\frac{\theta}{k})$, $F_ky=(\frac{w}{\sqrt{k}},\frac{\xi}{k})$, $y=(w,\xi)\in\mathbb C^n\times\mathbb R$, $A_{k\phi_0}(t,x,y)$ is as in
\eqref{e-gue210305yyd}. Let 
\[A_{(k)}(t): F^*_k\Omega^{0,\bullet}_c(B_{\log k})\To\mathcal{D}'(B_{\log k}, F^*_kT^{*0,\bullet}X)\] 
be the continuous operator given by 
\begin{equation}\label{e-gue210325yydI}
(A_{(k)}(t)u)(x)=\int A_{(k)}(t,x,y)u(y)m(F_ky)dy,\ \ u\in F^*_k\Omega^{0,q}_c(B_{\log k}).
\end{equation}
It was shown in~\cite{HZ23} that 
\begin{equation}\label{k(k)}
\Box_{\rho,(k)}(F_k^*u)=\frac{1}{k}F_k^*(\Box_{\rho,k\phi_0}u),\quad u\in\Omega^{0,\bullet}(F_k(B_{\log k})). 
\end{equation} 

For $u, v\in F^*_k\Omega^{0,\bullet}_c(B_{\log k})$ and every $\ell\in\mathbb N$, we have 
\begin{equation}\label{e-gue250424yyd}
(\,\frac{\pr^\ell}{\pr t^\ell}A_{(k)}(t)u\,|\,v\,)_{kF^*_k\phi_0}
=(\,A_{(k)}(t)(-1)^\ell(\Box_{\rho,(k)})^\ell u\,|\,v\,)_{kF^*_k\phi_0}
\end{equation}
and
\begin{equation}\label{e-gue250425yyd}
\abs{(\,\frac{\pr^\ell}{\pr t^\ell}A_{(k)}(t)u\,|\,v\,)_{kF^*_k\phi_0}}\leq\norm{(\Box_{\rho,(k)})^\ell u}_{kF^*_k\phi_0}\norm{v}_{kF^*_k\phi_0},
\end{equation}
for all $t>0$. 

We pause and introduce some notations about Heisenberg group. We identify $\mathbb R^{2n+1}$
with the Heisenberg group $H_n:=\mathbb C^{n}\times\mathbb R$. Using the same setting as in \eqref{local1} and \eqref{local2} to denote by $x=(z,\theta)$ the coordinates of $H_n$, $z=(z_1,\cdots,z_n)\in \mathbb C^n$, $\theta\in \mathbb R$, $x=(x_1,\ldots,x_{2n+1})$, $z_j=x_{2j-1}+ix_{2j}$, $j=1,\ldots,n$, $\theta=x_{2n+1}$. Let 
\[T^{1,0}H_n={\rm span\,}\set{\frac{\partial}{\partial z_j}-i\lambda_j\overline z_j\frac{\partial}{\partial\theta};\, j=1,\ldots,n}.\]
Then $(H_n,T^{1,0}H_n)$ is a CR manifold of dimension $2n+1$. Let $\langle\,\cdot\mid\cdot\,\rangle_{H_n}$ be the Hermitian metric on 
$\mathbb CTH_n$ so that 
\begin{equation}
\begin{aligned}
\left\{U_{j,H_n}, \ol U_{j,H_n}, T=-\frac{\pr}{\pr \theta};\, j=1,\cdots n\right\}
\end{aligned}
\end{equation}
is an orthonormal basis for $\mathbb CTH_n$, where
\begin{equation}
\begin{aligned}
U_{j,H_n}=\frac{\pr}{\pr z_j}-i\lambda_j\ol z_j\frac{\pr}{\pr\theta},\ \ j=1,\ldots,n.
\end{aligned}
\end{equation}
Then 
\begin{equation}
\begin{aligned}
\left\{dz_j, d\ol z_j, \omega_0=d\theta+\sum_{j=1}^{n}\left(i\lambda_j\ol z_jdz_j-i\lambda_jz_jd\ol z_j\right), j=1,\cdots, n\right\}
\end{aligned}
\end{equation}
is an orthonormal basis of the complexified cotangent bundle $\mathbb CT^*H_n$. The Levi form $\mathcal{L}_p$ of $H_n$ at $p\in H_n$ is given by $\mathcal L_p=\sum_{j=1}^{n}\lambda_jdz_j\wedge d\ol z_j$. The tangential Cauchy-Riemann operator $\ddbar_{b,H_n}$ on $H_n$ is given by
\begin{equation}\label{dbarbh}
\begin{aligned}
\ddbar_{b,H_n}=\sum_{j=1}^{n}d\ol z_j\wedge\ol U_{j,H_n}:\Omega^{0,q}(H_n)\To\Omega^{0,q+1}(H_n).
\end{aligned}
\end{equation}
Let 
\begin{equation}\label{e-gue250429yydI}
\begin{split}
&\Phi(x)=\beta\theta+\sum_{j,l=1}^{n}\mu_{j,l}z_j\ol z_l,\\
&\Phi_0(z)=\sum_{j,l=1}^{n}\mu_{j,l}z_j\ol z_l,
\end{split}\end{equation}
where $\beta$ and $\mu_{j,l}$ are the same as in \eqref{phi0}. It is easy to check that
\begin{equation}
\begin{aligned}
\sup_{(z,\theta)\in B_{\log k}}\left|kF_k^*\phi_0-\Phi\right|\to 0,\,\, \,\text{as}\,\, k\to\infty.
\end{aligned}
\end{equation}  
The Hermitian metric on $\mathbb CTH_n$ induces a Hermitian metric $\langle\,\cdot\mid\cdot\,\rangle_{H_n}$ on \[\Lambda^\bullet(\mathbb CT^*H_n):=\oplus^{2n+1}_{j=0}\Lambda^j(\mathbb CT^*H_n).\] 
Let $(\,\cdot\mid\cdot\,)_{\Phi}$ be the  inner product on $\mathcal{C}^\infty_c(H_n,\oplus^{2n+1}_{j=0}\Lambda^j(\mathbb CT^*H_n))$ with weight $\Phi$. Namely,
\begin{equation}\label{e-gue210506yyd}
\begin{aligned}
(\,u\mid v\,)_{\Phi}=\int_{H_n}\langle\,u\mid v\,\rangle_{H_n}e^{-\Phi}dv_{H_n},\quad u,v\in\mathcal{C}^\infty_c(H_n,\oplus^{2n+1}_{j=0}\Lambda^j(\mathbb CT^*H_n)),
\end{aligned}
\end{equation}
where $dv_{H_n}(x)=dv(z)d\theta$, $dv(z)=2^n dx_1\cdots dx_{2n}$, $z_j=x_{2j-1}+ix_{2j}$, $j=1,\cdots,n$. Similarly, let $(\,\cdot\mid\cdot\,)_{H_n}$ be the  inner product on $\mathcal{C}^\infty_c(H_n,\Lambda^\bullet(\mathbb CT^*H_n))$ given by
\begin{equation}\label{e-gue210506yydz}
\begin{aligned}
(\,u\mid v\,)_{H_n}=\int_{H_n}\langle\,u\mid v\,\rangle_{H_n}dv_{H_n},\quad u,v\in\mathcal{C}^\infty_c(H_n, \Lambda^*(\mathbb CT^*H_n)).
\end{aligned}
\end{equation}

Let $L^2_{(0,\bullet)}(H_n,\Phi)$ be the completion of $\Omega_c^{0,\bullet}(H_n)$ with respect to $(\,\cdot\mid\cdot\,)_{\Phi}$. Let $\|\cdot\|_{\Phi}$ be the corresponding norm. 

We extend $\ddbar_{b,H_n}$ to $L^2_{(0,\bullet)}(H_n,\Phi)$: 
\[\ddbar_{b,H_n}: {\rm Dom\,}\ddbar_{b,H_n}\subset L^2_{(0,\bullet)}(H_n,\Phi)\To L^2_{(0,\bullet)}(H_n,\Phi),\]
where ${\rm Dom\,}\ddbar_{b,H_n}=\set{u\in L^2_{(0,\bullet)}(H_n,\Phi);\, \ddbar_{b,H_n}u\in L^2_{(0,\bullet)}(H_n,\Phi)}$. 
Let 
$$\ddbar^*_{b,H_n}:{\rm Dom\,}\ddbar^*_{b,H_n}\subset L^2_{(0,\bullet)}(H_n,\Phi)\to L^2_{(0,\bullet)}(H_n,\Phi)$$
be the Hilbert space adjoint of $\ddbar_{b,H_n}$ with respect to $(\,\cdot\mid\cdot\,)_{\Phi}$. The (Gaffney extension) of Kohn Laplacian $\Box_{H_n,\Phi}$ is given by
\[\begin{split}
&\Box_{H_n,\Phi}: {\rm Dom\,}\Box_{H_n,\Phi}\subset L^2_{(0,\bullet)}(H_n,\Phi)\To L^2_{(0,\bullet)}(H_n,\Phi),\\
& {\rm Dom\,}\Box_{H_n,\Phi}=\{u\in L^2_{(0,\bullet)}(H_n,\Phi);\, u\in{\rm Dom\,}\ddbar_{b,H_n}\bigcap{\rm Dom\,}\ddbar^*_{b,H_n},\\
&\quad\ddbar_{b,H_n}u\in{\rm Dom\,}\ddbar^*_{b,H_n}, \ddbar^*_{b,H_n} u\in{\rm Dom\,}\ddbar_{b,H_n}\},\\
&\Box_{H_n,\Phi}=\ddbar^*_{b,H_n}\ddbar_{b,H_n}+\ddbar_{b,H_n}\ddbar^*_{b,H_n}\ \ \mbox{on ${\rm Dom\,}\Box_{H_n,\Phi}$}. 
\end{split}\] 
Let $e^{-t\Box_{H_n,\Phi}}$, $t>0$, be the heat operator for $\Box_{H_n,\Phi}$.

Let 
\[e^{-t\Box_{H_n,\Phi}}(x,y)\in\mathcal{D}'(\mathbb R_+\times H_n\times H_n,T^{*0,\bullet}H_n\boxtimes(T^{*0,\bullet}H_n)^*)\]
be the distribution kernel of $e^{-t\Box_{H_n,\Phi}}$ with respect to $(\,\cdot\mid\cdot\,)_{H_n}$, that is,
\[(e^{-t\Box_{H_n,\Phi}}u)(x)=\int e^{-t\Box_{H_n,\Phi}}(x,y)u(y)dv_{H_n}(y),\ \ u\in\Omega^{0,\bullet}_c(H_n).\]
We write $\Box^{(q)}_{H_n,\Phi}:=\Box_{H_n,\Phi}|_{{\rm Dom\,}\Box_{H_n,\Phi}\cap L^2_{(0,q)}(H_n)}$,  $e^{-t\Box^{(q)}_{H_n,\Phi}}:=e^{-t\Box_{H_n,\Phi}}|_{L^2_{(0,q)}(H_n)}$ and let $e^{-t\Box^{(q)}_{H_n,\Phi}}(x,y)\in\mathcal{D}'(\mathbb R_+\times H_n\times H_n,T^{*0,q}H_n\boxtimes(T^{*0,q}H_n)^*)$
    be the distribution kernel of $e^{-t\Box^{(q)}_{H_n,\Phi}}$ with respect to $(\,\cdot\,|\,\cdot\,)_{H_n}$. 

    Let 
    \[\Pi_{H_n,\Phi}: L^2_{(0,\bullet)}(H_n)\To{\rm Ker\,}\Box_{H_n,\Phi}\]
    be the orthogonal projection with respect to $(\,\cdot\,|\,\cdot\,)_\Phi$ and let $\Pi_{H_n,\Phi}(x,y)\in\mathcal{D}'(H_n\times H_n,T^{*0,\bullet}H_n\boxtimes(T^{*0,\bullet}H_n)^*)$ be the distribution kernel of $\Pi_{H_n,\Phi}$ with respect to $(\,\cdot\,|\,\cdot\,)_{H_n}$, that is, \[(\Pi_{H_n,\Phi}u)(x)=\int\Pi_{H_n,\Phi}(x,y)u(y)dv_{H_n}(y),\ \ u\in\Omega^{0,\bullet}_c(H_n).\]
Let $q\in\set{0,1,\ldots,n}$. Write $\Pi^{(q)}_{H_n,\Phi}:=\Pi_{H_n,\Phi}|_{L^2_{(0,q)}(H_n)}$ and let $\Pi^{(q)}_{H_n,\Phi}(x,y)\in\mathcal{D}'(H_n\times H_n,T^{*0,q}H_n\boxtimes(T^{*0,q}H_n)^*)$ be the distribution kernel of $\Pi^{(q)}_{H_n,\Phi}$.

    For any bounded open set $W\subset H_n$, we now identify $e^{-t\Box_{H_n,\Phi}}(x,y)$ and $A_{(k)}(t,x,y)$ as elements in 
$\mathcal{D}'(\mathbb R_+\times W\times W, \Lambda^*(\mathbb CT^*H_n)\boxtimes(\Lambda^*(\mathbb CT^*H_n))^*)$.

From \eqref{e-gue250424yyd}, \eqref{e-gue250425yyd} and notice that $\mathcal{C}^\infty_c(H_n,\Lambda^\bullet(\mathbb CT^*H_n))$ is separable, we can repeat the proof of~\cite[Theorem 3.9]{HZ23} with minor changes and deduce  

\begin{theorem}\label{t-gue250425yyd}
Let $p\in X$, let $x=(x_1,\ldots,x_{2n+1})$ be local coordinates of $X$ defined on an open set $D$ of $p$ with $x(p)=0$ and let $s$ be a local CR trivializing section of $L$ on $D$ such that 
\eqref{local1}, \eqref{local2}, \eqref{U} and \eqref{phi} hold. With the notations used above, 
let $I\subset\mathbb R_+$ be a bounded interval. 
Let $r>0$. We have 
\[\lim_{k\To+\infty}A_{(k)}(t,x,y)=e^{-t\Box_{H_n,\Phi}}(x,y)\]
in $\mathcal{D}'(I\times B_r\times B_r,\Lambda^*(\mathbb CT^*H_n)\boxtimes(\Lambda^*(\mathbb CT^*H_n))^*)$ topology, where $A_{(k)}(t,x,y)$ is as in \eqref{e-gue210325yyd}.
\end{theorem} 

We now introduce the explicit formula for $e^{-t\Box_{H_n,\Phi}}(x,y)$ which was obtained in~\cite[Theorem 4.3]{HZ23}. We first introduce some notations. Consider $\mathbb C^n$. Let $\langle\,\cdot\mid\cdot\,\rangle_{\mathbb C^n}$ be the Hermitian inner product on $T^{*0,q}\mathbb C^n$ so that 
\[\set{d\ol z^J;\, J=(j_1,\ldots,j_q), 1\leq j_1<\cdots<j_q\le n}\]
is an orthonormal basis for $T^{*0,q}\mathbb C^n$, where $T^{*0,q}\mathbb C^n$ denotes the bundle of $(0,q)$ forms of $\mathbb C^n$. 
Let $\Omega^{0,q}(\mathbb C^n)$ denote the space of smooth $(0,q)$ forms of $\mathbb C^n$ and put $\Omega^{0,q}_c(\mathbb C^n):=\Omega^{0,q}(\mathbb C^n)\bigcap\mathcal{E}'(\mathbb C^n,T^{*0,q}\mathbb C^n)$. Let $(\,\cdot\mid\cdot\,)_{\mathbb C^n}$ be the $L^2$ inner product on $\Omega^{0,q}_c(\mathbb C^n)$ 
induced by $\langle\,\cdot\mid\cdot\,\rangle_{\mathbb C^n}$ and let $L^2_{(0,q)}(\mathbb C^n)$ be the completion of $\Omega^{0,q}_c(\mathbb C^n)$ with respect to 
$(\,\cdot\mid\cdot\,)_{\mathbb C^n}$. 
For $\eta\in\mathbb R$, put
\begin{equation}\label{Phieta}
\Phi_{\eta}=-2\sum_{j=1}^{n}\eta\lambda_j|z_j|^2+\sum_{j,l=1}^{n}\mu_{j,l}\ol z_jz_l,
\end{equation}
where $\lambda_j$ and $\mu_{j,l}$, $j, l=1,\ldots,n$, are as in \eqref{U} and \eqref{phi0}. In particular, $\Phi_0=\sum^n_{j,l=1}\mu_{j,l}\ol z_jz_l$. 
Let
$$
\ddbar_{\eta}=\sum_{j=1}^nd\ol z_j\wedge\left(\frac{\pr}{\pr\ol z_j}+\frac{1}{2}\frac{\pr\Phi_\eta}{\pr\ol z_j}\right): \Omega^{0,q}(\mathbb C^n)\To\Omega^{0,q}(\mathbb C^n).
$$
Let 
\[\ddbar^*_\eta:  \Omega^{0,q+1}(\mathbb C^n)\To\Omega^{0,q}(\mathbb C^n)\]
be the formal adjoint of $\ddbar_\eta$ with respect to $(\,\cdot\mid\cdot\,)_{\mathbb C^n}$. It is easy to check that 
$$
\ddbar^*_{\eta}=\sum_{j=1}^n(d\ol z_j\wedge)^*\left(-\frac{\pr}{\pr z_j}+\frac{1}{2}\frac{\pr\Phi_\eta}{\pr z_j}\right),
$$
where $(d\ol z_j\wedge)^*$ is the adjoint of $d\ol z_j\wedge$ with respect to $\langle\,\cdot\mid\cdot\,\rangle_{\mathbb C^n}$, $j=1,\ldots,n$. Let 
\[\Box_\eta:=\ddbar_{\eta}\ddbar^*_\eta+\ddbar^*_\eta\ddbar_\eta: {\rm Dom\,}\Box_\eta\subset L^2_{(0,q)}(\mathbb C^n)\To L^2_{(0,q)}(\mathbb C^n),\]
where ${\rm Dom\,}\Box_\eta=\set{u\in L^2_{(0,q)}(\mathbb C^n);\,\Box_\eta u\in L^2_{(0,q)}(\mathbb C^n)}$. It is not difficult to see that $\Box_\eta$ is a non-negative self-adjoint operator. Let $e^{-t\Box_\eta}$ be the heat operator of $\Box_\eta$ and let $e^{-t\Box_\eta}(x,y)\in\mathcal{C}^\infty(\mathbb R_+\times\mathbb C^n\times\mathbb C^n, T^{*0,q}\mathbb C^n\boxtimes(T^{*0,q}\mathbb C^n)^*)$ be the distribution kernel of $e^{-t\Box_\eta}$. Let $\dot{R}^\eta:T^{1,0}\mathbb C^n\To T^{1,0}\mathbb C^n$ be the linear map defined by
\begin{equation}\label{Reta}
\langle\,\dot{R}^\eta U\mid\ol V\,\rangle_{\mathbb C^n}=\pr\ddbar\Phi_{\eta}(U,\ol V),\quad U,V\in T^{1,0}\mathbb C^n
\end{equation}
and let
\begin{equation}\label{omega}
\omega^\eta_{\mathbb C^n}:=\sum_{j,l=1}^n\frac{\pr^2\Phi_{\eta}}{\pr\ol z_j\pr z_l}d\ol z_j\wedge (d\ol z_l\wedge)^*,
\end{equation}
where $\Phi_\eta$ is as in \eqref{Phieta}.
The following is known (see~\cite[Appendix E.2]{MM},~\cite[(4.8)]{HZ23}):

\begin{equation}\label{e-gue210521yyd}
\begin{aligned}
e^{-t\Box_\eta}(x,y)=&\frac{1}{(2\pi)^n}\dfrac{\det\dot R^\eta}{\det\big(1-e^{-t \dot R^\eta}\big)}
\exp\bigg\{-t\omega^\eta_{\mathbb C^n}-\frac{1}{2}\Big\langle\,\frac{\dot R^\eta/2}{\tanh(t\dot R^\eta/2)}x\mid x\,\Big\rangle_{\mathbb C^n}\\
&-\frac{1}{2}\Big\langle\,\frac{\dot R^\eta/2}{\tanh(t\dot R^\eta/2)}y\mid y\,\Big\rangle_{\mathbb C^n}
+\Big\langle\,\frac{\dot R^\eta/2}{\sinh(t\dot R^\eta/2)}e^{t\dot R^\eta/2}x\mid y\,\Big\rangle_{\mathbb C^n}\bigg\}.
\end{aligned}
\end{equation}

The following was obtained by~\cite[Theorem 4.3]{HZ23}

\begin{theorem}\label{t-gue210521yyda}
With the notations used above, we have 
\begin{equation}\label{e-gue210521yydb}
\begin{split}
&e^{-t\Box_{H_n,\Phi}}(x,y)\\
&=\frac{1}{2\pi}\int e^{i<x_{2n+1}-y_{2n+1},\eta>+\frac{\beta}{2}\bigr((x_{2n+1}-y_{2n+1})+i\lambda(-\abs{z}^2+\abs{w}^2)\bigr)+\frac{-\Phi_0(w)+\Phi_0(z)}{2}}\\
&\quad\quad\times e^{-t\Box_\eta}(z,w)d\eta\in\mathcal{D}'(\mathbb R_+\times H_n\times H_n,\Lambda^\bullet(\mathbb CT^*H_n)\boxtimes(\Lambda^\bullet(\mathbb CT^*H_n))^*),
\end{split}
\end{equation}
where $z=(x_1,\ldots,x_{2n})$, $w=(y_1,\ldots,y_{2n})$, $e^{-t\Box_\eta}(z,w)$ is given by \eqref{e-gue210521yyd},
	where $\dot R^\eta$ and $\omega^\eta_{\mathbb C^n}$ are as in \eqref{Reta} and \eqref{omega} respectively, and $\lambda(-\abs{z}^2+\abs{w}^2):=\sum^n_{j=1}\lambda_j(-\abs{z_j}^2+\abs{w_j}^2)$.
\end{theorem}

We pause and introduce some notations. 

\begin{definition}\label{d-gue250921ycdt}
		Let $L$ be a CR line bundle over $X$ and $h^L$ be the Hermitian fiber metric on $L$ with local weight $\phi$. The curvature of $(L, h^L)$ at $x \in D$ with respect to $\phi$ is the Hermitian quadratic form $\mathcal{R}_x^\phi$ on $T^{1,0}_xX$ defined by
		\[
		\mathcal{R}^\phi_x(U, \overline{V}) = \frac{1}{2} \langle  d(\overline{\partial}_b \phi - \partial_b \phi)(x), U \wedge \overline{V}\rangle, \quad U, V \in T^{1,0}_xX,
		\]
		where $d$ is the usual exterior derivative.
	\end{definition}
	
	Let 
	\begin{equation}\label{e-gue250921ycdq}
		\begin{split}
			\dot{\mathcal{R}}^\phi_x : T^{1,0}_xX \to T^{1,0}_xX, \\
			\dot{\mathcal{L}}_x : T^{1,0}_xX \to T^{1,0}_xX,
		\end{split}
	\end{equation}
	be the linear maps given by $\langle \dot{\mathcal{R}}^\phi_xU \mid V \rangle = \mathcal{R}^\phi_x(U, \overline{V})$, $\langle \dot{\mathcal{L}}_x U \mid V \rangle = \mathcal{L}_x(U, \overline{V})$, for all $U, V \in T^{1,0}_xX$. For every $\eta \in \mathbb{R}$, let
	\[
	\det (\dot{\mathcal{R}}^\phi_x -2\eta \dot{\mathcal{L}}_x) = \mu_1(x) \cdots \mu_n(x),
	\]
	where $\mu_j(x)$, $j = 1, \cdots, n$, are the eigenvalues of $\dot{\mathcal{R}}^\phi_x -2\eta \dot{\mathcal{L}}_x$ with respect to $\langle \cdot \mid \cdot \rangle$. Let $\set{U_{1}(x),\ldots,U_{n}(x)}$ be an orthonormal frame of $T^{1,0}_xX$ and let
$\{\omega^j \}^n_{j=1} \subset T^{\ast1,0}X$ be the dual frame of $\set{U_{1},\ldots,U_{n}}$.
  For every $\eta\in\mathbb R$, Put 
	\begin{equation}\label{e-gue250920ycda}
		\omega^\eta_x = \sum_{j, l =1}^n (\dot{\mathcal{R}}^\phi_x -2\eta \dot{\mathcal{L}}_x)(U_l, \overline{U}_j) \overline{\omega}^j \wedge (\overline{\omega}^l \wedge)^\star : T^{\ast0,q}_x X \to T^{\ast0,q}_xX,
	\end{equation}
	where $\{\omega^j \}^n_{j=1} \subset T^{\ast1,0}X$ is the dual frame of $\set{U_{1,\eta},\ldots,U_{n,\eta}}$.
    
Let ${\rm Spec\,}\Box_{b,L^k}$ be the set of all spectrum of $\Box_{b,L^k}$. For $\lambda\in{\rm Spec\,}\Box_{b,L^k}$, let 
\[E_\lambda(X,L^k):=\set{u\in{\rm Dom\,}\Box_{b,L^k};\, \Box_{b,L^k}u=\lambda u}.\]
For $\mu\geq0$, put 
\begin{equation}\label{e-gue250709yyd}
\begin{split}
&E_{0<\lambda\leq\mu}(X,L^k):=\oplus_{\lambda\in{\rm Spec\,}\Box_{b,L^k},0<\lambda\leq\mu}E_\lambda(X,L^k),\\
&E_{\leq\mu}(X,L^k):=\oplus_{\lambda\in{\rm Spec\,}\Box_{b,L^k},0\leq\lambda\leq\mu}E_\lambda(X,L^k).
\end{split}
\end{equation}
For $q\in\set{0,1,\ldots,n}$, put 
\begin{equation}\label{e-gue250709yydz}
\begin{split}
&E^{(q)}_{0<\lambda\leq\mu}(X,L^k):=E_{0<\lambda\leq\mu}(X,L^k)\cap L^2_{(0,q)}(X,L^k),\\
&E^{(q)}_{\leq\mu}(X,L^k):=E_{\leq\mu}(X,L^k)\cap L^2_{(0,q)}(X,L^k). 
\end{split}
\end{equation}

Let 
\begin{equation}\label{e-gue250709yydI}
\Pi_{L^k,\mu}: L^2(X,L^k)\To E_{\leq\mu}(X,L^k)
\end{equation}be the orthogonal projection and let $\Pi_{L^k,\mu}(x,y)$ be the distribution kernel of $\Pi_{L^k,\mu}$. When $\mu=0$, $\Pi_{L^k,\mu}=\Pi_{L^k}$, $\Pi_{L^k,\mu}(x,y)=\Pi_{L^k}(x,y)$. 

We now come back to our situation. 
Let $p\in X$, let $x=(x_1,\ldots,x_{2n+1})$ be local coordinates of $X$ defined on an open set $D$ of $p$ with $x(p)=0$ and let $s$ be a local CR trivializing section of $L$ on $D$ such that 
\eqref{local1}, \eqref{local2}, \eqref{U} and \eqref{phi} hold. We will use the same notations as before. 
There exists $\Pi_{L^k, \frac{k}{\log k},s}(t,x,y)\in\mathcal{D}'(D\times D, T^{*0,\bullet}X\boxtimes(T^{*0,\bullet}X)^*)$ such that 
\begin{equation}\label{e-gue210303yydIz}
(\Pi_{L^k,\frac{k}{\log k}}u)(x)=s^k(x)\otimes e^{\frac{k\phi(x)}{2}}\int_D\Pi_{L^k,\frac{k}{\log k},s}(x,y)e^{-\frac{k\phi(y)}{2}}\hat u(y)dv_X(y)\ \ \mbox{on $D$}, 
\end{equation}
for every $u=s^k\otimes\hat u\in\Omega^{0,\bullet}_c(D,L^k)$, $\hat u\in\Omega^{0,\bullet}_c(D)$. Note that 
\begin{equation}\label{e-gue210303yydIIz}
\Pi_{L^k,\frac{k}{\log k}}(x,y)=s^k(x)\otimes\Pi_{L^k,\frac{k}{\log k},s}(t,x,y)e^{\frac{k\phi(x)-k\phi(y)}{2}}\otimes (s^k(y))^*\ \ \mbox{on $D$}. 
\end{equation}
Let 
\begin{equation}\label{e-gue210305yydz}
\begin{split}
&\Pi_{k\phi}(x,y):=e^{\frac{k\phi(x)}{2}}\Pi_{L^k,\frac{k}{\log k}s}(x,y)e^{-\frac{k\phi(y)}{2}},\\
&\Pi_{k\phi_0}(x,y):=e^{-k\rho(x)}\Pi_{k\phi}(x,y)e^{k\rho(y)}=e^{-k\rho(x)+\frac{k\phi(x)}{2}}\Pi_{L^k,\frac{k}{\log k},s}(x,y)e^{k\rho(y)-\frac{k\phi(y)}{2}},
\end{split}
\end{equation}
where $\rho$ is as in \eqref{rho} and $\phi_0$ is as in \eqref{phi0}. Let 
\[\Pi_{k\phi_0}: \Omega^{0,\bullet}_c(D)\To \mathcal{D}'(D,T^{*0,\bullet}X)\]
be the continuous operator with distribution kernel $\Pi_{k\phi_0}(x,y)$ with respect to $dv_X$. Note that 
\begin{equation}\label{e-gue210305yydIz}
(\Pi_{k\phi_0}u)(x)=\int \Pi_{k\phi_0}(x,y)u(y)dv_X(y),\ \ u\in\Omega^{0,\bullet}_c(D).
\end{equation}

Let 
\begin{equation}\label{e-gue250523yyd}
\Pi_{(k)}(x,y):=k^{-n-1}\Pi_{k\phi_0}(F_kx,F_ky)\in\mathcal{D}'(B_{\log k}\times B_{\log k},F^*_kT^{*0,\bullet}X\boxtimes(F^*_kT^{*0,\bullet}X)^*),
\end{equation}
where $F_kx=(\frac{z}{\sqrt{k}},\frac{\theta}{k})$. Let 
\[\Pi_{(k)}: F^*_k\Omega^{0,\bullet}_c(B_{\log k})\To\mathcal{D}'(B_{\log k}, F^*_kT^{*0,\bullet}X)\] 
be the continuous operator given by 
\begin{equation}\label{e-gue210325yydIz}
(\Pi_{(k)}u)(x)=\int\Pi_{(k)}(x,y)u(y)m(F_ky)dy,\ \ u\in F^*_k\Omega^{0,\bullet}_c(B_{\log k}).
\end{equation}

    For any bounded open set $W\subset H_n$, we now identify $\Pi_{H_n,\Phi}(x,y)$ and $\Pi_{(k)}(x,y)$ as elements in 
$\mathcal{D}'(W\times W, \Lambda^*(\mathbb CT^*H_n)\boxtimes(\Lambda^*(\mathbb CT^*H_n))^*)$. We can repeat the proof of~\cite[Theorem 3.9]{HZ23} with minor changes and deduce that there is a sequence $k_1<k_2<\cdots$, such that for any $r>0$, we have 
\begin{equation}\label{e-gue250524ycd}\lim_{k_j\To+\infty}\Pi_{(k_j)}(x,y)=B(x,y)\end{equation}
in $\mathcal{D}'(B_r\times B_r,\Lambda^*(\mathbb CT^*H_n)\boxtimes(\Lambda^*(\mathbb CT^*H_n))^*)$ topology, where $B(x,y)\in\mathcal{D}'(H_n\times H_n,T^{*0,\bullet}H_n\boxtimes(T^{*0,\bullet}H_n)^*)$ and 
\begin{equation}\label{e-gue250426ycd}
\begin{split}
&(\,Bu\,|\,v\,)_{\Phi}\leq\norm{u}_{\Phi}\norm{v}_{\Phi},\ \ \mbox{for all $u, v\in\Omega^{0,\bullet}_c(H_n)$},\\
&\Box_{H_n,\Phi}B=0\ \ \mbox{on $\mathcal{D}'(H_n,T^{*0,\bullet}H_n)$},
\end{split}
\end{equation}
where $B: \Omega^{0,\bullet}(H_n)\To\mathcal{D}'(H_n,T^{*0,\bullet}H_n)$ is the continuous operator with distribution kernel $B(x,y)$ with respect to $(\,\cdot\,|\,\cdot\,)_{H_n}$. 

We are going to prove that $B=\Pi_{H_n,\Phi}$. 
Recall that
\begin{equation}\label{e-gue250429yyd}
\Phi_\eta(z)=\Phi_0(z)-2\eta\sum^n_{j=1}\lambda_j(z)\abs{z_j}^2,
\end{equation}
where $\eta\in\mathbb R$, $\Phi_0(z)$ and $\lambda_j$, $j=1,\ldots,n$, are as in \eqref{U} and \eqref{e-gue250429yydI}. 
For every $q\in\set{0,1,\ldots,n}$, let
\begin{equation}\label{e-gue250429yyda}
\begin{split}
&I_q:=\{\eta\in\mathbb R;\, 
\mbox{$\left(\frac{\pr^2\Phi_\eta}{\pr z_j\pr\ol z_\ell}\right)^n_{j,\ell=1}$ has exactly $q$ negative}\\
&\quad\mbox{and $n-q$ positive eigenvalues}\}.
\end{split}
\end{equation}
We now fix $q\in\set{0,1,\ldots,n}$. For $\eta\in I_q$, write 
\begin{equation}\label{e-gue250429ycd}
\Phi_\eta(z)=\sum^n_{j=1}\mu_j(\eta)\abs{z_j(\eta)}^2,
\end{equation}
where $\mu_j(\eta)\in\mathbb R$, $\mu_j(\eta)<0$, $j=1,\ldots,q$, $\mu_j(\eta)>0$, $j=q+1,\ldots,n$
, $z_j(\eta)=\sum^n_{\ell=1}a_{j,\ell}(\eta)z_\ell$, $j=1,\ldots,n$, $a_{j,\ell}(\eta)\in\mathbb C$, $j, \ell=1,\ldots,n$, and the matrix $A_\eta:=\left(a_{j,\ell}(\eta)\right)^n_{j,\ell=1}$ satisfies $A^*_\eta A_\eta=A_\eta A^*_\eta=I$, $A^*_\eta=(\ol{A_\eta})^t$, $(\ol{A_\eta})^t$ 
is the transpose of $\ol{A_\eta}$, $\ol{A_\eta}$ is the conjugate of $A_\eta$.
Let 
\begin{equation}
\begin{split}
\Psi_\eta(z,w):=&i\sum^n_{j=1}\abs{\mu_j(\eta)}\abs{z_j(\eta)-w_j(\eta)}^2\\
&+i\sum^n_{j=1}\mu_j(\eta)(\ol{z_j(\eta)}w_j(\eta)-z_j(\eta)\ol {w_j(\eta)})\in\mathcal{C}^\infty(\mathbb C^n\times\mathbb C^n).
\end{split}
\end{equation}
For $x, y\in H_n$,
Let \begin{equation}\label{e-gue250523ycd}
\begin{split}
\tau^{(q)}_\eta(x,y):=&(d\ol{z_1(\eta)}\wedge\cdots\wedge d\ol{z_q(\eta)})(x)\otimes (d\ol{z_1(\eta)}\wedge\cdots\wedge d\ol{z_q(\eta)})^*(y): \\
\Lambda^\bullet_y(\mathbb CT^*H_n)&\To \Lambda^\bullet_x(\mathbb CT^*H_n),\\
u&\mapsto (d\ol{z_1(\eta)}\wedge\cdots\wedge d\ol{z_q(\eta)})(x)\langle\,u\,|\,(d\ol{z_1(\eta)}\wedge\cdots\wedge d\ol{z_q(\eta)})(y)\,\rangle_{H_n}
\end{split}
\end{equation}
 be the linear transform. Put
\begin{equation}\label{e-gue250429ycdII}
\tau^{(q)}_\eta: \Lambda^\bullet(\mathbb CT^*H_n)\To\set{cd\ol{z_1(\eta)}\wedge\cdots\wedge d\ol{z_q(\eta)};\, c\in\mathbb C}
\end{equation}
be the orthogonal projection with respect to $\langle\,\cdot\,|\,\cdot\,\rangle_{H_n}$. Note that the distribution kernel of $\tau^{(q)}_\eta$ is $\tau^{(q)}_\eta(x,y)$.
We can repeat the proof of~\cite[Theorems 1.1, 1.3]{HHL18} and deduce the following 
explicit formula for $\Pi^{(q)}_{H_n,\Phi}$.  

\begin{theorem}\label{t-gue250429yyd}
Let $q\in\set{0,1,\ldots,n}$. With the notations used above, we have 
\begin{equation}\label{e-gue250429ycdIII}
\begin{split}
&\Pi^{(q)}_{H_n,\Phi}(x,y)\\
&=\int_{\eta\in I_q}e^{\frac{1}{2}(\Phi_\eta(z)-\Phi_\eta(w))+i\frac{1}{2}\Psi_\eta(z,w)+(i\eta+\frac{\beta}{2})(x_{2n+1}-y_{2n+1})}\\
&\times e^{-(\eta-\frac{i}{2}\beta)\sum^n_{j=1}\lambda_j(\abs{w_j}^2-\abs{z_j}^2)}(2\pi)^{-n-1}\abs{\det\left(\frac{\pr^2\Phi_\eta}{\pr z_j\pr\ol z_\ell}\right)^n_{j,\ell=1}}\tau^{(q)}_\eta(x,y)d\eta\\
&\in\mathcal{D}'(H_n\times H_n,\Lambda^\bullet(\mathbb CT^*H_n)\boxtimes\Lambda^\bullet(\mathbb CT^*H_n)^*), 
\end{split}
\end{equation}
where $I_q$ and $\tau^{(q)}_\eta(x,y)$ are as in \eqref{e-gue250429yyda} and \eqref{e-gue250523ycd} respectively. 

Moreover, we have 
\begin{equation}\label{e-gue250524yyde}
\begin{split}
&\Pi_{H_n,\Phi}(x,y)\\
&=\sum^n_{q=0}\int_{\eta\in I_q}e^{\frac{1}{2}(\Phi_\eta(z)-\Phi_\eta(w))+i\frac{1}{2}\Psi_\eta(z,w)+(i\eta+\frac{\beta}{2})(x_{2n+1}-y_{2n+1})}\\
&\times e^{-(\eta-\frac{i}{2}\beta)\sum^n_{j=1}\lambda_j(\abs{w_j}^2-\abs{z_j}^2)}(2\pi)^{-n-1}\abs{\det\left(\frac{\pr^2\Phi_\eta}{\pr z_j\pr\ol z_\ell}\right)^n_{j,\ell=1}}\tau^{(q)}_\eta(x,y)d\eta\\
&\in\mathcal{D}'(H_n\times H_n,\Lambda^\bullet(\mathbb CT^*H_n)\boxtimes\Lambda^\bullet(\mathbb CT^*H_n)^*).
\end{split}
\end{equation}
\end{theorem} 
Now we are in a position to show the following

\begin{theorem}\label{t-gue250524yyd}
With the notations used above, and recall that we work with Assumption~\ref{a-gue250426yyd1}. We have 
\begin{equation}
B(x,y)=\Pi_{H_n,\Phi}(x,y)\ \ \mbox{on $\mathcal{D}'(H_n\times H_n,\Lambda^\bullet(\mathbb CT^*H_n)\boxtimes\Lambda^\bullet(\mathbb CT^*H_n)^*)$},
\end{equation}
where $B(x,y)$ is as in \eqref{e-gue250524ycd}. 
\end{theorem}

\begin{proof}
Let $M=\cup^n_{q=0}M_q$, where $M_q$ is any compact subset of $I_q$, $q=0,1,\ldots,n$. Let 
\begin{equation}\label{e-gue250524ycdb}
\begin{split}
&\Pi_{H_n,\Phi,M}(x,y)\\
&=\sum^n_{q=0}\int_{\eta\in I_q, \eta\in M_q}e^{\frac{1}{2}(\Phi_\eta(z)-\Phi_\eta(w))+i\frac{1}{2}\Psi_\eta(z,w)+(i\eta+\frac{\beta}{2})(x_{2n+1}-y_{2n+1})}\\
&\times e^{-(\eta-\frac{i}{2}\beta)\sum^n_{j=1}\lambda_j(\abs{w_j}^2-\abs{z_j}^2)}(2\pi)^{-n-1}\abs{\det\left(\frac{\pr^2\Phi_\eta}{\pr z_j\pr\ol z_\ell}\right)^n_{j,\ell=1}}\tau^{(q)}_\eta(x,y)d\eta\\
&\in\mathcal{D}'(H_n\times H_n,\Lambda^\bullet(\mathbb CT^*H_n)\boxtimes\Lambda^\bullet(\mathbb CT^*H_n)^*)
\end{split}
\end{equation}
and let $\Pi_{H_n,\Phi,M}: L^2_{(0,\bullet)}(H_n)\To L^2_{(0,\bullet)}(H_n)$ be the continuous operator with distribution kernel $\Pi_{H_n,\Phi,M}(x,y)$. Let 
$\chi, \tau\in\mathcal{C}^\infty_c(H_n)$, $0\leq\chi\leq 1$, $0\leq\tau\leq 1$, $\tau\equiv1$ near $0\in H_n$ and $\chi\equiv1$ near ${\rm supp\,}\tau$. For every $k\in\mathbb N$, put 
\begin{equation}\label{e-gue250524ycdg}
\begin{split}
&\chi_k(x)=\chi((\frac{\sqrt{k}x'}{\log k}, \frac{kx_{2n+1}}{\log k})),\\
&\tau_k(x)=\tau((\frac{\sqrt{k}x'}{\log k}, \frac{kx_{2n+1}}{\log k})),
\end{split}
\end{equation}
where $x'=(x_1,\ldots,x_{2n})$, $y'=(y_1,\ldots,y_{2n})$. 
Put 
\begin{equation}\label{e-gue250524ycdh}
\begin{split}
\Pi_{H_n,k,M}(x,y)&:=e^{k\rho(x)}s^k(x)\otimes k^{n+1}\chi_k(x)\times\\
&\Pi_{H_n,\Phi,M}((\frac{\sqrt{k}x'}{\log k}, \frac{kx_{2n+1}}{\log k}),(\frac{\sqrt{k}y'}{\log k}, \frac{ky_{2n+1}}{\log k}))
\tau_k(y)\otimes s^k(y)e^{-k\rho(y)},
\end{split}
\end{equation}
where $s$ is the CR trivializing section as in the discussion before \eqref{e-gue210303yydIz} and $\rho(x)$ is as in \eqref{rho}. For $k$ large, $\Pi_{H_n,k,M}(x,y)$ is well-defined as an element in $\mathcal{D}'(X\times X,\Lambda^\bullet(\mathbb CT^*X)\boxtimes(\Lambda^\bullet(\mathbb CT^*X)^*)$. Let 
\[\Pi_{H_n,k,M}: \mathcal{C}^\infty(X,\Lambda^\bullet(\mathbb CT^*X))\To\mathcal{D}'(X,\Lambda^\bullet(\mathbb CT^*X))\]
be the continuous operator with distribution kernel $\Pi_{H_n,k,M}(x,y)$. It is easy to check that 
\[\Pi_{H_n,k,M}: L^2(X,\Lambda^\bullet(\mathbb CT^*X))\To L^2(X,\Lambda^\bullet(\mathbb CT^*X))\]
is continuous. Let $u, v\in\mathcal{C}^\infty_c(H_n,\Lambda^\bullet(T^*H_n))$. Put 
\[\begin{split}
&u_k(x):=e^{k\rho(x)}s^k(x)\otimes u((\sqrt{k}x',kx_{2n+1}))\in\mathcal{C}^\infty(X,L^k\otimes\Lambda^\bullet(\mathbb CT^*X)),\\
&v_k(x):=e^{k\rho(x)}s^k(x)\otimes v((\sqrt{k}x',kx_{2n+1}))\in\mathcal{C}^\infty(X,L^k\otimes\Lambda^\bullet(\mathbb CT^*X)).
\end{split}\]
From \eqref{k(k)}, Theorem~\ref{t-gue250429yyd} and the coefficients of $\Box_{\rho,(k)}$ converge to the coefficients of $\Box_{H_n,\Phi}$ faster that $\frac{1}{\log k}$ on ${\rm  supp\,}\chi_k\cup{\rm supp\,}\tau_k$, it is straightforward to check that 
\begin{equation}\label{e-gue250524yydp}
\norm{\Box_{b,L^k}\Pi_{H_n,k,M}u_k}_{h^{L^k}}\leq k\delta_k\norm{u_k}_{h^{L^k}},
\end{equation}
where $\delta_k>0$ is a sequence with $\lim_{k\To+\infty}(\log k)\delta_k=0$. From \eqref{e-gue250524yydp}, we get 
\begin{equation}\label{e-gue250524yydq}
\lim_{k\To+\infty}\norm{\Pi_{H_n,k,M}u_k-\Pi_{L^k,\frac{k}{\log k}}\Pi_{H_n,k,M}u_k}_{h^{L^k}}=0.
\end{equation}
From \eqref{e-gue250524ycd} and \eqref{e-gue250524yydq}, we can check that 
\begin{equation}\label{e-gue250524ycdu}
\begin{split}
&(\,\Pi_{H_n,\Phi,M}u\,|\,Bv\,)_{H_n}\\
&=\lim_{k_j\To+\infty}(\,\Pi_{H_n,\Phi,M}u\,|\,\Pi_{(k_j)}v\,)_{H_n}\\
&=\lim_{k_j\To+\infty}(\,\Pi_{H_n,k_j,M}u_{k_j}\,|\,\Pi_{L^{k_j},\frac{k_j}{\log k_j}}v_{k_j}\,)_{L^{k^j}}\\
&=\lim_{k_j\To+\infty}(\,\Pi_{L^{k_j},\frac{k_j}{\log k_j}}\Pi_{H_n,k_j,M}u_{k_j}\,|\,v_{k_j}\,)_{L^{k^j}}\\
&=\lim_{k_j\To+\infty}(\,\Pi_{H_n,k_j,M}u_{k_j}\,|\,v_{k_j}\,)_{L^{k^j}}\\
&=(\,\Pi_{H_n,\Phi,M}u\,|\,v\,)_{H_n}. 
\end{split}
\end{equation}
Hence, $B\Pi_{H_n,\Phi,M}=\Pi_{H_n,\Phi,M}=\Pi_{H_n,\Phi,M}B$ on $\mathcal{C}^\infty_c(H_n,\Lambda^\bullet(\mathbb CT^*H_n))$. Thus, 
\begin{equation}\label{e-gue250524yydz}
\mbox{$B\Pi_{H_n,\Phi}=\Pi_{H_n,\Phi}=\Pi_{H_n,\Phi}B$ on $\mathcal{C}^\infty_c(H_n,\Lambda^\bullet(\mathbb CT^*H_n))$}.
\end{equation}
From \eqref{e-gue250426ycd}, we see that 
    \begin{equation}\label{e-gue250524ycdz}
\mbox{$B\Pi_{H_n,\Phi}=B$ on $\mathcal{C}^\infty_c(H_n,\Lambda^\bullet(\mathbb CT^*H_n))$}.
\end{equation}
From \eqref{e-gue250524yydz} and \eqref{e-gue250524ycdz}, we get $B=\Pi_{H_n,\Phi}$ on $\mathcal{C}^\infty_c(H_n,\Lambda^\bullet(\mathbb CT^*H_n))$. The theorem follows. 
\end{proof}

From Theorem~\ref{t-gue250524yyd}, we get the following 

\begin{theorem}\label{t-gue250524yydw}
Let $p\in X$, let $x=(x_1,\ldots,x_{2n+1})$ be local coordinates of $X$ defined on an open set $D$ of $p$ with $x(p)=0$ and let $s$ be a local CR trivializing section of $L$ on $D$ such that 
\eqref{local1}, \eqref{local2}, \eqref{U} and \eqref{phi} hold. Let $\Pi_{(k)}(x,y)\in\mathcal{D}'(B_{\log k}\times B_{\log k},F^*_kT^{*0,\bullet}X\boxtimes(F^*_kT^{*0,\bullet}X)^*)$ be as 
in \eqref{e-gue250523yyd}. We have for every $r>0$, 
\[\lim_{k\To+\infty}\Pi_{(k)}(x,y)=\Pi_{H_n,\Phi}(x,y)\]
in $\mathcal{D}'(B_r\times B_r,\Lambda^*(\mathbb CT^*H_n)\boxtimes(\Lambda^*(\mathbb CT^*H_n))^*)$ topology.
\end{theorem} 

\begin{theorem}\label{t-gue250528yyd}
Recall that we work with the assumption that the Levi form is non-degenerate of constant signature $(n_-,n_+)$. We assume further that $n_-=n_+$ or $\abs{n_--n_+}>1$.
With the notations used above, let $I\subset\mathbb R_+$ be a bounded open interval. Then, there is a constant $C>0$ independent of $k$ such that 
\begin{equation}\label{e-gue250528yyyd}
\norm{(e^{-\frac{t}{k}\Box_{b,L^k}}(I-\Pi_{L^k,\frac{k}{\log k}}))(x,x)}_{\mathcal{C}^0(I\times X)}\leq Ck^{n+1}.
\end{equation}
\end{theorem}

\begin{proof}
Let $\varepsilon>0$ be a small constant so that $I\subset I'=[\varepsilon,+\infty)$. Let $q\in\set{0,1,\ldots,n}$. Suppose $q\notin\set{n_-,n_+}$. Since $Y(q)$ holds, we can repeat the proofs of~\cite[Theorems 1.1, 1.3]{HZ23} and conclude that 
there is a constant $C>0$ independent of $k$ such that 
\begin{equation}\label{e-gue250528yyydI}
\norm{(e^{-\frac{t}{k}\Box^{(q)}_{b,L^k}}(I-\Pi^{(q)}_{L^k,\frac{k}{\log k}}))(x,x)}_{\mathcal{C}^0(I'\times X)}\leq Ck^{n+1}.
\end{equation}

Suppose $q=n_-$ or $q=n_+$. Since $n_-=n_+$ or $\abs{n_--n_+}>1$, we can repeat the proof of~\cite[Theorem 1.3]{HZ23} and conclude that 
there is a constant $C_1>0$ independent of $k$ such that 
\begin{equation}\label{e-gue250528yyydII}
\begin{split}
&\norm{(\ddbar_be^{-\frac{t}{k}\Box^{(q)}_{b,L^k}})(x,x)}_{\mathcal{C}^0(I'\times X)}\leq C_1k^{n+\frac{3}{2}},\\
&\norm{(\ddbar^*_be^{-\frac{t}{k}\Box^{(q)}_{b,L^k}})(x,x)}_{\mathcal{C}^0(I'\times X)}\leq C_1k^{n+\frac{3}{2}},
\end{split}
\end{equation}
and 
\begin{equation}\label{e-gue250529yyyd}
\begin{split}
&\norm{(\Box^{(q)}_{b,L^k}e^{-\frac{t}{k}\Box^{(q)}_{b,L^k}})(x,x)}_{\mathcal{C}^0(I'\times X)}\leq C_1k^{n+2}. 
\end{split}
\end{equation}
From \eqref{e-gue250529yyyd} and by integrating over $t$, we conclude that 
\begin{equation}\label{e-gue250529yyyda}
\begin{split}
&\norm{(e^{-\frac{t}{k}\Box^{(q)}_{b,L^k}}(I-\Pi^{(q)}_{L^k,\frac{k}{\log k}}))(x,x)}_{\mathcal{C}^0(I\times X)}\leq\frac{C_2}{k}\norm{(\Box_{b,L^k}e^{-\frac{t}{k}\Box^{(q)}_{b,L^k}})(x,x)}_{\mathcal{C}^0(I\times X)},
\end{split}
\end{equation}
where $C_2>0$ is a constant independent of $k$.

From \eqref{e-gue250529yyyd} and 
\eqref{e-gue250529yyyda}, the theorem follows.
\end{proof}


    Now we set \begin{equation}\label{e-gue250529ycdu}
	\begin{split}
	\mathbb R_x(q):=\{\eta&\in\mathbb R;\, \dot{\mathcal R}^\phi_x-2\eta\dot{\mathcal L}_x\,\text{has exactly $q$ negative eigenvalues} 
	\\&\text{and $n-q$ positive eigenvalues}\},
	\end{split}
	\end{equation}
    and let $1_{\mathbb R_x(q)}(\eta)=1$ if $\eta\in\mathbb R_x(q)$, $1_{\mathbb R_x(q)}(\eta)=0$ if $\eta\notin\mathbb R_x(q)$.
    
With the notations used in the discussion after \eqref{e-gue210303yydII} and Theorem~\ref{t-gue250429yyd}, we can check that 
    \begin{equation}\label{e-gue250529ycde}
\det\left(\frac{\pr^2\Phi_\eta}{\pr z_j\pr\ol z_\ell}\right)^n_{j,\ell=1}
={\rm det}\,(\dot{\mathcal{R}}^\phi_p -2\eta \dot{\mathcal{L}}_p)
    \end{equation}
    and 
\begin{equation}\label{e-gue250529ycdf}
I_q=\mathbb R_p(q),
    \end{equation}
for all $q=0,1,\ldots,n$, where $I_q$ is given by \eqref{e-gue250429yyda}. 
    

    Let $q\in\set{0,1,\ldots,n}$ and let $x\in X$. Let $A_x: \Lambda^{\bullet}_x(\mathbb CT^*X)\To\Lambda^{\bullet}_x(\mathbb CT^*X)$ be a linear transformation. For $q\in\set{0,1,\ldots,n}$, let
\begin{equation}\label{e-gue210529yydI}
\operatorname{Tr}^{(q)}A_x:=\sum^d_{j=1}\langle\,A_xv_j\,|\, v_j\,\rangle,
\end{equation}
where $\set{v_j}^d_{j=1}$ is an orthonormal basis for $T^{*0,q}_xX$.    
    
    From \cite[Proposition 4.2]{HM12}, it is easy to see that for every $q=0,1,\ldots,n$, every $x \in X$, the integral
	\[
	\int_{\eta\in\mathbb R}\Bigr(\frac{\det(\dot{\mathcal{R}}_x^\phi - 2 \eta \dot{\mathcal{L}}_x )}{\det \big( 1 - e^{-t ( \dot{\mathcal{R}}_x^\phi - 2 \eta \dot{\mathcal{L}}_x )}  \big)}{\rm Tr}^{(q)}\,e^{-t\omega_x^\eta}-\abs{\det(\dot{\mathcal{R}}_x^\phi - 2 \eta \dot{\mathcal{L}}_x)}1_{\mathbb R_x(q)}(\eta)\Bigr) d\eta 
	\]
    converges and is independent of the choice of local weight $\phi$ and the function
    \[x\in X\mapsto
	\int_{\eta\in\mathbb R}\Bigr(\frac{\det(\dot{\mathcal{R}}_x^\phi - 2 \eta \dot{\mathcal{L}}_x )}{\det \big( 1 - e^{-t ( \dot{\mathcal{R}}_x^\phi - 2 \eta \dot{\mathcal{L}}_x )}  \big)}{\rm Tr}^{(q)}\,e^{-t\omega_x^\eta}-\abs{\det(\dot{\mathcal{R}}_x^\phi - 2 \eta \dot{\mathcal{L}}_x)}1_{\mathbb R_x(q)}(\eta)\Bigr) d\eta 
	\]
is a globally defined continuous function, where $\omega^\eta_x$ is given by \eqref{e-gue250920ycda}. 

We come back to our situation. 
From Theorem~\ref{t-gue250425yyd}, Theorem~\ref{t-gue210521yyda}, Theorem~\ref{t-gue250429yyd}, Theorem~\ref{t-gue250524yyd} and Theorem~\ref{t-gue250528yyd}, we can repeat the proof of~\cite[Theorem 1.1]{HZ23} and get 

\begin{theorem}\label{t-gue250529yyd}
We assume that the Levi form is non-degenerate of constant signature $(n_-,n_+)$. We assume further that $n_-=n_+$ or $\abs{n_--n_+}>1$. With the notations used above, let $I\subset\mathbb R_+$ be a bounded open interval. Let $q\in\set{0,1,\ldots,n}$. Then, 
\begin{equation}\label{e-gue250529yydb}
\begin{split}
&\lim_{k\To+\infty}k^{-(n+1)}\operatorname{Tr}^{(q)}(e^{-\frac{t}{k}\Box^{(q)}_{b,L^k}}(I-\Pi^{(q)}_{L^k,\frac{k}{\log k}}))(x,x))=\\
&\frac{1}{(2\pi)^{n+1}}\int_{\mathbb R}\Bigr(\dfrac{\det(\dot{\mathcal R}^\phi_x-2\eta\dot{\mathcal L}_x)}{\det\big(1-e^{-t(\dot{\mathcal R}^\phi_x-2\eta\dot{\mathcal L}_x)}\big)}{\rm Tr\,}^{(q)}e^{-t\omega_x^\eta}-\abs{\det(\dot{\mathcal{R}}_x^\phi - 2 \eta \dot{\mathcal{L}}_x)}1_{\mathbb R_x(q)}(\eta)\Bigr)d\eta
\end{split}
\end{equation}
in $\mathcal{C}^0(I\times X)$ topology.
\end{theorem} 

We need the following

\begin{theorem}\label{t-gue250603yyd}
We assume that the Levi form is non-degenerate of constant signature $(n_-,n_+)$. We assume further that $n_-=n_+$ or $\abs{n_--n_+}>1$. Let $q\in\set{0,1,\ldots,n}$. We can find $k$-dependent smooth functions $b^{(q)}_{j,k}(x)\in\mathcal{C}^\infty(X)$, $j=0,1,\ldots$, with
\begin{equation}\label{e-gue250529ycdm}
k^{-n-1}b^{(q)}_{j,k}(x)=b^{(q)}_j(x)+k^{-\frac{1}{2}}r^{(q)}_{j,k}(x),\ \ j=0,1,\ldots,
\end{equation}
where $b^{(q)}_j(x)\in\mathcal{C}^\infty(X)$, $j=0,1,\ldots$, $\norm{r^{(q)}_{j,k}}_{\mathcal{C}^0(X)}\leq C_j$, $C_j>0$ is a constant independent of $k$, $j=0,1,\ldots$, such that for every $N\in\mathbb N$, we have 
\begin{equation}\label{e-gue250529ycdn}
\operatorname{Tr}^{(q)}(e^{-\frac{t}{k}\Box^{(q)}_{b,L^k}}(I-\Pi^{(q)}_{L^k})(x,x))=\sum^N_{j=0}t^{-n-1+j}b^{(q)}_{j,k}(x)+\delta^{(q)}_{N,k}(t,x),
\end{equation}
where $\delta^{(q)}_{N,k}(t,x)\in\mathcal{C}^\infty(\ol{\mathbb R}_+\times X)$ is a $k$-dependent smooth function and for every $m\in\mathbb N_0$, there is a constant $C_{m}>0$ independent of $k$ such that 
\begin{equation}\label{e-gue250529yydx}
\sup\,\set{\abs{\pr^m_t\delta^{(q)}_{N,k}(t,x)};\, x\in X}\leq t^{-n+N-m}k^{n+1}C_{m},
\end{equation}
for all $t\in(0,1]$ and all $k\gg1$.
\end{theorem}

\begin{proof}
Fix $p\in X$. 
We take local coordinates $x=(x_1,\ldots,x_{2n+1})=(z,\theta)=(z_1,\cdots,z_n,\theta)$ on an open set $D$ of $p$, such that \eqref{local1}, \eqref{local2}, \eqref{U} and \eqref{phi} hold. We will use the same notations as in the beginning of Section~\ref{s-gue250411yydI}. Fix $r>0$ and $N\in\mathbb N$, $N\gg n+1$. Let $\hat A_{(k),N}(t,x,y)$ be the proper supported complex Fourier integral operator on $B_r$ as in \eqref{e-gue250529ycdw} with respect to the scaled Kohn Laplacian $\Box_{\rho,(k)}$ (see \eqref{e-gue210614yydI}, for the definition of $\Box_{\rho,(k)}$). Since the coefficients of $\Box_{\rho,(k)}$ are uniform bounded in $k$ in $\mathcal{C}^\infty$ topology, from the construction of $\hat A_{(k),N}(t,x,y)$, the symbols and the phase of $\hat A_{(k),N}(t,x,y)$ are uniform bounded in $k$ in $\mathcal{C}^\infty$ topology and we can find $k$-dependent smooth functions $a_{j,k}(x)\in\mathcal{C}^\infty(B_r)$, $j=0,1,\ldots$, with
\begin{equation}\label{e-gue250529ycdmz}
a_{j,k}(0)=a_j+k^{-\frac{1}{2}}\gamma_{j,k},\ \ j=0,1,\ldots,
\end{equation}
where $a_j\in\mathbb R$, $j=0,1,\ldots$, $\abs{\gamma_{j,k}}\leq\hat C_j$, $\hat C_j>0$ is a constant independent of $k$ and the point $p$, $j=0,1,\ldots$, such that
\begin{equation}\label{e-gue250529ycds}
\operatorname{Tr}^{(q)}((\frac{\pr}{\pr t}\hat A_{(k),N})(t,x,x))=\sum^N_{j=0}t^{-n-2+j}a_{j,k}(x)+g_{N,k}(t,x),
\end{equation}
where $g_{N,k}(t,x)\in\mathcal{C}^\infty(\ol{\mathbb R_+}\times B_r)$ and
for every $m\in\mathbb N_0$ and every $0<r'<r$, there is a constant $\hat C_{r',m}>0$ independent of $k$, $t$ and the point $p$, such that 
\begin{equation}\label{e-gue250529yydk}
\sup\,\set{(\abs{\pr^m_tg_{N,k})(t,x)};\, x\in B_{r'}}\leq\hat C_{r',m}t^{-n-1+N-m},
\end{equation}
for all $k\gg1$ and all $t\in(0,1]$. 

Let $A_{(k)}(t,x,y)\in\mathcal{D}'(\mathbb R_+\times B_{\log k}\times B_{\log k}, F^*_kT^{*0,\bullet}X\boxtimes(F^*_kT^{*0,\bullet}X)^*)$ be as in \eqref{e-gue210325yyd}. We can repeat the proof of Theorem~\ref{T:011120251108} with minor change and deduce that on $B_r$
\begin{equation}\label{e-gue250530ycd}
A_{(k)}(t,x,y)=\hat A_{(k),N}(t,x,y)+\varepsilon_{N,k}(t,x,y),
\end{equation}
where $\varepsilon_{N,k}(t,x,y)\in\mathcal{C}^\infty(\ol R_+\times B_r\times B_r,F^*_kT^{*0,\bullet}X\boxtimes(F^*_kT^{*0,\bullet}X)^*)$ and for every $m\in\mathbb N_0$ and every $0<r'<r$, there is a constant $\Td C_{r',m}>0$ independent of $k$, $t$ and the point $p$, such that 
\begin{equation}\label{e-gue250529yydkz}
\sup\,\set{\abs{(\pr^m_t\varepsilon_{N,k})(t,x,y)};\, x, y\in B_{r'}}\leq\Td C_{r',m}t^{-n+N-m},
\end{equation}
for all $k\gg1$ and $t\in(0,1]$. Note that 
\begin{equation}\label{e-gue250530yyd}
(\frac{\pr}{\pr t}e^{-\frac{t}{k}\Box_{b,L^k}})(t,x,x)=k^{n+1}(\frac{\pr}{\pr t}A_{(k)})(t,x,x).
\end{equation}

From \eqref{e-gue250529ycdmz}, \eqref{e-gue250529ycds}, \eqref{e-gue250530ycd} and \eqref{e-gue250530yyd}, we deduce that 
\begin{equation}\label{e-gue250530yydz}
\operatorname{Tr}^{(q)}((\frac{\pr}{\pr t}e^{-\frac{t}{k}\Box_{b,L^k}})(t,p,p))=k^{n+1}\sum^N_{j=0}t^{-n-2+j}a_{j,k}(p)+k^{n+1}r_{N,k}(t),
\end{equation}
where $\abs{r_{N,k}(t)}\leq t^{-n-1+N}C_N$, for all $t\in(0,1]$ and all $k\geq1$, where $C_N>0$ is a constant independent of $k$ and the point $p$. 

From Theorem~\ref{t-gue250410ycdbz}, we see that 

\begin{equation}\label{e-gue250530yydb}
\mbox{$\operatorname{Tr}^{(q)}((\frac{\pr}{\pr t}e^{-t\Box_{b,L^k}})(t,x, x))\sim \sum_{j=0}^\infty t^{-n-2+j}\hat g_{j,k}(x)$ in $\tilde S^{-n-2}(\mathbb{R}_+ \times X)$},
		\end{equation}
where $\hat g_{j,k}(x)\in\mathcal{C}^\infty(X)$, $j=0,1,\ldots$. 

From \eqref{e-gue250530yydz} and \eqref{e-gue250530yydb}, we conclude that 
\[\hat g_{j,k}(p)=a_{j,k}(p),\ \ j=0,1,\ldots,N.\]
The proof is completed from this observation, \eqref{e-gue250529ycdmz}, \eqref{e-gue250530yydz} and \eqref{e-gue250530yydb}. 
\end{proof}

    \section{The asymptotics of the CR analytic torsion for CR line bundles}\label{s-gue250411}

	
	In this section, we first recall the Mellin transform, then we define the analytic torsion for the CR vector bundle $E$ over the CR manifold $X$, which answers the question raised by Bismut. The anomaly formula for the CR analytic torsion will also be given. Finally we will consider $E=L^k$, where $L^k$ is the $k$-th power of a CR line bundle $L$ over $X$ and establish Bismut-Vasserot type asymptotics of the analytic torsion for $L^k$ as $k\To+\infty$.

\subsection{Mellin transformation} 

     We first introduce some notations. Let $g(t)\in\mathcal{C}^\infty(\mathbb R_+)$. We write 
    \[\mbox{$g(t)\sim \sum^{\infty}_{j=0}g_{j}t^{-m+j}$ as $t\To0^+$},\]
    where $m\in\mathbb R$, $g_{j}\in\mathbb C$, $j=0,1,\ldots$, if for every $N>0$, we can find a constant $C_N>0$ independent of $t$ such that 
    \[g(t)-\sum^N_{j=0}g_{j}t^{-m+j}\leq C_Nt^{-m+N+1},\]
    for all $t>0$. 
    
	Let $\Gamma(z)$ be the Gamma function on $\Complex$. Then for ${\rm Re\,}z>0$, we have 
	\[\Gamma(z)=\int^\infty_0e^{-t}t^{z-1}dt.\]
	$\Gamma(z)^{-1}$ is an entire function on $\Complex$ and 
	\begin{equation}\label{e-gue160313b}
		\Gamma(z)^{-1}=z+O(z^2)\ \ \mbox{near $z=0$}. 
	\end{equation}
	
	We suppose that $f(t)\in \mathcal{C}^\infty(\mathbb R_+)$ verifies the following two conditions:
	\begin{itemize}
		\item[I.] 
		\begin{equation}\label{e-gue160420g}
			\mbox{$f(t)\sim \sum^{\infty}_{j=0}f_{j}t^{-m+j}$ as $t\To0^+$}, 
		\end{equation}
		where $m\in\mathbb N_0$, $f_{j}\in\mathbb C$, $j=0,1,2,\ldots$.
		\item[II.]  For every $\delta>0$, there exist $c>0$, $C>0$ such that 
		\begin{equation}\label{e-gue160420I}
			\abs{f(t)}\leq Ce^{-ct},\ \ \forall t\geq\delta. 
		\end{equation}
	\end{itemize}

	\begin{definition}\label{d-gue160313}
		The \emph{Mellin transformation} of $f$ is the function defined for ${\rm Re\,}z>m$, 
		\begin{equation}\label{e-gue160313bIII}
			M[f](z)=\frac{1}{\Gamma(z)}\int^\infty_0 f(t)t^{z-1}dt.
		\end{equation}
	\end{definition}
	
	We recall the following result, cf. \cite[Lemma 5.5.2]{MM}:
	
	\begin{theorem}\label{t-gue160313}
		$M[f]$ extends to a meromorphic function on $\Complex$ with poles contained in 
		\[\set{\ell-j;\, \ell,j\in\mathbb Z},\] 
		and its possible poles are simple. Moreover, $M[f]$ is holomorphic at $0$, 
		\begin{equation}\label{e-gue160428w}
			M[f](0)=f_0
		\end{equation}
		and
		\begin{equation}\label{e-gue160421s}
			\begin{split}
				&\frac{\pr M[f]}{\pr z}(0)=\int^1_0(f(t)-\sum^{k}_{j=0}f_{j}t^{-k+j})\frac{1}{t}dt\\
				&\quad+\int^\infty_1f(t)\frac{1}{t}dt+\sum^{k-1}_{j=0}\frac{f_{j}}{j-k}-\Gamma'(1)f_{0}.
			\end{split}
		\end{equation} 
	\end{theorem}
	
	\begin{proof}
		By \eqref{e-gue160420I}, the function $\int^\infty_1f(t)t^{z-1}dt$ is an entire function on $z\in\Complex$. For any $N\in\mathbb N$, we have 
		\begin{equation}\label{e-gue160421}
			\begin{split}
				&\int^1_0f(t)t^{z-1}dt\\
				&=\int^1_0(f(t)-\sum^N_{j=0}f_{j}t^{-m+j})t^{z-1}dt+\sum^N_{j=0}f_{j}\int^1_0t^{-m+j+z-1}dt\\
				&=\int^1_0(f(t)-\sum^N_{j=0}f_{j}t^{-m+j})t^{z-1}dt+\sum^N_{j=0}f_{j}\frac{1}{-m+j+z}. 
			\end{split}
		\end{equation}
		From \eqref{e-gue160420g} and \eqref{e-gue160420I}, we see that $\int^1_0(f(t)-\sum^N_{j=0}f_{j}t^{-m+j})t^{z-1}dt$ is a holomorphic function for ${\rm Re\,}z>m-N+1$. From this observation and \eqref{e-gue160421}, we conclude that $M[f]$ can be extended to a meromorphic function on $\Complex$ with poles contained in 
		$\set{\ell-j;\, \ell,j\in\mathbb Z}$, and its possible poles are simple.
		
		From \eqref{e-gue160313b}, we conclude that $M[f]$ is holomorphic at $z=0$. Take $N=k$ in \eqref{e-gue160421} and by some direct computation, we get \eqref{e-gue160428w} and \eqref{e-gue160421s}.
	\end{proof}

	\subsection{Definition of the CR analytic torsion}
	Let $N$ be the number operator on $T^{*0,\bullet}X$, i.e. $N$ acts on $T^{*0,q}X$ by multiplication by $q$.
	Fix $q=0, 1, \cdots, n$ and take a point $x \in X$. Let $e_1(x), \cdots, e_d(x)$ be an orthonormal frame of $T_x^{*0,q}X \otimes E_x$. Let $A \in (T_x^{*0,\bullet}X \otimes E_x)\boxtimes (T_x^{*0,\bullet}X \otimes E_x)$. Put $\operatorname{Tr}^{(q)} A := \sum_{j=1}^d \langle Ae_j | e_j \rangle_E$ and set
	\begin{eqnarray}
		\operatorname{Tr}A:= \sum_{j=0}^n \operatorname{Tr}^{(j)}A, \nonumber \\
		\operatorname{STr}A:= \sum_{j=0}^n (-1)^j \operatorname{Tr}^{(j)}A.
	\end{eqnarray}
	Let $A: \mathcal{C}^\infty(X,T^{*0,\bullet}X\otimes E)\To\mathcal{C}^\infty(X,T^{*0,\bullet}X\otimes E)$ be a continuous operator with distribution kernel $A(x,y)\in \mathcal{C}^\infty(X\times X,(T^{*0,\bullet}X\otimes E)\boxtimes(T^{*0,\bullet}X\otimes E)^*)$. 
	We set 
	\[\operatorname{Tr}^{(q)}\lbrack A\rbrack:=\int_X\operatorname{Tr}^{(q)} A(x,x)dv_X(x)\] and put 
	\begin{eqnarray}
		\operatorname{Tr}\lbrack A\rbrack:= \sum_{j=0}^n \operatorname{Tr}^{(j)}[A], \nonumber \\
		\operatorname{STr}\lbrack A\rbrack:= \sum_{j=0}^n (-1)^j \operatorname{Tr}^{(j)}[A].
	\end{eqnarray}
    We will use the same notations as in Section~\ref{s-gue250404yyds}. Recall that 
	$$
	\Pi_E : L^2(X, T^{*0,\bullet}X \otimes E) \to \Ker \Box_{b,E}
	$$  
	is the orthogonal projection with respect to $(\,\cdot\,|\,\cdot\,)_E$. Let 
	\[\begin{split}
	&\Pi^\perp_E : L^2(X, T^{*0,\bullet}X \otimes E) \to ({\rm Ker\,}\Box_{b,E})^\perp,\\
    &(\Pi^{(q)}_E)^\perp : L^2(X, T^{*0,q}X \otimes E) \to ({\rm Ker\,}\Box^{(q)}_{b,E})^\perp,\ \ q=0,1,\ldots,n,
	\end{split}\]
	be the orthogonal projections, where for $q=0,1,\ldots,n$,
	\[\begin{split}&({\rm Ker\,}\Box_{b,E})^\perp=\set{u\in L^2(X, T^{*0,\bullet}X \otimes E);\, (\,u\,|\,v\,)_E=0,\  \ \forall v\in{\rm Ker\,}\Box_{b,E}},\\
    &({\rm Ker\,}\Box^{(q)}_{b,E})^\perp=\set{u\in L^2(X, T^{*0,q}X \otimes E);\, (\,u\,|\,v\,)_E=0,\  \ \forall v\in{\rm Ker\,}\Box^{(q)}_{b,E}}.\end{split}\] 
	
	Assume that $\Box_{b,E}$ has $L^2$ closed range. 
    Note that  $\Box_{b,E}$ has $L^2$ closed range if and only if $\Box^{(q)}_{b,E}$ has closed range, for all $q=0,1,\ldots,n$. From Theorem~\ref{t-gue250410ycdbz}, we have the following asymptotic expansion: 
	\begin{equation}\label{E:112120241802}
		\operatorname{STr}  \lbrack N e^{-t \Box_{b,E}} \Pi^\perp_E \rbrack\sim\sum^\infty_{j=0} B_{j} t^{-n-1+j}\ \ \mbox{as $t\To0^+$}, 
	\end{equation}
	where $B_j \in \Complex$ is independent of $t$, $j=0,1,2,\ldots$. 
	
	\begin{lemma}\label{l-gue160420}
    Assume that $\Box_{b,E}$ has $L^2$ closed range.
		Fix $q=0,1,\ldots,n$. For every $\delta>0$, there exist $c>0$, $C>0$ such that 
		\begin{equation}\label{e-gue160420Ia}
			\abs{\operatorname{Tr}^{(q)}  \lbrack e^{-t \Box^{(q)}_{b,E}}(\Pi^{(q)}_E)^\perp \rbrack}\leq Ce^{-ct},\ \ \forall t\geq\delta. 
		\end{equation}
	\end{lemma}
	
	\begin{proof}
		Let $0<\lambda_1\leq\lambda_2\leq\lambda_3\leq\cdots$ be the non-zero eigenvalues of $\Box^{(q)}_{b,E}$. For $t>0$, we have 
		\begin{equation}\label{e-gue160421sI}
			\abs{\operatorname{Tr}^{(q)}  \lbrack e^{-t \Box^{(q)}_{b,E}}(\Pi^{(q)}_E)^\perp \rbrack}=e^{-\lambda_1t}+e^{-\lambda_2t}+\cdots.
		\end{equation}
		Let $\delta>0$. From \eqref{e-gue160421sI}, for $t\geq\delta$, we have 
		\begin{equation}\label{e-gue160421sII}
			\begin{split}
				\abs{\operatorname{Tr}^{(q)}  \lbrack e^{-t \Box^{(q)}_{b,E}} (\Pi^{(q)}_E)^\perp \rbrack}&=e^{-\lambda_1t}+e^{-\lambda_2t}+\cdots\\
				&=e^{-\lambda_1\frac{t}{2}}(e^{-\lambda_1\frac{t}{2}}+e^{-\lambda_2t+\lambda_1\frac{t}{2}}+e^{-\lambda_3t+\lambda_1\frac{t}{2}}+\cdots)\\
				&\leq e^{-\lambda_1\frac{t}{2}}(e^{-\lambda_1\frac{t}{2}}+e^{-\lambda_2\frac{t}{2}}+e^{-\lambda_3\frac{t}{2}}+\cdots)\\
				&\leq e^{-\lambda_1\frac{t}{2}}(e^{-\lambda_1\frac{\delta}{2}}+e^{-\lambda_2\frac{\delta}{2}}+e^{-\lambda_3\frac{\delta}{2}}+\cdots)\\
				&=e^{-\lambda_1\frac{t}{2}}\abs{\operatorname{Tr}^{(q)}  \lbrack e^{-\frac{\delta}{2}\Box^{(q)}_{b,E}}(\Pi^{(q)}_E)^\perp \rbrack}.
			\end{split}
		\end{equation}
		From \eqref{E:112120241802} and \eqref{e-gue160421sII}, the proof is completed. 
	\end{proof}
	
	From \eqref{E:112120241802} and Lemma~\ref{l-gue160420}, we see that $\operatorname{STr}  \lbrack N e^{-t \Box_{b,E}} \Pi^\perp_E \rbrack$ satisfies \eqref{e-gue160420g} and \eqref{e-gue160420I}. By Definition \ref{d-gue160313}, for $\operatorname{Re}(z)>n$, we can define 
	\begin{equation}\label{E:5.5.12}
		\theta_{b,E}(z)  = - M \left\lbrack \operatorname{STr}  \lbrack N e^{-t \Box_{b,E}}\Pi^\perp_E  \rbrack \right\rbrack  =   - \operatorname{STr} \left\lbrack N ({\Box}_{b,E})^{-z}\Pi^\perp_E \right\rbrack.  
	\end{equation}
	By Theorem~\ref{t-gue160313}, we have the following lemma.
	
	\begin{lemma}\label{L:mero}
    Assume that $\Box_{b,E}$ has $L^2$ closed range. Then, 
		$\theta_{b,E}(z)$ extends to a meromorphic function on $\mathbb{C}$ with poles contained in the set  
		\[\set{\ell- j;\, \ell,j\in\mathbb Z},\]
		its possible poles are simple, and $\theta_{b, E}(z)$ is holomorphic at $0$. Moreover,
		\begin{equation}\label{E:5.5.13}
			\begin{split}
				& \theta_{b,E}'(0)  =  -\int_0^1 \left\{  \operatorname{STr} \Big[  Ne^{-t\Box_{b,E}}\Pi^\perp_E\Big]-\sum^{n+1}_{j=0} B_{j} t^{-n-1+j}\right\} \frac{dt}{t}  \\
				& -\int_1^\infty \operatorname{STr} \Big[ Ne^{-t\Box_{b,E}} \Pi^\perp_E \Big] \frac{dt}{t} - \sum^{n}_{j=0} \frac{B_{j}}{j-n-1} + \Gamma'(1)B_{n+1}. 
			\end{split}
		\end{equation}
	\end{lemma}

\begin{definition}\label{d-gue160502w}
    Assume $\Box_{b,E}$ has $L^2$ closed range. 
		We define $\exp ( -\frac{1}{2} \theta_{b,E}'(0) )$ to be the $\ddbar_{b}$-torsion for the CR vector bundle $E$ over the CR manifold $X$. 
	\end{definition}

    From Lemma~\ref{L:mero}, we see that Definition~\ref{d-gue160502w} is well-defined.

\subsection{Variation of the analytic torsion}\label{s-gue260326}

In this subsection, we study the dependence of the analytic torsion on a compact oriented CR manifold with a nondegenerate Levi form and a CR vector bundle $E$ over $X$ on a change of the Hermitian metrics on $TX$ and $E$. The main result is Theorem \ref{t303262026}.

We now define the complex tangential Hodge $\ast$-operator, see also \cite[Proposition 8.8]{DT}, as a complex conjugate linear map
\begin{equation}\label{e-gue1530a05092026}
\ast_b : \Omega^{p,q}(X) \to \Omega^{n-p,n-q}(X)
\end{equation}
such that
\begin{equation}\label{e-gue1530b05092026}
\langle\, \phi\, | \, \psi \,\rangle \frac{(d\omega_0)^n}{n!} \, = \, \phi \wedge \ast_b \psi, \quad \ast_b \ast_b \phi = (-1)^{p+q} \phi,
\end{equation}
for any $\phi, \psi \in \Omega^{p,q}(X)$.

We denote by $H^*X$ the dual bundle of $HX$ and $conj$ the natural conjugate map induced by the bundle automorphism
\begin{equation}
H^*X\otimes_{\mathbb{R}}\mathbb{C} \to H^*X\otimes_{\mathbb{R}}\mathbb{C}, \quad u \otimes \lambda \mapsto u \otimes \overline{\lambda},
\end{equation}
for any $u \in H^*X, \lambda \in \mathbb{C}.$
Then 
$$
\widehat{\ast}_b := conj\ast_b
$$ 
is a complex linear map. Clearly, 
$$
\partial_b = conj \, \overline{\partial}_b \, conj \, : \, \Omega^{p,q}(X) \to \Omega^{p+1,q}(X)
$$
and
$$
\widehat{\ast}_b \, = \, conj \, \ast_b \, = \, \ast_b \, conj \, : \, \Omega^{p,q}(X) \to \Omega^{n-q,n-p}(X).
$$

Let $(\, \cdot \, |\, \cdot \, )$ be the $L^2$ inner product on $\Omega^{p,q}(X)$ induced by $\langle \, \cdot \, | \, \cdot \, \rangle$. Then, for all $\phi, \psi \in \Omega^{p,q}(X)$,
\[
(\, \phi \, |\, \psi\,) \, =\, \int_X \langle \, \phi \,| \, \psi \, \rangle dv_X \, =\, \int_X  \phi  \wedge  \ast_b  \psi   \wedge \omega_0,
\]
where
\begin{equation}\label{E:volume}
dv_X \, = \,  \frac{(d\omega_0)^n}{n!} \wedge \omega_0
\end{equation}
is the volume form.

We can easily check that
\[
\partial_b^* \phi = -\ast_b \partial_b \ast_b \phi = -\widehat{\ast}_b\, \overline{\partial}_b\, \widehat{\ast}_b\, \phi, 
\]
and
\[
 \overline{\partial}_b \psi = - \ast_b \overline{\partial}_b \ast_b \psi = - \widehat{\ast}_b\, {\partial}_b\, \widehat{\ast}_b\, \psi. 
\]

Denote by $\mu: E \to E^*$ the induced conjugate linear bundle isomorphism from the vector bundle $E$ to its dual vector bundle $E^*$. Let $(\, \cdot \, |\, \cdot \, )_E$ be the $L^2$ inner product on $\Omega^{p,q}(X,E)$ induced by $\langle \, \cdot \, | \, \cdot \, \rangle$ and $\langle \, \cdot \, | \, \cdot \, \rangle_E$. Then, for all $\alpha, \beta \in \Omega^{p,q}(X, E)$,
\[
(\, \alpha \, |\, \beta \,)_E \, =\, \int_X \langle \, \alpha \,| \, \beta \, \rangle_E dv_X \, =\, \int_X  \alpha  \wedge  \left(\ast_b \otimes \mu \right) \beta   \wedge \omega_0,
\]
where $dv_X$ is the volume form defined in $\eqref{E:volume}$. We write $\ddbar'_{b}$ to denote the tangential Cauchy-Riemann operator acting on forms with values in $E^*$:
\[
\ddbar'_{b}:\Omega^{0,\bullet}(X, E^*)\To\Omega^{0,\bullet}(X,E^*).
\]
We can check that the adjoint of $\overline{\partial}_{b}$ is 
\begin{equation}
\overline{\partial}^{*}_{b} = - \left( \ast_b^{-1} \otimes \mu \right)^{-1} \overline{\partial}'_{b} \left( \ast_b \otimes \mu \right). \nonumber
\end{equation}

Let $\langle \, \cdot \, |\, \cdot \, \rangle_{s}$ and $\langle \, \cdot \, |\, \cdot \, \rangle_{E, s}, s \in [0, 1]$ be smooth families of rigid Hermitian metrics on $TX$ and $E$, respectively.
Let $(\, \cdot \, |\, \cdot \, )_{E,s}$ be the $L^2$ inner products on $\Omega^{p,q}(X,E)$ induced by $\langle \, \cdot \, | \, \cdot \, \rangle_s$ and $\langle \, \cdot \, | \, \cdot \, \rangle_{E,s}$. Let $\ast_{b,s}$ be the tangential Hodge $\ast$-operators associated to the metrics $\langle \, \cdot \, |\, \cdot \, \rangle_{s}$ and $\mu_s$ be the induced conjugate linear bundle isomorphisms of $E$ and $E^*$ associated to the metric $\langle \, \cdot \, |\, \cdot \, \rangle_{E, s}$. Let 
$$
\Box_{b,E,s}:= \overline{\partial}_{b}\overline{\partial}^{\ast}_{b,s} + \overline{\partial}^{\ast}_{b,s}\overline{\partial}_b:\Omega^{0,\bullet}(X,E) \to \Omega^{0,\bullet}(X,E),
$$ 
where $\overline{\partial}^{\ast}_{b,s}$ denote the formal adjoint of $\overline{\partial}_b$ with respect to the $L^2$ scalar product $( \, \cdot \, | \, \cdot \, )_{E,s}$. Let $\theta_{b,E,s}(z)$ be the corresponding function as defined in \eqref{E:5.5.12} and let $\Pi_{E, s}$ be the orthogonal projection operator from $\Omega^{0,\bullet}(X,E)$ on $\operatorname{Ker}\left( D_{b,s}|_{\Omega^{0,\bullet}(X,E)} \right)$ for the Hermitian product $(\,\cdot\, | \,\cdot\,)_{E,s}$ on $\Omega^{0,\bullet}(X,E)$. Set 
\begin{equation}\label{E:qbs}
Q_{b, s} = - \left(\ast_{b,s} \otimes \mu_s \right)^{-1} \frac{\partial \left(\ast_{b, s} \otimes \mu_{k,s} \right)}{\partial s} = - \left(\ast^{-1}_{b, s} \frac{\partial \ast_{b, s}}{\partial s} + (\mu_{k,s})^{-1} \frac{\partial \mu_{k,s}}{\partial s} \right).
\end{equation}

We need the following lemma whose proof is a slight modification of the proof of \cite[Lemma 5.5.3]{MM}

\begin{lemma}\label{l03262026}
For any operators $K_1, K_2, K_3$ on $\Omega^{0,\bullet}(X, E)$ with distribution kernels $K_1(x, y)$, $K_2(x, y)$, $K_3(x, y)$ associate to $dv_X(y)$ such that the operators $K_1K_2-K_3$, $K_2K_1-K_3$ have smooth kernels, we have
\begin{equation}\label{e143503262026}
\operatorname{STr}[K_1K_2-K_3] = \operatorname{STr}[K_2K_1-K_3].
\end{equation}
\end{lemma}
\begin{proof}
If $K_1$ or $K_2$ preserves the $\mathbb{Z}_2$-grading of $\Omega^{0,\bullet}(X, E)$, then
\[
\begin{split}
& \operatorname{STr}[K_1K_2-K_3] \\
= &  \int_X \operatorname{STr} \left[ \int_X K_1(x, y)K_2(y,x)dv_X(y)\right] - K_3(x, x) dv_X(x) \\
= &  \int_X \operatorname{STr} \left[ \int_X K_2(y, x)K_1(x,y)dv_X(x)\right] - K_3(y, y) dv_X(y) \\
= &  \operatorname{STr}[K_2K_1-K_3].
\end{split}
\]
In the same way, we can verify \eqref{e143503262026} if $K_1$ and $K_2$ exchange the $\mathbb{Z}_2$-grading of $\Omega^{0,\bullet}(X,E)$.
\end{proof}

We have
\begin{theorem}\label{t303262026}
 As $t \to 0^+$, for any $\ell \in \mathbb{N}$, there is an asymptotic expansion
\begin{equation}\label{e205202042026}
\operatorname{STr} \left\lbrack Q_{b,s} \exp \left( -t\Box_{b,E, s} \right)(I-\Pi_{E,s}) \right\rbrack = \sum_{j=0}^{n+1+\ell} M_{j, s} t^{-n-1+j} +O(t^{\ell+1}), 
\end{equation}
where
\begin{equation}\label{e205102042026}
M_{n+1, s} = \frac{\partial}{\partial s}  \left( \frac{\partial}{\partial z}\theta_{b,E,s}  \right)(0).
\end{equation}
\end{theorem}
\begin{proof}
From Theorem \ref{t-gue250410ycdbz}, we get \eqref{e205202042026}. 

Denote by $D_{b,s}:= \overline{\partial}_{b} + \overline{\partial}^{\ast}_{b,s}:\Omega^{0,\bullet}(X,E) \to \Omega^{0,\bullet}(X,E)$. Then, we have 
\begin{equation}\label{e231402032026}
[D_{b,s}, N] = -\overline{\partial}_{b} + \overline{\partial}^{\ast}_{b,s}.
\end{equation}
For $u \in \Omega^{0,q}(X, E)$, we have
\begin{equation}\label{e234202032026}
\overline{\partial}^{*}_{b,s}u = - \left( \ast_{b,s}^{-1} \otimes \mu_s \right)^{-1} \overline{\partial}'_{b} \left( \ast_{b,s} \otimes \mu_s \right)u. 
\end{equation}
Then, by \eqref{E:qbs} and \eqref{e234202032026},
\begin{equation}\label{e095202042026}
\frac{\partial}{\partial s}D_{b,s} = \frac{\partial}{\partial_s}\overline{\partial}^{\ast}_{b,s}=-\left[\overline{\partial}^{\ast}_{b,s}, Q_{b,s}\right], \qquad \frac{\partial}{\partial s} \Box_{b, s} =\frac{\partial}{\partial s}D_{b,s}^2= \left[D_{b,s}, \frac{\partial}{\partial s}D_{b,s}\right].
\end{equation}
We can show that, for example, cf. \cite[Theorem D.16]{MM}, 
\begin{equation}\label{e102502042026}
\frac{\partial}{\partial s}e^{-t\Box_{b,E,s}} = \int_0^t e^{-(t-t_1)\Box_{b,E,s}}\left(-\frac{\partial}{\partial s}\Box_{b,E,s} \right)e^{-t_1\Box_{b,E,s}}dt_1.
\end{equation}
Notice that $\Box_{b,E,s}$ preserves the $\mathbb{Z}$-grading of $\Omega^{0,\bullet}(M,E)$, and commutes with $N$, and the operator $\left(-\frac{\partial}{\partial s}\Box_{b,E,s} \right)e^{-t\Box_{b,E,s}}$ is an operator with smooth kernel. By \eqref{e095202042026} and \eqref{e102502042026}, we have
\begin{equation}\label{e1036020242026}
\begin{split}
\frac{\partial}{\partial s}\operatorname{STr}\left[Ne^{-t\Box_{b,E,s}}(I-\Pi_{E,s})\right] & = \operatorname{STr}\left[\int_0^t \left(-\frac{\partial}{\partial s}\Box_{b,E,s}\right) e^{-t_1\Box_{b,E,s}}Ne^{-(t-t_1)\Box_{b,s}}dt_1 - N \frac{\partial}{\partial s}\Pi_{E,s}\right] \\
& = -t \operatorname{STr}\left[ \left[D_{b,s},\frac{\partial D_{b,s}}{\partial s} \right]N e^{-t\Box_{b,E,s}} - N \frac{\partial}{\partial s}\Pi_{E,s}\right]
\end{split}
\end{equation}
Lemma \ref{l03262026} implies that $D_{b,s}$ changes the $\mathbb{Z}_2$-grading of $\Omega^{0,\bullet}(X,E)$ and that $D_{b,s}$ commutes with $e^{-t\Box_{b,E,s}}$, we get for $0 < t_1 < t$, 
\begin{equation}\label{e1116020242026}
\begin{split}
& \operatorname{STr}\left[\left[D_{b,s}, \frac{\partial D_{b,s}}{\partial s} \right] N e^{-t\Box_{b,E,s}} - N \frac{\partial}{\partial s}\Pi_{E,s}\right] \\
= &  \operatorname{STr} \left[  D_{b,s} \frac{\partial D_{b,s}}{\partial s} N e^{-t\Box_{b,E,s}}+\frac{\partial D_{b,s}}{\partial s}D_{b,s}N e^{-t\Box_{b,E,s}} - N \frac{\partial}{\partial s}\Pi_{E,s}\right]\\ 
= &  \operatorname{STr}\left[e^{-(t-t_1)\Box_{b,E,s}} D_{b,s} \frac{\partial D_{b,s}}{\partial s} N e^{-t_1 \Box_{b,E,s}} + \frac{\partial D_{b,s}}{\partial s}D_{b,s}N e^{-t\Box_{b,E,s}}- N \frac{\partial}{\partial s}\Pi_{E,s}\right]      
\\ = &  \operatorname{STr}\left[ D_{b,s}e^{-(t-t_1)\Box_{b,E,s}} \frac{\partial D_{b,s}}{\partial s}  N e^{-t_1 \Box_{b,E,s}}+ \frac{\partial D_{b,s}}{\partial s}D_{b,s}N e^{-t\Box_{b,E,s}} - N \frac{\partial}{\partial s}\Pi_{E,s}\right] \\
= &  -\operatorname{STr}\left[ \frac{\partial D_{b,s}}{\partial s}N e^{-t_1 \Box_{b,E,s}}D_{b,s} e^{-(t-t_1)\Box_{b,E,s}}+\frac{\partial D_{b,s}}{\partial s}D_{b,s}N e^{-t\Box_{b,E,s}} - N \frac{\partial}{\partial s}\Pi_{E,s}\right]\\
= & -\operatorname{STr}\left[ \frac{\partial D_{b,s}}{\partial s}N D_{b,s}e^{-t\Box_{b,E,s}} + \frac{\partial D_{b,s}}{\partial s}D_{b,s}N e^{-t\Box_{b,E,s}} - N \frac{\partial}{\partial s}\Pi_{E,s} \right]\\
= & \operatorname{STr}\left[\frac{\partial D_{b,s}}{\partial s}\left[D_{b,s}, N\right] e^{-t\Box_{b,E,s}} - N \frac{\partial}{\partial s}\Pi_{E,s}\right].
\end{split}
\end{equation}
Using that $\overline{\partial}^{\ast}_{b,s}$ commutes with $e^{-t\Box_{b,E,s}}$, and relations \eqref{e231402032026}, \eqref{e095202042026} and \eqref{e1036020242026}, we get by using the same trick as in \eqref{e1116020242026}, we have
\begin{equation}\label{e1137020242026}
\begin{split}
& \frac{\partial}{\partial s}\operatorname{STr}\left[Ne^{-t\Box_{b,E,s}}(I-\Pi_{E,s}) \right] \\ = & -t \operatorname{STr}\left[Q_{b,s} \left[\overline{\partial}^{\ast}_{b,s}, -\overline{\partial}_b+\overline{\partial}^{\ast}_{b,s}   \right] e^{-t\Box_{b,s}}(I-\Pi_{E,s})\right] \\
 = & t \operatorname{STr}\left[Q_{b,s} \Box_{b,E,s} e^{-t\Box_{b,E,s}}(I-\Pi_{E,s})\right] \\
 = & -t \frac{\partial}{\partial t}\operatorname{STr}\left[Q_{b,s}  e^{-t\Box_{b,E,s}}(I-\Pi_{E,s})\right].
\end{split}
\end{equation}

Note that $\operatorname{STr}\left[Q_{b,s} e^{-t\Box_{b,E,s}}(I-\Pi_{E,s}) \right]$ decays exponentially when $t \to +\infty$. From \eqref{e1137020242026}, we obtain for $\operatorname{Re}(z)$ large enough,
\begin{equation}\label{e153202042026}
\begin{split}
\frac{\partial}{\partial s}\theta_{b,E,s}(z) & = \frac{1}{\Gamma(z)}\int_0^\infty \frac{\partial}{\partial s} \operatorname{STr}\left[N e^{-t\Box_{b,E,s}} (I-\Pi_{E,s}) \right] t^{z-1}dt \\
& = \frac{1}{\Gamma(z)}\int_0^\infty t^z \frac{\partial}{\partial t} \operatorname{STr}\left[Q_{b,s} e^{t\Box_{b,E,s}} (I-\Pi_{E,s}) \right]  dt \\
& = \frac{-z}{\Gamma(z)}\int_0^\infty t^{z-1}  \operatorname{STr}\left[Q_{b,s} e^{t\Box_{b,E,s}} (I-\Pi_{E,s}) \right]  dt.
\end{split}
\end{equation}

Using \eqref{e-gue160428w} and \eqref{e153202042026}, we get
\begin{equation}\label{e160002042026}
\frac{\partial}{\partial s} \left(\frac{\partial}{\partial z}\theta_{b,E,s} \right)(0) = -M_{n+1, s}.
\end{equation}

\end{proof}

\subsection{Asymptotics of the analytic torsion for $L^k$}\label{s-gue250530}

In Lemma~\ref{L:mero}, we can define CR analytic torsion when $\Box_{b,E}$ has $L^2$ closed range. In this subsection, we consider $E=L^k$, where $L^k$ is the $k$-th power of a CR line bundle $L$ over $X$. We use the same notations as Section~\ref{s-gue250704}. We will study asymptotic behavior of $\theta'_{b,L^k}$, when $k\To+\infty$.

The goal of this subsection is to establish Bismut-Vasserot type asymptotics of the analytic torsion for $L^k$ as $k\To+\infty$. We work with Assumption~\ref{a-gue250426yyd} below and assume further that $n_-=n_+$ or $\abs{n_--n_+}>1$. 

For every $x\in X$, let $Z_1(x),\ldots,Z_n(x)$ be an orthonormal basis of $T^{1,0}_xX$ such that 
\[\mathcal{L}_x(Z_j(x), \ol{Z_\ell(x)} )=\lambda_j\delta_{j,\ell},\ \  j,\ell=1,\ldots,n.\] 
We may assume that $\lambda_j<0$,$j=1,\ldots,n_-$, $\lambda_j>0$, $j=n_-+1,\ldots,n$. Let $\Lambda_-(x):={\rm span\,}\set{Z_1(x),\ldots,Z_{n_-}(x)}$, $\Lambda_+(x):={\rm span\,}\set{Z_{n_-+1}(x),\ldots,Z_{n}(x)}$. Let $\langle\,\cdot\,|\,\cdot\,\rangle_{\mathcal{L}}$ be the metric on $\mathbb CTX$ given by 
\begin{equation}\label{e-gue250623yyd}
\begin{split}
&\langle\,U\,|\,V\,\rangle_{\mathcal{L}}=0,\  \ U\in\Lambda_-(x), V\in\Lambda_+(x),\\
&\langle\,U\,|\,V\,\rangle_{\mathcal{L}}=\mathcal{L}_x(U,\ol V),\ \ U, V\in\Lambda_+(x),\\
&\langle\,U\,|\,V\,\rangle_{\mathcal{L}}=-\mathcal{L}_x(U,\ol V),\ \ U, V\in\Lambda_-(x),\\
&\langle\,T\,|\,T\,\rangle_{\mathcal{L}}=1,\\
&T\perp(T^{1,0}_xX\oplus T^{0,1}_xX).
\end{split}
\end{equation}
In this section, we take $\langle\,\cdot\,|\,\cdot\,\rangle$ to be $\langle\,\cdot\,|\,\cdot\,\rangle_{\mathcal{L}}$.

	For $x \in X$, $t>0$ and $\eta\in\mathbb R$, we set
	\begin{equation}\label{E:1103202422391}
		(\dot{\mathcal{R}}^\phi_x-2\eta \dot{\mathcal{L}_x})_t =\frac{1}{2\pi}\det \left( \frac{ \dot{\mathcal{R}}_x^\phi - 2 \eta \dot{\mathcal{L}}_x  }{2\pi} \right) \operatorname{Tr}[(\operatorname{Id} - \exp(t (\dot{\mathcal{R}}_x^\phi - 2 \eta \dot{\mathcal{L}}_x ) ))^{-1}].
	\end{equation}
	
	Let $\omega^1(x), \cdots, \omega^n(x) \in \mathscr{C}^\infty(X, T^{\ast1,0}X)$ be an orthonormal basis for $T^{\ast1,0}_xX$, for every $x \in X$. Define $\Theta(x) := i \sum^n_{j=1} \omega^j(x) \wedge \overline{\omega^j}(x) \in \mathscr{C}^\infty(X, T^{\ast1,1}X)$. Then, we have
	\begin{equation}\label{E:1103202422392}
		dv_X = \frac{1}{n!} \Theta^n \wedge (-\omega_0), \ \det \left( \frac{ \dot{\mathcal{R}}_x^\phi - 2 \eta \dot{\mathcal{L}}_x  }{2\pi} \right) = \frac{1}{n!} \left( \frac{i \mathcal{R}^\phi_x -2i \eta d\omega_0(x))}{2\pi} \right)^n \wedge (-\omega_0(x)).
	\end{equation}
	
	We need the following lemma.
	
	\begin{lemma}\label{l-gue160428}
		With the notations above, fix $t>0$, $x\in X$ and $\eta\in\mathbb R$, we have
		\begin{equation}\label{E:5.5.45}
			(2\pi)^{-n-1} \frac{\det(\dot{\mathcal{R}}^\phi_x - 2\eta \dot{\mathcal{L}}_x)\operatorname{STr} Ne^{-t\omega^\eta_x}}{\det(1-\exp(-t(\dot{\mathcal{R}}^\phi_x -2\eta \dot{\mathcal{L}_x }) ) )} = (\dot{\mathcal{R}}^\phi_x-2\eta \dot{\mathcal{L}}_x)_t.
		\end{equation}
	\end{lemma}
	
	\begin{proof}
		Let $\{ \omega_{j}(x) \}_{j=1}^n$ to be an orthonormal basis of $T^{1,0}_xX$ such that
		\begin{equation}\label{e-gue160428}
			\dot{\mathcal{R}}^\phi_x - 2\eta \dot{\mathcal{L}}_x = \operatorname{diag} (\mu_1(\eta), \cdots, \mu_n(\eta)) \in \operatorname{End}(T_x^{1,0}X),
		\end{equation}\
		then
		\begin{equation}\label{e-gue160428I}
			\begin{split}
				&\omega^\eta(x) = \sum^n_{j=1} \mu_j(\eta) \overline{\omega}^j(x) \wedge \iota_{\overline{\omega}_j}(x),\\
				&{\rm det\,}(\dot{\mathcal{R}}^\phi_x - 2\eta \dot{\mathcal{L}})(x) :=\mu_1(\eta)\cdots\mu_n(\eta),
			\end{split}
		\end{equation}
        where $\set{\overline{\omega}^j(x)}^n_{j=1}\subset T^{*0,1}_xX$ is the dual basis of $\{ \omega_{j}(x) \}_{j=1}^n$. Note that $\{ \omega_{j}(x)\}_{j=1}^n$ depends on $\eta$.

For simplicity, write $\mu_j:=\mu_j(\eta)$, $j=1,\ldots,n$. From \eqref{e-gue160428} and \eqref{e-gue160428I}, it is easy to check that 
		\begin{equation}\label{e-gue160428II}
			(\dot{\mathcal{R}}^\phi_x-2\eta \dot{\mathcal{L}}_x)_t=\frac{\mu_1\cdots\mu_n}{(2\pi)^{n+1}}\Bigr(\frac{1}{1-e^{\mu_1 t}}+\cdots+\frac{1}{1-e^{\mu_n t}}\Bigr)
		\end{equation}
		and 
		\begin{equation}\label{e-gue160428III}
			\begin{split}
				(2\pi)^{-n-1} &\frac{\det(\dot{\mathcal{R}}^\phi_x - 2\eta \dot{\mathcal{L}}_x)\operatorname{STr} Ne^{-t\omega^\eta_x}}{\det(1-\exp(-t(\dot{\mathcal{R}}^\phi_x -2\eta \dot{\mathcal{L} }_x) ) )}\\
				=&\frac{\mu_1\cdots\mu_n}{(2\pi)^{n+1}}\frac{1}{(1-e^{-\mu_1 t})\cdots(1-e^{-\mu_n t})}\times \\
                &\sum^n_{q=0}(-1)^qq\sum_{J=(j_1,\ldots,j_q),1\leq j_1<j_2<\cdots<j_q\leq n}e^{(-\mu_{j_1}-\cdots-\mu_{j_q})t}\\
				=&\frac{\mu_1\cdots\mu_n}{(2\pi)^{n+1}}\frac{e^{(\mu_1+\cdots+\mu_n)t}}{(1-e^{\mu_1 t})\cdots(1-e^{\mu_n t})}\times \\
                &\sum^n_{q=0}(-1)^{n+q}q\sum_{J=(j_1,\ldots,j_q),1\leq j_1<j_2<\cdots<j_q\leq n}e^{(-\mu_{j_1}-\cdots-\mu_{j_q})t}.
			\end{split}
		\end{equation}
		We have
		\begin{equation}\label{e-gue160428IV}
			\begin{split}
				&e^{-(\mu_1+\cdots+\mu_n)t}\sum^n_{j=1}(1-e^{\mu_1 t})\cdots(1-e^{\mu_{j-1} t})(1-e^{\mu_{j+1} t})\cdots(1-e^{\mu_n t})\\
				&=e^{-(\mu_1+\cdots+\mu_n)t}\sum^n_{j=1}\sum^n_{q=0}(-1)^q\sum_{J=(j_1,\ldots,j_q),1\leq j_1<\cdots<j_q\leq n, j\notin J}e^{(\mu_{j_1}+\cdots+\mu_{j_q})t}\\
				&=\sum^n_{j=1}\sum^n_{q=0}(-1)^{n+q}\sum_{J=(j_1,\ldots,j_q),1\leq j_1<\cdots<j_q\leq n, j\in J}e^{(-\mu_{j_1}-\cdots-\mu_{j_q})t}\\
				&=\sum^n_{q=0}(-1)^{n+q}q\sum_{J=(j_1,\ldots,j_q),1\leq j_1<j_2<\cdots<j_q\leq n}e^{(-\mu_{j_1}-\cdots-\mu_{j_q})t}.
			\end{split}
		\end{equation}
		From \eqref{e-gue160428II}, \eqref{e-gue160428III} and \eqref{e-gue160428IV}, the lemma follows. 
	\end{proof}

We assume the following as a variant of spectral gap condition.

\begin{ass}\label{a-gue250426yyd}
${\rm dim\,}E_{0<\lambda\leq\frac{k}{\log k}}(X,L^k)=o(k^{n+1})$, for $k\gg1$ and 
\[{\rm Spec\,}\Box_{b,L^k}\subset\set{0}\cup[C,+\infty),\]
for all $k\gg1$, where $C>0$ is a constant independent of $k$. 
\end{ass}


    From Theorem~\ref{t-gue250529yyd}, Lemma~\ref{l-gue160428} and Assumption~\ref{a-gue250426yyd}, we get 

    \begin{theorem}\label{t-gue250531yyd}
With the notations and assumptions  used before, we have for all $t>0$, 
\begin{equation}\label{e-gue250531yyd}
\begin{split}
&\lim_{k\To+\infty}k^{-(n+1)}{\rm STr\,}\lbrack Ne^{-\frac{t}{k}\Box_{b,L^k}}(I-\Pi_{L^k})\rbrack\\
&=\int_X\int_{\mathbb R}\Bigr((\dot{\mathcal{R}}^\phi_x-2\eta \dot{\mathcal{L}}_x)_t-(2\pi)^{-n-1}\sum^n_{q=0}q(-1)^q\abs{\det(\dot{\mathcal{R}}_x^\phi - 2 \eta \dot{\mathcal{L}}_x)}1_{\mathbb R_x(q)}(\eta)\Bigr)d\eta dv_X.
\end{split}
\end{equation}
    \end{theorem}

From Theorem~\ref{t-gue250603yyd}, we see that  we have the following asymptotic expansion, for $t \to 0$,
	\begin{equation}\label{E:110320242242z}
    \begin{split}
		&\int_{\mathbb R}\Bigr((\dot{\mathcal{R}}^\phi_x-2\eta \dot{\mathcal{L}}_x)_t-(2\pi)^{-n-1}\sum^n_{q=0}q(-1)^q\abs{\det(\dot{\mathcal{R}}_x^\phi - 2 \eta \dot{\mathcal{L}}_x)}1_{\mathbb R_x(q)}(\eta)\Bigr)d\eta\\
        &\sim\sum^{+\infty}_{j=0}\widehat{A}_j(x) t^{-n-1+j}\ \ \mbox{in $\tilde{S}^{-n-1}(\mathbb{R}_+\times X)$},
        \end{split}
	\end{equation}
    where $\widehat{A}_j(x)\in\mathcal{C}^\infty(X)$, $j=0,1,\ldots$. 

    In the proof of Theorem~\ref{T:5.5.8} below, we need to determine $\widehat{A}_{n+1}(x)$. Fix $\eta\in\mathbb R$, as in complex case, it
    is straightforward to check that 
    \begin{equation}\label{E:110320242242}
    \begin{split}
		&(\dot{\mathcal{R}}^\phi_x-2\eta \dot{\mathcal{L}}_x)_t-(2\pi)^{-n-1}\sum^n_{q=0}q(-1)^q\abs{\det(\dot{\mathcal{R}}_x^\phi - 2 \eta \dot{\mathcal{L}}_x)}1_{\mathbb R_x(q)}(\eta)\\
        &\sim\sum^{+\infty}_{j=0}F_{j,\eta}(x) t^{-n-1+j}\ \ \mbox{in $\tilde{S}^{-n-1}(\mathbb{R}_+\times X)$},
        \end{split}
	\end{equation}
    where 
\begin{equation}\label{e-gue250531yydk}
    \begin{split}
&F_{j,\eta}=0,\  \ j=0,1,\ldots,n-1,\\
&F_{n+1,\eta}(x)\\
&=\frac{1}{(2\pi)^{n+1}}\Bigr(\frac{n}{2}{\rm det\,}(\dot{\mathcal{R}}^\phi_x - 2\eta \dot{\mathcal{L}_x})-\sum^n_{q=0}(-1)^qq\abs{{\rm det\,}(\dot{\mathcal{R}}^\phi_x - 2\eta \dot{\mathcal{L}_x})}1_{\mathbb R_x(q)}(\eta)\Bigr).
\end{split}
\end{equation} 
Fix $C\gg1$. Thus, 
 \begin{equation}\label{e-gue250603ycd}
    \begin{split}
		&\int_{\abs{\eta}\leq C}\Bigr((\dot{\mathcal{R}}^\phi_x-2\eta \dot{\mathcal{L}}_x)_t-(2\pi)^{-n-1}\sum^n_{q=0}q(-1)^q\abs{\det(\dot{\mathcal{R}}_x^\phi - 2 \eta \dot{\mathcal{L}}_x)}1_{\mathbb R_x(q)}(\eta)\Bigr)d\eta\\
        &\sim\sum^{+\infty}_{j=0}F_{j}(x) t^{-n-1+j}\ \ \mbox{in $\tilde{S}^{-n-1}(\mathbb{R}_+\times X)$},
        \end{split}
	\end{equation}
 where $F_j(x)\in\mathcal{C}^\infty(X)$, $j=0,1,\ldots$, 
\begin{equation}\label{e-gue250603ycdI}
    \begin{split}
&F_{n+1}(x)\\
&=\int_{\abs{\eta}\leq C}\frac{1}{(2\pi)^{n+1}}\Bigr(\frac{n}{2}{\rm det\,}(\dot{\mathcal{R}}^\phi_x - 2\eta \dot{\mathcal{L}_x})-\sum^n_{q=0}(-1)^qq\abs{{\rm det\,}(\dot{\mathcal{R}}^\phi_x - 2\eta \dot{\mathcal{L}_x})}1_{\mathbb R_x(q)}(\eta)\Bigr)d\eta.
\end{split}
\end{equation} 

From \eqref{E:110320242242}, \eqref{e-gue250603ycd}, we deduce that 

\begin{equation}\label{e-gue250603ycde}
    \begin{split}
		&\int_{\abs{\eta}\geq C}\Bigr((\dot{\mathcal{R}}^\phi_x-2\eta \dot{\mathcal{L}}_x)_t-(2\pi)^{-n-1}\sum^n_{q=0}q(-1)^q\abs{\det(\dot{\mathcal{R}}_x^\phi - 2 \eta \dot{\mathcal{L}}_x)}1_{\mathbb R_x(q)}(\eta)\Bigr)d\eta\\
        &\sim\sum^{+\infty}_{j=0}G_{j}(x) t^{-n-1+j}\ \ \mbox{in $\tilde{S}^{-n-1}(\mathbb{R}_+\times X)$},
        \end{split}
	\end{equation}
where $G_j(x)\in\mathcal{C}^\infty(X)$, $j=0,1,\ldots$. We have 
\begin{equation}\label{e-gue250603yyds}
\widehat{A}_{n+1}(x)=F_{n+1}(x)+G_{n+1}(x). 
\end{equation}
We need to determine $G_{n+1}(x)$. Fix $C \gg 1$, let 
\begin{equation}\label{e-gue250603yyd}
\begin{split}
H_x(z):=&\frac{1}{\Gamma(z)}\int^{+\infty}_0\int_{\abs{\eta}\geq C}\Bigr((\dot{\mathcal{R}}^\phi_x-2\eta \dot{\mathcal{L}}_x)_t\\
&\quad-(2\pi)^{-n-1}\sum^n_{q=0}q(-1)^q\abs{\det(\dot{\mathcal{R}}_x^\phi - 2 \eta \dot{\mathcal{L}}_x)}1_{\mathbb R_x(q)}(\eta)\Bigr)t^{z-1}d\eta dt.
\end{split}
\end{equation}
We can repeat the proof of Theorem~\ref{t-gue160313} and deduce that for every $x\in X$,
		$H_x(z)$ extends to a meromorphic function on $\Complex$ with poles contained in 
        $\mathbb Z$ and its possible poles are simple. Moreover, $H_x(z)$ is holomorphic at $0$, and 
       \begin{equation}\label{e-gue250604yyd}
        G_{n+1}(x)=H_x(0),\ \ \mbox{for every $x\in X$}. 
	\end{equation} 
    
We now calculate $H_x(0)$. We introduce some notations. 
Let 
\[A: T^{1,0}_xX\To T^{1,0}_xX\]
be a self-adjoint invertible linear transformation with respect to $\langle\,\cdot\,|\,\cdot\,\rangle$. For every $x\in X$, let $Z_1(x),\ldots,Z_n(x)$ be an orthonormal basis of $T^{1,0}_xX$ such that $\langle\,A_xZ_j(x)\,|\,Z_\ell(x)\,\rangle=\lambda_j(x)\delta_{j,\ell}$, $j,\ell=1,\ldots,n$, where $\lambda_j(x)\neq0$, for every $j=1,\ldots,n$. We may assume that $\lambda_j(x)<0$, $j=1,\ldots,n_-$, $\lambda_j(x)>0$, $j=n_-+1,\ldots,n$. Recall that $\Lambda_-(x):={\rm span\,}\set{Z_1(x),\ldots,Z_{n_-}(x)}$, $\Lambda_+(x):={\rm span\,}\set{Z_{n_-+1}(x),\ldots,Z_{n}(x)}$.
Let $\tau_+=\tau_+(x): T^{1,0}_xX\To\Lambda_+(x)$ be the orthogonal projection with respect to $\langle\,\cdot\,|\,\cdot\,\rangle$. Let $\tau_-=\tau_-(x): T^{1,0}_xX\To\Lambda_-(x)$ be the orthogonal projection with respect to $\langle\,\cdot\,|\,\cdot\,\rangle$. Let 
\[A_{\mp}:=\tau_{\mp}\circ A\circ\tau_{\mp}: \Lambda_{\mp}\To\Lambda_{\mp}.\]
Let ${\rm Tr\,}A_{\mp}:=\sum^n_{j=1}\langle\,A_{\mp}U_{j,\mp}\,|\,U_{j,\mp}\,\rangle$,
where $U_{j,\mp}$, $j=1,\ldots,n$, is an orthonormal basis of $\Lambda_{\mp}$ with respect to $\langle\,\cdot\,|\,\cdot\rangle$. We let $\abs{A_{-}}:=-A_{-}$, $\abs{A_{+}}:=A_{+}$. Then, $\abs{A_{\mp}}$ is positive definite. 

We need

\begin{lemma}\label{l-gue250604yyd}
With the notations used above, fix $C \gg 1$, for ${\rm Re\,}z>n+1$, we have
\begin{equation}\label{e-gue250604yyda}
\begin{split}
&H_x(z) = (2\pi)^{-n-1} \zeta(z) \times\\
&\int_{\abs{\eta}\geq C}\det (\dot{\mathcal{R}}^\phi_x - 2 \eta \dot{\mathcal{L}}_x)\left[ \operatorname{Tr}\abs{(\dot{\mathcal{R}}^\phi_x - 2 \eta \dot{\mathcal{L}}_x  )_-}^{-z} - \operatorname{Tr}\abs{(\dot{\mathcal{R}}^\phi_x - 2 \eta \dot{\mathcal{L}}_x  )_+}^{-z} \right]  d \eta,
\end{split}
\end{equation}
where $\zeta(z)$ denotes the Riemann zeta function and recall that for ${\rm Re\,}z>1$, \[\zeta(z)=\frac{1}{\Gamma(z)}\int^{+\infty}_0\frac{e^{-t}}{1-e^{-t}}t^{z-1}dt.\]
\end{lemma}

\begin{proof}
Fix $\abs{\eta}\geq C$. Assume that 
$\dot{\mathcal{R}}^\phi_x-2\eta \dot{\mathcal{L}}_x$ has $q$ negative and $n-q$ positive eigenvalues. Assume that $\mu_1<0,\ldots,\mu_q<0$, $\mu_{q+1}>0,\ldots,\mu_n>0$. From \eqref{e-gue160428I} and \eqref{e-gue160428II}, we get
\[
\begin{split}
& (\dot{\mathcal{R}}^\phi_x-2\eta \dot{\mathcal{L}}_x)_t
-(2\pi)^{-n-1}\sum^n_{q=0}q(-1)^q\abs{\det(\dot{\mathcal{R}}_x^\phi - 2 \eta \dot{\mathcal{L}}_x)}1_{\mathbb R_x(q)}(\eta) \\
&  = \frac{\mu_1\cdots\mu_n}{(2\pi)^{n+1}}\Bigr(\frac{1}{1-e^{\mu_1 t}}+\cdots+\frac{1}{1-e^{\mu_n t}}\Bigr) - 
\sum^n_{q=0}q(-1)^q \frac{\abs{ \mu_1\cdots\mu_n }}{(2\pi)^{n+1} }1_{\mathbb R_x(q)}(\eta) \\
& =   \frac{\mu_1\cdots\mu_n}{(2\pi)^{n+1}}\Bigr(\frac{e^{\mu_1 t}}{1-e^{\mu_1 t}}+ \cdots + \frac{e^{\mu_{q} t}}{1-e^{\mu_{q} t}} + \frac{1}{1-e^{\mu_{q+1} t}} +\cdots+\frac{1}{1-e^{\mu_n t}}\Bigr). \\
\end{split}
\]
For $1\le i \le q$, we can replace $-\mu_i t$ by $t$ and get
\[
\begin{split}
 \frac{1}{\Gamma(z)} \int_0^{+\infty} \frac{e^{\mu_i t}}{1- e^{\mu_i t}} t^{z-1} dt & =  \frac{1}{\Gamma(z)} \int_0^{+\infty} \frac{e^{-t}}{1-e^{-t}}t^{z-1} dt \abs{\mu_i}^{-z} \\
& =  \zeta(z) \abs{\mu_i}^{-z}.
\end{split}
\]
Similarly, for $q+1 \le i \le n$, we can replace $\mu_i t$ by $t$ and get
\[
\begin{split}
 \frac{1}{\Gamma(z)} \int_0^{+\infty} \frac{1}{1- e^{\mu_i t}} t^{z-1} dt 
& =  \frac{1}{\Gamma(z)} \int_0^{+\infty} \frac{1}{1-e^{t}} t^{z-1} dt \mu_i^{-z} \\
& = - \zeta(z) |\mu_i|^{-z}.
\end{split}
\]
By \eqref{e-gue250603yyd} and the above computation, we have
\begin{equation*}
\begin{split}
&H_x(z) = (2\pi)^{-n-1} \zeta(z) \times\\
&\int_{\abs{\eta}\geq C}\det (\dot{\mathcal{R}}^\phi_x - 2 \eta \dot{\mathcal{L}}_x)\left[ \operatorname{Tr}\abs{(\dot{\mathcal{R}}^\phi_x - 2 \eta \dot{\mathcal{L}}_x  )_-}^{-z} - \operatorname{Tr}\abs{(\dot{\mathcal{R}}^\phi_x - 2 \eta \dot{\mathcal{L}}_x  )_+}^{-z} \right]  d \eta.
\end{split}
\end{equation*}
\end{proof} 
Put
\begin{equation}\label{e-gue250625yyd}
\begin{split}
&H_x(z) = (2\pi)^{-n-1} \zeta(z)\hat{H}_x(z),\\
&\hat{H}_x(z):=\\
&\int_{\abs{\eta}\geq C}\det (\dot{\mathcal{R}}^\phi_x - 2 \eta \dot{\mathcal{L}}_x)\left[ \operatorname{Tr}\abs{(\dot{\mathcal{R}}^\phi_x - 2 \eta \dot{\mathcal{L}}_x  )_-}^{-z} - \operatorname{Tr}\abs{(\dot{\mathcal{R}}^\phi_x - 2 \eta \dot{\mathcal{L}}_x  )_+}^{-z} \right]  d \eta.
\end{split}
\end{equation}

\begin{lemma}\label{l-gue250623yyd}
For $C \gg 1$, $H_x(z)$ extends to a meromorphic function on $\mathbb{C}$ with poles contained in
\[
\left\{n+1-j; j=0, 1, 2, \cdots \right\}.
\]
Moreover, $H_x(z)$ and $\hat H_x(z)$ are holomorphic at $0$. For $C \gg 1$, 
\begin{equation}\label{e-gue250623yydz}
\begin{split}
&H_x(0)  = (2\pi)^{-n-1} \zeta(0)\hat{H}_x(0),\ \ \zeta(0)=-\frac{1}{2},\\
&\hat{H}_x(0)=\left(F(C) + F(-C) \right) (n_- - n_+)  \\
 &-\left(2\operatorname{Tr} F \left(\bigr((2\dot{\mathcal{L}}_x)^{-1}  \dot{\mathcal{R}}^\phi_x \bigr)_- \right) - 2\operatorname{Tr}F \left( \bigr((2\dot{\mathcal{L}}_x)^{-1}  \dot{\mathcal{R}}^\phi_x \right)_+\right),
\end{split}
\end{equation}
where $F(\eta) = \int_0^\eta  \det(\dot{\mathcal{R}}^\phi_x - 2 s \dot{\mathcal{L}}_x )  ds$.
\end{lemma} 

\begin{proof}   
Fix $x \in X$. Let $A = \dot{\mathcal{R}}^\phi_x$ and $B=\dot{\mathcal{L}}_x$. For $\operatorname{Re} z>n+1$, write
\begin{equation*}
\begin{split}
&H_x(z) = (2\pi)^{-n-1} \zeta(z)\times\\
&\int_{\abs{\eta}>C} \det (A- 2\eta B) \left( \operatorname{Tr} \abs{(A - 2 \eta B)_-}^{-z}  -  \operatorname{Tr} \abs{(A - 2 \eta B)_+}^{-z}  \right) d\eta.
\end{split}
\end{equation*}
We are going to calculate $H_x(0)$. Recall that $B$ has constant signature $(n_-, n_+)$. 
For $\eta > C$, $\abs{(A - 2 \eta B)_-} = 2 \eta B_+ (\operatorname{Id} - (2 \eta B)^{-1}A)$. Thus,
\begin{equation}\label{E:062120252321}
\begin{split}
& \abs{(A - 2\eta B)_-}^{-z} \\
= & (2\eta)^{-z} B_+^{-z} e^{-z \log (\operatorname{Id} - (2\eta B)^{-1}A)} \\
= & (2\eta)^{-z} B_+^{-z} \left( (\operatorname{Id} - z \log (\operatorname{Id} - (2 \eta B)^{-1}A)) + O(|z|^2) \right) \\
= & I_1(z, \eta) + I_2(z, \eta) + O(|z|^2),
\end{split}
\end{equation}
where 
\begin{equation}\label{E:062120252322}
\left\{
\begin{split}
I_1(z, \eta) = & (2\eta)^{-z} B_+^{-z}, \\
I_2(z, \eta) = & (2\eta)^{-z} B_+^{-z} (-z) \log (\operatorname{Id} - (2\eta B)^{-1}A).
\end{split}
\right.
\end{equation}

Let
\begin{equation}\label{E:062120252153}
\det(A - 2\eta B) = \sum_{s=0}^n a_s \eta^s.
\end{equation}

From \eqref{E:062120252322} and \eqref{E:062120252153} 
\begin{equation}\label{E:062120252356}
\begin{split}
&  (2\pi)^{-n-1} \zeta(z) \int_{\eta>C} \det (A-2\eta B) {\rm Tr\,} I_1(z, \eta)  d\eta  \\
= &(2\pi)^{-n-1} \zeta(z) \int_{\eta>C}     \sum_{s=0}^n a_s \eta^{-z+s} 2^{-z} {\rm Tr\,} B_+^{-z} d \eta \\
= & - (2\pi)^{-n-1} \zeta(z)  \sum_{s=0}^n a_s \frac{C^{-z+s+1}}{-z+s+1} 2^{-z}   {\rm Tr\,} B_+^{-z}.
\end{split}
\end{equation}

Now, 
\[
\begin{split}
I_2(z, \eta) = & (2\eta)^{-z} B_+^{-z} (-z) \log ({\rm Id\,} - (2\eta B)^{-1}A) \\
= &  (2\eta)^{-z} B_+^{-z} (-z) \sum_{j=1}^\infty ( (2\eta B)^{-1}A)^j \frac{1}{-j}.
\end{split}
\]
Thus,
\[
{\rm Tr \,} I_2(z, \eta)=   (2\eta)^{-z}(-z) \sum_{j=1}^\infty  \frac{1}{-j} \eta^{-j}  {\rm Tr\,} \left( B_+^{-z} ( (2B)^{-1}A)^j \right).
\]

We have
\begin{equation}\label{E:062220252301}
\begin{split}
&  \int_{\eta > C} \det (A-2\eta B) {\rm Tr\,} I_2(z, \eta) d\eta  \\
= &  \int_{\eta > C}  \sum_{s=0}^n a_s \eta^s (2\eta)^{-z} (-z)  \sum_{j=1}^\infty  \operatorname{Tr}\left( B^{-z}_+ ((2B)^{-1}A)^j \right) \frac{\eta^{-j}}{-j} d\eta  \\
= &  \sum_{s=0}^n \sum_{j=1}^\infty a_s \frac{C^{-j -z+s+1}}{-j -z+s+1} (-z) 2^{-z}   \operatorname{Tr} \left( B_+^{-z} ((2B)^{-1}A)^j  \right) \frac{1}{j}.
\end{split}
\end{equation}

Note that $\zeta(z)$ has a simple pole at $z=1$. From \eqref{E:062120252356} and \eqref{E:062220252301}, we deduce that the function
$(2\pi)^{-n-1} \zeta(z) \int_{\eta > C} \det (A-2\eta B) \operatorname{Tr} \abs{(A - 2 \eta B)_-}^{-z} d\eta$ extends to a meromorphic function on $\mathbb{C}$ with poles contained in
\[
\left\{n+1-j; j=0, 1, 2, \cdots \right\}
\]
and holomorphic at 0.

Similarly, we can repeat the procedure above and deduce that $H_x(z)$ extends to a meromorphic function on $\mathbb{C}$ with poles contained in
\[
\left\{n+1-j; j=0, 1, 2, \cdots \right\}.
\]
Moreover, $H_x(z)$ is holomorphic at 0.

From \eqref{E:062120252356}, we have
\begin{equation}\label{E:062120252328}
\begin{split}
& \left( (2\pi)^{-n-1} \zeta(z) \int_{\eta>C} \det (A-2\eta B) {\rm Tr\,} I_1(z, \eta)  d\eta  \right)\Big|_{z=0} \\
= & (2\pi)^{-n-1} \zeta(0) (-F(C) n_+),
\end{split}
\end{equation}
where $F(\eta) = \int_0^\eta  \det(\dot{\mathcal{R}}^\phi_x - 2 s \dot{\mathcal{L}}_x )  ds$.

From \eqref{E:062220252301}, we have
\begin{equation}\label{E:062220250020}
\begin{split}
& \left( \int_{\eta > C} \det (A-2\eta B) {\rm Tr\,} I_2(z, \eta) d\eta \right)\Big|_{z=0} \\
= &  \sum_{s=0}^n  a_s {\rm Tr\,} \left( \bigr(((2B)^{-1}A)^{s+1}\bigr)_+  \right) \frac{1}{s+1} \\
= & \operatorname{Tr}F\left(\bigr((2 \dot{\mathcal{L}}_x)^{-1} \dot{\mathcal{R}}^\phi_x\bigr)_+\right).
\end{split}
\end{equation}

From \eqref{E:062120252321}, \eqref{E:062120252322}, \eqref{E:062120252328} and \eqref{E:062220250020}, we get
\begin{equation}\label{E:062220252156}
\begin{split}
& \left( (2\pi)^{-n-1} \zeta(z) \int_{\eta > C} \det (A-2\eta B) \operatorname{Tr} \abs{(A - 2 \eta B)_-}^{-z} d\eta \right)\Big|_{z=0} \\
= & (2\pi)^{-n-1}\zeta(0) \left(-F(C) n_+ +  \operatorname{Tr}F\left(\bigr((2 \dot{\mathcal{L}}_x)^{-1} \dot{\mathcal{R}}^\phi_x\bigr)_+\right)\right).
\end{split}
\end{equation}

Similarly,
\begin{equation}\label{E:062220252203}
\begin{split}
& \left( (2\pi)^{-n-1} \zeta(z) \int_{\eta > C} \det (A-2\eta B) \operatorname{Tr} \abs{(A - 2 \eta B)_+}^{-z} d\eta \right)\Big|_{z=0} \\
= & (2\pi)^{-n-1}\zeta(0) \left(-F(C) n_- +\operatorname{Tr}\left(F\bigr((2 \dot{\mathcal{L}}_x)^{-1} \dot{\mathcal{R}}^\phi_x \bigr)_-\right)\right).
\end{split}
\end{equation}
 
From \eqref{E:062220252156} and \eqref{E:062220252203},
 \begin{equation}\label{E:062220252234}
\begin{split}
& \left( (2\pi)^{-n-1} \zeta(z) \int_{\eta > C} \det (A-2\eta B) \left[     \operatorname{Tr} \abs{(A - 2 \eta B)_-}^{-z}  -    \operatorname{Tr} \abs{(A - 2 \eta B)_+}^{-z}     \right] d\eta \right)\Big|_{z=0} \\
= & (2\pi)^{-n-1}\zeta(0) \left(F(C) (n_- - n_+) +  \operatorname{Tr}F\left( \bigr( (2 \dot{\mathcal{L}}_x)^{-1} \dot{\mathcal{R}}^\phi_x \bigr)_+\right) - \operatorname{Tr}F\left( \bigr( (2 \dot{\mathcal{L}}_x)^{-1}\dot{\mathcal{R}}^\phi_x\bigr)_- \right) \right).
\end{split}
\end{equation}
 
Similarly, we deduce
\begin{equation}\label{E:062220252237}
\begin{split}
& \left( (2\pi)^{-n-1} \zeta(z) \int_{\eta < -C} \det (A-2\eta B) \left[     \operatorname{Tr} \abs{(A - 2 \eta B)_-}^{-z}  -    \operatorname{Tr} \abs{(A - 2 \eta B)_+}^{-z}     \right] d\eta \right)\Big|_{z=0} \\
= &  (2\pi)^{-n-1}\zeta(0) \left(F(-C) (n_- - n_+) +  \operatorname{Tr}F\left( \bigr((2 \dot{\mathcal{L}}_x)^{-1} \dot{\mathcal{R}}^\phi_x\bigr)_+)\right) - \operatorname{Tr}F\left(\bigr((2 \dot{\mathcal{L}}_x)^{-1} \dot{\mathcal{R}}^\phi_x\bigr)_-) \right)\right).
\end{split}
\end{equation}

From \eqref{E:062220252234} and \eqref{E:062220252237}, we get \eqref{e-gue250623yydz}.
The proof is completed.
\end{proof}

We need

\begin{lemma}\label{l-gue250623yydI}
Fix $x\in X$ and $C \gg 1$. Write 
\[\det(\dot{\mathcal{R}}^\phi_x - 2\eta\dot{\mathcal{L}}_x)=\sum^n_{s=0}a_s\eta^s,\]
and let
\[(\log(1-x))^2 = \sum_{j=1}^\infty c_j x^j,\ \ \abs{x}<1.\]
With the notations used above,  we have 
\begin{equation}\label{e-gue250623ycd}
\begin{split}
&H'_x(0)  = (2\pi)^{-n-1}\zeta'(0)\hat H_x(0)-\frac{1}{2}(2\pi)^{-n-1}\hat{H}'_x(0),\ \ \zeta'(0)=-\frac{1}{2}\log(2\pi),\\
&\hat{H}'_x(0)=\sum_{s=0}^n (-a_s) \frac{1}{(s+1)^2}\Bigr(C^{s+1}(n_+-n_-)-(-C)^{s+1}(n_--n_+)\Bigr) \\
&+ \sum_{s=0}^n(-a_s)\frac{1}{s+1} (-\log 2C)\Bigr(C^{s+1}(n_+-n_-)- (-C)^{s+1}(n_--n_+)\Bigr) \\
 & + \sum_{s=0}^n (-a_s)\frac{1}{s+1} \Bigr(C^{s+1}\bigr(\log(\det B_+)^{-1}-\log(\det(-B_-))^{-1}\bigr)\\
 &\quad-(-C)^{s+1}\bigr(\log (\det (-B_-))^{-1}-\log(\det B_+)^{-1}\bigr)\Bigr)\\
 &+ \sum_{s=0}^n \sum_{j=1}^\infty \sum_{j-s \not= 1} a_s\Bigr(\frac{C^{-j+s+1}}{-j+s+1}\Bigr(\operatorname{Tr}\left(\bigr ((2B)^{-1}A)^j\bigr)_+\right)-\operatorname{Tr}\left(\bigr ((2B)^{-1}A)^j\bigr)_-\right)\\
 &\quad-\frac{(-C)^{-j+s+1}}{-j+s+1}\Bigr(\operatorname{Tr}\left(\bigr ((2B)^{-1}A)^j\bigr)_-\right)-\operatorname{Tr}\left(\bigr ((2B)^{-1}A)^j\bigr)_+\right)\Bigr)\Bigr)(-\frac{1}{j})  \\
 & + \sum_{s=0}^n \sum_{j=1}^\infty \sum_{j-s = 1} a_s \Bigr(-\log 2C)\Bigr(\operatorname{Tr}\left(\bigr((2B)^{-1}A)^j\bigr)_+\right)-\operatorname{Tr}\left(\bigr((2B)^{-1}A)^j\bigr)_-\right)\Bigr)\\
 &\quad-(-\log 2C)\Bigr(\operatorname{Tr}\left(\bigr((2B)^{-1}A)^j\bigr)_-\right)-\operatorname{Tr}\left(\bigr((2B)^{-1}A)^j\bigr)_+\right)\Bigr)\Bigr)\frac{1}{j}   \\
  & + \sum_{s=0}^n \sum_{j=1}^\infty \sum_{j-s = 1} a_s\Bigr( \Bigr(\operatorname{Tr} \left((-\log B_+) ((2B)^{-1}A)^j\right)-\operatorname{Tr} \left((-\log(-B_-)) ((2B)^{-1}A)^j\right)\\
  &\quad- \Bigr(\operatorname{Tr} \left((-\log(-B_-)) ((2B)^{-1}A)^j\right)-\operatorname{Tr} \left((-\log B_+) ((2B)^{-1}A)^j\right)\frac{1}{j},\\
  &+\sum_{s=0}^n \sum_{j=1}^\infty \sum_{j-s=1}\frac{1}{2}a_s c_j\Bigr(\Bigr(\operatorname{Tr}\left(\bigr((2B)^{-1}A)^j\bigr)_+ \right)-\operatorname{Tr}\left(\bigr((2B)^{-1}A)^j\bigr)_-\right)\Bigr)\\
  &\quad-\operatorname{Tr}\left(\bigr((2B)^{-1}A)^j\bigr)_- \right)-\operatorname{Tr}\left(\bigr((2B)^{-1}A)^j\bigr)_+ \right)\Bigr)\Bigr),
\end{split}
\end{equation}
where $\hat{H}'_x(0)$ is given by \eqref{e-gue250623yydz}, $A=\dot{\mathcal{R}}^\phi_x$, $B=\dot{\mathcal{L}}_x$.
\end{lemma}

\begin{proof}
As before, let $A = \dot{\mathcal{R}}^\phi_x$ and $B=\dot{\mathcal{L}}_x$. 
Let 
\[
\Theta(z):= \int_{\eta>C} \det (A- 2\eta B) \operatorname{Tr} \abs{(A - 2 \eta B)_-}^{-z}   d\eta.
\]
We are going to calculate $\Theta'(0)$. As before, we have, for $\eta > C$, $\abs{(A - 2 \eta B)_-} = 2 \eta B_+ (\operatorname{Id} - (2 \eta B)^{-1}A)$. Thus,
\[
\begin{split}
& \abs{(A - 2\eta B)_-}^{-z} \\
= & (2\eta)^{-z} B_+^{-z} e^{-z \log (\operatorname{Id} - (2\eta B)^{-1}A)} \\
= & (2\eta)^{-z} B_+^{-z} \left( (\operatorname{Id} - z \log (\operatorname{Id} - (2 \eta B)^{-1}A)) + \frac{1}{2}z^2 (\log (\operatorname{Id} - (2 \eta B)^{-1}A))^2 + O(|z|^3) \right).
\end{split}
\]
Hence,
\[
\operatorname{Tr} \abs{(A - 2\eta B)_-}^{-z} = I(z, \eta) + II(z, \eta) + III(z, \eta) + O(|z|^3),
\]
where
\begin{equation}\label{E:062120252147}
\left\{
\begin{split}
I(z, \eta) & = (2\eta)^{-z} \operatorname{Tr} B_+^{-z} \\
II(z, \eta) & = (2\eta)^{-z}(-z) \operatorname{Tr} ( B_+^{-z}  \log (\operatorname{Id} - (2 \eta B)^{-1}A)  ) \\
III(z, \eta) & = \frac{1}{2} (2\eta)^{-z} z^2 \operatorname{Tr} ( B_+^{-z} (\log (\operatorname{Id} - (2 \eta B)^{-1}A))^2 ).
\end{split}
\right.
\end{equation}

For $\operatorname{Re} z > n+1$, let
\[
\Theta_1(z) := \int_{\eta > C} \det(A- 2\eta B){\rm Tr \,} I_1(z, \eta) d\eta.
\]
We have
\[
\begin{split}
\Theta_1(z) = & \int_{\eta > C}  \sum_{s=0}^n a_s \eta^s (2\eta)^{-z} \operatorname{Tr} B^{-z}_+  d\eta \\
= & - \sum_{s=0}^n a_s \frac{C^{-z+s+1}}{-z+s+1} 2^{-z} \operatorname{Tr} B_+^{-z}. 
\end{split}
\]
Thus,
\begin{equation}\label{E:062120252158}
\begin{split} 
\Theta_1'(0) = & \sum_{s=0}^n (-a_s) \frac{1}{(s+1)^2} C^{s+1}n_+ + \sum_{s=0}^n (-a_s) \frac{1}{s+1} (-\log 2C) C^{s+1}n_+ \\
 & + \sum_{s=0}^n (-a_s) C^{s+1} \frac{1}{s+1} \log (\det B_+)^{-1}. 
\end{split}
\end{equation}

Let
\[
\Theta_2(z) := \int_{\eta > C} \det(A- 2\eta B) II_1(z, \eta) d\eta.
\]
We have
\[
\begin{split}
\Theta_2(z) = & \int_{\eta > C}  \sum_{s=0}^n a_s \eta^s (2\eta)^{-z} (-z)  \sum_{j=1}^\infty  \operatorname{Tr}\left( B^{-z}_+ ((2B)^{-1}A)^j \right) \frac{\eta^{-j}}{-j} d\eta \\
= &  \sum_{s=0}^n \sum_{j=1}^\infty a_s \frac{C^{-j -z+s+1}}{-j -z+s+1} (-z) 2^{-z}   \operatorname{Tr} \left( B_+^{-z} ((2B)^{-1}A)^j  \right) \frac{1}{j}. 
\end{split}
\]
Thus,
\begin{equation}\label{E:062120252219}
\begin{split}
\Theta_2'(0) = & \sum_{s=0}^n \sum_{j=1}^\infty \sum_{j-s \not= 1} a_s \frac{C^{-j+s+1}}{-j+s+1}(-1)\operatorname{Tr}\left(\bigr((2B)^{-1}A)^j\bigr)_+\right) \frac{1}{j}  \\
 & + \sum_{s=0}^n \sum_{j=1}^\infty \sum_{j-s = 1} a_s (-\log 2C)  \operatorname{Tr}\left(\bigr((2B)^{-1}A)^j\bigr)_+\right)  \frac{1}{j}   \\
  & + \sum_{s=0}^n \sum_{j=1}^\infty \sum_{j-s = 1} a_s \operatorname{Tr} \left((-\log B_+) ((2B)^{-1}A)^j\right) \frac{1}{j}.  
\end{split}
\end{equation}

Let
\[
\Theta_3(z) := \int_{\eta > C} \det(A- 2\eta B) III_1(z, \eta) d\eta.
\]
Let $(\log(1-x))^2 = \sum_{j=1}^\infty c_j x^j$, $\abs{x}<1$. We have
\[
\begin{split}
\Theta_3(z) = & \int_{\eta > C}  \sum_{s=0}^n a_s \eta^s  (2\eta)^{-z} \left( \frac{1}{2}z^2\right)  \sum_{j=1}^\infty  c_j  \operatorname{Tr}\left( B^{-z}_+ ((2B)^{-1}A)^j \right)  d\eta \\
= & - \sum_{s=0}^n \sum_{j=1}^\infty a_s \frac{C^{-j -z+s+1}}{-j -z+s+1} \left( \frac{1}{2}z^2\right)  2^{-z} c_j  \operatorname{Tr} \left( B_+^{-z} ((2B)^{-1}A)^j  \right). 
\end{split}
\]
Thus,
\begin{equation}\label{E:062120252238}
\Theta'_3(0) = \sum_{s=0}^n \sum_{j=1}^\infty \sum_{j-s=1} \frac{1}{2} a_s c_j \operatorname{Tr}\left(\bigr((2B)^{-1}A)^j\bigr)_+ \right).
\end{equation}
We get 
\begin{equation}\label{e-gue250623ycda}
\Theta'(0) = \Theta_1'(0) + \Theta_2'(0) + \Theta_3'(0),
\end{equation}
where  $\Theta_1'(0)$, $\Theta_2'(0)$ and $\Theta_3'(0)$ are given by \eqref{E:062120252158}, \eqref{E:062120252219} and \eqref{E:062120252238} respectively.

Let 
\[
\hat\Theta(z):= (2\pi)^{-n-1} \zeta(z) \int_{\eta>C} \det (A- 2\eta B) \operatorname{Tr} \abs{(A - 2 \eta B)_+}^{-z}   d\eta.
\]
We can repeat the procedure above and deduce that 
\begin{equation}\label{e-gue250623ycdb}
\hat\Theta'(0) = \hat{\Theta}_1'(0) + \hat{\Theta}_2'(0) + \hat{\Theta}_3'(0),
\end{equation}
where
\begin{equation}\label{e-gue250623ycdc}
\begin{split}
\hat{\Theta}_1'(0)&=\sum_{s=0}^n (-a_s) \frac{1}{(s+1)^2} C^{s+1}n_- + \sum_{s=0}^n (-a_s) \frac{1}{s+1} (-\log 2C) C^{s+1}n_- \\
 & + \sum_{s=0}^n (-a_s) C^{s+1} \frac{1}{s+1} \log (\det (-B_-))^{-1},\\
 \hat{\Theta}_2'(0) = & \sum_{s=0}^n \sum_{j=1}^\infty \sum_{j-s \not= 1} a_s \frac{C^{-j+s+1}}{-j+s+1}(-1)\operatorname{Tr}\left(\bigr ((2B)^{-1}A)^j\bigr)_-\right) \frac{1}{j}  \\
 & + \sum_{s=0}^n \sum_{j=1}^\infty \sum_{j-s = 1} a_s (-\log 2C)  \operatorname{Tr}\left(\bigr((2B)^{-1}A)^j\bigr)_-\right)  \frac{1}{j}   \\
  & + \sum_{s=0}^n \sum_{j=1}^\infty \sum_{j-s = 1} a_s \operatorname{Tr} \left((-\log(-B_-)) ((2B)^{-1}A)^j\right) \frac{1}{j},\\
  \hat{\Theta}'_3(0)&= \sum_{s=0}^n \sum_{j=1}^\infty \sum_{j-s=1} \frac{1}{2} a_s c_j \operatorname{Tr}\left(\bigr((2B)^{-1}A)^j\bigr)_- \right).
\end{split}
\end{equation}

Let 
\[
\Td\Theta(z):= (2\pi)^{-n-1} \zeta(z) \int_{\eta<-C} \det (A- 2\eta B)\Bigr(\operatorname{Tr} \abs{(A - 2 \eta B)_-}^{-z} -\operatorname{Tr} \abs{(A - 2 \eta B)_+}^{-z}\Bigr)  d\eta.
\]
We can repeat the calculation above, and deduce that 
\begin{equation}\label{e-gue250624yyd}
\begin{split}
\Td{\Theta}'(0)&=\sum_{s=0}^n a_s \frac{1}{(s+1)^2} (-C)^{s+1}(n_--n_+) + \sum_{s=0}^na_s\frac{1}{s+1} (-\log 2C) (-C)^{s+1}(n_--n_+) \\
 & + \sum_{s=0}^n a_s(-C)^{s+1} \frac{1}{s+1}\Bigr(\log (\det (-B_-))^{-1}-\log (\det B_+)^{-1}\Bigr)\\
 &+ \sum_{s=0}^n \sum_{j=1}^\infty \sum_{j-s \not= 1} a_s \frac{(-C)^{-j+s+1}}{-j+s+1}(-1)\Bigr(\operatorname{Tr}\left(\bigr ((2B)^{-1}A)^j\bigr)_-\right)-\operatorname{Tr}\left(\bigr ((2B)^{-1}A)^j\bigr)_+\right)\Bigr)(-\frac{1}{j})  \\
 & + \sum_{s=0}^n \sum_{j=1}^\infty \sum_{j-s = 1} a_s (-\log 2C)\Bigr(\operatorname{Tr}\left(\bigr((2B)^{-1}A)^j\bigr)_-\right)-\operatorname{Tr}\left(\bigr((2B)^{-1}A)^j\bigr)_+\right)\Bigr) (-\frac{1}{j})   \\
  & + \sum_{s=0}^n \sum_{j=1}^\infty \sum_{j-s = 1} a_s \Bigr(\operatorname{Tr} \left((-\log(-B_-)) ((2B)^{-1}A)^j\right)-\operatorname{Tr} \left((-\log B_+) ((2B)^{-1}A)^j\right) (-\frac{1}{j}),\\
  &+\sum_{s=0}^n \sum_{j=1}^\infty \sum_{j-s=1}(-\frac{1}{2}) a_s c_j\Bigr(\operatorname{Tr}\left(\bigr((2B)^{-1}A)^j\bigr)_- \right)-\operatorname{Tr}\left(\bigr((2B)^{-1}A)^j\bigr)_+ \right)\Bigr).
\end{split}
\end{equation}

From \eqref{e-gue250623yydz}, \eqref{E:062120252158}, \eqref{E:062120252219}, \eqref{E:062120252238}, 
\eqref{e-gue250623ycdb} and \eqref{e-gue250624yyd}, we get \eqref{e-gue250623ycd}. 
\end{proof}
    
	\begin{theorem}\label{t-gue160427}
		With the notations and assumptions used before, let 
        \begin{equation}\label{e-gue250531yyda}
        B_{j,k}:=k^{-(n+1)}\sum^n_{q=0}q(-1)^q\int b^{(q)}_{j,k}(x)dv_X(x),\ \ j=0,1,\ldots, 
        \end{equation}
        where  $b^{(q)}_{j,k}(x)\in\mathcal{C}^\infty(X)$, $j=0,1,\ldots$, are as in \eqref{e-gue250529ycdm}. Then, 
        \begin{equation}\label{e-gue250605yyd}
B_{j,k}=\int_X\widehat{A}_j(x)dv_X+O(k^{-\frac{1}{2}}),\ \ j=0,1,\ldots,
        \end{equation}
where $\widehat{A}_{j}(x)$ is as in \eqref{E:110320242242z}, $j=0,1,\ldots$, and for every $N>0$, we have for all $t\in(0,1)$ and all $k\geq1$, 
		\begin{equation}\label{e-gue160427}
			k^{-(n+1)}{\rm STr\,}\lbrack Ne^{-\frac{t}{k}\Box_{b,L^k}}(I-\Pi_{L^k})\rbrack=\sum^N_{j=0}t^{-n-1+j}B_{j,k}+\varepsilon_{N,k}
		\end{equation}
        where 
        \[\abs{\varepsilon_{N,k}}\leq Ct^{-n+N},\] 
        for all $t\in(0,1)$ and all $k\geq1$, $C>0$ is a constant independent of $k$ and $t$.

Moreover, for all $t\in(0,1)$ and every $N\in\mathbb N$, we have 
		\begin{equation}\label{E:5.5.43}
			\lim_{k\To\infty}k^{-(n+1)}{\rm STr\,}\lbrack Ne^{-\frac{t}{k}\Box_{b,L^k}}(I-\Pi_{L^k})\rbrack=\sum^N_{j=n-1}t^{-n-1+j}\int_X\widehat{A}_j(x)dv_X(x)+\delta_N(t),
		\end{equation}
		where $\widehat{A}_{j}(x)$ is as in \eqref{E:110320242242}, $j=0,1,\ldots$, and  \[\abs{\delta_{N}(t)}\leq\hat Ct^{-n+N},\] 
        for all $t\in(0,1)$, $\hat C>0$ is a constant independent of $k$ and $t$. 
	\end{theorem}

\begin{proof} 
From \eqref{e-gue250529ycdm} and \eqref{e-gue250529ycdn}, we immediately get \eqref{e-gue250605yyd} and 
\eqref{e-gue160427}.  

From \eqref{e-gue250529ycdm}, \eqref{e-gue250529ycdn} and \eqref{e-gue250529yydx}, we get that for every $N\in\mathbb N$,  
\begin{equation}\label{e-gue250531yydm}
\begin{split}
			\lim_{k\To\infty}&k^{-(n+1)}{\rm STr\,}\lbrack Ne^{-\frac{t}{k}\Box_{b,L^k}}(I-\Pi_{L^k})\\
            &=\sum^N_{j=0}t^{-n-1+j}\int_X\sum^n_{q=0}(-1)^qqb^{(q)}_j(x)dv_X(x)+\hat\delta_N(t),
            \end{split}
		\end{equation}
        for all $t\in(0,1]$, where $b^{(q)}_j(x)$ is as in \eqref{e-gue250529ycdm}, $q=0,1,\ldots,n$, $j=0,1,\ldots$, and 
        \[\abs{\hat\delta_{N}(t)}\leq C_1t^{-n+N},\] 
        for all $t\in(0,1)$, $C_1>0$ is a constant independent of $t$.

From \eqref{e-gue250531yyd} and \eqref{E:110320242242z}, we get 
\begin{equation}\label{e-gue250531yydn}
\int_X\widehat{A}_j(x)dv_X(x)= \int_X\sum^n_{q=0}(-1)^qqb^{(q)}_j(x)(x)dv_X(x),\ \ j=0,1,\ldots.
\end{equation}
From \eqref{e-gue250531yydm} and \eqref{e-gue250531yydn}, we get \eqref{E:5.5.43}.
The theorem follows. 
	\end{proof}
	

	\begin{theorem}\label{T:110520241105}
		Assume that $n_-=n_+$ or $\abs{n_--n_+}>1$. For any $q\in\set{0,1,\ldots,n}$, $t \ge 1$, $k \in \mathbb{N}$, we have
		\begin{equation}\label{E:110520241105}
			k^{-(n+1)}\operatorname{Tr}^{(q)}[e^{-\frac{t}{k}\Box_{b,L^k}}(I-\Pi_{L^k})] \le  Ct^{-(n+1)},
		\end{equation}
        where $C>0$ is a constant independent of $k$. 
	\end{theorem}
	\begin{proof}
Assume $q\notin\set{n_-,n_+}$. By Assumption~\ref{a-gue250426yyd}, for $t \ge 1$, we can apply the scaling technique in ~\cite[Subsection 3.2, Lemma 3.4]{HZ23} and get \eqref{E:110520241105}. Assume that $q\in\set{n_-,n_+}$. Since \[\operatorname{Tr}^{(q)}[e^{-\frac{t}{k}\Box_{b,L^k}}(I-\Pi_{L^k})]\leq \operatorname{Tr}^{(q-1)}[e^{-\frac{t}{k}\Box_{b,L^k}}(I-\Pi_{L^k})]+\operatorname{Tr}^{(q+1)}[e^{-\frac{t}{k}\Box_{b,L^k}}(I-\Pi_{L^k})],\]
	we get \eqref{E:110520241105}.
	\end{proof}
	
	Now we are in a position to prove the following
	
	\begin{theorem}\label{T:5.5.8}

    Recall that we work with Assumption~\ref{a-gue250426yyd} and we assume that $n_-=n_+$ or $\abs{n_--n_+}>1$. Fix $C \gg 1$. As $k \to+\infty$, we have
		\begin{equation}\label{E:5.5.33}
        \begin{split}
			&\theta_{b,L^k}'(0)\\   
          &=(\log k)k^{n+1}\Bigr(\int_X\int_{\abs{\eta}\leq C}\frac{1}{(2\pi)^{n+1}}\bigr(\frac{n}{2}{\rm det\,}(\dot{\mathcal{R}}^\phi_x - 2\eta \dot{\mathcal{L}_x})\\
          &\quad-\sum^n_{q=0}(-1)^qq\abs{{\rm det\,}(\dot{\mathcal{R}}^\phi_x - 2\eta \dot{\mathcal{L}_x})}1_{\mathbb R_x(q)}(\eta)\bigr)d\eta dv_X(x)+\int_X H_x(0)dv_X(x)\Bigr)\\
&+k^{n+1}\Bigr(-\int_X H'_x(0)dv_X(x)\\
&\quad+\frac{1}{2} \log (2\pi)(2\pi)^{-n-1} \int_X\int_{\abs{\eta}\leq C} \det \left(\dot{\mathcal{R}}^\phi_x - 2 \eta \dot{\mathcal{L}}_x \right)(2q-n)1_{\mathbb R_x(q)}(\eta)d\eta dv_X(x)\\
                &+\frac{1}{2}(2\pi)^{-n-1} \int_X\int_{\abs{\eta}\leq C} \det \left(\dot{\mathcal{R}}^\phi_x - 2 \eta \dot{\mathcal{L}}_x \right)\bigr(-\log(\abs{{\rm det\,}(\dot{\mathcal{R}}^\phi_x - 2 \eta \dot{\mathcal{L}}_x)_-})\\
                &\quad+\log(\abs{{\rm det\,}(\dot{\mathcal{R}}^\phi_x - 2 \eta \dot{\mathcal{L}}_x)_+})\bigr)1_{\mathbb R_x(q)}(\eta)d\eta dv_X(x)\Bigr)+o(k^{n+1}),
          \end{split}
		\end{equation}
		where $H_x(z)$ is given by \eqref{e-gue250603yyd}, Lemma~\ref{l-gue250623yyd} and $H_x(0)$ and $H'_x(0)$ are computed in Lemma \ref{l-gue250623yyd} and Lemma \ref{l-gue250623yydI} respectively.
		
	\end{theorem}
	
	\begin{proof}
		For $k\gg1$, set
		\begin{equation}\label{E:5.5.50}
			\widetilde{\theta}_{b,L^k}(z) = -M \left\lbrack k^{-(n+1)} \operatorname{STr} \Big\lbrack Ne^{-\frac{t}{k}\Box_{b,L^k}}(I-\Pi_k)\Big\rbrack \right\rbrack(z).
		\end{equation}
		Clearly 
		\begin{equation}\label{E:5.5.51}
			k^{-(n+1)} \theta_{b,L^k}(z) = k^{-z} \widetilde{\theta}_{b,L^k}(z).
		\end{equation}
		By \eqref{E:110320242242z}, Theorem~\ref{t-gue160427} and \eqref{E:5.5.43} and Theorem~\ref{T:110520241105}, we have
		\begin{equation}\label{E:5.5.52}
			\begin{split}
				& k^{-(n+1)} \theta'_{b,L^k}(0)  = - \log(k) \widetilde{\theta}_{b,L^k}(0) + \widetilde{\theta}'_{b,L^k}(0),   \\
				&  \widetilde{\theta}_{b,L^k}(0)=-\int_X\widehat{A}_{n+1}(x)dv_X+O(k^{-\frac{1}{2}}).
			\end{split}
		\end{equation} 
		By \eqref{E:5.5.13}, Theorem~\ref{t-gue160427}, Theorem~\ref{T:110520241105} and Lebesgue's dominated convergence theorem, we get
		\begin{equation}\label{E:5.5.53}
			\begin{split}
				&\lim_{k \to \infty} \widetilde{\theta}'_{b,L^k}(0) \\
                =&  -\int_0^1 \lim_{k \to \infty} \left\{  k^{-(n+1)} \operatorname{STr} \left\lbrack Ne^{-\frac{t}{k}\Box_{b,L^k}}(I-\Pi_{L^k})\right\rbrack - \sum^{n+1}_{j=0} B_{j,k}t^{-n-1+j}\right\} \frac{dt}{t} \\
				& -\int_1^\infty \lim_{k \to \infty} \left\{  k^{-(n+1)} \operatorname{STr} \left\lbrack Ne^{-\frac{t}{k}\Box_{b,L^k}}(I-\Pi_{L^k})\right\rbrack \right\} \frac{dt}{t}  \\
				&-\lim_{k\to\infty}\sum^{n}_{j=0}\frac{B_{j,k}}{j-n-1}+\Gamma'(1)\lim_{k\to\infty}B_{n+1,k}.
			\end{split}
		\end{equation}

		
		For $z \in \mathbb{C}$ set
		\begin{equation}\label{E:5.5.55}
        \begin{split}
		&	\widetilde{\zeta}(z)  \\ & = -M
            \left\lbrack \int_X\int_{\mathbb R}\Bigr((\dot{\mathcal{R}}^\phi_x-2\eta \dot{\mathcal{L}}_x)_t -(2\pi)^{-n-1}\sum^n_{q=0}q(-1)^q\abs{\det(\dot{\mathcal{R}}_x^\phi - 2 \eta \dot{\mathcal{L}}_x)}1_{\mathbb R_x(q)}(\eta)\Bigr)d\eta dv_X\right\rbrack (z).
            \end{split}
		\end{equation}

From Theorem~\ref{t-gue250531yyd}, Theorem~\ref{t-gue160427}, Theorem~\ref{T:110520241105} and \eqref{E:5.5.53}, we have 
		\begin{equation}\label{E:5.5.56}
\lim_{k\to\infty}\widetilde{\theta}'_{b,L^k}(0)\\
				= \widetilde{\zeta}'(0).
		\end{equation}
		
    Let $\zeta(z) =\sum_{k=1}^\infty \frac{1}{n^z}$ be the Riemann zeta function. It is well-known that, for $\operatorname{Re}z>1$,
		\begin{equation}\label{E:5.5.58}
			\zeta(z) =\frac{1}{\Gamma(z)} \int_0^\infty t^{z-1} \frac{e^{-t}}{1-e^{-t}}dt.
		\end{equation} 
		Moreover,
		\begin{equation}\label{E:5.5.59}
			\zeta(0) = -\frac{1}{2}, \qquad \zeta'(0) = -\frac{1}{2} \log (2\pi).
		\end{equation}
		We can repeat the proof of Lemma~\ref{l-gue250604yyd} and deduce that for ${\rm Re\,}z>n+1$, $C\gg1$, 
\begin{equation}\label{e-gue250606yyd}
\begin{split}
&\widetilde{\zeta}(z)\\
&=-\int_X H_x(z)dv_X(x)-(2\pi)^{-n-1}\zeta(z)\int_X\int_{\abs{\eta}\leq C}\det\left(\dot{\mathcal{R}}^\phi_x - 2 \eta \dot{\mathcal{L}}_x\right)\\
&\quad\times\left[ \operatorname{Tr}\abs{(\dot{\mathcal{R}}^\phi_x - 2 \eta \dot{\mathcal{L}}_x  )_-}^{-z} - \operatorname{Tr}\abs{(\dot{\mathcal{R}}^\phi_x - 2 \eta \dot{\mathcal{L}}_x  )_+}^{-z} \right]1_{\mathbb R_x(q)}(\eta)d\eta dv_X(x).
\end{split}
\end{equation}

		By \eqref{E:5.5.58}, \eqref{E:5.5.59} and \eqref{e-gue250606yyd}, we get
		\begin{equation}\label{E:5.5.60}
			\begin{split}
				&\widetilde{\zeta}'(0)=-\int_X H'_x(0)dv_X(x)\\
                &+\frac{1}{2} \log (2\pi)(2\pi)^{-n-1} \int_X\int_{\abs{\eta}\leq C} \det \left(\dot{\mathcal{R}}^\phi_x - 2 \eta \dot{\mathcal{L}}_x \right)(2q-n)1_{\mathbb R_x(q)}(\eta)d\eta dv_X(x)\\
                &+\frac{1}{2}(2\pi)^{-n-1} \int_X\int_{\abs{\eta}\leq C} \det \left(\dot{\mathcal{R}}^\phi_x - 2 \eta \dot{\mathcal{L}}_x \right)\Bigr(-\log(\abs{{\rm det\,}(\dot{\mathcal{R}}^\phi_x - 2 \eta \dot{\mathcal{L}}_x)_-})\\
                &+\log(\abs{{\rm det\,}(\dot{\mathcal{R}}^\phi_x - 2 \eta \dot{\mathcal{L}}_x)_+})\Bigr)1_{\mathbb R_x(q)}(\eta)d\eta dv_X(x).
			\end{split}
		\end{equation}
		By \eqref{E:5.5.52}, \eqref{E:5.5.56} and \eqref{E:5.5.60}, we get \eqref{E:5.5.33}.
	\end{proof}

\subsection{Examples}\label{s-gue250714}

In this subsection, we will provide examples that satisfy all the conditions of Theorem \ref{T:5.5.8}. We now assume that the CR manifold $X$ is an irregular Sasakian manifold ($X$ can be irregular), that is, $X$ is strongly pseudoconvex and 
\[
[T, \mathcal{C}^\infty(X, T^{1,0}X)] \subset\mathcal{C}^\infty(X, T^{1,0}X).
\]
We take $\langle \, \cdot \,|\, \cdot \, \rangle$ so that $\langle \, \cdot \,|\, \cdot \, \rangle$ is $T$-invariant. Let $\eta : X \times X \to X$ be the $S^1$-action induced by the vector field $T$. 

\begin{theorem}
We can find an $S^1$-invariant CR line bundle $(L, h^L) \to X$ such that $h^L$ is $S^1$-invariant and $\mathcal{R}^\phi_x = \mathcal{L}_x$, for all $x \in X$. 
\end{theorem}
\begin{proof}
Let $x = (z, \theta)$ be a BRT chart of $X$ defined in an open set $D$ of $X$ (see discussion before Example 2.11 of \cite{HHL22} for the meaning of BRT charts). Then 
\[
\begin{split}
& T_x^{1,0}X  = \operatorname{span} \left\{\frac{\partial}{\partial z_j} - i \frac{\partial \phi}{\partial z_j} \frac{\partial}{\partial \theta}, j=1, \cdots, n \right\}, \quad \phi(z) \in \mathcal{C}^\infty(D), \\
& T  = -\frac{\partial}{\partial \theta}.
\end{split}
\]
Also, $\mathcal{L}_x = \sum_{j, \ell = 1}^n \frac{\partial^2 \phi}{\partial z_j \partial \overline{z}_\ell} dz_j \wedge d\overline{z}_\ell \Big|_{T^{1,0}_xX}$. Let $s_\phi:=1$ be the local CR trivialization on $D$ with $|1|^2_{h^L} = e^{-2\phi}$. Let $(w, \eta)$ be another BRT chart of $X$ on $D$, then
\[
T_x^{1,0}X  = \operatorname{span} \left\{\frac{\partial}{\partial w_j} - i \frac{\partial \widetilde{\phi}}{\partial w_j} \frac{\partial}{\partial \eta}, j=1, \cdots, n \right\},\ \ \widetilde{\phi}\in\mathcal{C}^\infty(D).
\]
We can check that
\begin{equation}\label{E:071420252233}
\left\{
\begin{split}
& w =(w_1, \cdots, w_n) = H(z) = (H_1(z), \cdots, H_n(z)), \\
& \overline{\partial}H  =0, \quad \widetilde{\phi}(H(z)) = \phi(z) + \log |g(z)|, \\
&\eta  = \theta + \operatorname{Im} g(z), \quad g \in\mathcal{C}^\infty(D), \quad g \not= 0, \quad \overline{\partial} g =0.
\end{split}
\right.
\end{equation}
 From \eqref{E:071420252233}, we have
 \[
 s_{\widetilde{\phi}} = g^{-1} s_\phi \qquad \text{on $D$}.
 \]
 Thus, $\{ s_\phi \}$ defines an $S^1$-invariant CR line bundle $L$.
\end{proof}

Let $\widehat{X}:= \left\{ v \in L^*;\, |v|_{h^{L^*}}^2  = 1 \right\}$. Then $\widehat{X}$ is a CR manifold of codimension two.
Let $T_1$ be the vector field on $\widehat{X}$ induced by the $S^1$-action on $\widehat{X}$ acting on the fiber. For $\gamma \in \operatorname{Spec} (-iT)$, $m \in \mathbb{Z}$, $q \in \{0, 1, \cdots, n\}$, put $\Omega^{0,q}_{\gamma, m}(X) := \{u \in \Omega^{0,q}(\widehat{X}); Tu = -i\gamma u, T_1 u = -imu \}$. Let $\{L_1, \cdots, L_N \} \subset\mathcal{C}^\infty(X, HX)$ such that for every $x \in X$, $\operatorname{span} \{L_1(x), \cdots, L_N(x) \} = H_xX$. For $u \in \Omega^{0,q}(\widehat{X})$, put $\| \widehat{u} \|^2_1:= \sum_{j=1}^n \| L_j u\|^2$. Since $\mathcal{R}^\phi_x = \mathcal{L}_x$, for all $x \in X$, we can repeat Kohn estimates (see Theorem 8.3.5 of \cite{CS}) and deduce that there is a constant $C>0$ such that 
\begin{equation}\label{E:071420252318}
|(\, (T+T_1)u\,| \, u \,)| + \| \widehat{u} \|_1^2 \le C( \| \Box_b^{(q)} u\|^2 + \| u \|^2),
\end{equation}
for all $u \in \Omega^{0,q}(\widehat{X})$, $q \in \{ 1, 2, \cdots, n-1\}$. We have the natural $L^2$ isometric map
\[
F: \Omega^{0,q}_{\gamma, m}(\widehat{X}) \to \Omega^{0,q}_{\gamma + m, 0}(\widehat{X})
\]
such that $F$ commutes with $\Box_b^{(q)}$. From this observation and \eqref{E:071420252318}, we deduce

\begin{lemma}\label{L:071520251512}
Let $q \in \{1, \cdots, n-1 \}$. Fix $C_0 \gg 1$. Then we can find a constant $C>0$ such that 
\[
\| \Box_b^{(q)} u\|^2 \ge C\| u\|^2,
\]
for all $u \in \Omega^{0,q}_{\gamma, m}(\widehat{X})$, $|\gamma + m | \le C_0$, $u \perp \operatorname{Ker} \Box_b^{(q)}$.
\end{lemma}

For $\gamma \in \operatorname{Spec}(-iT)$, put 
$\Omega_\gamma^{0,q}(X, L^k):= \{u \in \Omega^{0,q}(X, L^k); -iTu = \gamma u \}$. From Lemma \ref{L:071520251512}, we deduce

\begin{lemma}\label{L:071520251521}
Let $q \in \{1, \cdots, n-1\}$. Fix $C_0 \gg 1$. Then we can find a constant $c >0$ such that $\|\Box_{b, L^k}^{(q)}u \|_{L^k}^2 \ge c \|u\|_{L^k}^2$, for all $u \in \Omega_\gamma^{0,q}(X, L^k)$, $|\gamma + k| \le C_0$, $u \perp \operatorname{Ker} \Box_{b, L^k}^{(q)}$.
\end{lemma}

From Proposition 4.1 of \cite{Hsiao18} and Kohn estimates, we deduce the following

\begin{lemma}\label{L:071520251523}
Let $q \in \{1, \cdots, n-1\}$. There is a constant $C_1 >0$ such that
\begin{equation}
\big|k\|u\|^2 + (-iTu|u)\big| \le C_1(\|\Box^{(q)}_{b, L^k}u\|^2 + \|u\|^2),
\end{equation}
for all $u \in \Omega^{0,q}(X,L^k)$.
\end{lemma}

From Lemma \ref{L:071520251521} and Lemma \ref{L:071520251523}, we deduce 

\begin{theorem}\label{T:071520251528}
Let $q \in \{1, \cdots, n-1\}$. There is a constant $\widehat{C}_1 >0$ such that
$\|\Box_{b, L^k}^{(q)}u \|_{L^k}^2 \ge \widehat{C}_1 \|u\|_{L^k}^2$, for all $u \in \Omega^{0,q}(X, L^k)$,  $u \perp \operatorname{Ker} \Box_{b, L^k}^{(q)}$.
\end{theorem}

From Theorem \ref{T:071520251528}, we deduce that for $q=0, n$, 
\[
\operatorname{Spec} \Box^{(q)}_{b, L^k} \subset \{0\} \cup [C, +\infty),
\]
where $C>0$ is a constant independent of $k$. We obtain the following

\begin{theorem}
We have
\[
\operatorname{Spec} \Box_{b, L^k} \subset \{0\} \cup [C, +\infty),
\]
where $C>0$ is a constant independent of $k$.
\end{theorem}
    
Since $\mathcal{R}^\phi_x = \mathcal{L}_x$, $\mathbb{R}_x(q)=\emptyset$, for all $x \in X$ and all $q \in \{1, \cdots, n-1\}$, where $\mathbb{R}_x(q)$ is given by \eqref{e-gue250529ycdu}. From the proof of \cite[Theorem 1.5]{HM12}, we get
\[
\dim E^{(q)}_{0 < \lambda \le \frac{k}{\log k}}(X, L^k) = o(k^{n+1}), \quad q \in \{1, \cdots, n-1\}.
\]
Hence
\[
\dim E_{0 < \lambda \le \frac{k}{\log k}}(X, L^k) = o(k^{n+1}).
\]

\section{\texorpdfstring{The small time asymptotics of the heat kernel on CR manifolds with $S^1$-action}{}}\label{s23}
	
	Now we assume that $X$ admits an $S^1$-action. 
    In this section we will establish the small time asymptotics of the CR heat kernel in $S^1$-action case. 
    Let $T \in C^\infty(X, TX)$ be the infinitesimal generator of the $S^1$-action.

	Let 
	\[
	e^{-t\Box^{(q)}_{b, k, \le k\delta} } := e^{-t\Box^{(q)}_{b, k} } \circ Q_{X, \le k\delta} : L^2_{(0, q)}(X, L^k) \to  \mbox{Dom} \, \Box^{(q)}_{b, k, \le k\delta}, \quad t>0,
	\]
	and let
	\[
	A_{k, \delta}(t, x, y)  =  e^{-t\Box^{(q)}_{b, k, \le k\delta}}(x, y) \in C^\infty(\mathbb{R}_+ \times X \times X, (T^{*0,q}X \times L^k) \boxtimes (T^{*0,q}X \otimes L^k)^*)
	\]
	be the distribution kernel of $ e^{-t\Box^{(q)}_{b, k, \le k\delta}}$ with respect to $dv_X(x)$. We also write $A_{k, \delta}(t)$ to denote  $e^{-t\Box^{(q)}_{b, k, \le k\delta}}$. Note that$A_{k. \delta}\left( \frac{t}{k}\right)$ satisfies the following
	\begin{equation}
		\left\{ \begin{aligned}
			& \left( \frac{\partial}{\partial t}+ \frac{1}{k} \Box^{(q)}_{b,k}\right) A_{k, \delta}\left( \frac{t}{k}\right) =0  \\
			& \lim_{t \to 0}A_{k, \delta}(t) = Q_{X, \le k\delta} \quad \text{on} \ L^2_{(0, q)}(X, L^k).
		\end{aligned}   \right.
	\end{equation}
	
	As before, we use the canonical identification $\mbox{End}(L^k) = \mathbb{C}$, especially
	\[
	A_{k, \delta}(\frac{t}{k}, x, y)  =  e^{-\frac{t}{k} \Box^{(q)}_{b, k, \le k\delta}}(x, y) \in C^\infty(\mathbb{R}_+ \times X \times X, (T^{*0,q}X \times L^k) \boxtimes (T^{*0,q}X \otimes L^k)^*)
	\]
	Similarly to \eqref{e-gue210303yydI}, there exists $A_{k, s, \delta}(t, x, y) \in C^\infty(\mathbb{R}_+ \times D \times D, T^{*0,q}X \boxtimes (T^{*0,q}X)^*)$ such that
	\[
	(e^{-t\Box^{(q)}_{b, k, \le k\delta}}u)(x) = s^k(x) \otimes e^{\frac{k\phi(x)}{2}} \int_D  A_{k, s, \delta}(t, x, y) e^{-\frac{k\phi(y)}{2}} \hat{u}(y) dv_X(y) \ \text{on} \ D,
	\]
	for every $u = s^k \otimes \hat{u} \in \Omega^{0,q}_c(D, L^k)$, $\hat{u} \in \Omega^{0,q}_c(D)$. Note that
	\[
	A_{k, \delta}(t, x, y) = s^k(x) \otimes A_{k, s, \delta}(t, x, y)  e^{\frac{k\phi(x) - k\phi(y)}{2}}   \otimes (s^k(y))^*      \ \text{on} \ D.
	\]
	In particular,
	\[
	e^{-t\Box^{(q)}_{b, k, \le k\delta}}(x, x) = A_{k, \delta}(t, x, x) = A_{k, s, \delta}(t, x, x) \ \text{on} \ D.
	\] 

    Let $\Pi_{L^k,\le k\delta}: L^2_{(0,\bullet)}(X,L^k)\To{\rm Ker\,}\Box_{b,k,\le k\delta}$ be the orthogonal projection with respect to $(\,\cdot\,|\,\cdot\,)_{L^k}$ and let $\Pi_{L^k,\le k\delta}(x,y)\in\mathcal{D}'(X\times X,(L^k\otimes T^{*0,\bullet}X)\boxtimes(L^k\otimes T^{*0,\bullet}X)^*)$
    be the distribution kernel of $\Pi_{L^k,\le k\delta}$.  Let $q\in\set{0,1,\ldots,n}$. We write  $\Pi^{(q)}_{L^k,\le k\delta}:=\Pi_{L^k,\le k\delta}|_{L^2_{(0,q)}(X,L^k)}$ and let $\Pi^{(q)}_{L^k,\le k\delta}(x,y)\in\mathcal{D}'(X\times X,(L^k\otimes T^{*0,q}X)\boxtimes(L^k\otimes T^{*0,q}X)^*)$
    be the distribution kernel of $\Pi^{(q)}_{L^k,\le k\delta}$. We will also consider $\Pi_{L^k,\le k\delta}(x,y)$, $\Pi^{(q)}_{L^k, \le k\delta}(x,y)$, $q=0,1,\ldots,n$, to be elements in $\mathcal{D}'(X\times X,(L^k\otimes\Lambda^\bullet(\mathbb CT^*X))\boxtimes(L^k\otimes (\Lambda^\bullet(\mathbb CT^*X))^*)$ and $\Pi_{L^k, \le k\delta}$, $\Pi^{(q)}_{L^k,\le k\delta}$, $q=0,1,\ldots,n$, to be continuous operators: 
\[\Pi_{L^k,\le k\delta}, \Pi^{(q)}_{L^k, \le k\delta}: L^2(X,L^k\otimes\Lambda^\bullet(\mathbb CT^*X))\To  L^2(X,L^k\otimes\Lambda^\bullet(\mathbb CT^*X)),\]
$q=0,1,\ldots,n$.
    
	
	\subsection{\texorpdfstring{Spectral gap of $\Box^{(q)}_{b, k, \le k\delta}$}{}}

	Let $s \in \mathbb{N}_0$, denote by $\| \cdot \|_s$ the standard Sobolev norm for sections of $L^k$ with respect to $(\cdot \mid \cdot )_{h^{L^k}}$ of order $s$. 
	
	
	If we go over Kohn's poof (\cite[Theorem 8.4.2]{CS}, \cite{K65}), we see that the constant $c$ can be taken to be independent of $k$. We have the following theorem.
	
	\begin{theorem}\label{T:Kohnest2}
		Assume that $X$ admits a locally free CR transversal $S^1$-action. 
        There exists a constant $c > 0$ such that
		\begin{equation}\label{E:Kohnest2}
	ck^2\|u\|^2 \le \| \Box^{(q)}_{b, k, \le k\delta}u \|^2 + k\|(Tu\,|\, u) \|, \quad \text{for every $u \in \Omega^{0,q}_{\le k \delta}(X, L^k)$}.
		\end{equation}
	\end{theorem}

	\begin{theorem}\label{T:110520241057}
		Let $\mu^q_k$ be the lowest eigenvalue of $\Box^{(q)}_{b, k, \le k\delta}$. There exist constants $c_1>0$ not depending on $k$ such that, for $q \ge 1$ and $k \in \mathbb{N}$,
		\begin{equation}\label{E:110520241057}
			\mu^q_k \ge c_1 k.
		\end{equation}
	\end{theorem}
	\begin{proof}
		Fix $q \in \{1, 2, \cdots, n\}$. By \eqref{E:Kohnest2}, there is a constant $c>0$ independent of $k$ such that
		\begin{equation}\label{E:specgap1}
		ck^2\|u\|^2 \le \| \Box^{(q)}_{b, k, \le k\delta}u \|^2 + k\|(Tu\,|\, u) \| \le \| \Box^{(q)}_{b, k, \le k \delta}u \|^2 + k^2\delta \| u \|^2, \quad \text{for every $u \in \Omega^{0,q}_{\le k \delta}(X, L^k)$}.
		\end{equation}
		By \eqref{E:specgap1}, we have
		\begin{equation}\label{E:specgap2}
       (c-\delta)k^2 \|u\|^2  \le  \| \Box^{(q)}_{b, k, \le k\delta}u \|^2 , \quad \text{for every $u \in \Omega^{0,q}_{\le k \delta}(X, L^k)$}.
		\end{equation}
		Let $v \in \Omega^{0,q}_{\le k \delta}(X, L^k)$ be a nonzero eigenfunction of $\Box^{(q)}_{b, k, \le k\delta}$ with eigenvalue $\mu^q_k$. From \eqref{E:specgap2}, if $\delta > 0$ is small enough, then we have
		\[
		\sqrt{c-\delta} k \|v\| \le  \| \Box^{(q)}_{b, k, \le k\delta}v \| = \mu^q_k \|v\|.
		\]
		Hence, $\mu^q_k \ge \sqrt{c-\delta} k$. The theorem follows.
	\end{proof}

	
	\subsection{Small-time asymptotics of the CR heat kernel with $S^1$-action}
	
	In this subsection, we will establish small-time asymptotic expansions for the CR heat kernel with $S^1$-action. 
	
 Assume that \(X\) admits a locally free transversal CR $S^1$-action $e^{i\theta}$. Let \(T \in C^\infty(X, TX)\) be the vector field induced by $e^{i\theta}$. We notice that(see~\cite{HHL22}) $L, h^L$ and $R^L$, the Hermitian metric $\langle \cdot\, | \cdot \rangle$ and the $L^2$ inner product $(\cdot \, | \, \cdot )_k$ are $S^1$ invariant.

For $m\in \mathbb{Z}$, let
\begin{equation}\label{e01191}
\Omega^{0,\bullet}_{m}(X,L^k) = \{u \in \Omega^{0,\bullet}(X,L^k) :(e^{i\theta})^\ast u = e^{im\theta}u, \forall \ e^{i\theta} \in S^1\}.
\end{equation}
Let $L^{2,m}_{0,\bullet}(X,L^k)$ be the $L^2$ completion of $\Omega^{0,\bullet}_m(X,L^k)$ with respect to $(\cdot \,|\, \cdot)_k$.

Fix \(\delta > 0\) small. Let
\[
\Omega^{0,\bullet}_{\leq k\delta}(X,L^k) := \bigoplus_{|m| \leq k\delta,m\in\mathbb{Z}} \Omega^{0,\bullet}_{m}(X,L^k) ,
\]
and
\[
L^{2,\leq k\delta}_{0,\bullet}(X,L^k) := \bigoplus_{|m| \leq k\delta,m\in\mathbb{Z}} L^{2,m}_{0,\bullet}(X,L^k).
\]

We will consider the heat kernel and analytic torsion on \(L^{2,\leq k\delta}_{0,\bullet}(X,L^k)\). Fix \(m \in \mathbb{Z}\), consider
\begin{equation}\label{e01192}
e^{-t\Box_{b,m}}(x, y) = \frac{1}{2\pi} \int_{-\pi}^{\pi}e^{-t\Box_{b}} (x, e^{i\theta}\circ y)e^{-im\theta} d\theta.
\end{equation}

We have the following
\begin{theorem}\label{t05042026}
		With the notations and assumptions used before, we can find 
        $a_{j,k} \in \mathbb{R}$, 
        $j=0,1,\ldots$, $a_{0,k}\neq0$, such that
		\[
		\int_X e^{-t\Box^{(q)}_{b, k, \le k\delta} }(x, x)dv_X(x) \sim\sum^{+\infty}_{j=0} a_{j,k} t^{-n-1+j} \quad \text{as $t \to 0^+$}. 
		\]
	\end{theorem}
\begin{proof}
It suffices to show that \(e^{-t\Box_{b,m}}\) admits a small-time asymptotic expansion.
From Theorem \ref{T:0111202511001} we have
\begin{equation}\label{e01193}
\begin{split}
e^{-t\Box_{b,m}}(x, y) &= \frac{1}{2\pi} \int_{-\pi}^{\pi} \int_{\mathbb{R}^+} e^{i\frac{1}{t}\Phi_{-}(x,e^{i\theta}\circ y,s)-im\theta} b_{-}(t, x,e^{i\theta}\circ y,s)dsd\theta\\
&+\frac{1}{2\pi} \int_{-\pi}^{\pi} \int_{\mathbb{R}^+} e^{i\frac{1}{t}\Phi_{+}(x,e^{i\theta}\circ y,s)-im\theta} b_{+}(t, x,e^{i\theta}\circ y,s)dsd\theta+ O(t^{+\infty}).
\end{split}
\end{equation}
Recall that \(d_y \Phi_{\pm}(x, x, s) = \mp \omega_0(x)s\). Without loss of generality, suppose that \(\langle \omega_0, T \rangle > 0\).

Assume that \(m \geq 0\). We rewrite \eqref{e01193} as
\begin{equation}\label{e01194}
\begin{split}
e^{-t\Box_{b,m}}(x, y) &= \frac{1}{2\pi} \int_{-\pi}^{\pi} \int_{\mathbb{R}^+} e^{i\Psi_{+}(x,e^{i\theta}\circ y,t,s)-im\theta} tb_{+}(t,x,e^{i\theta}\circ y,ts)dsd\theta\\
&+\frac{1}{2\pi} \int_{-\pi}^{\pi} \int_{\mathbb{R}^+} e^{i\Psi_{-}(x,e^{i\theta}\circ y,t,s)-im\theta} tb_{-}(t,x,e^{i\theta}\circ y,ts)dsd\theta+ O(t^{+\infty})\\ & =I+II+ O(t^{+\infty}),
\end{split}
\end{equation}
where \(\Psi_{\pm}(x, y, t, s) := \frac{1}{t} \Phi_{\pm}(x, y, ts)\).

Since \(\langle \omega_0, T\rangle > 0\), we have
\[
\frac{d}{d\theta} \left( i\Psi_{+}(x, e^{i\theta}\circ x, t, s) - im\theta \right) \big|_{\theta=0} = -i\langle \omega_0(x), T_1 \rangle s-im\neq 0.
\]
We can integrate by parts in \(\theta\) several times and deduce that \(I = O(t^{+\infty})\).

We only need to consider II. Fix \(p \in X\). Suppose \(N_p := \{ \theta \in [0, 2\pi): e^{i\theta}\circ p = p \}
= \{ 0, \frac{2\pi}{\ell}, \frac{4\pi}{\ell}, \dots, \frac{(\ell-1)2\pi}{\ell} \}\). Near \(p\), we have
\begin{equation}\label{e01195}
\begin{split}
II= \sum_{j=0}^{\ell-1} \frac{1}{2\pi} \int_{\mathbb{R}}\int_{\mathbb{R}^+} &e^{i\Psi_{-}(x,e^{i\theta}\circ e^{i\frac{j}{\ell}2\pi}\circ y,t,s)-im\theta }
tb_{-}(t,x,e^{i\theta}\circ e^{i\frac{j}{\ell}2\pi}\circ y,ts)\chi(\theta)d\theta ds\\
&+O(t^{+\infty}),
\end{split}
\end{equation}
where \(\chi \in C_c^\infty((-\varepsilon, \varepsilon))\), \(0 \leq \chi \leq 1\), \(\chi \equiv 1\) on \([-\frac{\varepsilon}{2}, \frac{\varepsilon}{2}]\), \(\varepsilon > 0\) is a small constant. For \(j \in \{0, 1, \dots, \ell-1\}\), consider
\begin{equation}\label{e01196}
\Gamma_{j,m}(t,x,y) := \frac{1}{2\pi} \int_{|\theta|<\epsilon}\int_{\mathbb{R}^+} e^{i\Psi_{-}(x,e^{i\theta}\circ e^{i\frac{j}{\ell}2\pi}\circ y,t,s)-im\theta }
tb_{-}(t,x,e^{i\theta}\circ e^{i\frac{j}{\ell}2\pi}\circ y,ts)\chi(\theta)d\theta ds.
\end{equation}
Let \(X = (x_1, \dots, x_{2n+1})\) be local coordinates of \(X\) defined near \(p\) so that \(T = \frac{\partial}{\partial x_{2n+1}}\) near \(p\). Then,
\[
\Psi_{-}(x, y, t, s) = s(-x_{2n+1} + y_{2n+1}) + O(|(x, y)|^2)
\]
and
\begin{equation}\label{e01197}
\Psi_{-}(x, e^{i\theta}\circ y, t, s) = s\bigl(-x_{2n+1} + y_{2n+1} - \theta + O(\theta^2)\bigr) + O(|(x, y)|^2 |\theta|).
\end{equation}
From \eqref{e01197}, it is not difficult to see that \(s = -m\), \(\theta = 0\) are non-degenerate critical points of the system
\[
\begin{cases}
\frac{\partial}{\partial \theta} \left( \Psi_{-}(p, e^{i\theta}\circ e^{i\frac{j}{\ell}2\pi}\circ p, t, s) - m\theta \right) = 0, \\
\frac{\partial}{\partial s} \left( \Psi_{-}(p, e^{i\theta}\circ e^{i\frac{j}{\ell}2\pi}\circ p, t, s) - m\theta \right) = 0.
\end{cases}
\]
Thus, after changing $s$ by $\frac{s}{t}$ in the integral \eqref{e01196}, we can apply complex stationary phase formula of Melin-Sjöstrand to carry out the integral \eqref{e01196} and get
\begin{equation}\label{e01198}
\Gamma_{j,m}(t,x,y) = e^{i\frac{1}{t}\Phi_j(x,y)}a_{j,m}(x,y,t)+O(t^{+\infty}),
\end{equation}
moreover, we have
\[
a_{j,m}(x,y,t)\sim \sum_{r=0}^{+\infty}t^{-n+r}a_{j,m,r}(x,y)\quad \text{as} \quad t\rightarrow 0^+,
\]
and $\Phi_j(x,x)\approx \text{dist} \bigl(x, X^{j/\ell}_{\text{sing}} \bigr)^2$,
where $X^{j/\ell}_{\text{sing}}=\{x\in X: e^{i\frac{j}{\ell}2\pi}\circ x=x\}$.
From \eqref{e01196}, by using stationary phase formula to integrate over $X$, we deduce that $\int_X e^{-t\Box_{b, m}}(x,x)dv_X(x)$ admits a small-time asymptotic expansion. The proof is completed by summarizing the above arguments.
\end{proof}


    \section{Analytic torsion and Quillen metric on CR manifolds with $S^1$-action}\label{s-220404272026}
	
    
    In this section we assume that \(X\) admits a locally free transversal CR $S^1$-action $e^{i\theta}$. We will define the $S^1$-equivariant analytic torsion and Quillen metric for the rigid CR line bundle $L^k$ over the CR manifold $X$ with a transversal CR $S^1$-action.

	\subsection{Definitions of the $S^1$-equivariant CR analytic torsion and Quillen metric}\label{ss-220404272026}
	

	Let 
	$$
	\Pi_{L^k, \le k \delta} : L^2(X, T^{*0,\bullet}X \otimes L^k) \to \Ker \Box_{b, k, \le k\delta}
	$$  
	be the orthogonal projection and let 
	\[
    \begin{split}
	& \Pi_{L^k, \le k \delta}^\perp : L^2(X, T^{*0,\bullet}X \otimes L^k) \to (\Ker\Box_{b, k, \le k\delta})^\perp, \\
    & (\Pi_{L^k, \le k \delta}^{(q)})^\perp : L^2(X, T^{*0,\bullet}X \otimes L^k) \to (\Ker\Box_{b, k, \le k\delta}^{(q)})^\perp, \quad q = 0, 1, \cdots, n,
	\end{split}
    \]  
	be the orthogonal projection, where 
	\[
    \begin{split}
   & (\Ker\Box_{b, k, \le k\delta})^\perp=\set{u\in L^2(X, T^{*0,\bullet}X \otimes L^k);\, (\,u\,|\,v\,)_E=0,\  \ \forall v\in\Ker \Box_{b, k, \le k\delta}}, \\
   & (\Ker\Box_{b, k, \le k\delta}^{(q)})^\perp=\set{u\in L^2(X, T^{*0,\bullet}X \otimes L^k);\, (\,u\,|\,v\,)_E=0,\  \ \forall v\in\Ker \Box_{b, k, \le k\delta}^{(q)}}, \quad q =0, 1, \cdots, n.
    \end{split}
    \] 
	
	By Theorem~\ref{t05042026}, we have the following asymptotic expansion: 
	\begin{equation}\label{E:112120241802r1}
		\operatorname{STr}  \lbrack N e^{-t \Box_{b, k, \le k\delta}} \Pi_{L^k, \le k \delta}^\perp \rbrack\sim\sum^{+\infty}_{j=0} \widetilde{B}_{j, k, \le k\delta} t^{-n-1+j}\ \ \mbox{as $t\To0^+$}, 
	\end{equation}
	where $B_j \in \Complex$ is independent of $t$, $j=0,1,2,\ldots$. 

    The proof of the following lemma is similar to Lemma~\ref{l-gue160420}.
    
	\begin{lemma}\label{l-gue160420r3}
		Fix $q=0,1,\ldots,n$. For every $\delta>0$, there exist $c>0$, $C>0$ such that 
		\begin{equation}\label{e-gue160420Ia}
			\abs{\operatorname{Tr}^{(q)}  \lbrack e^{-t \Box_{b, k, k\delta}^{(q)}} (\Pi_{L^k, \le k \delta}^{(q)})^\perp \rbrack}\leq Ce^{-ct},\ \ \forall t\geq\delta. 
		\end{equation}
	\end{lemma}
	
	
	From \eqref{E:112120241802r1} and Lemma~\ref{l-gue160420r3}, we see that $\operatorname{STr}  \lbrack N e^{-t \Box_{b, k, \le k\delta}} (\Pi_{L^k, \le k \delta})^\perp \rbrack$ satisfies \eqref{e-gue160420g} and \eqref{e-gue160420I}. By Definition \ref{d-gue160313}, for $\operatorname{Re}(z)>n$, we can define 
	\begin{equation}\label{E:5.5.12r5}
		\theta_{b, L^k, \le k \delta}(z)  = - M \left\lbrack \operatorname{STr}  \lbrack N e^{-t \Box_{b, k, \le k\delta}} (\Pi_{L^k, \le k \delta})^\perp  \rbrack \right\rbrack  =   - \operatorname{STr} \left\lbrack N ({\Box}_{b, k, \le k \delta})^{-z} (\Pi_{L^k, \le k \delta})^\perp \right\rbrack.  
	\end{equation}
    
	By Theorem~\ref{t-gue160313}, we have the following
	\begin{lemma}\label{l-mero04262026}
		$\theta_{b, L^k, \le k \delta}(z)$ extends to a meromorphic function on $\mathbb{C}$ with poles contained in the set  
		\[\set{\ell- j;\, \ell,j\in\mathbb Z},\]
		its possible poles are simple, and $\theta_{b, L^k, \le k\delta}(z)$ is holomorphic at $0$. Moreover,
		\begin{equation}\label{e-gue191804262026}
			\begin{split}
				& \theta_{b, L^k, \le k\delta}'(0)  =  -\int_0^1 \left\{  \operatorname{STr} \Big[  Ne^{-t\Box_{b, k, \le k \delta}}(\Pi_{L^k, \le k \delta})^\perp\Big]-\sum^{n+1}_{j=0} B_{j, k, \le k\delta} t^{-n-1+j}\right\} \frac{dt}{t}  \\
				& -\int_1^\infty \operatorname{STr} \Big[ Ne^{-t\Box_{b, k, \le k \delta}} (\Pi_{L^k, \le k \delta})^\perp \Big] \frac{dt}{t} - \sum^{n}_{j=0} \frac{B_{j, k, \le k\delta}}{j-n-1} + \Gamma'(1)B_{n+1, k, \le k\delta}. 
			\end{split}
		\end{equation}
	\end{lemma}

Hence by \eqref{E:112120241802r1} and Lemma \ref{l-gue160420r3}, we can finally define the following
	\begin{definition}\label{d-gue160502wr6}
		We define $\exp ( -\frac{1}{2} \theta_{b, L^k, \le k \delta}'(0) )$ to be the $\ddbar_{b}$-torsion for the CR line bundle $L^k$ over the CR manifold $X$ with $S^1$-action. 
	\end{definition}
	
	Denote by 
$$
H^\bullet_{b, \le k \delta}(X, L^k) = \oplus_{q=0}^n H^q_{b,\le k \delta}(X,L^k).
$$ 
For a finite dimensional vector space $V$, we set 
$$
\det V := \wedge^{\text{max}}V.
$$
We then denote by 
$$
(\det V)^{-1}:= (\det V)^*,
$$ the dual line of $\det V$.
Then
\[
\det H^\bullet_{b,\le k \delta}(X,L^k) = \otimes_{q=0}^n \left( \det H^q_{b,\le k \delta}(X,L^k)  \right)^{(-1)^q}
\]
is the determinant line of the cohomology $H^\bullet_{b,\le k \delta}(X,L^k)$. We define 
\begin{equation}\label{E:5.5.14}
\lambda_{b,\le k \delta}(L^k) = \left( \det H^\bullet_{b,\le k \delta}(X,L^k)  \right)^{-1}. \nonumber
\end{equation}
The rigid Hermitian metrics $\langle \, \cdot \, |\, \cdot \, \rangle$ and $\langle \, \cdot \, |\, \cdot \, \rangle_{L^k}$ on $\mathbb{C}TX$ and $L^k$, respectively, induce a canonical $L^2$-metric $h^{H^\bullet_{b,\le k \delta}(X,L^k)}$ on $H^\bullet_{b,\le k \delta}(X,L^k)$. Let $|\cdot|_{\lambda_{b,\le k \delta}(L^k)}$ be the $L^2$-metric on $\lambda_{b,\le k \delta}(L^k)$ induced by $h^{H^\bullet_{b,\le k \delta}(X,L^k)}$. 

Now we can define the Quillen metric on $\det H^\bullet_{b,\le k \delta}(X,L^k)$.

\begin{definition}\label{D:5.5.5}
Fix $k \in \mathbb{N}$. The Quillen metric $\| \cdot  \|_{\lambda_{b, \le k \delta}(L^k)}$ on $\det H^\bullet_{b,\le k \delta}(X,L^k)$ is defined as
\begin{equation}\label{e204902042026}
\| \cdot  \|_{\lambda_{b, \le k \delta}(L^k)} \, := \,  |\cdot|_{\lambda_{b,\le k \delta}(L^k)} \cdot \exp ( -\frac{1}{2} \theta_{b,L^k,\le k \delta}'(0)  ).
\end{equation}
\end{definition}


\subsection{Variation of the $S^1$-equivariant Quillen metrics}
In this subsection, we study the dependence of the $S^1$-equivariant Quillen metrics on a compact CR manifold with transversal CR $S^1$-action on a change of the Hermitian metrics on $TX$ and $L^k$. The main result is Theorem \ref{t205302042026}.

Denote by $\mu_k: L^k \to (L^*)^k$ the induced conjugate linear bundle isomorphism from the vector bundle $L$ to its dual vector bundle $(L^*)^k$. We denote by 
$$
\Box_{b,k,\le k\delta, s}:=\Box_{b, s}|_{\Omega^{0,\bullet}_{\le k \delta}(X,L^k)}.
$$ 
Let $\| \cdot  \|_{\lambda_{b, \le k \delta}(L^k), s}$ be the corresponding Quillen metrics  on $\det H^\bullet_{b,\le k \delta}(X,L^k)$. Recall that $Q_{b,s}$ is defined in \eqref{E:qbs}. 

We have 
\begin{theorem}\label{t205302042026}
 As $t \to 0^+$, for any $\ell \in \mathbb{N}$, there is an asymptotic expansion
\begin{equation}\label{e205202042026r9}
\operatorname{STr} \left\lbrack Q_{b,s} \exp \left( -t\Box_{b,k,\le k\delta, s} \right) \right\rbrack = \sum_{j=0}^{\ell+n} M_{j,\le k \delta, s} t^{-n-1+j} +O(t^{\ell+1}), 
\end{equation}
where
\begin{equation}\label{e205102042026}
M_{n+1,\le k\delta, s} = \frac{\partial}{\partial s} \log \left(  \| \cdot  \|^{2}_{\lambda_{b, \le k \delta}(L^k), s} \right).
\end{equation}
\end{theorem}
\begin{proof}
By Theorem~\ref{t05042026}, we get \eqref{e205202042026r9}. 

By an argument similar to that used in the proof of Theorem \ref{t303262026} (cf. also \cite[Theorem 5.5.6]{MM}), we obtain the following analogue of \eqref{e1137020242026},
\begin{equation}\label{e1137020242026r21}
\frac{\partial}{\partial s}\operatorname{STr}\left[Ne^{-t\Box_{b,s}} \right] 
 = -t \frac{\partial}{\partial t}\operatorname{STr}\left[Q_{b,s}  e^{-t\Box_{b,s}}\right].
\end{equation}

Let $\Pi_{L^k, \le k \delta, s}$ be the orthogonal projection operator from $\Omega^{0,\bullet}(X,L^k)$ on $\operatorname{Ker}\left( D_{b,s}|_{\Omega^{0,\bullet}_{\le k \delta}(X,L^k)} \right)$ for the Hermitian product $(\,\cdot\, | \,\cdot\,)_{L^k,s}$ on $\Omega^{0,\bullet}(X,L^k)$. Note that $\operatorname{STr}\left[Q_{b,s} e^{-t\Box_{b,k,\le k \delta,s}}\Pi_{L^k, \le k\delta,s}^\perp \right]$ decays exponentially when $t \to +\infty$. From \eqref{e1137020242026r21}, we obtain for $\operatorname{Re}(z)$ large enough,
\begin{equation}\label{e153202042026r22}
\begin{split}
\frac{\partial}{\partial s}\theta_{b,k,\le k\delta,s}(z) & = \frac{1}{\Gamma(z)}\int_0^\infty \frac{\partial}{\partial s} \operatorname{STr}\left[N e^{-t\Box_{b,k,\le k\delta,s}} \Pi_{L^k,\le k\delta,s}^\perp \right] t^{z-1}dt \\
& = \frac{1}{\Gamma(z)}\int_0^\infty t^z \frac{\partial}{\partial t} \operatorname{STr}\left[Q_{b,s} e^{t\Box_{b,k,\le k\delta,s}} \Pi_{L^k,\le k\delta,s}^\perp \right]  dt \\
& = \frac{-z}{\Gamma(z)}\int_0^\infty t^{z-1}  \operatorname{STr}\left[Q_{b,s} e^{t\Box_{b,k,\le k\delta,s}} \Pi_{L^k,\le k\delta,s}^\perp \right]  dt.
\end{split}
\end{equation}

Using \eqref{e-gue160428w} and \eqref{e153202042026r22}, we get
\begin{equation}\label{e160002042026}
\frac{\partial}{\partial s} \left(\frac{\partial}{\partial z}\theta_{b,k,\le k\delta,s} \right)(0) = -\widetilde{M}_{n+1,\le k\delta, s} + \operatorname{STr}\left[Q_{b,s} \Pi_{L^k,\le k\delta,s} \right].
\end{equation}

The line bundle $\lambda_{b, \le k\delta, s}(L^k) = \otimes_{q=0}^n \left(\det \operatorname{Ker} \left(\Box_{b, k, \le k\delta,s}|_{\Omega^{0,q}(X,L^k)} \right)\right)^{(-1)^{q+1}}$.
Under this identification, the canonical isomorphism of the line bundle $\phi_s :\lambda_{b,\le k \delta, 0}(L^k) \to \lambda_{b,\le k\delta,s}(L^k)$ is defined by $\phi_s(\sigma) = \Pi_{L^k, \le k\delta, s}(\sigma)$ for $\sigma \in \lambda_{b,\le k \delta, 0}(L^k)$.

If $\sigma, \sigma' \in \operatorname{Ker} \Box_{b,k, \le k\delta, s}$, we have by definition
\begin{equation}\label{e200202042026}
\left( \Pi_{L^k, \le k\delta, s}(\sigma)\,|\, \Pi_{L^k, \le k\delta, s}(\sigma') \right)_{L^k, s} = \int_X \left\langle \Pi_{L^k, \le k\delta, s}(\sigma) \wedge (\ast_{b,s} \otimes \mu_s) \Pi_{L^k, \le k\delta, s}(\sigma')  \right\rangle_{L^k}dv_X.
\end{equation}
From $\Pi_{L^k, \le k\delta, s}^2 = \Pi_{L^k, \le k\delta, s}$, we get $\left(\frac{\partial}{\partial s}\Pi_{L^k, \le k\delta, s} \right)\Pi_{L^k, \le k\delta, s} + \Pi_{L^k, \le k\delta, s} \left(\frac{\partial}{\partial s}\Pi_{L^k, \le k\delta, s} \right) = \frac{\partial}{\partial s}\Pi_{L^k, \le k\delta, s}$, thus the operator $\frac{\partial}{\partial s}\Pi_{L^k, \le k\delta, s}$ sends $\operatorname{Ker}\Box_{b,s}$ to its orthogonal complement $(\,\cdot\,|\,\cdot\,)_{L^k,s}$. Therefore, from \eqref{e200202042026}, we get
\begin{equation}\label{e201302042026}
\begin{split}
\frac{\partial}{\partial s}\left( \Pi_{L^k, \le k\delta, s}(\sigma)\,|\, \Pi_{L^k, \le k\delta, s}(\sigma') \right)_{L^k, s} & = \int_X \left\langle \Pi_{L^k, \le k\delta, s}(\sigma) \wedge \frac{\partial}{\partial s}\left(\ast_{b,s} \otimes \mu_s\right) \Pi_{L^k, \le k\delta, s}(\sigma')  \right\rangle_{L^k}dv_X \\
& = -\left( \Pi_{L^k, \le k\delta, s}(\sigma)\,|\, Q_{b,s}\Pi_{L^k, \le k\delta, s}(\sigma') \right)_{L^k, s}. 
\end{split}
\end{equation}
Thus from \eqref{E:5.5.14} and \eqref{e201302042026}, we get
\begin{equation}\label{e201802042026}
\frac{\partial}{\partial s}\log  \left(\left|\Pi_{L^k, \le k\delta, s}(\sigma)\right|^2_{\lambda_{b,\le k\delta,s}(L^k)} \right) = \operatorname{STr}\left[Q_{b,s}\Pi_{L^k, \le k\delta, s} \right].
\end{equation}
From \eqref{e204902042026}, \eqref{e160002042026} and \eqref{e201802042026}, we get \eqref{e205102042026}.
\end{proof}


	\section{\texorpdfstring{The asymptotics of the analytic torsion on CR manifolds with $S^1$-action}{}}\label{sfinal}
	
     We use the same notations and assumptions as in Section~\ref{s-220404272026}. We will study asymptotic behavior of $\theta'_{b,L^k, \le k \delta}$, when $k\To+\infty$. The goal of this section is to establish Bismut-Vasserot type asymptotics of the analytic torsion for $\Box_{b,k,\le k\delta}$ as $k\To+\infty$.

	\subsection{The asymptotics of $e^{-\frac{t}{k} \Box^{(q)}_{b, k, \le k \delta}}$ as $k \to \infty$}

	In this subsection we recall the asymptotics of $e^{-\frac{t}{k} \Box^{(q)}_{b, k, \le k \delta}}(x, x) $ as $k \to \infty$ for all $x \in X$. We assume that $X$ admits a transversal CR $S^1$-action. Before proceeding to do so, we recall the relationship between $\dot{\mathcal{R}}^\phi$, the curvature of $L$, and the Levi form.

	 From \cite[Proposition 4.2]{HM12}, it is easy to see that for every $x \in X$, the map
	\[
	\int_{-\delta}^\delta \frac{\det(\dot{\mathcal{R}}_x^\phi - 2 \eta \dot{\mathcal{L}}_x )}{\det \big( 1 - e^{-t ( \dot{\mathcal{R}}_x^\phi - 2 \eta \dot{\mathcal{L}}_x )}  \big)} e^{-t\omega_x^\eta} d\eta : T^{\ast 0,q}_xX \to T^{\ast 0, q}_xX
	\]
	is independent of the choice of local weight $\phi$ and hence globally defined. 
	
	We recall the following theorem, \cite[Theorem 1.3]{HZ23}.
	\begin{theorem}\label{T:HZ23thm1.3}
		Let $X$ be an orientable (not necessarily compact) CR manifold of dimension $2n+1$, $n \ge 1$, with a transversal CR $S^1$-action. Suppose that $X$ admits an $S^1$-invariant complete Hermitian metric $\langle \cdot \mid \cdot \rangle$ on $\mathbb{C}TX$ so that we have the orthogonal decomposition 
		\[
		\mathbb{C}TX = T^{1,0}X \oplus T^{0,1}X \oplus \{ \lambda T : \lambda \in \mathbb{C} \}
		\]
		and $\mid T \mid^2 = \langle T \mid T \rangle = 1$, with $T \in \mathscr{C}^\infty(X, TX)$ is the infinitesimal generator of the $S^1$-action. Let $(L^k, h^{L^k})$ be the $k$-th tensor power of a rigid CR complex line bundle $(L, h^L)$ over $X$, where $h^L$ is an $S^1$-invariant Hermitian metric on $L$. Fix any $q \in \{0, 1, \cdots, n \}$. With the notations used above, let $I \subset \mathbb{R}^+$ be a compact interval and let $K \Subset X$ be a compact set. Then, there is a constant $C > 0$ independent of $k$ such that
		\[
		\Big| e^{-t \Box^{(q)}_{b, k, \le k \delta}}(x, x) \Big|_{\mathscr{L}(T_x^{\ast0,q}X, T_x^{\ast0, q}X)} \le Ck^{n+1}, \quad \text{for all $x \in K$ and $t \in I$.}
		\]
		Moreover, for every $x \in X$, we have
		\begin{equation}\label{e-gue070304262026I}
			\lim_{k \to \infty} k^{-(n+1)} e^{-\frac{t}{k} \Box^{(q)}_{b, k, \le k \delta}}(x, x)  = \frac{1}{(2\pi)^{n+1}} \int_{-\delta}^\delta \frac{\det(\dot{\mathcal{R}}_x^\phi - 2 \eta \dot{\mathcal{L}}_x )}{\det \big( 1 - e^{-t ( \dot{\mathcal{R}}_x^\phi - 2 \eta \dot{\mathcal{L}}_x )}  \big)} e^{-t\omega_x^\eta} d\eta.
		\end{equation}
	\end{theorem}

It is straightforward to check that for every $q=0, 1, \cdots, n$,
\begin{equation}\label{e-gue164904272026}
\begin{split}
& \lim_{t \to \infty}\dfrac{\det(\dot{\mathcal R}^\phi_x-2\eta\dot{\mathcal L}_x)}{\det\big(1-e^{-t(\dot{\mathcal R}^\phi_x-2\eta\dot{\mathcal L}_x)}\big)}{\rm Tr\,}^{(q)}e^{-t\omega_x^\eta} \\
= &   (-1)^q \det(\dot{\mathcal{R}}_x^\phi - 2 \eta \dot{\mathcal{L}}_x)1_{\mathbb R_x(q)}(\eta)        \\
= &   \abs{\det(\dot{\mathcal{R}}_x^\phi - 2 \eta \dot{\mathcal{L}}_x)}1_{\mathbb R_x(q)}(\eta),
\end{split}
\end{equation}
where $1_{\mathbb R_x(q)}(\eta)$ is the characteristic function of $\mathbb R_x(q)$.

By \eqref{e-gue070304262026I} and \eqref{e-gue164904272026}, we obtain
    
	\begin{theorem}\label{t-gue163904272026}
We assume that the Levi form is non-degenerate of constant signature $(n_-,n_+)$. With the notations and assumptions used in Theorem \ref{T:HZ23thm1.3}, let $I\subset\mathbb R_+$ be a bounded open interval. Let $q\in\set{0,1,\ldots,n}$. Then, 
\begin{equation}\label{e-gue164204272026}
\begin{split}
&\lim_{k\To+\infty}k^{-(n+1)}\operatorname{Tr}^{(q)}(e^{-\frac{t}{k}\Box^{(q)}_{b,L^k, \le k \delta}}(I-\Pi^{(q)}_{L^k,\le k\delta})(x,x))=\\
&\frac{1}{(2\pi)^{n+1}}\int_{-\delta}^\delta\Bigr(\dfrac{\det(\dot{\mathcal R}^\phi_x-2\eta\dot{\mathcal L}_x)}{\det\big(1-e^{-t(\dot{\mathcal R}^\phi_x-2\eta\dot{\mathcal L}_x)}\big)}{\rm Tr\,}^{(q)}e^{-t\omega_x^\eta}-\abs{\det(\dot{\mathcal{R}}_x^\phi - 2 \eta \dot{\mathcal{L}}_x)}1_{\mathbb R_x(q)}(\eta)\Bigr)d\eta
\end{split}
\end{equation}
in $\mathcal{C}^0(I\times X)$ topology.
\end{theorem} 

	
	


	\subsection{Asymptotics of the $S^1$-equivariant CR analytic torsion}\label{ss-213504262026}
We will study asymptotic behavior of $\theta'_{b,L^k, \le k \delta}$, when $k\To+\infty$. The goal of this subsection is to establish Bismut-Vasserot type asymptotics of the analytic torsion for $\Box_{b,k,\le k\delta}$ as $k\To+\infty$.

By Lemma \ref{l-gue160428} and Theorem \ref{t-gue163904272026}, we get

        \begin{theorem}\label{t-gue205004272026}
With the notations and assumptions used before, we have for all $t>0$, 
\begin{equation}\label{e-gue205004272026}
\begin{split}
&\lim_{k\To+\infty}k^{-(n+1)}{\rm STr\,}\lbrack Ne^{-\frac{t}{k}\Box_{b,L^k, \le k\delta}}(I-\Pi_{L^k, \le k\delta})\rbrack\\
&=\int_X\int_{-\delta}^\delta \Bigr((\dot{\mathcal{R}}^\phi_x-2\eta \dot{\mathcal{L}}_x)_t-(2\pi)^{-n-1}\sum^n_{q=0}q(-1)^q\abs{\det(\dot{\mathcal{R}}_x^\phi - 2 \eta \dot{\mathcal{L}}_x)}1_{\mathbb R_x(q)}(\eta)\Bigr)d\eta dv_X.
\end{split}
\end{equation}
    \end{theorem}

We need the following

\begin{theorem}\label{t-gue260426yyd}
With the notations and assumptions used before. Let $q\in\set{0,1,\ldots,n}$. We can find $k$-dependent smooth functions $b^{(q)}_{j,k,\le k \delta}(x)\in\mathcal{C}^\infty(X)$, $j=0,1,\ldots$, with
\begin{equation}\label{e-gue070704262026I}
k^{-n-1}b^{(q)}_{j,k,\le k\delta}(x)=\tilde{b}^{(q)}_j(x)+k^{-\frac{1}{2}}r^{(q)}_{j,k, \le k\delta}(x),\ \ j=0,1,\ldots,
\end{equation}
and $\gamma(t,x)\in\mathcal{C}^\infty(\mathbb R_+\times X)$, $\gamma(t,x)_{\mathcal{C}^0(X)}\leq C_Nt^{-n-1}k^{-N}$ on $\mathbb R_+\times X$, for every $N\in\mathbb N$, where $C_N>0$ is a constant independent of $k$ and $t$, 
where $b^{(q)}_j(x)\in\mathcal{C}^\infty(X)$, $j=0,1,\ldots$, $\norm{r^{(q)}_{j,k, \le k\delta}}_{\mathcal{C}^0(X)}\leq C_j$, $C_j>0$ is a constant independent of $k$, $j=0,1,\ldots$, such that for every $N\in\mathbb N$, we have 
\begin{equation}\label{e-gue070704262026II}
\gamma(t,x)+\operatorname{Tr}^{(q)}(e^{-\frac{t}{k}\Box^{(q)}_{b,k, \le k\delta}}(x,x))=\sum^N_{j=0}t^{-n-1+j}b^{(q)}_{j,k, \le k\delta}(x)+\delta^{(q)}_{N,k,\le k\delta}(t,x),
\end{equation}
where $\delta^{(q)}_{N,k,\le k\delta}(t,x)\in\mathcal{C}^\infty(\ol{\mathbb R}_+\times X)$ is a $k$-dependent smooth function and for every $m\in\mathbb N_0$, there is a constant $C_{m}>0$ independent of $k$ such that 
\begin{equation}\label{e-gue070704262026III}
\sup\,\set{\abs{\pr^m_t\delta^{(q)}_{N,k,\le k\delta}(t,x)};\, x\in X}\leq t^{-n+N-m}k^{n+1}C_{m},
\end{equation}
for all $t\in(0,1]$ and all $k\gg1$.
\end{theorem}

\begin{proof}
    Let $(X, T^{1,0}X)$ be a compact CR manifolds with a transversal and CR $S^1$-action. Let $\delta>0$. Let
    \[
Q_{X,\le k\delta} : L^2_{(0,q)}(X, L^k) \to L^2_{(0,q), \le k\delta}(X, L^k)
    \]
    be the orthogonal projection with respect to $(\,\cdot\,|\,\cdot\,)_{L^k}$. It was proved in Subsection \ref{s-gue250401yyda} that for every $N \in \mathbb{N}$,
    \begin{equation}\label{e-gue2604232159}
    \begin{split}
    e^{-\frac{t}{k}\Box_{b, k}^{(q)}}(x, y) = &  \int_0^{+\infty} e^{i\frac{k}{t}\phi_{k,-}(x, y,s)} a_{k,-}(\frac{k}{t}, x, y, s)ds
    \\ & +   \int_0^{+\infty} e^{i\frac{k}{t}\phi_{k,+}(x, y,s)} a_{k,+}(\frac{k}{t}, x, y, s)ds + R_{N}(\frac{t}{k}, x ,y),
    \end{split}
    \end{equation}
    where 
    \begin{equation}\label{e-gue2604232200}
    \begin{split}
    & R_{N}(\frac{t}{k}, x ,y) \in C^\infty(\bar{\mathbb{R}}_+ \times X \times X, (T^{*0,q}X \otimes L^k) \boxtimes (T^{*0,q}X \otimes L^k)^*), \\
    &  R_{N}(\frac{t}{k}, x ,y) = O(t^N) \quad \text{as} \quad t \to 0^+, \\
    & a_{k,\pm}(t, x, y, s) \sim \sum_{j=0}^{+\infty}t^{-n-1+j}a_{k, \pm}^j(x, y, s)\quad \text{in} \quad S^{n,-n-1}_\varepsilon(\mathbb{R}_+ \times X \times X \times \bar{\mathbb{R}}_+). 
    \end{split}    
    \end{equation}
    Let $p \in X$. Let $x=(x_1,\cdots,x_{2n+1})$ be BRT coordinates of $X$ defined near $p$ such that $p \leftrightarrow 0$ (see discussion before Example 2.11 of \cite{HHL22} for the meaning of BRT charts). Let $x' =(x_1, x_2, \cdots, x_{2n})$. Then, it was proved in Subsection \ref{s-gue250401yyda} that $a^j_{k,\pm}(x,y,s) \equiv a^j_{k,\pm}(x',y',s)$, $j=0,1,\cdots$ can be taken to be independent of $x_{2n+1}, y_{2n+1}$, and
    \begin{equation}\label{e-gue2604232201}
    k^{-(n+1)}a^j_{k,\pm}\big(\frac{x'}{\sqrt{k}}, \frac{y'}{\sqrt{k}}, s\big) \big) = a^j_{\pm}(x', y', s)+O(k^{-\frac{1}{2}}) 
    \quad \text{on} \quad \{|x'| < M\} \times \{|y'| < M\}
    \end{equation}
    for all $j=0,1, \cdots$, for all $k\gg 1$, where $M>0$ is a constant and
    \begin{equation}\label{e-gue2604232202}
  \big|  k^{-(n+1)}R_N\big(\frac{t}{k}, \big(\frac{x'}{\sqrt{k}},\frac{x_{2n+1}}{k}), (\frac{y'}{\sqrt{k}},\frac{y_{2n+1}}{k}\big) \big) \big| \le Ct^N
    \quad
    \text{on} \quad \{|x'| < M\} \times \{|y'| < M\}, 
    \end{equation}
    for all $k\gg 1$, $t>0$, where $C>0$ is a constant independent of $k$ and $t$. Moreover, from the construction of the approximated heat kernel, we see that if $T=\frac{\partial}{\partial x_{2n+1}}$, then
    \begin{equation}\label{e-gue2604232203}
    \left\{ \begin{split} & \phi_{-,k}(x,y,s) = s(x_{2n+1}-y_{2n+1}) + \widehat{\phi}_{-,k}(x',y',s) \\ & \phi_{+,k}(x,y,s) = s(-x_{2n+1}+y_{2n+1}) + \widehat{\phi}_{+,k}(x',y',s)\end{split}    \right.,
    \end{equation}
     where 
     \begin{equation}\label{e-gue2604232204}
     \widehat{\phi}_{\mp,k}\big(\frac{x'}{\sqrt{k}},\frac{y'}{\sqrt{k}},s\big) \to \widehat{\phi}_\mp(x',y',s),\quad \text{in $C^\infty$-topology on $\{|x'|<M\} \times \{|y'|<M\}$},
     \end{equation}
      where $s(x_{2n+1}-y_{2n+1}) + \widehat{\phi}_{-,k}(x',y',s)$, $s(-x_{2n+1}+y_{2n+1}) + \widehat{\phi}_{+,k}(x',y',s)$ are the phase functions appearing in the heat kernel of the model case. Let 
    \begin{equation}\label{e-gue2604232205}
    \begin{split} \widehat{Q}_{\le k\delta} : & \Omega^{0,q}_c(D, L^k) \to \Omega^{0,q}(D, L^k) \\
   & u \mapsto \frac{k}{2\pi}\int e^{ik\langle x_{2n+1}-y_{2n+1}, \eta_{2n+1}\rangle} 1_{[-\delta, \delta]}(\eta_{2n+1}) u(x', y_{2n+1})  dy_{2n+1} d\eta_{2n+1}
    \end{split}
    \end{equation}
    It was proved in \cite[Lemma 6.5 $\&$ Lemma 6.8]{HZ23} that
    \begin{equation}\label{e-gue2604232206}
    \widehat{Q}_{\le k\delta} - Q_{X,\le k\delta} = O(k^{-\infty}) : H^s_{\operatorname{comp}}(D) \to H^s_{\operatorname{loc}}(D),
    \end{equation}
    for all $s\in\mathbb N$. From Fourier inversion formula, we have 
   \begin{equation}\label{e-gue2604232207}
   \begin{split}
   & \big( e^{-\frac{t}{k}\Box_{b, k}^{(q)}} \circ \widehat{Q}_{\le k\delta} \big)\big(\big(\frac{x'}{\sqrt{k}}, \frac{x_{2n+1}}{k}\big), \big(\frac{y'}{\sqrt{k}}, \frac{y_{2n+1}}{k}\big)\big) \\
   = & \frac{k}{2\pi} \int e^{i\frac{k}{t} \big( \widehat{\phi}_{-,k}\big(\frac{x'}{\sqrt{k}}, \frac{y'}{\sqrt{k}}\big)  + s \big( \frac{x_{2n+1}}{k}-\frac{u_{2n+1}}{k} \big) \big)+ ik \langle u_{2n+1} -y_{2n+1}, \eta_{2n+1}\rangle } \\
   & \qquad \qquad \times a_{k,-}\big(\frac{k}{t}, \frac{x'}{\sqrt{k}}, \frac{y'}{\sqrt{k}}, s\big) 1_{[-\delta, \delta]}(\eta_{2n+1})  d\eta_{2n+1}ds du_{2n+1} 
  \\
   & + \frac{k}{2\pi}\int e^{i\frac{k}{t} \big( \widehat{\phi}_{+,k}\big(\frac{x'}{\sqrt{k}}, \frac{y'}{\sqrt{k}}\big)  + s \big( -\frac{x_{2n+1}}{k}+\frac{u_{2n+1}}{k} \big) \big)+ ik \langle u_{2n+1} -y_{2n+1}, \eta_{2n+1}\rangle }\\
   & \qquad \qquad \times a_{k, +}\big(\frac{k}{t}, \frac{x'}{\sqrt{k}}, \frac{y'}{\sqrt{k}}, s\big) 1_{[-\delta, \delta]}(\eta_{2n+1})  d\eta_{2n+1}ds du_{2n+1} \\ 
   = &  \int_{-\delta}^\delta e^{i\frac{k}{t} \big( \widehat{\phi}_{-,k}\big(\frac{x'}{\sqrt{k}}, \frac{y'}{\sqrt{k}}\big)  + s \big( \frac{x_{2n+1}}{k}-\frac{y_{2n+1}}{k} \big)  \big)} a_{k,-}\big(\frac{k}{t}, \frac{x'}{\sqrt{k}}, \frac{y'}{\sqrt{k}}, s\big)  ds
  \\
   & + \int_{-\delta}^\delta e^{i\frac{k}{t} \big( \widehat{\phi}_{+,k}\big(\frac{x'}{\sqrt{k}}, \frac{y'}{\sqrt{k}}\big)  + s \big( -\frac{x_{2n+1}}{k}+\frac{y_{2n+1}}{k} \big)  \big)} a_{k, +}\big(\frac{k}{t}, \frac{x'}{\sqrt{k}}, \frac{y'}{\sqrt{k}}, s\big)   ds.
   \end{split}
   \end{equation}
   From \eqref{e-gue2604232159}, \eqref{e-gue2604232200}, \eqref{e-gue2604232201}, \eqref{e-gue2604232202}, \eqref{e-gue2604232203}, \eqref{e-gue2604232204}, \eqref{e-gue2604232205}, \eqref{e-gue2604232206}, \eqref{e-gue2604232207} and notice that 
    \[{\rm Tr\,}a^j_{k, +}(0,0)+{\rm Tr\,}a^j_{k,-}(0,0)=b^{(q)}_{j,k,\leq k\delta}(0),\ \ j=0,1,\ldots,\] we obtain the result.
\end{proof}

From Theorem~\ref{t-gue260426yyd}, we have the following asymptotic expansion, for $t \to 0$,
	\begin{equation}\label{e-gue233704262026}
	\begin{split}
		&\int_{-\delta}^\delta\Bigr((\dot{\mathcal{R}}^\phi_x-2\eta \dot{\mathcal{L}}_x)_t-(2\pi)^{-n-1}\sum^n_{q=0}q(-1)^q\abs{\det(\dot{\mathcal{R}}_x^\phi - 2 \eta \dot{\mathcal{L}}_x)}1_{\mathbb R_x(q)}(\eta)\Bigr)d\eta\\
        &\sim\sum^{+\infty}_{j=0}\widetilde{A}_{j}(x) t^{-n-1+j}\ \ \mbox{in $\tilde{S}^{-n-1}(\mathbb{R}_+\times X)$},
        \end{split}
	\end{equation}
    where $\widetilde{A}_{j}(x)\in\mathcal{C}^\infty(X)$, $j=0,1,\ldots$.

   In the proof of Theorem~\ref{t-que042520262248} below, we need to determine $\widetilde{A}_{n+1}(x)$. By \eqref{E:110320242242} and \eqref{e-gue250531yydk}, we have 
\begin{equation}\label{e-gue083604272026IV}
\widetilde{A}_{n+1}(x)=\int_{-\delta}^\delta \frac{1}{(2\pi)^{n+1}}\Bigr(\frac{n}{2}{\rm det\,}(\dot{\mathcal{R}}^\phi_x - 2\eta \dot{\mathcal{L}_x})-\sum^n_{q=0}(-1)^qq\abs{{\rm det\,}(\dot{\mathcal{R}}^\phi_x - 2\eta \dot{\mathcal{L}_x})}1_{\mathbb R_x(q)}(\eta)\Bigr)d\eta.
\end{equation}

\begin{theorem}\label{t-gue070804262026}
		With the notations and assumptions used before, let 
        \begin{equation}\label{e-gue144704272026I}
        B_{j,k, \le k\delta}:=k^{-(n+1)}\sum^n_{q=0}q(-1)^q\int b^{(q)}_{j,k, \le k\delta}(x)dv_X(x),\ \ j=0,1,\ldots, 
        \end{equation}
        where  $b^{(q)}_{j,k, \le k\delta}(x)\in\mathcal{C}^\infty(X)$, $j=0,1,\ldots$, are as in \eqref{e-gue070704262026II}. Then, 
        \begin{equation}\label{e-144704272026II}
    B_{j,k, \le k\delta}=\int_X \widetilde{A}_{j} (x)dv_X+O(k^{-\frac{1}{2}}),\ \ j=0,1,\ldots,
        \end{equation}
where $\widetilde{A}_{j}(x)$ is as in \eqref{e-gue233704262026}, $j=0,1,\ldots$, and for every $N>0$, we have for all $t\in(0,1)$ and all $k\geq1$, 
		\begin{equation}\label{e-gue144704272026III}
			k^{-(n+1)}{\rm STr\,}\lbrack Ne^{-\frac{t}{k}\Box_{b,L^k, \le k\delta}}(I-\Pi_{L^k, \le k\delta})\rbrack=\sum^N_{j=0}t^{-n-1+j}B_{j,k, \le k\delta}+\varepsilon_{N,k, \le k\delta}
		\end{equation}
        where 
        \[\abs{\varepsilon_{N,k, \le k\delta}}\leq Ct^{-n+N},\] 
        for all $t\in(0,1)$ and all $k\geq1$, $C>0$ is a constant independent of $k$ and $t$.

Moreover, for all $t\in(0,1)$ and every $N\in\mathbb N$, we have 
		\begin{equation}\label{e-gue144704272026IV}
			\lim_{k\To\infty}k^{-(n+1)}{\rm STr\,}\lbrack Ne^{-\frac{t}{k}\Box_{b,L^k, \le k\delta}}(I-\Pi_{L^k, \le k\delta})\rbrack=\sum^N_{j=n-1}t^{-n-1+j}\int_X\widetilde{A}_{j}(x)dv_X(x)+\widetilde{\delta}_{N}(t),
		\end{equation}
		where $\widetilde{A}_{j}(x)$ is as in \eqref{e-gue233704262026}, $j=0,1,\ldots$, and  \[\abs{\widetilde{\delta}_{N}(t)}\leq\tilde Ct^{-n+N},\] 
        for all $t\in(0,1)$, $\tilde C>0$ is a constant independent of $k$ and $t$. 
	\end{theorem}
\begin{proof}
    From \eqref{e-gue070704262026I} and \eqref{e-gue070704262026II}, we immediately deduce \eqref{e-144704272026II} and \eqref{e-gue144704272026III}.
    
    From \eqref{e-gue070704262026I}, \eqref{e-gue070704262026II} and \eqref{e-gue070704262026III}, we get for every $N \in \mathbb{N}$,
\begin{equation}\label{e-gue183804272026}
\begin{split}
			\lim_{k\To\infty}&k^{-(n+1)}{\rm STr\,}\lbrack Ne^{-\frac{t}{k}\Box_{b,L^k, \le k \delta}}(I-\Pi_{L^k, \le k \delta})\\
            &=\sum^N_{j=0}t^{-n-1+j}\int_X\sum^n_{q=0}(-1)^qq\tilde{b}^{(q)}_{j}(x)dv_X(x)+\hat\delta_{N}(t),
            \end{split}
		\end{equation}
        for all $t\in(0,1]$, where $\tilde{b}^{(q)}_{j}(x)$ is as in \eqref{e-gue070704262026I}, $q=0,1,\ldots,n$, $j=0,1,\ldots$, and 
        \[\abs{\hat\delta_{N}(t)}\leq C_1t^{-n+N},\] 
        for all $t\in(0,1)$, $C_1>0$ is a constant independent of $t$.

By \eqref{e-gue205004272026} and \eqref{e-gue233704262026}, we have
\begin{equation}\label{e-gue205804272026}
\int_X\widetilde{A}_j(x)dv_X(x)=\int_X\sum^n_{q=0}(-1)^qq\tilde{b}^{(q)}_{j}(x)(x)dv_X(x),\ \ j=0,1,\ldots.
\end{equation}
From \eqref{e-gue183804272026} and \eqref{e-gue205804272026}, we get \eqref{e-gue144704272026IV}.
The theorem follows. 
    
\end{proof}

	\begin{theorem}\label{t-gue070904262026I}
		With the notations and assumptions used before, there exist $C, c, c'$ such that for any $q \ge 1$, $t \ge 1$, $k \in \mathbb{N}$, we have
		\begin{equation}\label{e-gue065504262026}
			k^{-(n+1)}\operatorname{Tr}^{(q)}[e^{-\frac{t}{k} \Box^{(q)}_{b, k, \le k \delta}}(I - \pi_{L^k, \le k \delta})] \le C \exp(-(c-c'/k)t).
		\end{equation}
	\end{theorem}
	\begin{proof}
		By Theorem~\ref{T:110520241057} for $t \ge 1$, $q\ge 1$, we have
		\begin{equation}\label{E:110520241106}
			\operatorname{Tr}^{(q)}[e^{-\frac{t}{k} \Box^{(q)}_{b, k, \le k \delta}}(I - \pi_{L^k, \le k \delta})] \le \operatorname{Tr}^{(q)}[e^{-\frac{1}{k} \Box^{(q)}_{b, k, \le k \delta}}(I - \pi_{L^k, \le k \delta})] \exp \left( - \frac{(t-1)}{k} (c_1 k-c_2) \right).
		\end{equation}
		By \eqref{e-gue070304262026I}, we know that $k^{-(n+1)}\operatorname{Tr}^{(q)}[e^{-\frac{1}{k} \Box^{(q)}_{b, k, \le k \delta}}(I - \pi_{L^k, \le k \delta})]$ has a finite limit as $k \to +\infty$. We obtain \eqref{e-gue065504262026} by \eqref{E:110520241106}.
	\end{proof}
	
	We have the following theorem
	
	\begin{theorem}\label{t-que042520262248}
		With the notations and assumptions used before, as $k \to+\infty$, we have
		\begin{equation}\label{e-gue074304262026}
			  \begin{split}
			&\theta_{b,L^k, \le k \delta}'(0)\\   
          &=(\log k)k^{n+1}\Bigr(\int_X\int_{-\delta}^\delta \frac{1}{(2\pi)^{n+1}}\bigr(\frac{n}{2}{\rm det\,}(\dot{\mathcal{R}}^\phi_x - 2\eta \dot{\mathcal{L}_x})\\
          &\quad-\sum^n_{q=0}(-1)^qq\abs{{\rm det\,}(\dot{\mathcal{R}}^\phi_x - 2\eta \dot{\mathcal{L}_x})}1_{\mathbb R_x(q)}(\eta)\bigr)d\eta dv_X(x)\Bigr)\\
&+k^{n+1}\Bigr(\frac{1}{2} \log (2\pi)(2\pi)^{-n-1} \int_X\int_{-\delta}^\delta \det \left(\dot{\mathcal{R}}^\phi_x - 2 \eta \dot{\mathcal{L}}_x \right)(2q-n)1_{\mathbb R_x(q)}(\eta)d\eta dv_X(x)\\
                &+\frac{1}{2}(2\pi)^{-n-1} \int_X\int_{-\delta}^\delta \det \left(\dot{\mathcal{R}}^\phi_x - 2 \eta \dot{\mathcal{L}}_x \right)\bigr(-\log(\abs{{\rm det\,}(\dot{\mathcal{R}}^\phi_x - 2 \eta \dot{\mathcal{L}}_x)_-})\\
                &\quad+\log(\abs{{\rm det\,}(\dot{\mathcal{R}}^\phi_x - 2 \eta \dot{\mathcal{L}}_x)_+})\bigr)1_{\mathbb R_x(q)}(\eta)d\eta dv_X(x)\Bigr)+o(k^{n+1}),
          \end{split}
		\end{equation}
	\end{theorem}
	
	\begin{proof}
	For $k\gg1$, set
		\begin{equation}\label{e-gue070904262026II}
			\widetilde{\theta}_{b,L^k, \le k \delta}(z) = -M \left\lbrack k^{-(n+1)} \operatorname{STr} \Big\lbrack Ne^{-\frac{t}{k}\Box_{b,L^k, \le k\delta}}(I-\Pi_{k, \le k\delta})\Big\rbrack \right\rbrack(z).
		\end{equation}
		Clearly 
		\begin{equation}\label{e-gue071004262026I}
			k^{-(n+1)} \theta_{b,L^k, \le k\delta}(z) = k^{-z} \widetilde{\theta}_{b,L^k, \le k\delta}(z).
		\end{equation}
		By  \eqref{e-gue233704262026}, Theorem \ref{t-gue070804262026}, Theorem~\ref{t-gue070904262026I} and \eqref{e-gue071004262026I}, we have
		\begin{equation}\label{e-gue071004262026III}
			\begin{split}
				& k^{-(n+1)} \theta'_{b,L^k, \le k\delta}(0)  = - \log(k) \widetilde{\theta}_{b,L^k, \le k\delta}(0) + \widetilde{\theta}'_{b,L^k, \le k\delta}(0),   \\
				&  \widetilde{\theta}_{b,L^k, \le k\delta}(0)=-\int_X\widetilde{A}_{n+1}(x)dv_X+O(k^{-\frac{1}{2}}).
			\end{split}
		\end{equation} 
		By \eqref{E:5.5.13}, Theorem \ref{t-gue070804262026}, Theorem~\ref{t-gue070904262026I} and Lebesgue's dominated convergence theorem, for any $t>0$,
		\begin{equation}\label{e-gue071104262026I}
			\begin{split}
				&\lim_{k \to \infty} \widetilde{\theta}'_{b,L^k, \le k\delta}(0) \\
                =&  -\int_0^1 \lim_{k \to \infty} \left\{  k^{-(n+1)} \operatorname{STr} \left\lbrack Ne^{-\frac{t}{k}\Box_{b,L^k, \le k\delta}}(I-\Pi_{L^k, \le k\delta})\right\rbrack - \sum^{n+1}_{j=0} B_{j,k, \le k\delta}t^{-n-1+j}\right\} \frac{dt}{t} \\
				& -\int_1^\infty \lim_{k \to \infty} \left\{  k^{-(n+1)} \operatorname{STr} \left\lbrack Ne^{-\frac{t}{k}\Box_{b,L^k, \le k\delta}}(I-\Pi_{L^k, \le k\delta})\right\rbrack \right\} \frac{dt}{t}  \\
				&-\lim_{k\to\infty}\sum^{n}_{j=0}\frac{B_{j,k, \le k\delta}}{j-n-1}+\Gamma'(1)\lim_{k\to\infty}B_{n+1,k,\le k\delta}.
			\end{split}
		\end{equation}

		For $z \in \mathbb{C}$ set
		\begin{equation}\label{e-gue212404272026}
        \begin{split}
		&	\widehat{\zeta}(z)  \\ & = -M
            \left\lbrack \int_X\int_{-\delta}^\delta\Bigr((\dot{\mathcal{R}}^\phi_x-2\eta \dot{\mathcal{L}}_x)_t -(2\pi)^{-n-1}\sum^n_{q=0}q(-1)^q\abs{\det(\dot{\mathcal{R}}_x^\phi - 2 \eta \dot{\mathcal{L}}_x)}1_{\mathbb R_x(q)}(\eta)\Bigr)d\eta dv_X\right\rbrack (z).
            \end{split}
		\end{equation}

From Theorem~\ref{t-gue205004272026}, Theorem \ref{t-gue070804262026}, Theorem~\ref{t-gue070904262026I} and \eqref{e-gue071104262026I}, we have 
		\begin{equation}\label{e-gue212504272026}
\lim_{k\to\infty}\widetilde{\theta}'_{b,L^k, \le k\delta}(0)\\
				= \widehat{\zeta}'(0).
		\end{equation}
		
Recall that $\zeta(z)$ denotes the Riemann zeta function and, for $\operatorname{Re}z>1$,
		\[
			\zeta(z) =\frac{1}{\Gamma(z)} \int_0^\infty t^{z-1} \frac{e^{-t}}{1-e^{-t}}dt.
		\]
		Moreover,
		\[
        \zeta(0) = -\frac{1}{2} \qquad \zeta'(0) = -\frac{1}{2} \log (2\pi).
        \]
   
		We can repeat the proof of Lemma~\ref{l-gue250604yyd} and deduce that for ${\rm Re\,}z>n+1$, 
\begin{equation}\label{e-gue212604272026}
\begin{split}
\widehat{\zeta}(z)=&-(2\pi)^{-n-1}\zeta(z)\int_X\int_{-\delta}^\delta\det\left(\dot{\mathcal{R}}^\phi_x - 2 \eta \dot{\mathcal{L}}_x\right)\\
&\quad\times\left[ \operatorname{Tr}\abs{(\dot{\mathcal{R}}^\phi_x - 2 \eta \dot{\mathcal{L}}_x  )_-}^{-z} - \operatorname{Tr}\abs{(\dot{\mathcal{R}}^\phi_x - 2 \eta \dot{\mathcal{L}}_x  )_+}^{-z} \right]1_{\mathbb R_x(q)}(\eta)d\eta dv_X(x).
\end{split}
\end{equation}

		By \eqref{e-gue212604272026}, we get
		\begin{equation}\label{e-gue212704272026}
			\begin{split}
				\widehat{\zeta}'(0)=&\frac{1}{2} \log (2\pi)(2\pi)^{-n-1} \int_X\int_{-\delta}^\delta \det \left(\dot{\mathcal{R}}^\phi_x - 2 \eta \dot{\mathcal{L}}_x \right)(2q-n)1_{\mathbb R_x(q)}(\eta)d\eta dv_X(x)\\
                &+\frac{1}{2}(2\pi)^{-n-1} \int_X\int_{-\delta}^\delta \det \left(\dot{\mathcal{R}}^\phi_x - 2 \eta \dot{\mathcal{L}}_x \right)\Bigr(-\log(\abs{{\rm det\,}(\dot{\mathcal{R}}^\phi_x - 2 \eta \dot{\mathcal{L}}_x)_-})\\
                &+\log(\abs{{\rm det\,}(\dot{\mathcal{R}}^\phi_x - 2 \eta \dot{\mathcal{L}}_x)_+})\Bigr)1_{\mathbb R_x(q)}(\eta)d\eta dv_X(x).
			\end{split}
		\end{equation}
		By \eqref{e-gue071004262026III}, \eqref{e-gue212504272026} and \eqref{e-gue212704272026}, we get \eqref{e-gue074304262026}.
	\end{proof}

	
	\bibliographystyle{plain}

\end{document}